\documentclass [leqno] {amsart}
  
\usepackage {amssymb}
\usepackage {amsthm}
\usepackage {amscd}
\usepackage{bm}

\usepackage{pstricks}
\usepackage{color}
\usepackage{graphicx}

\usepackage[all]{xy}
\usepackage {hyperref}

\theoremstyle{plain}
\newtheorem{introthm}{Theorem}
\newtheorem{introcorollary}{Corollary}
\newtheorem{proposition}[subsubsection]{Proposition}
\newtheorem{lemma}[subsubsection]{Lemma}
\newtheorem{corollary}[subsubsection]{Corollary}
\newtheorem{thm}[subsubsection]{Theorem}
\newtheorem*{thm*}{Theorem}

\theoremstyle{definition}
\newtheorem{definition}[subsubsection]{Definition}
\newtheorem{notation}[subsubsection]{Notation}

\newtheorem{example}[subsubsection]{Example}

\theoremstyle{remark}
\newtheorem{remark}[subsubsection]{Remark}

\numberwithin{equation}{subsection}
\numberwithin{subsubsection}{subsection}
\newcommand{\eq}[2]{\begin{equation}\label{#1}#2 \end{equation}}
\newcommand{\ml}[2]{\begin{multline}\label{#1}#2 \end{multline}}
\newcommand{\mlnl}[1]{\begin{multline*}#1 \end{multline*}}

\newcommand{\arir}{\ar@{^{(}->}}
\newcommand{\aril}{\ar@{_{(}->}}
\newcommand{\are}{\ar@{>>}}

\newcommand{\xr}[1] {\xrightarrow{#1}}
\newcommand{\xl}[1] {\xleftarrow{#1}}
\newcommand{\lra}{\longrightarrow}
\newcommand{\inj}{\hookrightarrow}

\newcommand{\mc}[1]{\mathcal{#1}}

\newcommand{\codim}{{\rm codim}}

\newcommand{\Hom}{{\rm Hom}}
\newcommand{\sHom}{{\rm \mathcal{H}om}}
\newcommand{\sExt}{{\rm \mathcal{E}xt}}
\newcommand{\Ker}{{\rm Ker}}
\newcommand{\Spec}{{\rm Spec \,}}
\newcommand{\im}{{\rm im}}

\newcommand{\Tr}{{\rm Tr}}

\newcommand{\Trf}{{\rm Trf}}

\newcommand{\id}{{\rm id}}
\newcommand{\ul}{\underline}
\newcommand{\ol}{\overline}
\newcommand{\RS}{\mc{A}}

\newcommand{\Ext}{\mathrm{Ext}}

\newcommand{\drw}{{\rm dRW}}
\newcommand{\cdrw}{{\widehat{\rm dRW}}}

\newcommand{\XP}{{(X,\Phi)}}
\newcommand{\YP}{{(Y,\Psi)}}

\newcommand{\sA}{{\mathcal A}}

\newcommand{\sC}{{\mathcal C}}

\newcommand{\sF}{{\mathcal F}}
\newcommand{\sG}{{\mathcal G}}
\newcommand{\sH}{{\mathcal H}}
\newcommand{\sI}{{\mathcal I}}

\newcommand{\sK}{{\mathcal K}}
\newcommand{\sL}{{\mathcal L}}

\newcommand{\sN}{{\mathcal N}}
\newcommand{\sO}{{\mathcal O}}

\newcommand{\sR}{{\mathcal R}}

\newcommand{\A}{{\mathbb A}}

\newcommand{\C}{{\mathbb C}}

\newcommand{\G}{{\mathbb G}}

\newcommand{\N}{{\mathbb N}}
\renewcommand{\P}{{\mathbb P}}
\newcommand{\Q}{{\mathbb Q}}

\newcommand{\Z}{\mathbb{Z}}

\newcommand{\fU}{\mathfrak{U}}

\newcommand{\fm}{\mathfrak{m}}

\newcommand{\OO}{\sO}

\newcommand{\triples}{(F_*,F^*,T,e)}

\newcommand{\pr}{{\rm pr}}
\newcommand{\GrAb}{\text{\bf{GrAb}}}
\newcommand{\CH}{{\rm CH}}
\newcommand{\hH}{\hat{\sH}}
 
\begin{document}

\title{Hodge-Witt cohomology and Witt-rational singularities}

\begin{abstract}
We prove the vanishing modulo torsion of the higher direct images of the sheaf of Witt vectors
(and the Witt canonical sheaf) for a purely inseparable projective alteration 
between normal finite 
quotients over a perfect field. 
For this, we show that the relative Hodge-Witt cohomology admits an action of 
correspondences. As an application we define Witt-rational singularities 
which form a broader class than rational singularities. In particular, 
finite quotients have Witt-rational singularities. In addition, we prove
that the torsion part of the Witt vector cohomology of a smooth, proper scheme 
is a birational invariant.    
\end{abstract}

\author{Andre Chatzistamatiou and Kay R\"ulling}
\address{Fachbereich Mathematik \\ Universit\"at Duisburg-Essen \\ 45117 Essen, Germany}
\email{a.chatzistamatiou@uni-due.de}
\email{kay.ruelling@uni-due.de}

\thanks{This work has been supported by the SFB/TR 45 ``Periods, moduli spaces and arithmetic of algebraic varieties''}

\maketitle

\tableofcontents

\section*{Introduction}
\addtocontents{toc}{\protect\setcounter{tocdepth}{1}}

An important class of singularities over fields of characteristic zero are the 
rational singularities. For example, quotient singularities and log terminal 
singularities are rational singularities. Over fields with positive 
characteristic the situation is more subtle. The definition of rational
singularities requires resolution of singularities which is not yet available in 
all dimensions. Moreover, quotient singularities are only rational singularities 
under a further tameness condition, but in general fail to be rational singularities.

The purpose of this paper is to define a broader class of singularities in 
positive characteristic, which we call Witt-rational singularities. 
The main idea is that we replace the structure sheaf $\OO_X$ and the 
canonical sheaf $\omega_X$ by the Witt sheaves $W\OO_{X,\Q}$ and $W \omega_{X,\Q}$.
One important difference is that multiplication with $p$ is invertible in 
$W\OO_{X,\Q}$ and $W \omega_{X,\Q}$.
Instead of resolution of singularities we can use alterations. 

Witt-rational singularities have been first introduced by Blickle and Esnault \cite{BE}.
In this paper we use a slightly different and more restrictive definition, which seems
to be more accessible. Conjecturally, our definition agrees with the one of Blickle-Esnault
by using a Grauert-Riemenschneider vanishing theorem for the Witt canonical sheaf $W \omega_{X,\Q}$.
We hope to say more about this in the future.   

\subsection{} 
Let $k$ be a perfect field of positive characteristic. We denote by $W=W(k)$
the ring of Witt vectors and by $K_0={\rm Frac}(W)$ the field of fractions.
For smooth projective $k$-scheme 
the crystalline cohomology $H^*_{\rm crys}(X/W)$ has, by the work of Bloch and Deligne-Illusie, a natural interpretation
as hypercohomology of the de Rham-Witt complex $W\Omega^\bullet_X$,
$$
H^*_{\rm crys}(X/W)\cong H^*(X,W\Omega^\bullet_X).
$$
After inverting $p$, the slope spectral sequence degenerates which yields 
a decomposition 
$$
H^{n}(X/K_0)=\bigoplus_{i+j=n} H^j(X, W\Omega^i_{X})\otimes_W K_0. 
$$
The de Rham-Witt complex is the limit of a pro-complex $(W_n\Omega^\bullet_X)_n$,
and for us $W_n\OO_X$ and $W_n\omega_X=W_n\Omega^{\dim X}_X$ will be most important. 
The sheaf $W_n\OO_X$ is the sheaf of Witt vectors of length $n$, and 
defines a scheme structure $W_nX$ on the topological space $X$. 
The structure map $\pi:X\xr{} \Spec(k)$ induces a morphism 
$W_n(\pi): W_n X\to \Spec W_n(k)$, but $W_n(\pi)$ is almost never flat.  
By the work of Ekedahl (see \cite{EI}) $W_n\omega_X$ equals $W_n(\pi)^! W_n[-\dim X]$, hence $W_n\omega_X$ is a dualizing sheaf for $W_nX$.

The main technical problem in order to define Witt-rational singularities 
is to prove the independence of the chosen alteration. Our approach 
is to use the action of algebraic cycles in a similar way as in \cite{CR}.
For this, we have to extend the work of Gros \cite{Gros} on the de Rham-Witt complex in Theorem \ref{thm1} below. 

For a $k$-scheme $S$ we denote by $\mc{C_S}$ the category whose objects are  
$S$-schemes which are smooth and quasi-projective over $k$. For two objects $f:X\to S$ and 
$g: Y\to S$ in $\mc{C}_S$, the 
morphisms $\Hom_{\mc{S}}(f:X\to S, g:Y\to S)$ are defined by $\varinjlim_Z \CH(Z)$, where the limit is taken over all proper correspondences over $S$ between $X$ and $Y$, i.e.~closed subschemes $Z\subset X\times_S Y$ such that the projection to $Y$ is proper ($\CH(Z)=\oplus_i \CH_i(Z)$ denotes the Chow group). The composition of two morphisms is defined using Fulton's refined intersection product.
The following theorem on the action of proper correspondences on relative Hodge-Witt cohomology is the main technical tool of the article.
\begin{introthm}[cf.~Proposition~\ref{proposition-C_S-to-dRW_S}]\label{thm1}
There is a functor
\begin{align*}
\hat{\sH}(?/S): \mc{C_S}&\to (W\sO_S-\text{modules}),\\
\hat{\sH}(f: X\to S /S)&= \bigoplus_{i,j} R^if_*W\Omega^j_X,
\end{align*}
with the following properties.

If $h: X\to Y$ is an $S$-morphism between two smooth $k$-schemes and $\Gamma_h^t\subset Y\times_S X$ denotes the transpose of its graph, then
$\hat{\sH}([\Gamma_h^t]/S)$ is the natural pull-back.
If in addition $h$ is projective  then $\hat{\sH}([\Gamma_h]/S)$  
is the pushforward defined by Gros in \cite{Gros} using Ekedahl's duality theory \cite{EI}.

For a morphism $\alpha\in \Hom_{\mc{C}_S}(X/S, Y/S)$ in $\mc{C}_S$, the map $\hat{\sH}(\alpha/S)$ is compatible with Frobenius, Verschiebung and the differential.
\end{introthm}

 
\subsection{} We say that an integral normal  $k$-scheme $X$ is a \emph{finite quotient} if there exists a finite and surjective morphism from a smooth $k$-scheme $Y\to X$ (e.g. $X=Y/G$ for some finite group $G$ acting on $Y$.)
We say that a normal integral scheme $X$ is  a {\em topological finite quotient} if there exists a finite, surjective and purely inseparable morphism  $u: X\to X'$, where $X'$ is a finite quotient.
The morphism  $u$ is in fact a universal homeomorphism. Finally we say that a morphism $f:X\to Y$ between two integral $k$-schemes is a {\em quasi-resolution of} $Y$  if
$X$ is a topologically finite quotient and the morphism $f$ is surjective, projective, generically finite and purely inseparable.
(In characteristic zero these conditions imply that $X$ is a finite quotient and $f$ is projective and birational.)
By a result of de Jong (see \cite{deJong1}, \cite{deJong2}) quasi-resolutions always exist. When working with $\Q$-coefficients the notion of quasi-resolutions
suffices to define an analog of rational singularities. This follows from the following theorem.
\begin{introthm}[Theorem~\ref{vanishing-HDI}]\label{thm2}
Let $Y$ be a topological finite quotient and $f: X\to Y$ a quasi-resolution. Then
\[Rf_*W\sO_{X,\Q}\cong W\sO_{Y,\Q}.\] 
\end{introthm}
If $X$ and $Y$ are smooth and $f$ is birational, this is a direct consequence of Theorem \ref{thm1} and the vanishing Lemmas \ref{lemma-van1} and \ref{lemma-van2}. Indeed, in $\CH(X\times_Y X)$ the diagonal $\Delta_X\subset X\times_Y X$ can be written as $[\Gamma^t_f]\circ [\Gamma_f]+ E$, where $E$ is a cycle whose projections  
to $X$ have at least codimension $\ge 1$. Thus $E$ acts as zero on the $W\sO$ part and hence $[\Gamma^t_f]\circ [\Gamma_f]$ acts as the identity on $R^if_*W\sO_{X,\Q}$,
but it factors over $0$ for $i>0$; this will prove the theorem in case $X$ and $Y$ is smooth. 
Because the Frobenius is invertible when working with $\Q$-coefficients we can neglect all
purely inseparable phenomena. Therefore the main point in the general case is to realize the higher direct images of $R^if_*W\sO_{X,\Q}$  (and also for $Y$) as 
certain direct factors in the relative cohomology of smooth schemes, which is possible since $X$ and $Y$ are topological finite quotients.

\subsection{} Before explaining our definition of Witt-rational singularities we need to introduce some  notations.
If $X$ is a $k$-scheme of pure dimension $d$ and with structure map $\pi: X\to \Spec k$, then we define the Witt canonical sheaf of length $n$ by
\[W_n\omega_X:= H^{-d}(W_n(\pi)^! W_n).\]
It follows from the duality theory developed by Ekedahl in \cite{EI}, that these sheaves form a projective system $W_\bullet\omega_X$ with Frobenius, Verschiebung and Cartier morphisms.
Further properties are (the first two are due to Ekedahl, see \cite{EI} and Proposition \ref{properties-Witt-canonical})
\begin{enumerate}
 \item If $X$ is smooth, then $W_\bullet\omega_X\cong W_\bullet\Omega^d_X$.
 \item If $X$ is Cohen-Macaulay, then $W_n\omega_X[d]\cong W_n(\pi)^!W_n$, in particular $W_n\omega_X$ is dualizing.
\item If $f: X\to Y$ is a proper morphism between $k$-schemes of the same pure dimension, then there is a $W_\bullet\sO_Y$-linear morphism 
        \[f_*: f_*W_\bullet\omega_X\to W_\bullet\omega_Y,\]
       which is compatible with composition and localization.
\end{enumerate}
We define $W\omega_X:=\varprojlim W_\bullet\omega_X$. 

We say that an integral $k$-scheme $S$ has {\em Witt-rational singularities} (Definition~\ref{definition-Witt-rational}) if for any quasi-resolution 
 $f: X\to S$ the following conditions are satisfied:
\begin{enumerate}
 \item  $f^*: W\sO_{S,\Q}\xr{\simeq} f_*W\sO_{X,\Q}$ is an isomorphism,
 \item $R^if_* W\sO_{X,\Q}=0$, for all $i\ge 1$,
 \item $R^if_* W\omega_{X,\Q}=0$, for all $i\ge 1$.
\end{enumerate}
In case only the first two properties are satisfied we say that $S$ has {\em $W\sO$-rational singularities}. 
Condition (1) is satisfied provided that $S$ is normal.

Our main example for varieties with Witt-rational singularities are topologically finite quotients,
because the vanishing property in Theorem \ref{thm2} also holds  for $W\omega$. 
\begin{introthm}[Corollary~\ref{finite-quotients-are Witt-rational}]\label{thm3}
Topological finite quotient have Witt-rational singularities.
\end{introthm}

A particular case are normalizations of smooth schemes $X$ in a purely inseparable 
finite field extension of the function field of $X$. More generally, if 
$u:Y\xr{} X$ is a universal homeomorphism between normal schemes then $Y$ has Witt-rational 
singularities if and only if $X$ has Witt-rational singularities (Proposition~\ref{proposition-Witt-rational-universal-homeomorphism}).

Every scheme with rational singularities has Witt-rational singularities, but
varieties with Witt-rational singularities form a broader class. For example, 
finite quotients may fail to be Cohen-Macaulay and thus are in general not rational singularities.

A different definition of Witt-rational singularities has been introduced by 
Blickle and Esnault as follows. 
Let $S$ be an integral $k$-scheme and $f:X\xr{} S$ a generically \'etale alteration with $X$ a smooth $k$-scheme.
We say that $S$ has \emph{BE-Witt-rational singularities} if the 
natural morphism 
$$
W\OO_{S,\Q} \xr{} Rf_*W\OO_{X,\Q}
$$
admits a splitting in the derived category of sheaves of abelian groups on $X$. 
A scheme with Witt-rational singularities in our sense has BE-Witt-rational singularities (Proposition~\ref{comparison}). 
We conjecture that the converse is also true.

The existence of quasi-resolutions implies the following corollary.
\begin{introcorollary}[Corollary~\ref{independence}]\label{cor-1}
 Let $S$ be a $k$-scheme and $X$ and $Y$ two integral $S$-schemes. Suppose that 
there exists a commutative diagram
\[\xymatrix@-1pc{            &  Z\ar[dl]_{\pi_X}\ar[dr]^{\pi_Y}  &  \\
               X\ar[dr]_{f} &                                   &  Y\ar[dl]^{g}  \\
                        &   S,                               &  }\]
with  $\pi_X$ and $\pi_Y$ quasi-resolutions. Suppose that $X,Y$ have Witt-rational singularities.
Then we get induced isomorphisms in $D^b(W\OO_S)$
\begin{equation*} 
\tag{1}\label{intro-independence1} Rf_*W\sO_{X,\Q}\cong Rg_*W\sO_{Y,\Q}, \quad Rf_*W\omega_{X,\Q}\cong Rg_*W\omega_{Y,\Q}.\end{equation*}
The isomorphisms are compatible with the action of the Frobenius and the Verschiebung. 
\end{introcorollary}

If $f:X\xr{} S$ and $g:Y\xr{} S$ are quasi-resolutions then the isomorphisms 
in \eqref{intro-independence1} are independent of the choice of $Z$ (Corollary~\ref{independence3}). In this 
way we obtain natural complexes 
\[\mc{WS}_{0,S}:=Rf_*W\sO_{X,\Q}, \quad \mc{WS}_{\dim(S), S}:= Rf_*W\omega_{S,\Q},\]
(Definition~\ref{canonical-complex-for-singularities}).


\subsection{} 
By using the work of Berthelot-Bloch-Esnault Corollary \ref{cor-1} yields 
congruences for the number of rational points over finite fields. 

\begin{introcorollary}[Corollary~\ref{points}]\label{intro-points} 
 Let $S=\Spec k$ be a finite field. Let $X$ and $Y$ be as in 
Corollary \ref{cor-1}, and suppose that $X,Y$ are proper. 
Then for any finite field extension $k'$ of $k$ we have
\[|X(k')|\equiv |Y(k')|\quad \text{mod } |k'|.\]
\end{introcorollary} 
If $X,Y$ are smooth this is a theorem due to Ekedahl \cite{E83}.

For a normal integral scheme $S$ with an isolated singularity $s\in S$ we can give 
a criterion for the $W\OO$-rationality of $S$, provided that a resolution
of singularities $f:X\xr{} S$ exists such that $f:f^{-1}(S\backslash \{s\})\xr{} S\backslash \{s\}$
is an isomorphism; we denote by $E:=f^{-1}(s)$ the fibre over $s$.
Then $S$ has $W\OO$-rational singularities if and only if
\begin{equation*}\label{eq-van}\tag{2}
H^i(E,W\OO_{E,\Q})=0 \quad \text{for all $i>0$,}
\end{equation*}
(Corollary~\ref{corollary-direct-image-isolated-singularity}).
This implies that a normal surface has $W\OO$-rational singularities if and
only if the exceptional divisor is a tree of smooth rational curves. 

For cones $C$ of smooth projective schemes $X$, we obtain that $C$ has 
$W\OO$-rational singularities if and only if $H^i(X,W\OO_{X,\Q})=0$ for $i>0$.
We can show that $C$ has Witt-rational singularities provided that Kodaira vanishing 
holds for $X$ (Section \ref{section-on-cones}). We expect that this assumption can be dropped; in general, 
a Grauert-Riemenschneider type vanishing theorem for $W\omega$ should imply 
that $W\OO$-rationality is equivalent to Witt-rationality. 
 
Over a finite field $k$ we use a weight argument to refine the criterion \eqref{eq-van}  if $E$ is a strict normal crossing divisor. 
Let $E_i$ be the irreducible components 
of $E$, via the restriction maps we obtain for all $t\geq 0$ a complex $C_t(E)$:
$$
\underset{\deg=0}{\underbrace{\bigoplus_{\imath_0} H^t(E_{\imath_0},W\OO_{E_{\imath_0},\Q})}}\xr{} \bigoplus_{\imath_0<\imath_1}
H^t(E_{\imath_0}\cap E_{\imath_1},W\OO_{E_{\imath_0}\cap E_{\imath_1},\Q})\xr{} \dots
$$

\begin{introthm}[Theorem~\ref{thm-criterion-WO-rational}]\label{thm-intro-direct-image-isolated-singularity}
Let $k$ be a finite field. In the above situation, $S$ has $W\OO$-rational singularities if and only 
if 
$$
H^i(C_t(E))=0 \quad \text{for all $(i,t)\neq (0,0)$.}
$$
\end{introthm}

Theorem \ref{thm-intro-direct-image-isolated-singularity} is inspired by 
the results of Kerz-Saito \cite[Theorem~8.2]{KS} on the weight homology 
of the exceptional divisor.

For morphisms with generically smooth fibre with trivial Chow group of zero cycles we can 
show the following vanishing theorem. 

\begin{introthm}[Theorem \ref{thm-rational-connected-fibres}] \label{thm-intro-rational-connected-fibres}
Let $X$ be an integral scheme with Witt-rational singularities.
Let $f:X\xr{} Y$ be a projective morphism to an integral, normal and quasi-projective scheme $Y$.  
We denote by $\eta$ the generic point of $Y$, and $X_{\eta}$ denotes the generic
fibre of $f$. Suppose that $X_{\eta}$ is smooth and for every field extension 
$L\supset k(\eta)$ the degree map 
$$
\CH_0(X_{\eta}\times_{k(\eta)}L)\otimes_{\Z} \Q\xr{} \Q
$$
is an isomorphism.
Then, for all $i>0$, 
$$
R^if_*W\OO_{X,\Q}\cong H^i(\mc{WS}_{0,Y}) , \quad  R^if_*W\omega_{X,\Q}\cong H^i(\mc{WS}_{\dim(Y),Y}). 
$$
In particular, if $Y$ has Witt-rational singularities then 
$$
R^if_*W\OO_{X,\Q}=0 , \quad  R^if_*W\omega_{X,\Q}=0, \quad \text{for all $i>0$.} 
$$
\end{introthm}

\subsection{} 
For smooth schemes we can show the following result which takes the torsion 
into account.

\begin{introthm}[Theorem \ref{torison1}] \label{intro-torison1}
 Let $S$ be a $k$-scheme. Let $f: X\to S$ and $g: Y\to S$ be two $S$-schemes which are integral and smooth over $k$ and have dimension $N$.
Assume $X$ and $Y$ are properly birational over $S$, i.e. there exists a closed integral subscheme $Z\subset X\times_S Y$, such that the projections
  $Z\to X$ and $Z\to Y$ are proper and birational.
There are isomorphisms in $D^b(S,W(k))$:
\[Rf_*W\OO_X\cong Rg_*W\OO_Y,\quad Rf_*W\Omega_X^N\cong Rg_*W\Omega^N_Y.\]
Taking cohomology we obtain isomorphisms of $W\sO_S$-modules which are compatible with Frobenius and Verschiebung: 
\[R^if_*W\sO_X\cong R^ig_*W\sO_Y, \quad R^i f_*W\Omega^N_X\cong R^i g_* W\Omega^N_Y,\quad \text{for all }i\ge 0.\] 
\end{introthm}

If $X$ and $Y$ are tame finite quotients and there exists a proper and birational 
morphism $h:X\xr{} Y$ then a similar statement holds (see Theorem \ref{torsion2}).

If $X$ and $Y$ are two smooth and proper $k$-schemes, which are birational and of pure dimension $N$. Then we obtain isomorphisms of $W(k)[F, V]$-modules
\[H^i(X, W\sO_X)\cong H^i(Y, W\sO_Y), \quad H^i(X, W\Omega^N_X)\cong H^i(Y, W\Omega^N_Y), \quad \text{for all } i\ge 0.\]

Modulo torsion the statement for $W\OO$ is a theorem due to Ekedahl.

\subsection{} We give a brief overview of the content of each section. In Section 1 we introduce the category $\drw_X$ of de Rham-Witt systems on a $k$-scheme $X$. In the language of Ekedahl \cite{EI} an object in $\drw_X$
is both, a direct and an inverse de Rham-Witt system at the same time. 
Furthermore, we introduce  the derived pushforward, derived cohomology with supports and $R\varprojlim$ on $D^b(\drw_X)$.
  We recall the definition of Witt-dualizing systems from \cite{EI}  in \ref{1.4}, and some facts about
residual complexes in \ref{section-Residual-complexes-and-traces}. In particular, we observe that if $f: X\to Y$ is an morphism between $k$-schemes, which is proper along a family of supports $\Phi$ on $X$, then
for any residual complex $K$ on $Y$ the trace morphism $f_*f^\Delta K\to K$, which always exists as a map of graded sheaves, induces a morphism of complexes $f_*\ul{\Gamma}_\Phi f^\Delta K\to K$.
In \ref{section-Witt-residual-complexes} we show that for any $\pi: X\to \Spec k$ the residual complexes $W_n\pi^\Delta W_n(k)$ form a projective system $K_X$, which is term-wise a Witt-dualizing system.
In \ref{section-The-dualizing-functor} we define the functor $D_X=\sHom(-, K_X)$ on $D(\drw_{X,{\rm qc}})^o$. (It is only defined on complexes of quasi-coherent de Rham-Witt systems.) 
In \ref{section-E-results} we recall the results of Ekedahl in the smooth case relating $K_X$ to $W_\bullet\Omega^{\dim X}_X$, and in \ref{section-The-trace-morphism-for-a-regular-closed-immersion} we calculate the trace morphism for a regular closed immersion.
A similar description is given in \cite{Gros}, but it refers to work in progress by Ekedahl, which we could not find in the literature, therefore we give another argument.

In Section 2 we introduce relative Hodge-Witt cohomology with supports on smooth and quasi-projective $k$-schemes, which are defined over some base scheme $S$. We define a pullback for arbitrary morphisms
and using the trace map from Section 1 also a pushforward for morphisms which are proper along a family of supports. Then in \ref{section-Compatibility} we give an explicit description of the pushforward in the case of a regular closed immersion and 
also for the projection $\P^n_X\to X$, where $X$ is a smooth scheme $X$. From this description we deduce the expected compatibility between pushforward and pullback with respect to maps in a certain cartesian diagram.

In Section 3 we collect and prove the remaining facts, which we need to show that $(X,\Phi)\mapsto \oplus_{i,j} H^i_\Phi(X, W\Omega^j_X)$ is a weak cohomology theory with supports in the sense of \cite{CR}.
In particular, we need the cycle class constructed by Gros in \cite{Gros}.
From this we deduce Theorem \ref{thm1} above. In \ref{section-vanishing-results} we prove the two vanishing Lemmas, which give a criterion for certain correspondences to act as zero on certain parts of the Hodge-Witt cohomology.
In \ref{section-dRW-mod-torsion} we introduce the notation $\drw_{X,\Q}$, which is the $\Q$-linearization of $\drw_X$. In general, 
for $M\in \drw_X$ the notation 
$M_\Q$ means the image of $M$ in $\drw_{X,\Q}$ (which is not the same as $M\otimes_\Z\Q$).

In Section 4 we introduce the Witt canonical system $W_\bullet\omega_X$ for a pure-dimensional $k$-scheme $X$ and prove some of its properties. Moreover we show in \ref{section-top-finite-quot} that the cohomology  of $W\sO$ and $W\omega$ 
for a topological finite quotient is a direct summand of the Hodge-Witt cohomology of a certain smooth scheme. Then we prove Theorem \ref{thm2}
and define Witt rational singularities. It follows some elaboration on this notion, in particular the Theorems \ref{thm3}, \ref{thm-intro-direct-image-isolated-singularity}, \ref{thm-intro-rational-connected-fibres}.

Finally in Section 5 we prove some results on torsion, as in Theorem \ref{intro-torison1}. In order to do this, we show that a correspondence actually gives rise to a morphism in the derived category
of modules over the Cartier-Dieudonn\'e-Raynaud ring and then use Ekedahl's Nakayama Lemma to deduce the statement from \cite{CR}.

We advise the reader who is mostly interested in the geometric application to start for a first time reading with Section \ref{1.1} and \ref{drws} to get some basic notations and then jump directly to Section 4.

\subsection{Notation and general conventions}
We are working over a perfect ground field $k$ of characteristic $p>0$. We denote by $W_n=W_n(k)$ the ring of Witt vectors of length $n$ over $k$ and by $W=W(k)$ the ring of infinite Witt vectors.
By a $k$-scheme we always mean a scheme $X$, which is separated and of finite type over $k$. If $X$ and $Y$ are $k$-schemes, then a morphism $X\to Y$ is always assumed to be a $k$-morphism.

\addtocontents{toc}{\protect\setcounter{tocdepth}{2}}
\section{De Rham-Witt systems after Ekedahl} \label{section-De-Rham-Witt-systems-after-Ekedahl}

\subsection{Witt schemes}\label{1.1} For the following facts see e.g. \cite[0.1.5]{IlDRW}, \cite[Appendix A]{LZ}.
Let $X$ be a $k$-scheme. For $n\ge 1$, we denote 
\[W_n X=(|X|, W_n\sO_X)=\Spec W_n\sO_X,\] 
where $W_n\sO_X$ is the sheaf of rings of Witt vectors of length $n$. 
This construction yields a functor from the category of $k$-schemes to the category of separated, finite type $W_n$-schemes. 
If $f: X\to Y$ is a separated (resp. finite type, proper or \'etale) morphism of $k$-schemes, then $W_nf: W_nX\to W_nY$ is a separated (resp. finite type, proper or \'etale) morphism
of $W_n$-schemes. If $f$ is an open (resp. closed) immersion, so is $W_nf$. 
We denote by $ i_n : W_{n-1}X \inj W_nX$ (or sometimes by $i_{n, X}$) the nilimmersion induced by the restriction $W_n\sO_X\to W_{n-1}\sO_X$.
We will write $\pi: W_n\sO_X\to i_{n*}W_{n-1}\sO_X$ instead of $i_n^*$. 
The absolute Frobenius on $X$ is denoted by $F_{X}: X\to X$. The morphism $W_n(F_X) : W_nX\to W_nX$ is finite for all $n$.
With this notation the Frobenius and Verschiebung morphisms on the Witt vectors become morphisms of $W_n\sO_X$-modules 
\[ F=W_n(F_X)^*\circ \pi: W_n\sO_X\to (W_n(F_X)i_n)_*W_{n-1}\sO_X,\]
\[ V: (W_n(F_X)i_n)_*W_{n-1}\sO_X\to W_n\sO_X.\]
Further ``lift and multiply by $p$ '' induces a morphism of $W_n\sO_X$-modules
\[\ul{p}: i_{n*}W_{n-1}\sO_X\to W_n\sO_X.\]
If $f:X\to Y$ is a morphism of $k$-schemes, then we have $W_n(f)  i_{n,X}=i_{n,Y}W_{n-1} f$ and $W_n(f) W_n(F_X)=W_n(F_Y)W_n(f)$. 
If $f: X\to Y$ is {\em \'etale}, then the following diagrams are {\em cartesian}:
\eq{1.1.1}{
\xymatrix{W_{n-1} X\arir[r]^{i_n}\ar[d]_{W_{n-1}f}  &        W_nX\ar[d]^{W_nf}\\
                    W_{n-1}Y\arir[r]^{i_n}                     &  W_nY, }\quad\quad\quad
\xymatrix{W_n X\ar[r]^{W_n(F_X)}\ar[d]_{W_nf}  &        W_nX\ar[d]^{W_nf}\\
                    W_nY\ar[r]^{W_n(F_Y)}        &  W_nY. }
}

\subsection{De Rham-Witt systems}\label{drws}

\begin{definition}
For an integer $n\geq 1$ we denote by $\mc{C}_n$ the category of 
$\Z$-graded $W_n\OO_X$-modules on $X$. We define 
$$
\mc{C}_{\N}:=\prod_{n\in \Z, n\geq 1} \mc{C}_n.
$$
\end{definition}


For an object $M\in \mc{C}_{\N}$ and $n\geq 1$ we denote by $M_n$ the $n$-th component. An object  $M$ in $\sC_\N$ is (quasi-)coherent, if all $M_n$ are (quasi-)coherent $W_n(\sO_X)$-modules.
We denote by $\sC_{\N, {\rm qc}}$ (resp. $\sC_{\N,{\rm c}}$) the full subcategory of (quasi-)coherent objects of $\sC_\N$.
There are two natural endo-functors: 
\begin{equation*}
\begin{split}
i_*: \mc{C}_{\N} &\xr{} \mc{C}_{\N} \\
(i_*M)_n&:=\begin{cases}i_{n*}M_{n-1} &\text{if $n>1$,} \\ 0 &\text{if $n=1$,}\end{cases}\\
\sigma: \mc{C}_{\N} &\xr{} \mc{C}_{\N} \\
(\sigma_*M)_n&:= W_n(F_X)_*M_n
\end{split}
\end{equation*}
The two functors commute 
\begin{equation}\label{isigmacomm}
\sigma_*i_*=i_*\sigma_*,
\end{equation}
since $W_n(F_X)_*i_{n*}=i_{n*}W_{n-1}(F_X)_*$. 

We will also need the following functor:
\begin{equation*} 
\begin{split}
\Sigma_*:\mc{C}_{\N}&\xr{} \mc{C}_{\N} \\
(\Sigma_*M)_n&:=W_n(F_X)^n_*M_n.
\end{split}
\end{equation*}
We have the equalities 
\begin{equation}\label{Sigmacomm}
\sigma_*i_*\Sigma_*=\Sigma_*i_*, \quad \sigma_*\Sigma_* = \Sigma_*\sigma_*.
\end{equation}

Furthermore, since the components of $M\in \mc{C}_{\N}$ are $\Z$-graded we can define
for all $i\in \Z$ the shift functor 
\begin{equation}
\label{firstshift}
M(i)_n:=M_n(i).
\end{equation}
The shift functor commutes in an obvious way with $i_*,\sigma_*,\Sigma_*$.

\begin{definition}\label{definition-W}
A \emph{graded Witt system} $(M,F,V,\pi,\ul{p})$ on $X$ is an object $M$ in $\mc{C}_{\N}$ equipped with 
morphisms in $\mc{C}_{\N}$:
$$
F:M\xr{} \sigma_*i_*M, \quad V: \sigma_*i_*M \xr{} M, \quad \pi: M \xr{} i_*M, \quad \ul{p}:i_*M\xr{} M,
$$ 
such that 
\begin{itemize}
\item[(a)] $V\circ F$ is multiplication with $p$,
\item[(b)] $F\circ V$ is multiplication with $p$,
\item[(c)] $\sigma_*i_*(\pi)\circ F= i_*(F) \circ \pi$,
\item[(d)] $\pi\circ V=i_*(V\circ \sigma_*(\pi))$,
\item[(e)] $i_*(\sigma_*(\ul{p})\circ F)=F\circ \ul{p}$,
\item[(f)] $V\circ \sigma_*i_*(\ul{p}) = \ul{p}\circ i_*(V),$
\item[(g)] $i_*(\ul{p}\circ \pi)=\pi \circ \ul{p}$.
\end{itemize} 
Graded Witt systems form in the obvious way a category which we denote by $W_X$.
It is straightforward to check that $W_X$ is abelian.
\end{definition}
We have an obvious forgetful functor $W_X\xr{} \mc{C}_{\N}$. We say that 
$(M,F,V,\pi,\ul{p})$ is {\em (quasi-)coherent} if $M_n$ is (quasi-)coherent for every $n$.

\begin{remark}
One should memorise (c) as ``$\pi\circ F=F\circ \pi$'', (d) as ``$\pi\circ V=V\circ \pi$'',
(e) as ``$\ul{p}\circ F=F\circ \ul{p}$'', (f) as ``$V\circ \ul{p}=\ul{p}\circ V$'', and
(g) as ``$\ul{p}\circ \pi=\pi \circ \ul{p}$''.  
\end{remark}


\begin{definition}\label{definition-drw}
A {\em de Rham-Witt system} $(M,d)$ is a graded Witt system $M$ together with a morphism 
in $W_X$:
$$
d:\Sigma_*M \xr{} \Sigma_*M(1), 
$$
such that the following conditions are satisfied: 
\begin{itemize}
\item [(a)] $\Sigma_*F(1) \circ d\circ \Sigma_*V= \sigma_*^2i_*d$ (we used \ref{Sigmacomm}), 
\item [(b)] $\Sigma_*\pi(1) \circ d= \sigma_*i_*d \circ  \Sigma_*(\pi)$ (we used \ref{Sigmacomm}),
\item [(c)] $d\circ \Sigma_*(\ul{p})= \Sigma(\ul{p})\circ \sigma_*i_*d$ (again, we used \ref{Sigmacomm}).
\item [(d)] $d(1)\circ d=0$. 
\end{itemize}
De Rham Witt systems form in the obvious way a category which we denote by $\drw_X$.
We say that a de Rham-Witt system is {\em (quasi-)coherent} if the underlying graded Witt system is.
We denote the category of (quasi-)coherent de Rham-Witt systems by $\drw_{X, {\rm qc}}$ (resp. $\drw_{X, {\rm c}}$).
It is straightforward to check that $\drw_X$, $\drw_{X,{\rm qc}}$ and $\drw_{X, {\rm c}}$ are abelian. 
We denote by $D^+(\drw_X)$, $D^+(\drw_{X,{\rm qc}})$ and $D^+(\drw_{X, {\rm c}})$ the corresponding derived categories of bounded below complexes.
\end{definition}

\begin{remark}
One should memorise (a) as ``$F\circ d \circ V=d$'', (b) as ``$\pi \circ d=d\circ \pi$'', and
(c) as ``$d\circ \ul{p}=\ul{p}\circ d$''.
\end{remark}

\begin{definition}\label{definition-cdRW}
 A \emph{de Rham-Witt module} $(M, F, V, d )$ is a graded $W\sO_X$-module $M$ together with morphisms of $W\sO_X$-modules
$$
F:M\xr{} W(F_X)_*M, \quad V: W(F_X)_*M \xr{} M
$$
and a morphism of $W(k)$-modules
\[d: M\to M(1)\]
such that
\begin{itemize}
 \item [(a)] $F\circ V$ is multiplication with $p$,
 \item [(b)] $V\circ F$ is multiplication with $p$,
 \item [(c)] $F\circ d\circ V=d$,
 \item [(d)] $d(1)\circ d=0$.
\end{itemize}
De Rham-Witt modules form in the obvious way a category which we denote by $\cdrw_X$.
It is straightforward to check that $\cdrw_X$ is abelian. We denote by $D^+(\cdrw_X)$ the derived category
of bounded below complexes of de Rham-Witt modules.

\end{definition}

\begin{example}\label{1.2.4}
Let $X$ be a $k$-scheme
\begin{enumerate}
\item The sheaves of Witt vectors of finite length on $X$  define a coherent graded Witt system
     \[W_\bullet\sO_X=(\{W_n\sO_X\}_{n\ge 1}, \pi, F, V, \ul{p}),\]
      which is concentrated in degree 0.
       If $X=\Spec k$, we simply write $W_\bullet$ instead of $W_\bullet k$.
\item The de Rham Witt complex of Bloch-Deligne-Illusie $W_\bullet\Omega_X$ is a coherent de Rham-Witt system (see \cite{IlDRW}) and 
         $W\Omega_X=\varprojlim_n W_n\Omega_X$ is a de Rham-Witt module.
\item Let $M$ be a de Rham-Witt system on $X$ and $i\in \Z$. Then we define 
       \[M(i):=(\{M_n(i)\}_{n\ge 1}, \pi_{M}, F_{M} , V_{M}, (-1)^i d_{M}, \ul{p}_{M} )\in \drw_X.\]
\end{enumerate}
\end{example}

\subsection{Direct image, inverse image and inverse limit} \label{direct-inverse-image} 

\subsubsection{} 
Let $f: X\to Y$ be a morphism between $k$-schemes. We get an induced 
functor 
$$
f_*:\mc{C}_{\N,X}\xr{} \mc{C}_{\N,Y}, \quad (M_n)\mapsto (W_n(f)_*M_n)
$$
which commutes in the obvious way with $i_*,\sigma,\Sigma_*$. We thus obtain 
a functor 
$$
f_*:\drw_X\xr{} \drw_Y.
$$

\subsubsection{}
Let $f: X\to Y$ be an \emph{\'etale} morphism between $k$-schemes. We get an induced 
functor 
$$
f^*:\mc{C}_{\N,Y}\xr{} \mc{C}_{\N,X},\quad (M_n)\mapsto (W_n(f)^*M_n)
$$
which by \eqref{1.1.1} commutes  with $i_*,\sigma,\Sigma_*$. We thus obtain 
a functor 
$$
f^*:\drw_Y\xr{} \drw_X.
$$

\subsubsection{}
Let $(M, F, V, \pi, \ul{p}, d)$ be a de Rham Witt system. Then $(M,\pi)$ forms naturally a projective system of $W\sO_X$-modules,
$F$ and $V$ induce morphisms of projective systems of $W\sO_X$-modules $F: (M,\pi)\to (W(F_X)_*M, W(F_X)_*\pi)$, 
$V:(W(F_X)_*M,W(F_X)_*\pi)\to (M,\pi)$ and induces a morphism of projective systems of $W(k)$-modules $d:(M,\pi)\to (M(1),\pi(1))$.
We thus obtain a functor
\[\varprojlim: \drw_X\to \cdrw_X.\]

\subsection{Global sections with support}
\begin{definition}\label{def:famsupp}
A family of supports $\Phi$ on $X$ is a non-empty set of closed subsets of $X$ such that
the following holds:
\begin{itemize}
\item[(i)] The union of two elements in $\Phi$ is contained in $\Phi$.
\item[(ii)] Every closed subset of an element in $\Phi$ is contained in $\Phi$. 
\end{itemize}  
\end{definition}

Let $A$ be any set of closed subsets of $X$. The smallest family of supports 
$\Phi_A$ which contains $A$ is given by
\begin{equation}\label{Phiset}
\Phi_A:=\{\bigcup_{i=1}^n Z'_i\,;\, Z'_i \underset{\text{closed}}{\subset} Z_i\in A\}.
\end{equation}
For a closed subset $Z\subset X$ we write $\Phi_Z$ for $\Phi_{\{Z\}}$.

\begin{notation}
Let $f:X\xr{} Y$ be a morphism of schemes and $\Phi$ resp. $\Psi$ a family of 
supports of $X$ resp. $Y$.
\begin{enumerate}
\item We denote by $f^{-1}(\Psi)$ the smallest family of supports on $X$ which contains 
      $\{f^{-1}(Z);Z\in \Psi\}$.
\item We say that $f\mid \Phi$ is proper if $f\mid Z$ is proper for every $Z\in \Phi$.
If $f\mid\Phi$ is proper then $f(\Phi)$ is a family of supports on $Y$. 
\item If $\Phi_1,\Phi_2$ are two families of supports then $\Phi_1\cap \Phi_2$ is a family of supports.
\item If $\Phi$ resp. $\Psi$ is a family of supports of $X$ resp. $Y$ then we denote 
by $\Phi\times \Psi$ the smallest family of supports on $X\times_k Y$ which contains
$\{Z_1\times Z_2; Z_1\in \Phi, Z_2\in \Psi\}$.
\end{enumerate}
\end{notation}

\subsubsection{} 
Let $\Phi$ be a family of supports on $X$. There is a well-known functor
$$
\ul{\Gamma}_{\Phi}:\mc{C}_{\N,X} \xr{} \mc{C}_{\N,X}, \quad (M_n)\mapsto (\ul{\Gamma}_{\Phi}(M_n)).
$$
Since $\ul{\Gamma}_{\Phi}$ commutes in the obvious way with $i_*,\sigma,\Sigma_*$ we obtain 
$$
\ul{\Gamma}_{\Phi}:\drw_X \xr{} \drw_X.
$$
For a closed subset $Z\subset X$  we also write $\ul{\Gamma}_Z$ instead of $\ul{\Gamma}_{\Phi_Z}$.

If $f: X\to Y$ is a morphism and $\Psi$ a family of supports on $Y$, then
\eq{supp-push}{\ul{\Gamma}_{\Psi}f_*=f_*\ul{\Gamma}_{f^{-1}(\Psi)}.} 

If $f:X\to Y$ is a morphism and $\Phi$ is a family of supports on $X$, then we define
\begin{align}
f_\Phi:= f_*\circ \ul{\Gamma}_\Phi&: \drw_X\to \drw_Y, \label{fPhi} \\
\hat{f}_\Phi:= \varprojlim\circ f_\Phi&: \drw_X\to \cdrw_Y \label{fPhihat}.
\end{align}
Notice that if $\Phi=\Phi_Z$, with $Z$ a closed subset of $X$, then
\eq{Phihatrel}{\hat{f}_{\Phi_Z}= f_*\circ \Gamma_{\Phi_Z}\circ \varprojlim.}
This relation does not hold for arbitrary families of support on $X$.

\subsection{Derived functors}
\begin{lemma}\label{vanishing-of-Rlim} 
Let $(X,\OO_X)$ be a ringed space and $E=(E_n)$ a projective system of 
$\sO_X$-modules (indexed by integers $n\geq 1$). Let $\mc{B}$ be a basis of the topology 
of $X$. We consider the following two conditions:
\begin{enumerate}
 \item[a)] For all $U\in \mc{B}$, $H^i(U, E_n)=0$ for all $i,n \ge 1$.
 \item[b)] For all $U\in \mc{B}$, the projective system $(H^0(U,E_n))_{n\ge 1}$ satisfies the Mittag-Leffler condition.
\end{enumerate}
Then
\begin{enumerate}
 \item  If $E$ satisfies condition a), then $R^i\varprojlim_n E_n=0$, for all $i\ge 2$.
\item If $E$ satisfies the conditions a) and b), then $R^i\varprojlim_n E_n= 0$, for all $i\ge 1$, i.e. $E$ is $\varprojlim$-acyclic.
\end{enumerate}
\end{lemma}

\begin{proof}
It is a basic fact that there are sufficiently many injective $\sO_X$-modules.

Notice that a projective system of $\sO_X$-modules $I=(I_n)$ is injective if and only if each $I_n$ is an injective $\sO_X$-module and the transition maps $I_{n+1}\to I_n$ are split surjective.
(The ``if'' direction is easy, as well as $I$ injective implies each $I_n$ is injective. If $I$ is injective, let $J$ be the projective system with $J_n=I_1\oplus\ldots \oplus I_n$ and projections as transition maps. We have an obvious inclusion of projective systems $I\inj J$, hence a surjection $\Hom(J, I)\to \Hom(I, I)$. Now a lift of the identity on $I$ together with the split surjectivity of the transition maps of $J$ gives the splitting of the transition maps of $I$.)

Now let $E\to I^\bullet$ be an injective resolution (which always exist). The 
transition maps of the projective system (of abelian groups)
$(\Gamma(U, I_n^q))_n$ are surjective (since split) for all $q\ge 0$ and all open subsets $U\subset X$. 
Hence they 
satisfy the Mittag-Leffler 
condition and are $\varprojlim$-acyclic. 

On the other hand, $\varprojlim_n I^q_n$ is an injective $\OO_X$-module for 
every $q$.
Indeed, since the transition 
maps $I^q_{n+1}\xr{} I^q_n$ are surjective and split, we may write 
$I^q_n\cong \oplus_{i=1}^n I'_{i}$ for $I'_i=\ker(I^q_{i}\xr{} I^q_{i-1})$. The $\sO_X$-modules
$I'_n$ are injective for all $n$, and the transition maps 
$$
\oplus_{i=1}^{n+1} I'_{i} \cong I^q_{n+1} \xr{} I^q_{n} \cong \oplus_{i=1}^{n} I'_{i}  
$$
are the obvious projections. Thus $\varprojlim_n I^q_n=\prod_{i\geq 1} I'_n$ is 
injective. 
 
By using $\varprojlim \circ \Gamma_U = \Gamma_U \circ \varprojlim$ we obtain a spectral sequence 
\[R^i\varprojlim H^j(U, E_n)\Longrightarrow H^{i+j}(U, R\varprojlim E_n),\]
where $R\varprojlim E_n= \varprojlim_n I^\bullet_n$. If $U\in \mc{B}$, condition a) implies $R^i\varprojlim H^0(U, E_n)=H^i(U,R\varprojlim E_n)=
H^i(\varprojlim I^{\bullet}(U))$.
We know that $R^i\varprojlim H^0(U, E_n)$ is zero for all $i\ge 2$ and in case condition b) is satisfied also for all $i\ge 1$.
Now the assertion follows from 
$$
\varprojlim_{U\in \mc{B}, U\ni x} H^i(U,R\varprojlim E_n)= \varprojlim_{U\in \mc{B}, U\ni x} H^i(\varprojlim I^{\bullet}(U)) = (R^i\varprojlim E_n)_x, 
$$ 
for all $x\in X$. 
\end{proof}

\begin{lemma}\label{supp-of-flasque-is-flasque}
Let $A$ be a sheaf of abelian groups on a noetherian topological space $X$. 
If $A$ is flasque, so is $\ul{\Gamma}_\Phi(A)$ for all families of supports $\Phi$ on $X$.
\end{lemma}

\begin{proof}
Let $Y$ and $Z$ be two closed subsets of $X$. Since $\ul{\Gamma}_Z(I)$ is injective if $I$ is (\cite[Exp. I, Cor. 1.4]{SGA2}), there exists a spectral sequence
$H^i_Y(X,\sH^j_Z(A))\Longrightarrow H^{i+j}_{Y\cap Z}(X, A).$
Now assume $A$ is flasque, then $\sH^j_Z(A)=0$ for $j\neq 0$. In particular
$H^1_Y(X, \ul{\Gamma}_Z(A))= H^1_{Y\cap Z}(X, A)=0.$
Thus $\ul{\Gamma}_Z(A)$ is flasque. The space $X$ is noetherian and therefore 
$\ul{\Gamma}_{\Phi}(A)=\varinjlim_{Z\in \Phi} \ul{\Gamma}_Z(A)$ is also flasque.
\end{proof}

\begin{definition}\label{flasque-dRW}
 We say that a de Rham-Witt system on a $k$-scheme $X$ is {\em flasque}, if for all $n$ 
\[0\to K_n\to M_n\xr{\pi} M_{n-1}\to 0\]
is an exact sequence of flasque abelian sheaves on  $X$, where $K_n=\Ker(\pi: M_{n}\to M_{n-1})$. 
\end{definition}

\begin{lemma}\label{properties-of-flasque}
Let $X$ be a $k$-scheme.
\begin{enumerate}
 \item Let $0\to M'\to M\to M''\to 0$
      be a short exact sequence of de Rham-Witt systems on $X$ and assume that $M'$ is flasque.
      Then $M$ is flasque iff $M''$ is.
\item Let $\Phi$ be a family of supports on $X$. Then $\ul{\Gamma}_\Phi$ restricts to an exact endo-functor on the full subcategory of flasque de Rham-Witt systems.
\item Let $f:X\to Y$ be a morphism. Then $f_*$ restricts to an exact functor between the full subcategories of flasque de Rham-Witt systems on $X$ and $Y$.
\item The functor $\varprojlim: \drw_X\to \cdrw_X$ restricts to an exact functor from the full subcategory of flasque de Rham-Witt systems to the
              full subcategory of flasque de Rham-Witt modules (i.e. de Rham-Witt modules, which are flasque as abelian sheaves on $X$). 
\end{enumerate}
\end{lemma}

\begin{proof}
The proof of (1) is straightforward.
(2) follows from Lemma \ref{supp-of-flasque-is-flasque}. (3) is clear. Finally (4). It follows directly from the definition, that the transition maps on the sections over any open $U\subset X$
of a flasque de Rham-Witt systems are surjective. The exactness of $\varprojlim$ on the category of flasque de Rham-Witt systems, thus follows from Lemma \ref{vanishing-of-Rlim}, (2).
Now let $M$ be a flasque de Rham-Witt system. It remains to show that $\varprojlim M$ is flasque again. For this let $U\subset X$ be open and define
$L_n=\Ker(\Gamma(X,M_n)\to \Gamma(U,M_n))$. Thus we have an exact sequence
\eq{pof1}{\varprojlim \Gamma(X, M)\to \varprojlim\Gamma(U,M)\to R^1\varprojlim_n L_n.}
Consider the following diagram:
\[\xymatrix@=15pt{  &  & 0\ar[d] & 0\ar[d] & \\ 
              &  & \Gamma(X,K_n)\ar[d]\ar[r]& \Gamma(U,K_n)\ar[d]\ar[r] & 0\\
      0\ar[r]& L_n\ar[d]^a\ar[r] & \Gamma(X,M_n)\ar[d]\ar[r] & \Gamma(U,M_n)\ar[d]\ar[r] & 0\\
      0\ar[r]& L_{n-1}\ar[r] & \Gamma(X,M_{n-1})\ar[d]\ar[r] & \Gamma(U,M_{n-1})\ar[d]\ar[r] & 0\\
            &                      &      0                 &  0. &
}\]
All rows and columns are exact, since $M$ is flasque. Now it follows from an easy diagram chase that $a$ is surjective.
Therefore $R^1\varprojlim_n L_n=0$ and the flasqueness of $\varprojlim M$ follows from \eqref{pof1}.
\end{proof}

\begin{lemma}\label{enough-flasque}
The categories $\drw_X$ and $\cdrw_X$ have enough flasque objects, i.e. any $M$ in $\drw_X$ (or in $\cdrw_X$) admits an injection into a flasque object.
\end{lemma}
\begin{proof}
For the de Rham-Witt modules this is just the usual Godement construction. For the de Rham-Witt systems this has to be refined as follows:
Let $M$ be a de Rham-Witt system. Denote by $G(M_n)$ the $W_n\sO_X$-module given by
      \[G(M_n)(U)=\prod_{x\in U } M_{n,x},\quad U\subset X \text{ open},\]
       with the restriction maps given by projection.
       These sheaves fit together to form a de Rham-Witt system $G(M)=\{G(M_n)\}_{n\ge 1}$, such that the natural map
       $M \to G(M)$ is a morphism of de Rham-Witt systems. 

 For $m<n$ we denote by $i_{m,n}: W_mX\inj W_nX$ the closed immersion induced by the restriction $W_n\sO_X\to W_m\sO_X$, in particular $i_{n-1,n}=i_n$. 
We set 
\[\tilde{G}_n(M):= i_{1, n*}G(M_1)\oplus \ldots \oplus i_{n-1, n*}G(M_{n-1})\oplus G(M_n).\]
Then $\tilde{G}_n(M)$ is a graded $W_n\sO_X$-module. We define $W_n\sO_X$-linear maps $\pi$, $F$, $d$,  $V$, $\ul{p}$, as follows
\[
\begin{array}{ll}
\pi: \tilde{G}_n\to i_{n*}\tilde{G}_{n-1},& (m_1, \ldots, m_n)\mapsto (m_1,\ldots, m_{n-1}),\\
F: \tilde{G}_n\to (W_n(F_X)i_n)_*\tilde{G}_{n-1}, & (m_1, \ldots, m_n)\mapsto (Fm_2, \ldots, Fm_n),\\
d: W_n(F^n_{X})_*\tilde{G}_n\to W_n(F^n_{X})_*\tilde{G}_n(1),& (m_1, \ldots, m_n)\mapsto (dm_1,\ldots, dm_n),\\
V: (W_n(F_X)i_n)_*\tilde{G}_{n-1}\to \tilde{G}_n,& (m_1, \ldots, m_{n-1})\mapsto (0, Vm_1,\ldots, Vm_{n-1}),\\
\ul{p}: i_{n*}\tilde{G}_{n-1}\to \tilde{G}_n,& (m_1,\ldots, m_{n-1})\mapsto (0, \ul{p}m_1,\ldots, \ul{p}m_{n-1}).
\end{array} \]
It is straightforward to check that $\tilde{G}(M)=(\{\tilde{G}_n(M)\}_{n\ge 1}, \pi, F,d,V,\ul{p})$ 
becomes a de Rham-Witt system and it is flasque by its definition. Also, the inclusion $M\inj G(M)$ induces  an inclusion 
\[M_n\inj \tilde{G}_n(M),\quad m\mapsto (\pi^{n-1}(m), \ldots, \pi(m), m).\]
By definition this yields an inclusion of de Rham-Witt systems $M\inj \tilde{G}(M)$ and we are done.
\end{proof}


\begin{proposition}\label{derived-functors-exist}
 Let $f:X\to Y$ be a morphism between $k$-schemes and $\Phi$ a family of supports on $X$. Then the right derived functors 
\[R\ul{\Gamma}_\Phi: D^+(\drw_X)\to D^+(\drw_X),\]
\[Rf_*: D^+(\drw_X)\to D^+(\drw_Y),\]
\[R\varprojlim : D^+(\drw_X)\to D^+(\cdrw_X)\]
\[Rf_\Phi: D^+(\drw_X)\to D^+(\drw_Y),\]
\[R\hat{f}_\Phi: D^+(\drw_X)\to D^+(\cdrw_Y),\]
exist. Furthermore there are the following natural isomorphisms:
\begin{enumerate}
 \item Let $f:X\to Y$ and $g: Y\to Z$  morphisms, then $Rg_*Rf_*= R(g\circ f)_*$.
 \item Let $\Phi$ and $\Psi$ be two families of supports on $X$, then $R\ul{\Gamma}_{\Phi}R\ul{\Gamma}_\Psi= R\ul{\Gamma}_{\Phi\cap \Psi}$.
\item  Let $f: X\to Y$ be a morphism and $\Psi$ a family of supports on $Y$, then $R\ul{\Gamma}_{\Psi}Rf_*= Rf_* R\ul{\Gamma}_{f^{-1}(\Psi)}$.
\item Let $f:X\to Y$ be a morphism, then $R\varprojlim Rf_* =Rf_* R\varprojlim$.
\item Let $f:X\to Y$ be a morphism and $\Phi$ a family of supports on $X$. Then $Rf_\Phi= Rf_*R\ul{\Gamma}_\Phi$ and $R\hat{f}_\Phi= R\varprojlim Rf_\Phi$.
       If $Z$ is a closed subset of $X$ and $\Phi=\Phi_Z$, then also $R\hat{f}_{\Phi_Z}= Rf_*R\ul{\Gamma}_ZR\varprojlim$. 
\end{enumerate}
\end{proposition}

\begin{proof}
The existence follows from \cite[I, Cor. 5.3, $\beta$]{Ha} (take $P$ there to be the flasque objects) together with the Lemmas \ref{enough-flasque} and \ref{properties-of-flasque}.
The compatibility isomorphisms follow from \cite[I, Cor 5.5]{Ha} and Lemma \ref{properties-of-flasque}, (2)-(4).
\end{proof}

\begin{remark}\label{etale-pullback}
Let $f: X\to Y$ be an \'etale morphism between $k$-schemes. Then $W_n(f)$ is \'etale and thus $W_n(f)^*$ is exact on the category of $W_n(\sO_Y)$-modules.
Therefore $f^*: \drw_Y\to \drw_X$ is exact and thus extends to 
\[f^*: D^+(\drw_Y)\to D^+(\drw_X).\]
In case $j: U\inj X$ is an open immersion we write $M_{|U}$ instead of $j^*M$ for $M\in D^+(\drw_X)$.
\end{remark}

\subsubsection{Cousin-complex for de Rham-Witt systems.}\label{CousinWitt}
Let $X$ be a $k$-scheme and $Z^\bullet$ the codimension filtration of $X$, i.e. $Z^q$ is the family of supports on $X$ consisting of all closed subsets of $X$ whose codimension is at least
$q$. Let $M$ be a de Rham-Witt system on $X$. Take a complex of flasque de Rham-Witt systems $G$ on $X$, which is a resolution of $M$, i.e. there is a quasi-isomorphism $M[0]\to G$.
The filtration of complexes of de Rham-Witt systems 
\[G\supset \ul{\Gamma}_{Z^1}(G)\supset\ldots\supset \ul{\Gamma}_{Z^q}(G)\supset\ldots\]
defines a spectral sequence of de Rham-Witt systems
\[E_{1}^{i,j}= \sH^{i+j}_{Z^i/Z^{i+1}}(M)\Longrightarrow \sH^{i+j}(M),\]
where we put $\sH^{i+j}_{Z^i/Z^{i+1}}(M)= H^{i+j}(\ul{\Gamma}_{Z^i}(G)/\ul{\Gamma}_{Z^{i+1}}(G))$.
We define the {\em Cousin complex of $M$} (with respect to the codimension filtration) $E(M)$ to be the complex
$E_1^{\bullet, 0}$ coming from this spectral sequence, i.e. it is the complex of de Rham-Witt systems
\[E(M): \sH^{0}_{Z^0/Z^1}(M)\xr{d^{0,0}_1} \sH^{1}_{Z^1/Z^0}(M)\xr{d^{1,0}_1}\ldots \xr{}\sH^{i}_{Z^i/Z^{i+1}}(M)\xr{d^{i,0}_1}\ldots . \]
It satisfies the following properties:
\begin{enumerate}
 \item $(E(M))_n= E(M_n)$ is the usual Cousin complex associated to $M_n$ (see e.g. \cite[IV, \S 2]{Ha} or \cite[p. 107-109]{Co}).
 \item \[E^i(M)= \sH^i_{Z^i/Z^{i+1}}(M)=\bigoplus_{x\in X^{(i)}}i_{x*}H^{i}_x(M),\] 
        where $H^{i}_x(M)= (\varinjlim_{U\ni x}H^i_{\ol{\{x\}}\cap U}(U,M_n))_n$, which is a de Rham-Witt system on $\Spec \sO_{X,x}$ supported in the closed point $x$,
       $i_{x}: \Spec\sO_{X,x}\to X$ is the natural map and $X^{(i)}$ is the set of points $x$ of codimension $i$ in $X$ (i.e. $\dim \sO_{X,x}=i$). 
 \item The natural augmentation $M\to E(M)$ is a resolution of $M$ if and only if  $H^i_x(M_n)=0$ for all $x\in X^{(j)}$ with $j\neq i$ and for all $n\ge 1$.
\end{enumerate}
((1) holds since each $G_n$ is a flasque resolution of $M_n$; (2) follows from (1) and \cite[IV, \S 1, Var. 8, Motif F]{Ha};
      (3) follows from (1) and \cite[IV, Prop. 2.6, (iii)$\Longleftrightarrow$(iv)]{Ha} and \cite[IV, \S 1, Var. 8, Motif F]{Ha}.)

\begin{lemma}\label{drw-is-CM}
Let $X$ be a smooth $k$-scheme. Then $E(W_\bullet\Omega_X)$ is a flasque resolution of quasi-coherent de Rham-Witt systems of $W_\bullet\Omega_X$.
\end{lemma}

\begin{proof}
By \cite[I, Cor. 3.9]{IlDRW} the graded pieces of the standard filtration on $W_n\Omega^q_X$ are  extensions of locally free $\sO_X$-modules. 
Thus 
\eq{loc-coh-van}{H^i_x(W_n\Omega^q_X)=0\quad  \text{for all } x\in X^{(j)}, \text{ with }j\neq i, \text{ and all } q, n\ge 1.}
Thus $E(W_\bullet\Omega_X)$ is a quasi-coherent resolution of $W_\bullet\Omega_X$. 
Next we claim, that the transition morphisms
\eq{trans-surj}{H^i_x(W_n\Omega^q_X)\to H^i_x(W_{n-1}\Omega^q_X)}
 are surjective for all $x\in X^{(i)}$ and $n\ge 2$. 
Indeed, for $x\in X^{(i)}$ we can always find an open affine neighborhood 
$U=\Spec A$ of $x$ and sections $t_1,\ldots, t_i$ such that $\ol{\{x\}}\cap U=V(t_1,\ldots, t_i)$. This also implies for all $n\ge 1$,
$W_n(\ol{\{x\}})\cap W_nU=V([t_1],\ldots, [t_i])\subset W_n U$, where $[t]\in W_nA$ is the Teichm\"uller lift of $t\in A$. Then by \cite[Exp. II, Prop. 5]{SGA2} 
\[H^i_{\ol{\{x\}}\cap U}(U, W_n\Omega_X)= \varinjlim_r \frac{\Gamma(U,W_n\Omega_X)}{([t_1]^r,\ldots, [t_i]^r)\Gamma(U,W_n\Omega_X)}.\]
 In particular the transition maps \eqref{trans-surj} are surjective.
If we denote the kernel of the restriction morphism $W_n\Omega_X\to W_{n-1}\Omega_X$ by $K_n$, then this and \eqref{loc-coh-van} implies, that the sequence
\[0\to E^i(K_n)\to E^i(W_n\Omega_X)\to E^i(W_{n-1}\Omega_X)\to 0\]
is an exact sequence of flasque abelian sheaves on $X$ and this proves the lemma.
\end{proof}

\subsection{Witt-dualizing systems}\label{1.4} 

\subsubsection{}\label{fflat}
Let $f:X\to Y$ be a finite morphism between two finite dimensional noetherian schemes.
Using the notation from \cite[III, \S6]{Ha} we denote by $f^\flat: D^+(\sO_Y)\to D^+(\sO_X)$, the functor which sends a complex $C$ to 
\[f^\flat(C)=f^{-1}R\sHom_{\sO_Y}(f_*\sO_X, C)\otimes_{f^{-1}f_*\sO_X}\sO_X.\]
Evaluation by 1 induces the finite trace morphism on $D^+_{\rm qc}(\sO_Y)$ (see \cite[III, Prop. 6.5]{Ha})
\eq{Trf}{\Trf: f_*f^\flat\to \id_{D^+_{\rm qc}(\sO_Y)}} 
and composition with the natural map 
\eq{epsilonf}{\epsilon_f: f_*R\sHom_X(-,-)\to R\sHom_Y(f_*(-),f_*(-)) }
induces an isomorphism for any $ A\in D^-_{\rm qc}(\sO_X)$, $B\in D^+_{\rm qc}(\sO_Y)$
\eq{finite-duality}{\theta_f=\Trf_f\circ\epsilon_f: f_*R\sHom_X(A, f^\flat B)\xr{\simeq} R\sHom_Y(f_*A, B).}
In particular, we see that for any morphism $\varphi: f_*A\to B$ in $D_{\rm qc}(\sO_Y)$ , with $A$ bounded above and $B$ bounded below there exists a morphism
$^a\varphi: A\to f^\flat B$ in $D_{\rm qc}(\sO_X)$, such that $\varphi$ equals the composition
\[f_*A\xr{f_*(^a\varphi)} f_*f^\flat B\xr{\Trf_f} B.\]
We call $^a\varphi$ {\em the adjoint} of $\varphi$.

\subsubsection{}\label{iflat-sigmaflat}
Let $X$ be a $k$-scheme and denote by $D(\sC_{\N, X})=\prod_{n\ge 1} D(\sC_{n, X})$ the derived category of $\sC_\N$. 
Since the morphisms $i_n$ and $W_n(F_X)$ are finite for all $n$, the functors $i_*$, $\sigma_*$, $\Sigma_*$ are exact and extend to functors
on $D(\sC_{\N })$, which still satisfy the identities \eqref{isigmacomm}, \eqref{Sigmacomm}. On $D^+_{\rm qc}(\sC_\N)$ 
we define $i^\flat$, $\sigma^\flat$, $\Sigma^\flat$ as follows:
\[(i^\flat M)_n:= i_{n+1}^\flat M_{n+1},\quad (\sigma^\flat M)_n:= W_n(F_X)^\flat M_n,\quad (\Sigma^\flat M)_n= W_n(F^n_X)^\flat M_n.\]
There is an obvious way to define $\Trf_i$, $\Trf_\sigma$, $\Trf_\Sigma$, $\epsilon_i$, $\epsilon_\sigma$, $\epsilon_\Sigma$ such that the compositions
\[i_*R\sHom(M, i^\flat N)\xr{\epsilon_i}R\sHom(i_*M, i_*i^\flat N)\xr{\Trf_i} R\sHom(i_*M, N),\]
\[\sigma_*R\sHom(M, \sigma^\flat N)\xr{\epsilon_\sigma}R\sHom(\sigma_*M, \sigma_*\sigma^\flat N)\xr{\Trf_\sigma} R\sHom(\sigma_*M, N),\]
\[\Sigma_*R\sHom(M, \Sigma^\flat N)\xr{\epsilon_\Sigma}R\sHom(\Sigma_*M, \Sigma_*\Sigma^\flat N)\xr{\Trf_\Sigma} R\sHom(\Sigma_*M, N)\]
are isomorphisms for $M\in D^-_{\rm qc}(\sC_\N)$ and $N\in D^+_{\rm qc}(\sC_\N)$.
 
\begin{definition}[{\cite[III, Def. 2.2]{EI}}]\label{1.4.1}
A {\em Witt quasi-dualizing system} on $X$ is a collection 
$(Q,\ul{p},C,V)$ where  $Q$ is an object in $\mc{C}_{\N,{\rm qc}}$ and  
$$
\ul{p}:i_*Q\xr{} Q, \quad C:\Sigma_*Q\xr{} Q, \quad V:\sigma_*i_*Q\xr{} Q
$$ 
are morphisms in $\mc{C}_{\N}$ such that the following holds:
\begin{itemize}
\item[(a)] $V\circ \sigma_*i_*C = C \circ \Sigma_*\ul{p}$,
\item[(b)] $\ul{p}\circ i_*V=V\circ \sigma_*i_*\ul{p}$.
\end{itemize}

A {\em Witt dualizing system} is a Witt quasi-dualizing system, which has the additional property, that the adjoints
\eq{eq-adjoints}{^a\ul{p}: Q\xr{\simeq} i^\flat Q, \quad ^aC: Q\xr{\simeq}\Sigma^\flat Q, \quad ^aV: Q\xr{\simeq}i^\flat \sigma^\flat Q}
are quasi-isomorphisms.
 
A morphism $\varphi$ between Witt (quasi-) dualizing systems is a morphism
in $\mc{C}_{\N}$ commuting with $\ul{p}, V,$ and $C$. 
\end{definition}

A Witt  (quasi-) dualizing system $(Q,\ul{p},C,V)$ on $X$ is called {\em coherent} if  $Q_n$ is coherent for all $n\ge 1$. 

\begin{example}\label{1.4.4}
 \begin{enumerate}
  \item The system 
 \[W_\bullet\omega:=(\{W_n\}_{n\ge 1},\, \ul{p},\, C:=\{ W_n(F_{\Spec k})^{-n}\}, \, V:=\{W_n(F_{\Spec k})^{-1}\ul{p}\})\] 
is a Witt dualizing system on $\Spec k$, where $\ul{p}$ is the usual map ``lift and multiply by $p$'', which is concentrated degree $0$.
For this, first notice that $W_n$ is an injective $W_n$-module for all $n\ge 1$. 
Then one easily checks that the following maps are isomorphisms and adjoint to $\ul{p}$, $C$ and $V$ respectively: 
\[ W_\bullet\xr{\simeq} i^\flat W_\bullet=Hom_{W_\bullet}(i_*W_\bullet, W_\bullet),\quad a \mapsto (b\mapsto \ul{p}ab),\] 
\[W_\bullet\xr{\simeq} \Sigma^\flat W_\bullet= Hom_{W_\bullet}((\Sigma_*W_\bullet), W_\bullet), \quad a \mapsto (b\mapsto Cab), \]
\[W_\bullet\xr{\simeq} i^\flat\sigma^\flat W_\bullet= Hom_{W_\bullet}(\sigma_* i_* W_\bullet, W_\bullet), \quad  a\mapsto (b\mapsto V ab).\]
\item Let $X$ be a smooth $k$-scheme of pure dimension $N$. Then 
\[W_\bullet\omega_X:=(\{W_n\Omega^N_X\}_{n\ge 1},\, \ul{p},\, C,\, V) \]
is a Witt dualizing system which by definition is concentrated in degree $N$. Here $\ul{p}$ is ``lift and multiply by $p$'' and $V$ is the Verschiebung. On the $n$-th level $C$ is defined as the composition:
\[C^n: (\Sigma_*W_\bullet\Omega^N_X)_n\to (\Sigma_*W_\bullet\Omega^N_X/d(W_\bullet\Omega^{N-1}_X))_n\xr{(C^{-n})^{-1}} W_n\Omega^N_X,\]
where $C^{-n}: W_n\Omega^N_X\xr{\simeq} W_n(F^n_X)_*W_n\Omega^N_X/d(W_n\Omega^{N-1}_X)$ is the inverse Cartier isomorphism from \cite[III, Prop. (1.4).]{IR}.
One easily checks that $\ul{p}$, $C^n$ and $V$ satisfy the relations (a), (b) in Definition \eqref{1.4.1}. The condition on the adjoints \eqref{eq-adjoints} is harder and follows
from Ekedahl's result $W_n\Omega^N_X= W_n(f)^!W_n$, with $f: X\to \Spec k$ the structure map, see \cite[I and II, Ex. 2.2.1.]{EI}.
Notice that $W_\bullet\omega_{\Spec k}= W_\bullet\omega$.
\end{enumerate}

\end{example}

\subsubsection{} \label{1.4.3}
Let $(Q,\ul{p},C,V)$ be a Witt dualizing system on $X$. 
We may express the equalities in \eqref{eq-adjoints} as 
\[H^0(^a\varphi): Q\xr{\simeq} H^0(f^\flat Q) \quad \text{and} \quad H^0(f^\flat Q)[0]\cong f^\flat Q,\]
 where $(f,\varphi)\in \{(i,\ul{p}),(\sigma,V), (\Sigma,C)\}$. 
Therefore by the definition of the adjoints, $\ul{p}$, $C$ and $V$ factor as follows:
\[\ul{p}: i_{*}Q\xr{\simeq,\, H^0(^a\ul{p})}i_{*}H^0(i^\flat Q)\xr{H^0(\Trf_{i})} Q,\]
\[C: \Sigma_*Q\xr{\simeq, \, H^0(^aC)} \Sigma_*H^0(\Sigma^\flat Q)\xr{H^0(\Tr_{\Sigma})} Q,\]
\[V: \sigma_*i_*Q\xr{\simeq, \, H^0(^aV)} \sigma_* i_* H^0(i^\flat \sigma^\flat Q)\xr{H^0(\Tr_{\sigma i})} Q.\]

Furthermore it follows from \ref{iflat-sigmaflat}, that the natural transformations
\[\ul{p}\circ \epsilon_{i}: i_*\sHom_{\mc{C}_{\N}}((-), Q)\xr{} \sHom_{\mc{C}_{\N}}(i_{*}(-),Q),\]
\[C\circ \epsilon_{\Sigma}: \Sigma_*\sHom_{\mc{C}_{\N}}((-), Q)\xr{} \sHom_{\mc{C}_{\N}}(\Sigma_*(-),Q),\]
\[V\circ\epsilon_{\sigma i}: \sigma_*i_*\sHom_{\mc{C}_{\N}}((-), Q)\xr{} \sHom_{\mc{C}_{\N}}(\sigma_* i_*(-),Q)\]
are isomorphisms when restricted to the category $\sC_{\N, {\rm qc}}$.

 \subsubsection{}\label{dualizing-functor} Let $M$ be a \emph{quasi-coherent} de Rham-Witt system and $Q$ a Witt dualizing system on $X$. Then we may define
maps $\pi$, $F$, $V$, $d$ and $\ul{p}$ on $\sHom_{\sC_\N}(M,Q)\in\sC_\N$ as follows: 
\[\pi: \sHom(M, Q)\xr{\circ \ul{p}} \sHom(i_*M, Q)\xr{(\ul{p}\circ \epsilon_{i})^{-1}} i_{*}\sHom(M, Q),  \]
\[F: \sHom(M, Q)\xr{\circ V}\sHom(\sigma_*i_*M, Q)
             \xr{(V\circ\epsilon_{\sigma i})^{-1}}\sigma_*i_*\sHom(M, Q),\]
\[V: \sigma_*i_*\sHom(M, Q)\xr{(V\circ\epsilon_{\sigma i})} \sHom(\sigma_*i_*M, Q)
             \xr{\circ F} \sHom(M,Q),\]
\[\ul{p}: i_{*}\sHom(M, Q)\xr{\ul{p}\circ\epsilon_{i}} \sHom(i_{*}M, Q)\xr{\circ \pi}\sHom(M, Q),\]
\mlnl{d: \Sigma_*\sHom(M, Q)\xr{(C\circ\epsilon_{\Sigma})}\sHom(\Sigma_*M, Q)\xr{\circ d} \sHom(\Sigma_*M(-1), Q)\\
                \xr{\simeq, \, \alpha} \sHom(\Sigma_*M, Q)(1)\xr{(C\circ\epsilon_{\Sigma})^{-1}} \Sigma_*\sHom(M, Q)(1), }
where the isomorphism $\alpha: \sHom(\Sigma_*M(-1), Q)\xr{\simeq} \sHom(\Sigma_*M, Q)(1)$ is given by multiplication with $(-1)^{q+1}$ in degree $q$.

\begin{proposition}[{\cite[III, 2.]{EI}}]\label{1.4.7}
The above construction yields a functor
\[\sHom(-,Q): (\drw_{X,{\rm qc}})^o\lra \drw_{X},\]
which has the following properties
\begin{enumerate}
 \item A morphism of Witt dualizing systems $Q\to Q'$ induces a natural transformation of functors $\sHom(-,Q) \to \sHom(-,Q')$.
 \item The functor $\sHom(-,Q)$  restricts to a functor 
$(\drw_{X,{\rm c}})^o\to \drw_{X, {\rm qc}}$ and if $Q$ is coherent, then also to 
       $(\drw_{X,{\rm c}})^o\to \drw_{X, {\rm c}}$.
\item For all $M\in \drw_{X,{\rm qc}}$ and all $m\in \Z$ there is a natural isomorphism 
      \[ \sHom(M(m),Q)\simeq \sHom(M,Q)(-m), \]
      given by multiplication with $(-1)^{qm+\frac{m(m-1)}{2}}$ in degree $q$ and a natural isomorphism
       \[\sHom(M,Q(m))\simeq \sHom(M,Q)(m)\]
      given by the identity in each degree. (There is some freedom in defining these isomorphisms; our choice is compatible with the sign convention for complexes in \cite{Co}.) 
    
\end{enumerate}
\end{proposition}

\begin{proof}
It is straightforward to check the relations in Definition \ref{definition-drw}. 
\end{proof}

\subsection{Residual complexes and traces}\label{section-Residual-complexes-and-traces}
In this section $A$  will always be a regular local ring, all schemes are of finite type over $A$ and all morphisms will be $A$-morphisms.
The results of this section will be applied in the next sections in the case $A=W$ and schemes of finite type over $W_n$, some $n$. 

\subsubsection{Review of residual complexes.}\label{1.6.1} 
The general references for residual complexes are \cite[VI]{Ha}, \cite[3.2]{Co}. Let $X$ be an $A$-scheme.
A {\em residual complex on $X$} is a bounded complex $K$ of quasi-coherent and injective
$\sO_X$-modules, which has {\em coherent cohomology} and such that there exists an isomorphism of $\sO_X$-modules
$\bigoplus_{q\in\Z} K^q\cong \bigoplus_{x\in X} i_{x*}J(x)$, where 
$i_x: \Spec \sO_{X,x}\inj X$ is the inclusion and $J(x)$ is an injective hull of $k(x)$ in $\sO_{X,x}$
(i.e. it is an injective $\sO_{X,x}$-module which contains $k(x)$ and such that,
for any $0\neq a\in J(x)$ exists a $b\in\sO_{X,x}$ with $0\neq ba\in k(x)$). It follows that $i_{x*}J(x)$ is 
supported in $\overline{\{x\}}$. 
{\em The codimension function on $X$ associated to $K$} is the unique function 
$d_{K}: |X|\to \Z$ such that $K^q=\bigoplus_{d_{K^\bullet}(x)=q} i_{x*}J(x)$ for all $q$.
If  $x_0$ is an immediate specialization of $x\in X$, then $d_{K}(x_0)=d_{K}(x)+1$.
{\em The filtration $\ldots\subset Z^q_{K}\subset Z^{q-1}_{K}\subset\ldots\subset X$ associated
to $K$} is defined by $Z^q_{K}:=\{x\in X\,|\, d_{K}(x)\ge q\}$. On each irreducible
component of $X$ this filtration equals the shifted codimension filtration.

If $R\in D^b_{\rm c}(X)$ is a dualizing complex with associated codimension filtration $Z^\bullet$
(see \cite[3.1]{Co} for these notions), then the Cousin complex $E_{Z^\bullet}(R)$
of $R$ with respect to $Z^\bullet$ is a residual complex with associated filtration also $Z^\bullet$.
In $D^b_{\rm c}(X)$ we have $R\cong E_{Z^\bullet}(R)$ (since a dualizing complex
is Gorenstein).
 Particular examples of dualizing complexes
are $W_n$ on $X=\Spec W_n$,  $\sO_X[0]$ in case $X$ is regular, $\omega_X[N]=\Omega^N_{X/A}[N]$ in case $X$ is smooth of pure dimension $N$
and more general $f^!R \in D^b_c(X')$, where $f: X'\to X$ is a finite type morphism and $R$ is a dualizing complex on $X$.

On the other hand any residual complex on $X$ is a dualizing complex when viewed as a complex in $D^b_c(X)$. 
Furthermore, if $K$ is a residual complex on $Y$ with filtration $Z^\bullet$, then we have an
equality of complexes $E_{Z_\bullet}(K)=K$.

\subsubsection{$(-)^!$ for residual complexes}\label{1.6.2}
Let $f: X\to Y$  be a finite type morphism between $A$-schemes and
$K$ a residual complex on $Y$ with associated filtration $Z^\bullet_{K}=:Z^\bullet$ (which exists by the above).
Then there is a functor $f^\Delta$ from the category of residual complexes with filtration $Z^\bullet$ on
$Y$ to the category of residual complexes on $X$ with a fixed filtration denoted by $f^\Delta Z^\bullet$,
having the following properties:
\begin{enumerate}
 \item  For two finite type morphisms $f: X\to Y$ and $g: Y\to Z$ of $A$-schemes, there is an isomorphism 
        $c_{f, g}: (gf)^\Delta\xr{\simeq}f^\Delta g^\Delta$, which is compatible with triple compositions
         and such that $c_{\id, f}=\id=c_{f,\id}$.
\item  If $f:X\to Y$ is smooth and separated of relative dimension $r$, then there is an isomorphism
    \[\varphi_f: f^\Delta K\xr{\simeq} E_{f^{-1}Z^\bullet[r]}(\Omega^r_{X/Y}[r]\otimes f^* K).\] 
      Here $E_{f^{-1}Z^\bullet[r]}(\Omega^r_{X/Y}[r]\otimes f^* K)$ is the Cousin complex 
       associated to the complex $\Omega^r_{X/Y}[r]\otimes f^* K$  and the filtration $f^{-1}Z^\bullet[r]$. 
       If we have two smooth separated maps of some fixed relative dimension $f$ and $g$ as in (1),
       then $c_{f,g}$ is compatible with the natural map on the right.
\item If $f: X\to Y$ is finite, then there is an isomorphism
       \[\psi_f: f^\Delta K\xr{\simeq}E_{f^{-1}Z^\bullet}(\bar{f}^*R\sHom_{\sO_Y}(f_*\sO_{X}, K))
            =\bar{f}^*\sHom_{\sO_Y}(f_*\sO_{X}, K),\]
       where we set $\bar{f}^*(-):=f^{-1}(-)\otimes_{f^{-1}f_*\sO_X}\sO_X$ (which is an exact functor).
       If we have two finite maps $f$ and $g$ as in (1), then $c_{f, g}$ is compatible with the natural map on the 
       right of the above isomorphism.
\item Let 
\[\xymatrix{X_U\ar[r]^{u'}\ar[d]^{f'}\ar[dr]^h & X\ar[d]^f \\
             U\ar[r]^{u} & Y}\]
be a cartesian diagram of $A$-schemes with $u$ \'etale. Then there is an isomorphism 
           \[d_{u,f} : {f'}^\Delta u^*\xr{\simeq}{u'}^*f^\Delta,\]
which is compatible with compositions in $u$ and $f$ and with the isomorphisms in (2) and (3).
Furthermore, by (2) we have $u^*\cong u^\Delta$ and ${u'}^*\cong {u'}^\Delta$ and the following diagram commutes
\[\xymatrix{ {f'}^\Delta u^*\ar[rr]^-{d_{u,f}}\ar[d]^\simeq &    &          {u'}^*f^\Delta\ar[d]^\simeq \\
              {f'}^\Delta u^\Delta & h^\Delta\ar[l]_-{c_{f',u}}\ar[r]^-{c_{u', f}} &  {u'}^\Delta f^\Delta.}  \]
(To prove this commutativity one may assume that $f$ factors as $X\xr{i} P\xr{\pi} Y$, with $i$ a closed immersion and $\pi$ smooth and then use (2) and (3), cf. \cite[(3.3.34)]{Co}.)
Finally if $f$ and $g$ are \'etale, then $c_{f,g }: (gf)^\Delta\cong f^\Delta g^\Delta$ corresponds to the natural isomorphism
$(gf)^*\cong f^*g^*$.
\end{enumerate}
In fact $f^\Delta$ is defined by locally factoring $f: X\to Y$ into a closed immersion followed by 
a smooth morphism and then use (2) and (3) and glue (and then show that what you obtain is independent of all the choices). In $D^b_{\rm c}(X)$ we have
$f^\Delta K\cong f^!K$. For details, as well as more compatibilities,
see \cite[VI]{Ha} and \cite[3.2]{Co}. 

\subsubsection{The Trace for residual complexes}\label{1.6.3} 
The reference for this section is \cite[VI, 4,5, VII, 2]{Ha} and \cite[3.4]{Co}. Let $f: X\to Y$ be a finite type morphism between $A$-schemes. Let
$K$ be a residual complex on $Y$.
Then there exists a morphism of {\em graded sheaves} (in general not of complexes, which we indicate by the dotted arrow)
\[\Tr_f: \xymatrix{f_*f^\Delta K\ar@{.>}[r] &  K,}\]
which satisfies the following properties (and is also uniquely determined by the first three of them):
\begin{enumerate}
 \item $\Tr_f$ is functorial with respect to maps between residual complexes with same associated filtration and $\Tr_\id=\id$.
 \item If $g: Y\to Z$ is another morphism of finite type between $A$-schemes, then
        \[\Tr_{gf}= \Tr_g\circ g_*(\Tr_f) \circ(gf)_*c_{f,g}.\]
 \item If $f$ is finite, then $\Tr_f$ is a morphism of complexes and equals the composition
       \[\Tr_f: f_*f^\Delta K\xr{\simeq,\, \psi_f} \sHom_{\sO_Y}(f_*\sO_X, K)
                   \xr{\text{ev. at } 1} K. \]
\item $\Tr_f$ is compatible with \'etale base change (using the maps $d_{u,f}$ from \ref{1.6.2}, (4)).
\item If $f: X\to Y$ is {\em proper}, then $\Tr_f: f_*f^\Delta K\to K$ is a morphism of complexes.
\item If $f: \P^n_X\to X$ is the projection, then $\Tr_f$ is the composition (in $D^b_{\rm c}(X)$)
       \[\Tr_f: f_*f^\Delta K\xr{\simeq,\, \varphi_f}Rf_*(\Omega^n_{\P^n_X/X}[n])\otimes  K\simeq
                   K,\]
        where for the first isomorphism we used \ref{1.6.2}, (2) and the projection formula, the second
        isomorphism is induced by base change from the isomorphism
         \[\Z\xr{\simeq} H^n(\P^n_\Z, \Omega^n_{\P^n_\Z/\Z})=\check{H}^n(\fU, \Omega^n_{\P^n_\Z/\Z}),\quad 
             1\mapsto (-1)^{\frac{n(n+1)}{2}} \frac{dt_1\wedge\ldots \wedge dt_n}{t_1\cdots t_n},
         \]
where $\fU=\{U_0, \ldots, U_n\}$ is the standard covering of $\P^n_\Z$ and the $t_i$'s are the coordinate functions
on $U_0$.
\item(Grothendieck-Serre duality, special case) 
       If $f: X\to Y$ is {\em proper}, then for any $C\in D^-_{\rm qc}(X)$ the composition
         \[Rf_*R\sHom_X(C, f^\Delta K)\xr{\epsilon_f}
           R\sHom_{Y}(Rf_*C, f_*f^\Delta K)\xr{\Tr_f}
            R\sHom_{Y}(Rf_*C, K) \] 
       is an isomorphism. It is compatible with \'etale base change. 
\end{enumerate}

\begin{definition}\label{1.6.3.5}
We denote by ${\rm Sch}_{A,*} $ the category with objects given by pairs $(X, \Phi)$, where $X$ is a scheme of finite type over $A$ and $\Phi$ is a family of supports on $X$,
and the morphisms $f: (X, \Phi)\to (Y, \Psi)$ are given by separated 
$A$-morphisms $f: X\to Y$, whose restriction to $\Phi$ is {\em proper} and 
which satisfy $f(\Phi)\subset \Psi$.
\end{definition}

\begin{remark}\label{1.6.3.5.5}
Let $(X,\Phi)$ be an object in ${\rm Sch}_{A,*}$ and $K$ a residual complex on $X$. Then $\ul{\Gamma}_{\Phi} K$ is a direct summand of $K$ and is a complex
of quasi-coherent and injective  $\sO_X$-modules. Indeed the isomorphism $\bigoplus_{q\in\Z} K^q\cong \bigoplus_{x\in X} i_{x*}J(x)$ (see \ref{1.6.1}) implies
$\bigoplus_{q\in \Z}\ul{\Gamma}_\Phi K^q\cong \bigoplus_{\ol{\{x\}}\in \Phi} i_{x*}J(x).$
\end{remark}

\begin{corollary}\label{1.6.3.6}
 Let $f: (X, \Phi)\to (Y, \Psi)$ be a morphism in ${\rm Sch}_{A,*}$ and $K$ a residual complex on $Y$, then there exists a morphism of {\em complexes}
\[\Tr_f: f_{\Phi}f^\Delta K\to \ul{\Gamma}_{\Psi}(K),\]
where we set $f_{\Phi}:=f_*\circ\ul{\Gamma}_{\Phi}$, which satisfies the following properties:
\begin{enumerate}
 \item $\Tr_f$ is functorial with respect to maps between residual complexes with same associated filtration and $\Tr_\id=\id$.
\item The following diagram commutes
 \[\xymatrix{ f_\Phi f^\Delta K\ar[d]\ar[r]^{\Tr_f} & \ul{\Gamma}_\Psi (K) \ar[d]\\
                f_*f^\Delta K\ar@{.>}[r]^{\Tr_f} & K, }\]
 where the vertical maps are the natural ones and the lower horizontal map is the trace from \ref{1.6.3}, which is only a map of graded sheaves (visualized by the dotted arrow).
    By abuse of notation we write $f$ for both, the morphism $(X, \Phi)\to (Y,\Psi)$ in ${\rm Sch}_{A,*}$ and the underlying morphism of schemes $X\to Y$.
 \item If $g: (Y, \Psi)\to (Z, \Xi)$ is another morphism in ${\rm Sch}_{A, *} $, then
        \mlnl{\Tr_{gf}= \Tr_g\circ g_*(\Tr_f) \circ(gf)_*c_{f,g}: \\
           (gf)_{\Phi}(gf)^\Delta K\xr{c_{f,g}} g_{\Psi}f_{\Phi}f^\Delta g^\Delta K\xr{\Tr_{f}} g_{\Psi}g^\Delta K\xr{\Tr_g} \ul{\Gamma}_{\Xi}(K).}
\item $\Tr_f$ is compatible with \'etale base change in the following sense: Let 
      \[\xymatrix{X_U\ar[r]^{u'}\ar[d]_{f'} & X\ar[d]^f\\
                  U\ar[r]^u & Y }\]
      be a cartesian square of finite type $A$-schemes with $u$ \'etale; let $\Phi$, $\Psi$, $\Phi'$ and $\Psi'$ be families of supports on $X$, $Y$, $X_U$ and $U$, such that
      $f:(X, \Phi)\to (Y,\Psi)$ and $f': (X_U, \Phi')\to (U, \Psi')$ are in ${\rm Sch}_{A, *}$ and additionally $u^{-1}(\Psi)\subset \Psi'$, ${u'}^{-1}(\Phi)\subset \Phi'$.
      Then $u^*K\simeq u^\Delta K$ is a residual complex and the following diagram commutes:
  \[\xymatrix{u^*f_\Phi f^\Delta K\ar[r]^{u^*(\Tr_f)}\ar[d] & u^*\ul{\Gamma}_\Psi K\ar[d]\\
              f'_{\Phi'}{f'}^\Delta u^* K\ar[r]^{\Tr_{f'}} & \ul{\Gamma}_{\Psi'}u^* K.}\]
Here the vertical maps are given as follows: First, the composition $\ul{\Gamma}_\Psi\to \ul{\Gamma}_\Psi u_*u^*\simeq u_*\ul{\Gamma}_{u^{-1}(\Psi)}u^*\to u_*\ul{\Gamma}_{\Psi'}u^*$
gives by adjunction a map $u^*\ul{\Gamma}_{\Psi}\to \ul{\Gamma}_{\Psi'}u^*$, yielding the vertical map on the right in the diagram; similar we have a map ${u'}^*\ul{\Gamma}_{\Phi}\to \ul{\Gamma}_{\Phi'}{u'}^*$ inducing a map $u^*f_\Phi\to f'_{\Phi'}{u'}^*$; the vertical map on the left in the diagram is then given by the composition
$u^*f_\Phi f^\Delta\to f'_{\Phi'}{u'}^*f^\Delta \xr{d_{u,f}^{-1}} f'_{\Phi'}{f'}^\Delta u^*$, where $d_{u,f}$ is the map from \ref{1.6.2}, (4).
\item Let $j:U\to X$ be an open immersion such that $\Phi$ 
is contained in $U$. Then $j: (U, \Phi)\inj \XP$ is a morphism in ${\rm Sch}_{A,*}$ and $\Tr_j: j_\Phi j^\Delta  K\to \ul{\Gamma}_{\Phi} K$ is the 
       excision isomorphism, more precisely: $\Tr_j$ is given by the composition 
       \eq{1.6.3.6.1}{\Tr_j: j_\Phi j^\Delta K= \ul{\Gamma}_{\XP}j_* j^\Delta K\simeq \ul{\Gamma}_{\XP}j_* j^* K \stackrel{\text{exc.}}{=} \ul{\Gamma}_{\XP} K . }
\end{enumerate}
\end{corollary}

\begin{proof}
We define $\Tr_f : f_\Phi f^\Delta K\to \ul{\Gamma}_\Psi K$ to be the following composition
\eq{1.6.3.6.2}{f_\Phi f^\Delta K\to f_*\ul{\Gamma}_{f^{-1}(\Psi)} f^\Delta K\simeq  \xymatrix{\ul{\Gamma}_\Psi f_*f^\Delta K \ar@{.>}[r]^{\text{\ref{1.6.3}}}  & \ul{\Gamma}_{\Psi} K}.}
A priori this is only a map of graded sheaves. But we already observe, that the properties 1)- 4) follow immediately from the definition and the corresponding properties in \ref{1.6.3}. 
If $j:(U,\Phi)\inj (X,\Phi)$ is an open immersion as in (5), we may apply the excision identity $\ul{\Gamma}_{\XP}=\ul{\Gamma}_{\XP}j_*j^*$ to the map of graded sheaves 
$\Tr_j: \xymatrix{j_*j^\Delta (-)\ar@{.>}[r] & (-)}$ to obtain a commutative diagram
\[\xymatrix{ j_\Phi j^\Delta K=\ul{\Gamma}_\XP j_*j^\Delta K\ar@{=}[d]\ar@{.>}[r]^-{\Tr_j} & \ul{\Gamma}_\XP K\ar@{=}[d]\\
             \ul{\Gamma}_\XP j_*j^*(j_*j^\Delta K)\ar@{.>}[r]^-{j_*j^*(\Tr_j)} & \ul{\Gamma}_{\XP}j_*j^* K.    }\]
Now using the compatibility with base change as in (4) (in the situation $j=u=f$) implies that going around the diagram from the top left corner to the top right corner counterclockwise is
the isomorphism ({\em of complexes}) \eqref{1.6.3.6.1}. This gives (5) and in particular $\Tr_j$ is a morphism of complexes. 

It remains to show that $\Tr_f$ as defined in \eqref{1.6.3.6.2}  is in fact a morphism of complexes. For this we factor $f: X\to Y$ into an open immersion
$j:X\inj \bar{X}$ followed by a proper $A$-morphism $\bar{f}:\bar{X}\to Y$ (Nagata compactification). Since the restriction of $f$ to $\Phi$ is proper, it
follows that $\Phi$ also defines a family of supports on $\bar{X}$.  We consider  the following diagram  
\[\xymatrix{ f_\Phi f^\Delta K\ar[d]_\simeq\ar[r] &  \ul{\Gamma}_\Psi f_*f^\Delta K\ar@{.>}[r]^-{\Tr_f}\ar@{.>}[d]_{\Tr_j\circ c_{\bar{f}, j}} & \ul{\Gamma}_\Psi(K)\\
             \bar{f}_\Phi \bar{f}^\Delta K\ar[r] & \ul{\Gamma}_\Psi \bar{f}_*\bar{f}^\Delta K \ar[ur]_{\Tr_{\bar{f}}}, }\]
where the vertical isomorphism on the left is a morphism of complexes, which is given by 
\[f_*\ul{\Gamma}_{(X,\Phi)}f^\Delta K=\bar{f}_*\ul{\Gamma}_{(\bar{X}, \Phi)}j_* f^\Delta K  \xr{c_{\bar{f},j}} \bar{f}_*\ul{\Gamma}_{(\bar{X}, \Phi)}j_*j^\Delta \bar{f}^\Delta K
\xr{\Tr_j,\, (5)} \bar{f}_*\ul{\Gamma}_{(\bar{X}, \Phi)}\bar{f}^\Delta K.\]
Further $\Tr_{\bar{f}}$ is a morphism of complexes by (\ref{1.6.3}, (5)).
The diagram is commutative by \ref{1.6.3}, 2) and hence $\Tr_f$ as defined in \eqref{1.6.3.6.2} is a morphism of complexes.
\end{proof}

\subsubsection{}\label{1.6.4} Let $X$ be a finite type $A$-scheme and 
$\emptyset=Z^r\subset \ldots \subset Z^{q+1}\subset Z^q\subset\ldots \subset Z^s=X $ be a filtration with $r>s\in \Z$ and
such that $Z^q$ is stable under specialization and any $y\in Z^q\setminus Z^{q+1}$ is not a 
specialization of any other point in $Z^q$. Recall that a {\em Cousin complex on $X$ with respect to $Z^\bullet$}
is a complex $C^\bullet$ of quasi-coherent $\sO_X$-modules, such that for all $q$
the terms $C^q$ are {\em supported in the $Z^q/Z^{q+1}$-skeleton}, i.e.
$C^q\cong \bigoplus_{x\in Z^q\setminus Z^{q+1}} i_{x*}M_x$, where
$i_x:\Spec \sO_{X,x}\inj X$ is the inclusion and $M_x$ is a quasi-coherent sheaf on $\Spec \sO_{X,x}$
supported only in the closed point $x$. Notice that $i_{x*}M_x$ is the extension by zero 
of the constant sheaf $M_x$ on $\overline{\{x\}}$. Any residual complex with associated filtration $Z^\bullet$ 
is in particular a Cousin complex with respect to $Z^\bullet$. If $G$ is any complex of quasi-coherent 
$\sO_X$-modules, then $E_{Z^\bullet}(G)$ is a Cousin complex with respect to $Z^\bullet$. 

\begin{lemma}\label{1.6.4.5}
If $f: X\to Y$ is finite and $D$ is a Cousin complex on $X$ 
with respect to $f^{-1}Z^\bullet$ ($Z^\bullet$ as above), 
then $f_*D$ is a Cousin complex on $Y$ with respect to $Z^\bullet$.
\end{lemma}
\begin{proof}
Write $D^q\cong \bigoplus_{x\in f^{-1}Z^q\setminus f^{-1}Z^{q+1}} i_{x*}N_x$ as above, in particular $N_x$ is supported in $x$. Then 
\[f_*D^q\cong \bigoplus_{y\in Z^q\setminus Z^{q+1}} i_{y*}M_y,\quad \text{with } 
                                M_y:= f_{|X_{(y)}*}\bigoplus_{x\in f^{-1}(y)} i_{x, f^{-1}(y)*}N_x,\]
where $i_{x, f^{-1}(y)}: \Spec \sO_{X, x}\inj X_{(y)}=X\times_Y\Spec \sO_{Y,y}$ is the inclusion.
$M_y$ is supported in $y$ and this gives the claim.
\end{proof}

\begin{corollary}\label{1.6.5}
 Let $f: X\to Y$ be a {\em finite} morphism between finite-type $A$-schemes
       and $K$ a residual complex on $Y$ with filtration $Z^\bullet$ and 
       $C$ a Cousin complex on $X$ with respect to $f^{-1}Z^\bullet$. 
   Then the isomorphism of \ref{1.6.3}, (7) induces an isomorphism
   \[\Hom_X(C, f^\Delta K)\cong \Hom_Y(f_*C,  K).\]
  This isomorphism is compatible with \'etale base change and is concretely induced by
       sending a morphism of complexes $\varphi:C\to f^\Delta K$ to the composition
       $f_*C\xr{f_*(\varphi)} f_*f^\Delta K\xr{\Tr_f} K$.
\end{corollary}
\begin{proof}
 Since $K$ and $f^\Delta K$ are complexes of injectives and $f$ is finite, 
 \ref{1.6.3}, (7) immediately gives (for all $C$)
  \[\Hom_X(C, f^\Delta K)/{\rm homotopy}\cong \Hom_Y(f_*C,  K)/{\rm homotopy}.\]
By Lemma \ref{1.6.4.5} above $f_*C^q$ is supported in the $Z^q/Z^{q+1}$-skeleton, which by definition also holds for
$K^q$, for all $q$. Thus there is only the trivial homotopy between $f_*C^\bullet$ and $K$. 
Similar with $f^\Delta K$ and $C$.
\end{proof}

\begin{lemma}\label{1.6.6}
 Let $f: X\to Y$ be a finite morphism between finite type $A$-schemes and  $K$ a residual complex
on $Y$ with associated filtration $Z^\bullet$. 
Then for all $q\in \Z$ the following equality holds in $D^b_{\rm qc}(X)$
\[(f^\Delta K)^q\cong \bar{f}^*\sHom_Y(f_*\sO_X, K^q)\cong f^{\flat}(K^q)\cong 
\bigoplus_{x\in f^{-1}(Z^q)\setminus f^{-1}(Z^{q+1})} i_{x*}J(x),\]
where $\bar{f}^*(-):=f^{-1}(-)\otimes_{f^{-1}f_*\sO_X}\sO_X$, $i_x:\Spec \sO_{X,x}\inj X$ is the 
inclusion and $J(y')$ is an injective hull of $k(y')$ in $\sO_{Y',y'}$. 
\end{lemma}

\begin{proof}
 The first isomorphism follows from \ref{1.6.2}, (3), the second holds since $K^q$ is injective 
and the third is \cite[VI, Lem. 4.1]{Ha}.
\end{proof}

\subsection{Witt residual complexes}\label{section-Witt-residual-complexes}

Let $X$ be a $k$-scheme.
\subsubsection{} Let $C(\mc{C}_{\N})$ be the category of complexes of $\mc{C}_{\N}$. 
Recall that it is equipped with endo-functors $i_*,\sigma_*,\Sigma_*$. 

\begin{notation}
We say that a complex $K$ in $C(\mc{C}_{\N})$ is a \emph{residual complex} if $K_n$ is a residual complex on $W_n(X)$ for all $n$ and
the associated filtrations on $|X|=|W_n(X)|$ are all the same.

We say that a complex $C$ in $C(\mc{C}_{\N})$ is a \emph{Cousin complex} if $C_n$ is a Cousin complex for all $n$ with respect to the same filtration.
\end{notation}

For a residual complex $K\in C(\mc{C}_{\N})$ we define 
\begin{equation*}
\begin{split}
(i^\Delta K)_n&:=i_{n+1}^{\Delta} K_{n+1}\\
(\sigma^\Delta K)_n&:=W_n(F_X)^{\Delta} K_{n}\\
(\Sigma^\Delta K)_n&:=W_n(F_X^n)^{\Delta} K_{n}.
\end{split}
\end{equation*}
This yields  endo-functors on the full subcategory of residual complexes with some fixed filtration of $C(\sC_\N)$. 

In view of Corollary \ref{1.6.5} we get 
\begin{equation} \label{Cousin-adjoint-residual}
\begin{split}
\Hom_{C(\mc{C}_{\N})}(i_*C,K)&=\Hom_{C(\mc{C}_{\N})}(C,i^\Delta K)\\
\Hom_{C(\mc{C}_{\N})}(\sigma_*C,K)&=\Hom_{C(\mc{C}_{\N})}(C,\sigma^\Delta K)\\
\Hom_{C(\mc{C}_{\N})}(\Sigma_*C,K)&=\Hom_{C(\mc{C}_{\N})}(C,\Sigma^\Delta K)
\end{split}
\end{equation}
if $C$ is a Cousin complex with respect to the filtration of $K$.

\begin{definition}\label{definition-Wittresidual}
Let $X$ be a $k$-scheme. A {\em Witt residual complex on $X$} is a collection
$(K,\ul{p},C,V)$ where $K\in C(\mc{C}_{\N})$ is a complex  and 
\[\ul{p}: i_{*}K\to K,\, C : \Sigma_{*}K\to K,\, V: \sigma_*i_*K\to K,\]
are morphisms of complexes such that:
\begin{itemize}
\item [(a)] $K$ a residual complex,  
\item [(b)] the morphisms $\ul{p}$, $C$,  $V$, satisfy the relations
\[V\circ \sigma_*i_*C = C \circ \Sigma_*\ul{p}, \quad \ul{p}\circ i_*V=V\circ \sigma_*i_*\ul{p},\] 
and the adjoints of $\ul{p}$, $C$ and $V$ under \eqref{Cousin-adjoint-residual} 
\[^a\ul{p}:K\xr{\simeq} i^\Delta K,\, 
  ^aC: K\xr{\simeq} {\Sigma}^\Delta K,\, 
   ^aV: K\xr{\simeq} i^{\Delta}\sigma^\Delta K\]
are isomorphisms of complexes. 
\end{itemize}
A morphism between Witt residual complexes is a morphism in $C(\mc{C}_{\N})$ which
is compatible with $\ul{p}, C, V$ in the obvious sense. 

\end{definition}

\begin{remark}
 One should memorise the relations in (b) above as $VC=C\ul{p}$ and $V\ul{p}=\ul{p}V$.
\end{remark}

\begin{remark}\label{1.6.8}
It follows from Lemma \ref{1.6.6}, that if $(K,\ul{p},C,V)$ is a Witt residual complex on $X$ , then
for all $q\in \Z$ the systems $(\{K^q_n\}_{n\ge 1}, \ul{p}^q, C^q , V^q)$ 
are Witt dualizing systems. Furthermore, if $Z^\bullet$ is the filtration of 
$K_1$, then
\[K^q_n=\bigoplus_{x\in Z^q\setminus Z^{q+1}} i_{W_n(x)*} J_n(x),\]
where $i_{W_n(x)}: \Spec W_n\sO_{X,x}\inj W_n X$ is the inclusion and $J_n(x)$ is an injective hull of 
$k(x)$ in $W_n\sO_{X,x}$.
\end{remark}

\subsubsection{}\label{1.6.10}
By extending the notions from \ref{1.6.2} and Corollary \ref{1.6.3.6} term by term to $C(\sC_\N)$ we obtain:
For any finite-type morphism  $f: X\to Y$ between $k$-schemes a functor $f^{\Delta}$ of residual complexes in $C(\mc{C}_{\N,Y})$ with associated filtration $Z^\bullet$ to  
residual complexes in $C(\mc{C}_{\N,X})$ with filtration $f^\Delta Z^\bullet$.  For $g: Y\to Z$ another morphism a canonical isomorphism $c_{f,g}: (g\circ f)^\Delta\cong f^\Delta g^\Delta$
of natural transformations on the category of residual complexes on $C(\sC_{\N, Z})$. For $f:(X,\Phi)\to (Y,\Psi)$ a morphism in ${\rm Sch}_{k,*}$ (see Definition \ref{1.6.3.5}) 
a natural transformation $\Tr_f: f_\Phi f^\Delta\to \ul{\Gamma}_\Psi$ on the category of residual complexes with fixed filtration in $C(\sC_{\N, Y})$.
And these data satisfies the compatibilities from \ref{1.6.2} and Corollary \ref{1.6.3.6}.

\subsubsection{}\label{1.6.10.5}
Let $f: X\to Y$ be a finite-type morphism between $k$-schemes. Let $K$ be a Witt
residual complex.
We use \eqref{Cousin-adjoint-residual} to define  morphisms of complexes
$$
\ul{p}: i_{*}f^\Delta K\to f^\Delta K, \quad
C: \Sigma_{*} f^\Delta K\to f^\Delta K, \quad  
V: \sigma_*i_* f^\Delta K \to f^\Delta K,
$$
as adjoints of the compositions 
\[f^\Delta K\xr{f^\Delta(^a\ul{p}_{K}),\, \simeq} 
        f^\Delta i^\Delta K \cong i^\Delta f^\Delta K\]
\[f^\Delta K\xr{f^\Delta(^aC_K),\,\simeq} 
         f^\Delta \Sigma^\Delta K\cong \Sigma^\Delta f^\Delta K,\]
\[f^\Delta K\xr{f^\Delta(^aV_K),\,\simeq} 
      f^\Delta  i^{\Delta}\sigma^\Delta K\cong i^{\Delta}\sigma^\Delta f^\Delta K.\]
Here $\ul{p}_K$, $C_K$ and $V_K$ denote the corresponding morphisms for $K$ and the isomorphisms on the right are induced by the isomorphisms 
$c_{W_nf, i_{n,Y}}$ etc. from \ref{1.6.2}, (1).

\begin{proposition}\label{1.6.11}
Let $f: X\to Y$ be a finite-type morphism between $k$-schemes and $K$ a Witt residual complex 
on $Y$. Then the system
\[(f^\Delta K, \ul{p}, C, V)\]
defined above is a Witt residual complex on $X$. Furthermore, 
if $g:Y\to Z$ is another finite-type morphism of $k$-schemes, then 
the isomorphism $c_{f,g}: (g\circ f)^\Delta K\cong f^\Delta g^\Delta K$
from \ref{1.6.10} defines an isomorphism of Witt residual complexes,
which is compatible with triple compositions.

If $f:X\to Y$ is \'etale, then  $f^* K$ is isomorphic to $f^\Delta K$. If $f: X\to Y$ and $g: Y\to Z$ are \'etale,
then the isomorphism $(gf)^\Delta \cong f^\Delta g^\Delta$ is induced by the isomorphisms $(gf)^*\cong f^*g^*$.
\end{proposition}

\begin{proof}
First notice that if $K$ and $L$ are two residual complexes, then their associated filtrations are
the same iff their associated codimension functions (see \ref{1.6.1}) are the same.
Since the codimension functions of the $K_n$'s  are the same by assumption we obtain
(using the formula of \cite[(3.2.4)]{Co}) for any $x\in |X|=|W_nX|$
\begin{eqnarray*}
 d_{(W_nf)^\Delta K_n}(x) & = & d_{K_n}(W_nf(x))-{\rm trdeg}(k(x)/ k(W_nf(x)))\\
                                  &= &  d_{K_1}(f(x))-{\rm trdeg}(k(x)/ k(f(x)))\\
                                  &=& d_{f^\Delta K_1}(x).
\end{eqnarray*}
Thus the $(W_nf)^\Delta K_n$ are residual complexes and all have the same associated filtration.
The condition on the adjoints of $\ul{p}$, $C$ and $V$ holds by definition. It remains to check
the relations $VC=C\ul{p}$ and $V\ul{p}=\ul{p}V$, which are equivalent to
${^aV}\, {^aC}= {^aC}\, {^a\ul{p}}$ and ${^aV}\, {^a\ul{p}}= {^a\ul{p}} \,{^aV}$.
To prove the first equality consider the following diagram 
\[\xymatrix{f^\Delta K\ar[r]^{^a\ul{p}_K}\ar[d]_{^aC_K} &
            f^\Delta i^\Delta K\ar[r]^\simeq\ar[d]^{^aC_K} &
            i^\Delta f^\Delta K\ar[d]^{^aC_K}\\
            f^\Delta \Sigma^\Delta K\ar[r]^{^a V_K}\ar[d]_\simeq &
            f^\Delta i^\Delta \Sigma^\Delta K\ar[r]^\simeq\ar[d]^\simeq &
            i^\Delta f^\Delta (\Sigma)^\Delta K\ar[d]^\simeq\\
            \Sigma^\Delta f^\Delta K\ar[r]^-{^a V_K} &
            \Sigma^\Delta f^\Delta i^\Delta\sigma^\Delta K\ar[r]^-\simeq &
            i^\Delta \Sigma^\Delta f^\Delta  K.
}\]
Here all arrows labeled by $\simeq$ are induced by compositions of the isomorphisms $c_{f,g}$ 
and their inverses with $f,g$ appropriate and
also in the two arrows labeled by $^aV_K$ there is such an isomorphism involved. (We also need the relation \eqref{Sigmacomm}.) The square in the upper
left commutes since $K$ is a Witt residual complex, all other squares commute because of
the functoriality of the $c_{f,g}$'s and their compatibility with compositions. Thus going around the diagram
from the top left corner to the lower right corner clockwise is the same as going around counter clockwise,
which yields ${^aC}\, {^a\ul{p}}={^aV}\, {^aC}$. The other relation is proved by drawing a similar 
diagram. Thus $f^\Delta K$ is a Witt residual complex. 

The second statement amounts to prove that $c_{g, f}$ is compatible with $\ul{p}$, $C^n$ and $V$.
This follows again from the functoriality of the $c_{(-,-)}$'s and their compatibility with compositions by
drawing diagrams as above, which we omit. Finally, the statement about \'etale morphisms follows from \ref{1.6.2}, (4).
\end{proof}

\begin{lemma}\label{1.6.13}
Let $(X,\Phi)\to (Y, \Psi)$ be a morphism in ${\rm Sch}_{k,*}$ and $K$ a Witt residual complex on $Y$. Then the complex $f_\Phi f^\Delta K$, with  $f_\Phi:= f_*\circ \ul{\Gamma}_\Phi$,
is a complex of Witt quasi-dualizing systems (see Definition \ref{1.4.1}) and each term
 $(f_\Phi f^\Delta K)_n$ is a bounded complex of quasi-coherent, flasque sheaves.
Furthermore 
\[\Tr_f: f_\Phi f^\Delta K \lra \ul{\Gamma}_\Psi K, \]
is a morphism of complexes of Witt quasi-dualizing systems,
which is compatible with composition and \'etale base change, as in Corollary \ref{1.6.3.6}, (3) and (4), and it is functorial with respect to maps between  Witt residual complexes with the same associated filtration
(e.g. isomorphisms). Finally, if $f$ is an open immersion and $\Psi=\Phi$, then $\Tr_f$ is the excision isomorphism as in Corollary \ref{1.6.3.6}, (5).
\end{lemma}

\begin{proof}
We define the maps $\ul{p}$, $C$ and $V$ on $f_\Phi f^\Delta K$ to be the following compositions
\[\ul{p}: i_*f_\Phi f^\Delta K \cong f_\Phi i_* f^\Delta K\xr{\ul{p}_{f^\Delta K }} f_\Phi f^\Delta K,\]
\[C : \Sigma_*f_\Phi f^\Delta K\cong f_\Phi \Sigma_* f^\Delta K\xr{C_{f^\Delta K}} f_\Phi f^\Delta K,\]
\[V: \sigma_*i_* f_\Phi f^\Delta K\cong f_\Phi \sigma_* i_* f^\Delta K\xr{V_{f^\Delta}} f_\Phi f^\Delta K.\]
These maps obviously make $f_\Phi f^\Delta K$ into a complex of Witt quasi-dualizing systems. 
Also the $(f^\Delta K)_n$ are bounded complexes of quasi-coherent injective $W_n\sO_X$-modules (since the $(f^\Delta K)_n$ are residual complexes), thus all the $(f_\Phi f^\Delta K)_n$ are complexes of quasi-coherent and 
flasque sheaves. It remains to check, that the trace morphism commutes with $\ul{p}$, $C^n$ and  $V$. We will prove
\eq{1.6.13.1}{\Tr_f\circ\ul{p}= \ul{p}\circ i_*(\Tr_f).}
For this, we consider the following diagram:
\[\xymatrix{ &  i_* f_\Phi f^\Delta K\ar[r]^\simeq\ar[d]^{^a\ul{p}}\ar[dl]_{\Tr_f}  & f_\Phi i_*f^\Delta K\ar[d]^{^a\ul{p}}\\
             i_*\ul{\Gamma}_\Psi K\ar[dr]_{^a\ul{p}} &   i_*f_\Phi f^\Delta i^\Delta K\ar[d]^{\Tr_f}\ar[r]^\simeq & f_\Phi i_* f^\Delta i^\Delta K\ar[d]^\simeq\\
            & i_*\ul{\Gamma}_\Psi i^\Delta  K\ar[d]^{\Tr_i} &  f_\Phi i_*i^\Delta f^\Delta K\ar[d]^{\Tr_i}\\
           & \ul{\Gamma}_\Psi K &   f_\Phi f^\Delta K\ar[l]_{\Tr_f}.
}\]
The triangle in the upper left (which should be a square) commutes by the functoriality of $\Tr_f$ (Corollary \ref{1.6.3.6}, (1)), the upper square in the middle commutes by the functoriality of the
isomorphism $i_* f_\Phi \cong f_\Phi i_*$ and the big square in the middle commutes by Corollary \ref{1.6.3.6}, (3). Thus going around the diagram from the upper left to the lower left corner
clockwise, which is  $\Tr_f\circ\ul{p}$, is the same as going around counterclockwise, which is $\ul{p}\circ i_*(\Tr_f)$. This gives \eqref{1.6.13.1}.
Similar diagrams prove the relations $\Tr_f\circ C\cong C\circ \Sigma_*(\Tr_f)$ and $V\circ \sigma_*i_*(\Tr_f)\cong \Tr_f\circ V$. The compatibility of
$\Tr_f$ with composition and \'etale base change follows directly from Corollary \ref{1.6.3.6}, (3) and (4). The functoriality statement follows from the corresponding statement in
\ref{1.6.3.6}, (1). If $f$ is an open immersion, then so is $W_nf$ for all $n$ and hence the last statement follows from \ref{1.6.3.6}, (5).
\end{proof}

\subsection{The dualizing functor}\label{section-The-dualizing-functor}
\subsubsection{}\label{1.6.14}
Let $X$ be a $k$-scheme and $K$ a Witt residual complex on $X$. If $M$ is a complex of {\em quasi-coherent} de Rham-Witt systems on $X$, then by Proposition \ref{1.4.7} and
Remark \ref{1.6.8}, we have de Rham-Witt systems $\sHom(M^i, K^j)\in \drw_X$ for all $i,j$. Thus we obtain a complex of de Rham-Witt systems on $X$, 
$\sHom(M, K)\in C(\drw_X)$
which is defined by
\[\sHom^q(M, K):=\prod_{i\in\Z}\sHom(M^i, K^{i+q}),\]
 \[ d^q_{\sHom}((f^i)):=(d^{i+q}_{K}\circ f^i + (-1)^{q+1} f^{i+1}\circ d^i_{M}). \]
In this way we clearly obtain a functor 
\[D_K: C(\drw_{X, {\rm qc}})^o\to C(\drw_X), \quad M\mapsto \sHom(M, K),\]
It restricts to a functor $D_K: C(\drw_{X,{\rm c}})^o\to C(\drw_{X, {\rm qc}})$.
Since $K_n$ is a bounded complex of injective $W_n\sO_X$-modules for each $n\ge 1$, $\sHom(M, K)$ is acyclic, whenever $M$ is (\cite[II, Lem. 3.1]{Ha}). 
Thus $D_K$ preserves quasi-isomorphisms and therefore induces functors
\[D_K: D(\drw_{X, {\rm qc}})^o\to D(\drw_X),\quad D_K: D(\drw_{X, {\rm c}})^o\to D(\drw_{X, {\rm qc}}).\]
A morphism $K\to L$ of Witt residual complexes clearly induces a natural transformation of functors 
\eq{1.6.14. 1}{D_K\to D_L.}
\begin{notation}\label{KX}
Let $X$ be a $k$-scheme with structure map $\rho_X:X\to \Spec k$. Then
we denote by  $K_X$ the Witt residual complex $\rho_X^\Delta W_\bullet\omega$ and write
\[D_X=D_{K_X}=\sHom(-,K_{X}).\]
Notice that $K_X$ is concentrated in degree 0, since $W_\bullet\omega$ is.
\end{notation}
 
\begin{remark}\label{1.6.14.4}
Let $X$ be a scheme, $\Phi$ a family of supports on $X$ and $K$ a Witt residual complex. Then we have the following equalities on $C(\sC_{\N,{\rm qc}})$
\[\ul{\Gamma}_\Phi D_K(-)= \ul{\Gamma}_\Phi\sHom(-, K)= \sHom(-, \ul{\Gamma}_\Phi K).\]
Therefore, if $M$ is a complex of quasi-coherent de Rham-Witt systems on $X$, then $\sHom(M, \ul{\Gamma}_\Phi K)$ inherits the structure of a complex of de Rham-Witt systems from
$\ul{\Gamma}_\Phi D_K(M)$. Furthermore $\ul{\Gamma}_\Phi(K)_n$ is a complex of injectives (by Remark\ref{1.6.3.5.5}), 
\[\sHom(-, \ul{\Gamma}_\Phi K):D(\drw_{X,{\rm qc}})^o\to D(\drw_{X}).\]
\end{remark}

\begin{proposition}\label{1.6.14.4.5}
Let $f: (X,\Phi)\to (Y,\Psi)$ be a morphism in ${\rm Sch}_{k,*}$, $K$ a Witt residual complex on $Y$ and $M$ a bounded above complex of quasi-coherent de Rham-Witt systems on $X$.
Let 
\[\vartheta_f: f_\Phi D_{f^\Delta K}(M)\to\ul{\Gamma}_\Psi D_K(f_*M)\]
be the composition  
\[f_\Phi \sHom(M, f^\Delta K)\xr{\rm nat.} \sHom(f_*M, f_\Phi f^\Delta K) \xr{\Tr_f} \sHom(f_*M, \ul{\Gamma}_\Psi K). \]
Then $\vartheta_f$ is a morphism of complexes of de Rham-Witt systems on $Y$ and it has the following properties: 
\begin{enumerate}
\item It defines a natural transformation between the bifunctors $f_\Phi\sHom(-,f^\Delta -)$ and 
$\sHom(f_*-, \ul{\Gamma}_\Psi(-))$ defined on the product of
the category $C^-(\drw_{X, {\rm qc}})$ with the category of Witt residual complexes with the same associated filtration. 
\item It is compatible with composition, i.e.  
if $g: (Y, \Psi)\to (Z, \Xi)$ is another morphism in ${\rm Sch_{k,*}}$,
then  $\vartheta_{gf}\cong \vartheta_g\circ g_*(\vartheta_f)\circ c_{f,g}$. 
\item It is compatible with \'etale base change as in Corollary \ref{1.6.3.6}, (4).
\item Let $j: U\inj X$ be an open immersion, such that $\Phi$ is contained in $U$, then 
$\vartheta_j: j_\Phi\sHom(-, j^\Delta K)\xr{\simeq} \sHom(j_*(-), \ul{\Gamma}_\Phi K)$ is the excision isomorphism (cf. Corollary \ref{1.6.3.6}, (5)).
\end{enumerate}
\end{proposition}
\begin{proof}
We have to show, that $\vartheta_f$ commutes with $\pi, F, d, V, \ul{p}$. Recall from \ref{dualizing-functor}, how the de Rham-Witt system structure on 
$\sHom(M, Q)$ is defined, where $M$ is a de Rham-Witt - and $Q$ a Witt dualizing system. For example the map $\pi$ is uniquely determined by the property that it makes the following diagram commutative:
\[\xymatrix{ \sHom(M, Q)\ar[r]^{\circ \ul{p}_M}\ar[d]^{\pi} & \sHom(i_*M_, Q_)\\
             i_*\sHom(M, Q)\ar[r]^-{\epsilon_i} & \sHom(i_*M, i_*Q)\ar[u]_{\ul{p}_Q}.   }\]
The maps $F,d,V, \ul{p}$ are defined uniquely by making similar diagrams commutative. Using these diagrams, the second part of  \ref{1.4.3} and Lemma \ref{1.6.13}  it is straightforward to check 
the compatibility of $\vartheta_f$ with the de Rham-Witt system structure. (Notice, that $\vartheta_f$ factors over 
$\sHom(f_*M,f_\Phi f^\Delta K)$, which in general will not be a de Rham-Witt system, since $f_\Phi f^\Delta K$ will not be a complex of Witt dualizing systems. But it
is a complex of Witt quasi-dualizing systems, which is sufficient to conclude.) Thus $\vartheta_f$ gives a morphism of complexes of de Rham-Witt systems. The functoriality of $\vartheta_f$ in $M$ is clear and in 
$K$ follows from the functoriality of $\Tr_f: f_\Phi f^\Delta \to \id$ (see Lemma \ref{1.6.13}). 
The properties (2)-(4) follow from the corresponding properties of $\Tr_f$, see Lemma \ref{1.6.13}.

\end{proof}

\subsection{Ekedahl's results}\label{section-E-results}

\begin{thm}[{\cite[II, Thm 2.2 and III, Prop. 2.4]{EI}}]\label{1.4.8}
Let $X$ be a smooth $k$-scheme of pure dimension $N$. Then the morphism induced by multiplication 
\[W_\bullet\Omega_X\xr{\simeq} \sHom(W_\bullet\Omega_X, W_\bullet\omega_X),\quad    \alpha\mapsto (\beta\mapsto \alpha\beta ).\] 
is an  isomorphism of de Rham-Witt systems. Furthermore
\[\sExt^i_{W_n\sO_X}(W_n\Omega_X, W_n\omega_X)=0,\quad \text{for all } i,n\ge 1.\] 
\end{thm}

\begin{proof}
  This is all due to Ekedahl. Let us just point out that $W_\bullet\omega_X$ sits in degree $N$, hence the multiplication map is a graded morphism; and that $C^nd=0$ gives the equality
\[C^n(xdy)= (-1)^{q+1}C^n(d(x)y), \quad \text{for all } x\in W_n\Omega^q_X,\, y\in W_n\Omega^{N-q}_X\]
and this together with the sign $\alpha$ introduced in the definition of $d$ in \ref{dualizing-functor} (which is missing in \cite{EI}) gives the compatibility with $d$.
\end{proof}

\begin{lemma}[{\cite[Prop. 8.4, (ii)]{BER}}]\label{1.6.12.6}
Let $X'$ be a smooth $W_n$-scheme and denote by $X$ its reduction modulo $p$.
Let $\varphi: W_n\Omega_X\xr{\simeq} \sigma^n_*\sH^\bullet_{\rm DR}(X'/W_n)$, with $\sigma=W_n(F_{\Spec k})$, be the canonical isomorphism of graded $W_n$-algebras from \cite[III, (1.5)]{IR}. 
Then $\varphi$ is the unique morphism, which makes the following diagram commutative 
\[\xymatrix{W_{n+1}\Omega_{X'/W_n} \ar[d]\ar[r]^{F^n} & Z\Omega_{X'/W_n}\ar[d]\\
             W_n\Omega_X\ar[r]^{\varphi} & \sH^\bullet_{\rm DR}(X'/W_n).  }\]
Here $W_{n+1}\Omega_{X'/W_n}$ is the relative de Rham-Witt complex defined in \cite{LZ} and the
vertical maps are the natural surjections.
In particular we have
\eq{1.6.12.6.1}{\varphi(\sum_{i=0}^{n-1}V^i([a_i]))=\sum_{i=0}^{n-1}F^nV^i([\tilde{a}_i])=
                    \tilde{a}_0^{p^n}+p\tilde{a}_1^{p^{n-1}}+\ldots+ p^{n-1}\tilde{a}_{n-1}^p}
and 
\eq{1.6.12.6.2}{\varphi(\sum_{i=0}^{n-1}dV^i([a_i]))=\sum_{i=0}^{n-1}F^ndV^i([\tilde{a}_i])
       =\tilde{a}_0^{p^n-1}d\tilde{a}_0+ \tilde{a}_1^{p^{n-1}-1}d\tilde{a}_1+\ldots +
                                                           \tilde{a}_{n-1}^{p-1}d\tilde{a}_{n-1},} 
where the $\tilde{a}_i$ are any liftings of $a_i\in\sO_{X}$ to $\sO_{X'}$.
\end{lemma}

\begin{thm}[Ekedahl]\label{1.6.12}
Let $X$ be a smooth $k$-scheme of pure dimension $N$ with structure map $\rho_X: X\to \Spec k$. 
Then $(W_n\rho_X)^!W_n\in D^b_c(W_n X)$ is concentrated in degree $-N$, for all $n\ge 1$ and there is a quasi-isomorphism of complexes of Witt dualizing systems
(see Example \ref{1.4.4}, (1) and (2))
\[\tau: W_\bullet\omega_X(N)[N]\xr{\rm qis} K_{X}=\rho_X^\Delta W_\bullet\omega,\]
such that
\begin{enumerate}
 \item The map $\tau_1: \Omega^N_X[N]\xr{\rm qis} \rho_X^\Delta k$ is the classical (ungraded) quasi-isomorphism, i.e. it is the composition of the natural quasi-isomorphism
        $\Omega^N_X[N]\to E_{Z^\bullet[N]}(\Omega^N_X[N])$ with the inverse of \ref{1.6.2}, (2). Here $Z^\bullet[N]$ is the shifted codimension filtration of $X$ and $E_{Z^\bullet[N]}$ the Cousin functor.
\item  $\tau$ is compatible with \'etale pullback, i.e. if $f: U\to X$ is \'etale, then the following diagram commutes
       \[\xymatrix{f^*W_\bullet\omega_X(N)[N]\ar[d]^\simeq\ar[r]^{\tau_X} & f^*K_X=f^*\rho_X^\Delta W_\bullet\omega\ar[d]^\simeq\\
                    W_\bullet\omega_U(N)[N]\ar[r]^{\tau_U} & K_U=\rho_U^\Delta W_\bullet\omega,
         }\] 
where $\rho_U: U\to \Spec k$ is the structure morphism of $U$, the vertical isomorphism on the right is induced by $f^*\rho_X^\Delta\cong f^\Delta\rho_X^\Delta\cong \rho_U^\Delta$
and the isomorphism on the left is induced by $(W_nf)^*W_n\Omega^N_X=W_n\Omega^N_U$, $n\ge 1$ (see \cite[I, Prop. 1.14]{IlDRW}).
\end{enumerate}
 
\end{thm}
\begin{proof}
We do not write the shift $(N)$ in the following (the statement about the grading being obvious). In \cite[I,Thm.~4.1]{EI} it is proven that one has an isomorphism $\tau_n: W_n\Omega^N_X \xr{\simeq} (W_n\rho_X)^!W_n[-N]$ in $D^b_c(W_n X)$. (We give a sketch of the construction in case $X$ admits a smooth lift over $W_n$
in the remark below.) Let $Z^\bullet$ be the filtration by codimension on $W_n X$.
For $C$ any complex let $E_{Z^\bullet}(C)$ be the associated Cousin complex. Then $E_{Z^\bullet}$ induces an equivalence from the category of dualizing complexes with filtration $Z^\bullet$
in $D^b_c(W_n X)$ and the category of residual complexes with associated filtration $Z^\bullet$ on $W_n X $ with quasi inverse the natural localization functor (see \cite[Lem. 3.2.1]{Co}). Thus we obtain
an isomorphism of residual complexes
\[ E_{Z^\bullet}(W_n\Omega^N_X)\xr{\simeq} E_{Z^\bullet}((W_n\rho_X)^!W_n[-N])\cong (W_n\rho_X)^\Delta W_n[-N].\]
By Lemma \ref{drw-is-CM} the natural map $W_n\Omega^N_X\to E_{Z^\bullet}(W_n\Omega^N_X)$ is a resolution, for all $n\ge 1$ (by \cite[IV, Prop. 2.6.]{Ha}). 
Therefore we obtain a quasi-isomorphism (by abuse of notation again denoted by $\tau_n$)
\[\tau_n: W_n\Omega^N_X[N]\xr{qis} (W_n\rho)^\Delta W_n.\]
It follows from \cite[I, Lem 3.3 and II, Lem. 2.1]{EI}, that the $\tau_n$, $n\ge 1$, are compatible with $\ul{p}, C^n, V$ and hence induce a morphism of complexes of Witt dualizing systems as in the statement.
The property (1) is proved at the end of the proof of \cite[I, Thm. 4.1, p. 198 ]{EI} and (2) follows from the construction of $\tau_n$ in \cite[I, 2.]{EI} (see in particular the second paragraph on page 194).
\end{proof}

\begin{remark}\label{1.6.12.5}
 Let us sketch for later purposes how the isomorphism 
 $\tau_n : W_n\omega_X\to (W_n\rho_X)^! W_n [-N]$ is constructed in case $X$ admits a smooth lifting
$\rho_{X'}: X'\to \Spec W_n$. For details see \cite[I, 2.]{EI}. 
Let 
\eq{1.6.12.5.1}{\varphi: W_n\Omega_X\xr{\simeq} \sigma^n_*\sH^\bullet_{\rm DR}(X'/W_n)}
be the  canonical isomorphism of graded $W_n$-algebras of \cite[III, (1.5)]{IR} (see Lemma \ref{1.6.12.6}).
The composition
$W_n\sO_X\xr{\varphi^0}\sigma^n_*\sH^0_{\rm DR}(X'/W_n)\to \sO_{X'} $ defines a finite morphism 
$\epsilon: X'\to W_n X$, which fits into a commutative diagram
\[\xymatrix{X'\ar[r]^\epsilon\ar[d]_{\rho_{X'}} & W_nX\ar[d]^{W_n\rho_X}\\
             \Spec W_n\ar[r]^{\sigma^n} & \Spec W_n. }\]
It follows that $\varphi$ becomes an isomorphism $W_n\Omega_X\xr{\simeq} \epsilon_*\sH^\bullet_{\rm DR}(X'/W_n)$.
(Here we abuse the notation $\epsilon_*$ to indicate where the $W_n\sO_X$-module structure on the $W_n$-module $\sH^\bullet_{\rm DR}(X'/W_n)$ is coming from, which is
of course not an $\sO_{X'}$-module.)
Since $\rho_{X'}$ is smooth we have a canonical isomorphism 
$\tau_{X'}: \omega_{X'/W_n}\xr{\simeq} \rho_{X'}^! W_n[-N]$. Now Ekedahl shows, that
the composition
\mlnl{\epsilon_*\omega_{X'/W_n}\xr{\tau_{X'}} \epsilon_*\rho_{X'}^! W_n[-N]\simeq
            \epsilon_*\rho_{X'}^!{\sigma^n}^!W_n[-N]\\ \simeq \epsilon_*\epsilon^!(W_n\rho_X)^!W_n[-N]
        \xr{\Tr_{\epsilon}} (W_n\rho_X)^!W_n[-N] }
factors over $\epsilon_* \sH^N_{\rm DR}(X'/W_n)$. Then $\tau_n$ is defined as the composition
\[W_n\omega_X\xr{\varphi,\, \simeq}\epsilon_* \sH^N_{\rm DR}(X'/W_n)\to (W_n\rho_X)^!W_n[-N].\]
\end{remark}

\begin{corollary}\label{1.6.17}
Let $X$ be a smooth $k$-scheme of pure dimension $N$. Denote by $E(W_\bullet\omega_X)$ the Cousin complex associated to $W_\bullet\omega_X$ with respect to the codimension filtration on $X$ (cf. \ref{CousinWitt}). 
Then the maps $\ul{p}$, $C^n$ , $V$ on $W_\bullet\omega_X$ induce morphisms of complexes on the system $\{E(W_n\omega_X)(N)[N]\}_n$; this
defines a Witt residual complex $E(W_\bullet\omega_X)(N)[N]$, which is isomorphic as Witt residual complex to $K_X$ (see Notation \ref{KX}). This isomorphism is compatible with
\'etale base change. 
\end{corollary}

\begin{corollary}\label{1.6.18}
Let the notations be as above. Then there is an isomorphism in $C^b(\drw_{X, {\rm qc}})$ 
                  \[\mu: E(W_\bullet\Omega_X)\xr{\simeq} D_X(W_\bullet\Omega_X)(-N)[-N],\]
given by the composition
\mlnl{E(W_\bullet\Omega_X)\xr{E(\ref{1.4.8}),\, \simeq}E(\sHom(W_\bullet\Omega_X, W_\bullet\omega_X))\xr{\simeq}\\
\sHom(W_\bullet\Omega_X, E(W_\bullet\omega_X))\xr{\text{Cor.~\ref{1.6.17}}, \,\simeq}\sHom(W_\bullet\Omega_X, K_X)(-N)[-N].}
The map $\mu$ is compatible with \'etale base change.

\end{corollary}

\subsection{The trace morphism for a regular closed immersion}\label{section-The-trace-morphism-for-a-regular-closed-immersion}

\subsubsection{Local Cohomology}\label{1.7.6}
Let $X=\Spec A$ be an affine   Cohen-Macaulay scheme and $Z\subset X$ a closed subscheme of pure codimension $c$, defined by the ideal $I\subset A$.
Let $t=t_1,\ldots,t_c\in I$ be an $A$-regular sequence with $\sqrt{(t)}=\sqrt{I}$, here $(t)$ denotes the ideal $(t_1,\ldots, t_c)\subset A$.
(After shrinking $X$ such a sequence always exists.)
We denote by $K^\bullet(t)$ the  Koszul complex of the sequence $t$,
i.e. $K^{-q}(t)=K_q(t)= \bigwedge^q A^c$, $q=0,\ldots, c$, and if $\{e_1,\ldots, e_c\}$ is the standard basis of $A^c$ and
$e_{i_1,\ldots, i_q}:= e_{i_1}\wedge\ldots\wedge e_{i_q}$, then the differential is given by
\[d^{-q}_{K^\bullet}(e_{i_1,\ldots, i_q})=d^{K_\bullet}_{q}(e_{i_1,\ldots, i_q})=
                        \sum_{j=1}^{q}(-1)^{j+1} t_{i_j}e_{i_1,\ldots,\widehat{i_j},\ldots i_q}.\]
We define the complex 
\[K^\bullet(t,M):= \Hom_A(K^{-\bullet}(t), M),\]
and denote its $n$-th cohomology by $H^n(t, M)$. The map 
\[\Hom_A(\bigwedge^c A^c, M)\to M/(t)M,\quad \varphi\mapsto \varphi(e_{1,\ldots,c})\] induces a canonical isomorphism $H^c(t, M)\simeq M/(t)M$. 

If $t$ and $t'$  are two sequences as above with $(t')\subset (t)$, then there exists a $c\times c$-matrix $T$ with coefficients in $A$ such that
$t'= T t$ and $T$ induces a morphism of complexes $K^\bullet(t')\to K^\bullet(t)$, which is the unique (up to homotopy) morphism lifting
the natural map $A/(t')\to A/(t)$. Furthermore we observe that, for any pair of sequences $t$, $t'$
as above there exists an $N\ge 0$ such that $(t^N)\subset (t')$, where $t^N$ denotes the sequence $t_1^N,\ldots, t_c^N$. Thus 
the sequences $t$ form a directed set and $H^c(t, M)\to H^c(t', M)$, $(t')\subset (t)$, becomes a direct system. 
It follows from \cite[Exp. II, Prop. 5]{SGA2}, that we have an isomorphism
\eq{1.7.6.1}{\varinjlim_t M/(t)M=\varinjlim_t H^c(t, M) \cong H^c_Z(X, \widetilde{M}),}
where the limit is over all $A$-regular sequences $t=t_1,\ldots, t_c$  with $V((t))=Z$ and $\widetilde{M}$ is the sheaf associated to $M$. In fact it is enough to take the limit over the powers 
$t^n=t_1^n,\ldots, t_c^n$, $n\ge 1$, of just one regular sequence $t$ with $V((t))=Z$.
We denote by
\[\genfrac{[}{]}{0pt}{}{m}{t}\]
the image of $m\in M$ under the composition 
\[M\to M/(t)M\to H^c(t, M)\to H_Z^c(X, \widetilde{M}).\]
It is a consequence of the above explanations that we have the following properties:
\begin{enumerate}
 \item Let $t$ and $t'$ be two sequences as above with $(t')\subset (t)$. Let $T$ be a $c\times c$-matrix 
       with $t'=Tt$, then
      \[\genfrac{[}{]}{0pt}{}{\det(T)\,m}{t'}=\genfrac{[}{]}{0pt}{}{m}{t}.\]
\item\[\genfrac{[}{]}{0pt}{}{m+m'}{t}=\genfrac{[}{]}{0pt}{}{m}{t}+ \genfrac{[}{]}{0pt}{}{m'}{t}, \quad 
                 \genfrac{[}{]}{0pt}{}{t_i m}{t}=0 \quad\text{all }i.  \]
\item If $M$ is any $A$-module, then
        \[H_Z^c(X,\sO_X)\otimes_A M\xr{\simeq} H^c_Z(X, \widetilde{M}), \quad 
                                    \genfrac{[}{]}{0pt}{}{a}{t}\otimes m\mapsto \genfrac{[}{]}{0pt}{}{am}{t}\]
is an isomorphism.
\end{enumerate}
 
\begin{remark}\label{1.7.7}
Since for an $A$-regular sequence $t$ as above $K^\bullet(t)\to A/(t)$ is a free resolution, we have an isomorphism
\[\Ext^n(A/(t), M)\simeq \Hom^\bullet_A(K^\bullet(t), M).\] 
 Notice that we also have an isomorphism 
\[\Hom^\bullet_A(K^\bullet(t), M)\simeq K^\bullet(t,M),\]
which is multiplication with $(-1)^{n(n+1)/2}$ in degree $n$ (see \cite[(1.3.28)]{Co}). 
We obtain an isomorphism
\[\psi_{t,M}: \Ext^c(A/(t),M)\xr{\simeq} H^c(t,M)= M/(t)M, \]
which has the sign $(-1)^{c(c+1)/2}$ in it. In particular under the composition
\[\Ext^c(A/(t), M)\xr{\psi_{t,M}} M/(t)M \to H^c_Z(X,\widetilde{M})\]
the class of a map $\varphi\in \Hom_A(\bigwedge^c A^c,M)$ is sent to 
\[(-1)^{c(c+1)/2}\genfrac{[}{]}{0pt}{}{\varphi(e_{1,\ldots,c})}{t}.\]
\end{remark}

\begin{remark}\label{1.7.8} 
 If $X=\Spec A$ is a $k$-scheme and $t=t_1,\ldots, t_c$ is a regular sequence of sections of $\sO_X$, then for any
$n$ the sequence  $[t]=[t_1],\ldots, [t_c]$ of sections of $W_n\sO_X$ is a regular sequence. 
Here $[t_i]$ denotes the Teichm\"uller lift of $t_i$. Indeed the sequence $[t]$ is regular iff
its Koszul complex is a resolution of $W_n A/([t])$ and thus the statement follows by induction from the
short exact sequence
\[0\to K_\bullet([t]^p, W_{n-1}A)\xr{V} K_\bullet([t], W_n A)\to K_\bullet(t,A)\to 0.\] 
\end{remark}

\begin{proposition}\label{1.7.10}
Let $i: Z\inj X$ be a closed immersion between smooth $k$-schemes of pure dimension $N_Z$, $N_X$ and structure maps $\rho_Z$, $\rho_X$ and let $c$ be the codimension of $Z$ in $X$, i.e. $c=N_X-N_Z$. 
Then we have the following isomorphism in $D^b_{\rm qc}(W_n\sO_X)$,
\eq{1.7.10.0}{R\ul{\Gamma}_Z (W_n\omega_X) \cong \sH^c_Z (W_n\omega_X)[-c] }
Assume furthermore, that the ideal sheaf of
$Z$ is generated by a regular sequence $t=t_1, \ldots, t_c$ of global sections of $\sO_X$
and define a morphism $\imath_{Z,n}$ by
\[\imath_{Z,n}: (W_n i)_*W_n\omega_Z\lra \sH^c_Z(W_n\omega_X), \quad 
\alpha\mapsto (-1)^c \genfrac{[}{]}{0pt}{}{d[t]\tilde{\alpha}}{[t]},\] 
with $\tilde{\alpha}\in W_n\Omega^{N_Z}_X$ any lift of $\alpha$, and $d[t]=d[t_1]\cdot\ldots\cdot d[t_c]$.
Then the following diagram in $D^b_{\rm qc}(W_n\sO_X)$ is commutative
\[\xymatrix{ (W_n i)_*(W_n \rho_Z)^\Delta W_n\ar[r]^-{c_{W_ni, W_n\rho_X}}& (W_ni)_*(W_ni)^\Delta (W_n\rho_X)^\Delta W_n\ar[r]^-{\Tr_{W_ni}} &  R\ul{\Gamma}_Z(W_n\rho_X)^\Delta W_n\\ 
            (W_n i)_*W_n\omega_Z[N_Z]\ar[u]^{\tau_{Z,n}}_\simeq\ar[r]^{\imath_{Z,n}} & \sH^c_Z( W_n\omega_X) [N_Z] \ar[r]^{\simeq} & R\ul{\Gamma}_Z( W_n\omega_X) [N_X]\ar[u]_{\tau_{X,n}}^\simeq         }\]
where $\Tr_{W_n i}$ is the $n$-th level of the trace morphism from Lemma \ref{1.6.13}. 
\end{proposition}

\begin{proof}
We write $i_n=W_ni$, $\rho_{X,n}=W_n\rho_X$ and $W_n\rho_Z=\rho_{Z,n}$.
By \eqref{loc-coh-van} we have $\sH^i_Z (W_n\omega_X)=0$ for $i<c$. Furthermore $\sH^i_Z (W_n\omega_X)=0$ for $i>c$, by $\check{C}$ech considerations. (Locally the ideal of $Z$ is generated by a regular sequence of length $c$
and thus $U\setminus Z$ may locally be covered by $c$ affine open subschemes.)  This gives the isomorphism \eqref{1.7.10.0}.
To prove the commutativity of the diagram in the statement of the proposition we have to show that two elements in $\Hom_{W_n\sO_X}(i_{n*}W_n\omega_Z, \sH^c_Z (W_n\omega_X) )$
are equal. This is a local question. We may thus assume, that the closed immersion $i$ lifts to 
a closed immersion $i': Z'\inj X'$ between smooth $W_n$-schemes and that there exists a regular sequence
$t'=t_1',\ldots, t_c'$ in $\sO_{X'}$ which generates the ideal of $Z'$ and reduces modulo $p$ to $t$.
As in Remark \ref{1.6.12.5} we obtain a commutative diagram
\[\xymatrix{Z'\ar[r]^{\epsilon_Z}\ar@{^(->}[d]^{i'}\ar@/_1.8pc/[dd]_{\rho_{Z'}}  &  W_nZ\ar@{^(->}[d]_{i_n}\ar@/^ 1.8pc/[dd]^{\rho_{Z,n}}\\
            X'\ar[r]^{\epsilon_X}\ar[d]^{\rho_{X'}} & W_nX\ar[d]_{\rho_{X,n}}\\
              \Spec W_n\ar[r]^{\sigma^n} & \Spec W_n,   } \]
where we write $\sigma^n$ instead of $W_n(F^n_{\Spec k})$.

Consider the following diagram, in which we set $\Lambda:=W_n[-N_Z]$ and write $(-)^!$ instead of $(-)^\Delta$:
\begin{tiny}
\eq{1.7.10.1}{\xymatrix@-1pc{ i_{n*}\epsilon_{Z*}\omega_{Z'/W_n}\ar[d]\ar[r]^-\simeq & i_{n*} \epsilon_{Z*} \rho_{Z'}^!\Lambda\ar[d]\ar[r]^-\simeq & \epsilon_{X*}i'_*{i'}^!\rho_{X'}^!\Lambda\ar[d]\ar[r]^-{\Tr_{i'}} &
                  \epsilon_{X*}R\ul{\Gamma}_{Z} \rho_{X'}^!\Lambda\ar[d] & \epsilon_{X*}R\ul{\Gamma}_{Z}\omega_{X'/W_n}[c]\ar[d]\ar[l]_-{\simeq}  \\
         i_{n*}\epsilon_{Z*}\sH^{N_Z}_{\rm DR}(Z')\ar@{.>}[r]^-{\exists} & i_{n*}\rho_{Z,n}^!\Lambda \ar[r]^-\simeq & i_{n*}i_n^!\rho_{X,n}^!\Lambda\ar[r]^-{\Tr_{i_n}} & R\ul{\Gamma}_Z \rho_{X,n}^!\Lambda  & 
                                                     \epsilon_{X*}R\ul{\Gamma}_Z(\sH^{N_X}_{\rm DR}(X'))[c]\ar@{.>}[l]_-{\exists}   \\
              i_{n*}W_n\omega_Z\ar[u]^\simeq\ar[ur]_{\tau_{Z,n}}  &    &     &    &  R\ul{\Gamma}_Z (W_n\omega_X)[c]\ar[u]_\simeq\ar[ul]^{\tau_{X,n}}\\ 
                                                                              & & & & \sH^c_Z( W_n\omega_X).\ar[u]_\simeq\\
                 } }
\end{tiny}
We give some explanations: We have $\Lambda={\sigma^n}^!\Lambda$ (see Example \ref{1.4.4}, (1)) and the three vertical maps in the middle are the compositions
\[i_{n*} \epsilon_{Z*} \rho_{Z'}^!\Lambda\cong i_{n*} \epsilon_{Z*} \rho_{Z'}^!{\sigma^n}^!\Lambda \cong i_{n*} \epsilon_{Z*} \epsilon_Z^! \rho_{Z,n}^! \Lambda \xr{\Tr_{\epsilon_Z}} i_{n*}\rho_{Z,n}^!\Lambda,\]
\mlnl{\epsilon_{X*}i'_*{i'}^! \rho_{X'}^! \Lambda\cong i_{n*}\epsilon_{Z*}{i'}^! \rho_{X'}^! {\sigma^n}^!\Lambda\cong i_{n*}\epsilon_{Z*}{i'}^!\epsilon_X^!\rho_{X,n}^!\Lambda
             \\ \xr{\simeq} i_{n*}\epsilon_{Z*}\epsilon_Z^! {i_n}^!\rho_{X,n}^! \Lambda \xr{\Tr_{\epsilon_Z}}  i_{n*}i_n^! \rho_{X,n}^!\Lambda,}
\[\epsilon_{X*}R\ul{\Gamma}_Z\rho_{X'}^! \Lambda\cong R\ul{\Gamma}_Z\epsilon_{X*}\rho_{X'}^!{\sigma^n}^!\Lambda\cong R\ul{\Gamma}_Z\epsilon_{X*}\epsilon_X^! \rho_{X,n}^! \Lambda
                                                                                                                                       \xr{\Tr_{\epsilon_X}}  R\ul{\Gamma}_Z\rho_{X,n}^! \Lambda. \]
Now one easily checks that the two squares in the middle commute. In fact, to see that the first of the two middle squares commute, one only needs, that $c_{(f,g)}: (gf)^!\xr{\simeq} f^!g^!$ is a natural transformation,
which is compatible with triple compositions (see \cite[p.139, (VAR1)]{Co}) and that $\Tr_{\epsilon_Z}: \epsilon_{Z*}\epsilon_Z^!\to \id$ is a natural transformation. For the commutativity of the second
middle square we need the naturality of $\Tr_{i'}: i'_*{i'}^!\to \id$ and the formula (see \cite[Lem 3.4.3, (TRA1)]{Co})
\mlnl{\Tr_{i_n}\circ i_{n*}(\Tr_{\epsilon_Z})\circ (i_n\epsilon_Z)_*(c_{\epsilon_Z, i_n})=\Tr_{i_n\epsilon_Z}\\ 
             =\Tr_{\epsilon_Xi'}=\Tr_{\epsilon_X}\circ \epsilon_{X*}(\Tr_{i'})\circ (\epsilon_Xi')_*(c_{i',\epsilon_X}).}
That there exist two unique dotted arrows, which make the two outer diagrams commutative was proved by Ekedahl, see Remark \ref{1.6.12.5}. The two vertical maps in the two lower triangles
are the isomorphisms from Lemma \ref{1.6.12.6}. Thus the two triangles commute by definition of $\tau_{Z,n}$ and $\tau_{X,n}$, see Remark \ref{1.6.12.5}. It follows that the whole diagram commutes.

Let $\sigma_Z$ and $\sigma_X$ be the following compositions
\[\sigma_Z: \epsilon_{Z*}\omega_{Z'/W_n} \to \epsilon_{Z*}\sH^{N_Z}_{\rm DR}(Z'/W_N)\xr{\simeq} W_n\omega_Z,\]
\[\sigma_X: \epsilon_{X*}\sH^c_{Z}(\omega_{X'/W_n})\to \epsilon_{X*}\sH^c_Z(\sH^{N_X}_{\rm DR}(X'/W_n))\xr{\simeq} \sH^c_Z( W_n\omega_X).\]
We obtain maps
\eq{1.7.10.2}{\xymatrix{ & \Hom(\epsilon_{X*} i_{n*}\omega_{Z'/W_n}, \epsilon_{X*}\sH^c_{Z} (\omega_{X'/W_n} ))\ar[d]^{\sigma_X}\\
            \Hom(i_{n*} W_n\omega_Z, \sH^c_Z (W_n\omega_X))\ar@{^(->}[r]^-{D(\sigma_Z)} & \Hom(\epsilon_{X*} i_{n*}\omega_{Z'/W_n},\sH^c_Z( W_n\omega_X)), } }
with the horizontal map being injective (since $\sigma_Z$ is surjective).
We denote by $a$ the following composition
\[a:=H^0(\tau_{X,n})^{-1}\circ H^0(\Tr_{i_n}\circ c_{i_n, \rho_{X,n}}\circ\tau_{Z,n}) \in \Hom(i_{n*} W_n\omega_Z, \sH^c_Z (W_n\omega_X)). \]
The commutativity of the diagram in the statement of the proposition means, that $a$ equals $\imath_{Z,n}$. Thus it is enough to show
\eq{1.7.10.3}{ D(\sigma_Z)(a)=D(\sigma_Z)(\imath_{Z,n}) \quad \text{in } \Hom(\epsilon_{X*} i_{n*}\omega_{Z'/W_n},\sH^c_Z (W_n\omega_X)).}

We define $a'\in \Hom(i'_*\omega_{Z'/W_n}, \sH^c_Z(\omega_{X/W_n}))$ to be the following composition

\[a': H^0\left( i'_*\omega_{Z'/W_n}\simeq  i'_* \rho_{Z'}^!\Lambda\simeq i'_*{i'}^!\rho_{X'}^!\Lambda\xr{\Tr_{i'}} R\ul{\Gamma}_{Z}( \rho_{X'}^!\Lambda)\simeq R\ul{\Gamma}_{Z}(\omega_{X'/W_n})[c]   \right).\]
Then diagram \eqref{1.7.10.1} says
\eq{1.7.10.4}{D(\sigma_Z)(a)= \sigma_X(\epsilon_{X*}(a')). }

We define $\imath_{Z'}$ by
\[\imath_{Z'}: i_{Z'*}\omega_{Z'/W_n}\to \sH^{c}_Z(\omega_{X'/W_n}),\quad \beta\mapsto (-1)^c\genfrac{[}{]}{0pt}{}{dt'\tilde\beta}{t'},\]
with $\tilde\beta\in\Omega^{N_Z}_{X'/W_n}$ a lift of $\beta$.  
We have
\eq{1.7.10.5}{D(\sigma_Z)(\imath_{Z,n}) = \sigma_X(\epsilon_{X*}(\imath_{Z'})).}
Indeed, this follows from  the concrete description of $\sigma_X$ and $\sigma_Z$ given by Lemma \ref{1.6.12.6} and from the fact that by \ref{1.7.6}, 1) the following equality holds for all $\gamma\in W_{n+1}\Omega^{N_Z}_{X'/W_n}$
\[\genfrac{[}{]}{0pt}{}{dt'F^n(\gamma)}{t'}= \genfrac{[}{]}{0pt}{}{{t'}^{p^n-1}dt'F^n(\gamma)}{{t'}^{p^n}}=F^n\left(\genfrac{[}{]}{0pt}{}{d[t']\gamma}{[t']}\right),\]
where we set $ {t'}^{p^n-1}:= {t'_1}^{p^n-1}\cdots {t'_c}^{p^n-1}$.
By \eqref{1.7.10.4} and \eqref{1.7.10.5} we are thus reduced to show 
\[ \imath_{Z'}= a' \quad \text{in } \Hom(i'_*\omega_{Z'/W_n}, \sH^c_{Z}( \omega_{X'/W_n} )),\]
which is well-known (see e.g. \cite[Lemma A.2.2]{CR}).
\end{proof}

\section{Pushforward and pullback for Hodge-Witt cohomology with supports}

\noindent In this section all schemes are assumed to be quasi-projective over $k$. We fix a $k$-scheme $S$.

\subsection{Relative Hodge-Witt cohomology with supports}
\begin{definition}\label{definition-Sm_*^*}
We denote by $(Sm_*/S)$ and $(Sm^*/S)$ the following two categories: Both have as objects pairs $(X, \Phi)$ with $X$ an $S$-scheme, which is smooth and {\em quasi-projective} over $k$ (for short we will say $X$ is a smooth $k$-scheme over $S$) and $\Phi$ a family of supports on $X$ and the morphisms are given by
\[\Hom_{Sm_*}((X,\Phi), (Y, \Psi))= \{f\in \Hom_S(X, Y)\,|\,  f_{|\Phi} \text{ is proper and } f(\Phi)\subset \Psi \} \]
 and 
\[\Hom_{Sm^*}((X,\Phi), (Y, \Psi))=\{f\in \Hom_S(X, Y)\,|\,  f^{-1}(\Psi)\subset \Phi \}.\]
\end{definition}

 If $X$ is a smooth $k$-scheme over $S$ and $Z\subset X$ a closed subset, we write $(X, Z)=(X, \Phi_Z)$ and $X=(X,X)=(X, \Phi_X)\in {\rm Obj}(Sm_*)={\rm Obj}(Sm^*)$ 
(see Definition \ref{def:famsupp} for the notation).

We will say that a morphism $f:\XP\to \YP$ in $(Sm_*/S)$ or $(Sm^*/S)$ is \'etale, flat, smooth, etc. if the corresponding property holds for the underlying morphism of schemes $X\to Y$.
We will say that a diagram
\[\xymatrix{ (X',\Phi')\ar[r]\ar[d] & (Y',\Psi')\ar[d]\\
             \XP\ar[r] & \YP}\]
is cartesian, if the underlying diagram of schemes is cartesian.


\subsubsection{}\label{2.1.1}
For $(X, \Phi)\in {\rm obj}(Sm_*/S)={\rm obj}(Sm^*/S)$, with structure map $a: X\to S$ we denote by $\sH(\XP/S)$ the de Rham-Witt system 
\[ \sH(\XP/S):= \bigoplus_{i\ge 0} R^i a_\Phi W_\bullet\Omega_X\in \drw_S\]
 and its level $n$ part by
$ \sH_n(\XP/S)$. We denote by $\hat{\sH}(\XP/S)$ the de Rham-Witt module
\[\hat{\sH}(\XP/S):= \bigoplus_{i\ge 0} R^i\hat{a}_\Phi W_\bullet\Omega_X\in \widehat{\drw}_S.\]
We write
\[\sH^{i,q}(\XP/S):=R^i a_\Phi W_\bullet\Omega_X^q,\quad \hat{\sH}^{i,q}(\XP/S):=R^i \hat{a}_\Phi W_\bullet\Omega_X^q.\]

By definition we have:
\begin{enumerate}
 \item If $Z\subset X$ is closed then, $\hat{\sH}((X,Z)/S)=\bigoplus R^i a_Z W\Omega_X$. In particular $\hat{\sH}((X,Z)/X)= \bigoplus \sH^i_Z(W\Omega_X)$, in particular $\hat{\sH}(X/X)=W\Omega_X$.
\item Let $U\subset S$ be open and set $X_U:= X\times_S U$ and $\Phi_U=\Phi\cap X_U$. Then 
       \[\sH(\XP/S)_{|U}= \sH((X_U, \Phi_U)/U),\quad \hat{\sH}(\XP/S)_{|U}= \hat{\sH}((X_U,\Phi_U)/U).\]
\end{enumerate}

\subsection{Pullback}\label{2.1.1.5}

 \begin{definition}\label{2.1.2}
Let $f: \XP\to \YP$ be a morphism in $(Sm^*/S)$. Then applying $R\ul{\Gamma}_\Psi$ to the functoriality morphism $W_\bullet\Omega_Y\to Rf_*W_\bullet\Omega_X$ in $D^+(\drw_Y)$
and composing it with the natural map $R\ul{\Gamma}_\Psi Rf_*= Rf_* R\ul{\Gamma}_{f^{-1}(\Psi)}\to Rf_*R\ul{\Gamma}_\Phi$
yields a morphism in $D^+(\drw_Y)$
\eq{2.1.2.1}{R\ul{\Gamma}_\Psi(W_\bullet\Omega_Y)\to R\ul{\Gamma}_\Psi Rf_*W_\bullet\Omega_X=Rf_*R\ul{\Gamma}_{f^{-1}(\Psi)}W_\bullet\Omega_X\to  Rf_*R\ul{\Gamma}_\Phi W_\bullet\Omega_X.}

Let $a: Y\to S$ be the structure morphism. Then we define  
\[f^*:= \bigoplus_i H^i(Ra_*\eqref{2.1.2.1}): \sH(\YP/S)\to  \sH(\XP/S) \quad \text{in } \drw_S\]
and 
\[f^*:=\bigoplus_i H^i(R\varprojlim Ra_*\eqref{2.1.2.1}) :\hat{\sH}(\YP/S)\to \hat{\sH}(\XP/S)\quad \text{in }\widehat{\drw}_S.\]
We call these morphisms the {\em pullback morphisms}. Notice that by definition $f^*$ always factors as
\[f^*: \sH(\YP/S)\to \sH((X, f^{-1}(\Psi)/S))\xr{\rm nat.} \sH(\XP/S), \]
same with $\hat{\sH}$.
\end{definition}

\begin{proposition}\label{2.1.3}
The assignments  
\[\sH: (Sm^*/S)^o\to \drw_S,\quad \XP\mapsto \sH(\XP/S)\]
and
\[\hat{\sH}: (Sm^*/S)^o \to \cdrw_S,\quad \XP\mapsto \hat{\sH}(\XP/S)\]
define functors, where we set $\sH(f)=f^*$ and $\hat{\sH}(f)=f^*$, for a morphism  $f:\XP\to\YP$ in $(Sm^*/S)$.

Furthermore, if $U\subset S$ is open and $f_U: (X_U, \Phi_U)\to (Y_U,\Psi_U)$ is the pullback of $f:\XP\to \YP$ over $U$, then
$f_U^*= (f^*)_{|U}$ on $\sH((Y_U,\Psi_U)/U)=\sH(\YP/S)_{|U}$ (resp. on $\hat{\sH}$).
\end{proposition}

\begin{proof}
 This is all straightforward.
\end{proof}

\subsection{Pushforward}\label{2.1.5}
Let $X$ be a smooth $k$-scheme of pure dimension $N$ with structure map $\rho: X\to \Spec k$. Recall that we set $K_X=\rho_X^\Delta W_\bullet \omega$
and that this is isomorphic to $E(W_\bullet\omega_X)(N)[N]$, where $E$ denotes the Cousin complex with respect to the codimension filtration (see Corollary \ref{1.6.17}).
Furthermore we have the  dualizing functor 
\[D_X(-)=\sHom(-, K_X)\]
(see \ref{dualizing-functor}), which we may view as a functor from $C(\drw_{X, {\rm qc}})^o$ to $C(\drw_X)$
or from $D(\drw_{X, {\rm qc}})^o$ to $D(\drw_{X})$.

For any morphism $f:\XP\to \YP$ in $(Sm_*/S)$ the composition of the isomorphism 
$K_X\simeq f^\Delta K_Y$ (see Proposition \ref{1.6.11}) with $\vartheta_f$ from Proposition \ref{1.6.14.4.5}, 
defines a natural transformation of functors on $C(\drw_{X,{\rm qc}})^o$ 
\eq{2.1.5.1}{\theta_f: f_\Phi D_X(-)\to \ul{\Gamma}_\Psi D_Y(f_*(-)). }
The map $\theta_f$ is compatible with compositions and \'etale base change. Furthermore, recall that the multiplication pairing 
$W_\bullet\Omega^i_X\otimes_W W_\bullet\Omega^j_X\to W_\bullet\Omega_X^{i+j}$ induces an isomorphism  in $C^b(\drw_{X, {\rm qc}})$ (see Corollary \ref{1.6.18})
\eq{2.1.5.2}{\mu_X: E(W_\bullet\Omega_X)\xr{\simeq} D_X(W_\bullet\Omega_X)(-N)[-N].}
The map $\mu_X$ is compatible with \'etale base change.

\begin{definition}\label{2.1.6}
Let $f: \XP\to \YP$ be a morphism in $(Sm_*/S)$. We assume that $X$ and $Y$ are of pure dimension $N_X$ and $N_Y$.
Consider the following  composition of morphisms of complexes of de Rham-Witt systems on $Y$
\ml{2.1.6.1}{ f_\Phi E(W_\bullet\Omega_X)(N_X)[N_X]\xr{\mu_X} f_\Phi D_X(W_\bullet\Omega_X)\xr{\theta_f}\ul{\Gamma}_\Psi D_Y(f_*W_\bullet\Omega_X)\\
                   \xr{D(f^*)} \ul{\Gamma}_\Psi D_Y(W_\bullet\Omega_Y)\xr{\mu^{-1}_Y}\ul{\Gamma}_\Psi E(W_\bullet\Omega_Y)(N_Y)[N_Y].}
By Lemma \ref{drw-is-CM} this induces a morphism in $D^b(\drw_Y)$
\eq{derived-pf}{Rf_\Phi W_\bullet\Omega_X\to R\ul{\Gamma}_\Psi W_\bullet\Omega_Y(-r)[-r],} 
where $r=N_X-N_Y$ is the relative dimension.

Let $a: Y\to S$ be the structure morphism. Then we define
\[f_*:= \bigoplus_i H^i(Ra_* \eqref{derived-pf}): \sH(\XP/S)\to \sH(\YP/S)(-r)\quad \text{in } \drw_S\]
and
\[f_*:= \bigoplus_i H^i(R\varprojlim Ra_* \eqref{derived-pf}): \hat{\sH}(\XP/S)\to \hat{\sH}(\YP/S)(-r)\quad \text{in } \cdrw_S.\]
We extend these definitions additively to all morphisms $f:\XP\to \YP$ in $(Sm_*/S)$ (where $X$ and $Y$ don't need to be pure dimensional).
We call these morphisms the  {\em push-forward morphisms}. Notice, that since $f_{|\Phi}$ is proper, $f(\Phi)$ is a family of supports on $Y$ and
$f_*$ always factors as
\[f_*: \sH(\XP/S)\to \sH((Y, f(\Phi))/S)(-r)\xr{\rm nat.} \sH(\YP/S)(-r),\]
 same with $ \hat{\sH}$.
\end{definition}

\begin{remark}\label{2.1.7}
Without supports and for a {\em proper} morphism between smooth $k$-schemes a similar push-forward was already defined in \cite[II, Def. 1.2.1]{Gros}.
(Although the verification, that it is compatible with $F$, $V$, $d$ and $\pi$ was omitted.)
\end{remark}

\begin{proposition}\label{2.1.9}
The assignments 
\[\sH: (Sm_*/S)\to \drw_S,\quad \XP\mapsto \sH(\XP/S)\]
and 
\[\hat{\sH}: (Sm_*/S) \to \cdrw_S,\quad \XP\mapsto \hat{\sH}(\XP/S)\]
define functors, where we set $\sH(f)=f_*$ and $\hat{\sH}(f)=f_*$, for $f:\XP\to\YP$ a morphism in $(Sm_*/S)$.

Furthermore: (1) If $U\subset S$ is open and $f_U: (X_U, \Phi_U)\to (Y_U,\Psi_U)$ is the pullback of $f:\XP\to \YP$ over $U$, then
$f_{U*}= (f_*)_{|U}$ on $\sH((Y_U,\Psi_U)/U)=\sH(\YP/S)_{|U}$ (resp. on $\hat{\sH}$).
(2) If $j:(U,\Phi)\inj \XP$ is an open immersion in $(Sm_*/S)$, then $j_*$ is the excision isomorphism.
\end{proposition}
\begin{proof}
 This follows from the corresponding properties of \eqref{2.1.6.1}, which follow from Proposition \ref{1.6.14.4.5} and Corollary \ref{1.6.18}.
\end{proof}

 For later use we record:

\begin{proposition}[{\cite[II, Prop. 4.2.9]{Gros}}]\label{2.1.8}
 Let $f: X\to Y$ be a finite morphism between two connected smooth $k$-schemes. We may view $f$ as a morphism in $(Sm_*/Y)$ or as a morphism in $(Sm^*/Y)$.
Then the composition 
\[f_*\circ f^*: W_\bullet\Omega_Y\to W_\bullet\Omega_Y\]
equals multiplication with the degree of $f$.
\end{proposition}

\subsection{Compatibility of pushforward and pullback}\label{section-Compatibility}

\begin{proposition}[Gros]\label{1.7.9}
Let $i: Z\inj (X,Z)$ be a closed immersion of pure codimension $c$ in $(Sm_*/S)$.
Then 
\[R\ul{\Gamma}_Z W_\bullet \Omega_X[c]\cong \sH^c_Z(W_\bullet\Omega_X) \quad \text{in } D^+(\drw_X).\]
Suppose further the ideal sheaf of $Z$ in $\sO_X$ is generated by a regular sequence $t=t_1,\ldots, t_c$  of global sections of $X$, then the projection of 
$i_*: \sH(Z/X)\to \sH((X,Z)/X)$  to the $n$-th level is given by 
  \eq{1.7.9.1}{{W_n(i)}_*W_n\Omega_Z\to \sH^c_Z (W_n\Omega_X)(c), 
                                 \quad \alpha\mapsto (-1)^c\genfrac{[}{]}{0pt}{}{d[t]\tilde\alpha}{[t]}, }
       where we set $[t]:=[t_1], \ldots, [t_c]$ and $d[t]:=d[t_1]\cdots d[t_c]\in W_n\Omega^c_X$, 
       with $[t_i]\in W_n\sO_X$ the Teichm\"uller lift of $t_i$, and $\tilde{\alpha}$ is any lift
       of $\alpha\in {W_n(i)}_*W_n\Omega_Z$ to $W_n\Omega_X$.
\end{proposition}
The above proposition is proved in \cite[II, 3.4]{Gros}, but since the proof uses a result by Ekedahl, 
which is referred to as work in progress and which we could not find in the literature we reprove the above proposition 
(using Proposition \ref{1.7.10}, instead of a comparison of Ekedahl's trace with Berthelot's trace in crystalline cohomology as Gros does). 

\begin{proof}
The first statement is proven as in \eqref{1.7.10.0}. It remains to prove the explicit description \eqref{1.7.9.1}. Let 
$\rho_{Z,n}: W_n Z\to \Spec W_n$ and $\rho_{X,n}: W_n X\to \Spec W_n$ be the structure maps,
$i_n:=W_n(i) : W_n Z\inj W_n X$ the closed immersion 
(for this proof we don't use the $i_n$ defined in \ref{1.1}) and write $\dim X=N_X$, $\dim Z=N_Z$.  By Definition \ref{2.1.6} the projection to the $n$-th level of $i_*$ is given by the 
following composition in $D^b_{\rm qc}(\sC_{X,n})$ 
\begin{eqnarray*}
i_{n*}W_n\Omega_Z (N_Z)[N_Z]   & \xr{\mu_{Z,n},\, \simeq} &  i_{n*} R\sHom(W_n\Omega_Z, K_{Z,n})\\
                               &  \xr{\theta_{i_n}}&   R\ul{\Gamma}_Z R\sHom(i_{n*}W_n\Omega_Z, K_{X,n} )\\
               & \xr{ D(i^*_n)} & R\sHom(W_n\Omega_X, \ul{\Gamma}_Z K_{X,n} )\\
              & \xr{\mu_{X,n}^{-1}, \, \simeq} & R\ul{\Gamma}_Z(W_n\Omega_X)(N_X)[N_X].
\end{eqnarray*}
Here $\mu_{Z,n}$, $\mu_{X,n}$ and $\theta_{i_n}$ are images in the derived category of the projections to the $n$-th level of the corresponding maps from \ref{2.1.5}.
We denote by $\imath_{Z,n}$ the following morphism
\[\imath_{Z,n}: (W_n i)_*W_n\omega_Z\lra  \sH^c_Z(W_n\omega_X)(c), \quad 
\alpha\mapsto (-1)^c\genfrac{[}{]}{0pt}{}{d[t]\tilde{\alpha}}{[t]},\] 
with $\tilde{\alpha}\in W_n\Omega^{N_Z}_X$ any lift of $\alpha$. Then it follows from Proposition \ref{1.7.10}, that in the derived category $(i_*)_n$ equals
\begin{eqnarray*}
i_{n*}W_n\Omega_Z(N_Z)[N_Z]   & \xr{\text{multipl.}} &  \sHom(i_{n*}W_n\Omega_Z, i_{n*}W_n\omega_Z(N_Z)[N_Z])\\
                     & \xr{\imath_{Z,n}} &\sHom(i_{n*}W_n\Omega_Z,\sH^c_Z(W_n\omega_X)(N_X)[N_Z] ) \\
               & \xr{D(i_n^*)}& \sHom(W_n\Omega_X, \sH^c_Z(W_n\omega_X)(N_X)[N_Z])\\
               & \xr{\simeq,\,(*)} & \sH^c_Z(W_n\Omega_X)(N_X)[-c][N_X]\\
              &\xr{\simeq} & R\ul{\Gamma}_Z (W_n\Omega_X)(N_X)[N_X].
\end{eqnarray*}
Here the  isomorphism $(*)$ is the inverse of
\[\sH^c_Z(W_n\Omega_X)(N_X)\xr{\simeq} \sHom(W_n\Omega_X, \sH^c_Z(W_n\omega_X)), \quad
   \genfrac{[}{]}{0pt}{}{\alpha}{t^n}\mapsto \left(\beta\mapsto  \genfrac{[}{]}{0pt}{}{\alpha\beta}{t^n}\right).\]
It is straightforward to check, that taking $H^{-N_Z}(-)$ of this composition gives \eqref{1.7.9.1} and hence the claim.
\end{proof}

\begin{corollary}\label{2.1.9.5}
Let $i: Y\inj X$ be a closed immersion between smooth affine $k$-schemes and assume that the ideal of $Y$ in $X$ is generated by a regular sequence $t_1,\ldots, t_c$.
Let $Z\subset Y$ be a closed subset which is equal to the vanishing set $V(f_1,\ldots f_i)$ of global sections $f_1,\ldots, f_i\in \Gamma(Y,\sO_Y)$. Denote by
$i_Z: (Y,Z)\inj (X,Z)$ the morphism in $(Sm_*/\Spec k)$ induced by $i$. Then the projection  to the $n$-th level of $i_{Z*}$ is given by
\[i_{Z*,n}: H^i_Z(Y,W_n\Omega_Y)\to H^{i+c}_Z(X, W_n\Omega_X)(c),\quad \genfrac{[}{]}{0pt}{}{\alpha}{[f]}\mapsto (-1)^c \genfrac{[}{]}{0pt}{}{d[t]\tilde{\alpha}}{[t],[f]},\]
where $\alpha\in \Gamma(Y,W_n\Omega_Y)$, $\tilde{\alpha}\in \Gamma(X, W_n\Omega_X)$ is some lift, $[f]=[f_1],\ldots, [f_i]$ and $[t]$, $d[t]$ as in in Lemma \ref{1.7.9}.
\end{corollary}

\begin{proof}
Choose $\tilde{f}_1,\ldots, \tilde{f}_i\in \Gamma(X, \sO_X)$ lifts of $f_1,\ldots,f_i$ and define the set $\tilde{Z}\subset X$ as the vanishing locus of these $\tilde{f}_i$'s.
Denote by $i_{Y/X}$ the closed immersion $Y\inj (X,Y)$ viewed as a morphism in $(Sm_*/X)$. Then it follows from the first part of Proposition \ref{1.7.9} that
$i_{Z*,n}$ equals 
\mlnl{H^i_{\tilde{Z}}(X, i_{Y/X*,n}): H^i_Z(Y, W_n\Omega_Y)\cong H^i_{\tilde{Z}}(X, i_*W_n\Omega_Y)\\ \lra H^i_{\tilde{Z}}(X, \sH_Y^c(W_n\Omega_X))\cong H^{i+c}_{Z}(X, W_n\Omega_X)[c].}
Now denote by $[t]=[t_1],\ldots, [t_c]$ and $[\tilde{f}]=[\tilde{f}_1],\ldots, [\tilde{f}_i]$ the Teichm\"uller lifts of the sequences $t$ and $f$ to $\Gamma(X,W_n\sO_X)$ and
by $\sK([t])$ and $\sK([\tilde{f}])$ the corresponding sheafified Koszul complexes. We define a morphism of complexes
\[\varphi: W_n\Omega_X[0]\to \sHom(\sK^{-\bullet}([t]), W_n\Omega_X(c))[c], \quad \alpha\mapsto (e_1\wedge\ldots\wedge e_c\mapsto (-1)^c d[t]\alpha) \]
where $e_1,\ldots, e_c$ is the standard basis of $(W_n\sO_X)^c$. Then the second part of Proposition \ref{1.7.9} maybe rephrased by saying that the diagram
\[\xymatrix{W_n\Omega_X\ar[d]_{i^*}\ar[r]^-{\varphi} & \sHom(\sK^{-\bullet}([t]), W_n\Omega_X(c))[c]\ar[d]^{\eqref{1.7.6.1}}\\
            i_*W_n\Omega_Y\ar[r]^-{i_{Y/X*}} & \sH_Y^c(W_n\Omega_X)(c) }\]
commutes. This yields the following commutative diagram of complexes (in which the lower square is concentrated in degree 0):
\[\xymatrix{\Hom(\sK^{-\bullet}([\tilde{f}]), W_n\Omega_X)[i]\ar[r]^-{\varphi}\ar[d]_{\eqref{1.7.6.1}} & \Hom(\sK^{-\bullet}([\tilde{f}]), \sHom(\sK^{-\bullet}([t]), W_n\Omega_X(c))[c])[i]\ar[d]^{\eqref{1.7.6.1}}\\
              H^i_{\tilde{Z}}(X,W_n\Omega_X)\ar[d]_{i^*}\ar[r]^-{\varphi} & H^i_{\tilde{Z}}(X,\sHom(\sK^{-\bullet}([t]), W_n\Omega_X(c))[c])\ar[d]^{\eqref{1.7.6.1}} \\
            H^i_{\tilde{Z}}(X, i_*W_n\Omega_Y)\ar[r]^-{i_{Y/X*}} & H^i_{\tilde{Z}}(X,\sH_Y^c(W_n\Omega_X)(c)).   }\]
Now the claim follows from $\sK([t])\otimes_{W_n\sO_X}\sK([\tilde{f}])=\sK([t], [\tilde{f}])$ and the commutativity of the following diagram, in which we denote $W_n\Omega_X(c)$ simply by $\Lambda$ 
\[\xymatrix{ \Hom(\sK^{-\bullet}([\tilde{f}]), \sHom(\sK^{-\bullet}([t]), \Lambda)[c])[i]\ar[d]\ar[r]^-{\simeq} & \Hom(\sK^{-\bullet}([t],[\tilde{f}]), \Lambda)[c+i]\ar[d]\\
             H^i_{\tilde{Z}}(X, \sH^c_Y(\Lambda))\ar[r]^-\simeq  &  H^{i+c}_{Z}(X, \Lambda).  }\]
\end{proof}

\begin{lemma}\label{2.1.10}
Let $Y$ be a smooth $k$-scheme, $f:X=\P^r_Y\to Y$ the projection and $\Psi$ a family of supports on $Y$. We denote by 
$f_{\Psi/S}:(X, f^{-1}(\Psi))\to (Y,\Psi)$ the morphism in $(Sm_*/S)$ induced by $f$ and by 
\[f_{\Psi/S*}: \sH((X, f^{-1}(\Psi))/S)\to \sH(\YP/S)(-r)\]
the  corresponding push-forward (same for $\hat{\sH}$). Let $a: Y\to S$ be the structure morphism. Then:
\begin{enumerate}
 \item The higher direct images $R^if_*W_\bullet\Omega_X$ vanish for $i\ge r+1$. Therefore there is a natural morphism in $D^+(\drw_Y)$
       \[e:Rf_*W_\bullet\Omega_X\to R^rf_*W_\bullet\Omega_X[-r]\]
       and the push-forward $f_{\Psi/S*}$ factors as
     \[\bigoplus_i H^i(Ra_\Psi Rf_*W_\bullet\Omega_X)\xr{e} \bigoplus_i R^{i-r}a_\Psi( R^r f_*W_\bullet\Omega_X)\xr{f_{Y/Y*}} \bigoplus_i Ra^{i-r}_\Psi W_\bullet\Omega_Y(-r).\]
 Similar for $\hat{\sH}$ (replace $a_\Psi$ by $\hat{a}_\Psi$ in the above formula).
 \item The push-forward $f_{Y/Y*}: \sH(X/Y)\to \sH(Y/Y)(-r)$ induces an isomorphism
        \[f_{Y/Y*}:R^r f_*W_\bullet \Omega_X\xr{\simeq} W_\bullet \Omega_Y(-r),\]
        with inverse induced by the push-forward $i_*:\sH_\bullet(Y/Y)(-r)\to \sH_\bullet(X/Y)$, where $i: Y\to X$ is any section of $f$.
\end{enumerate}

\end{lemma}

\begin{proof} First of all observe that it suffices to prove the statements on each finite level $n$ separately.
 
(1) By \cite[I, Cor 4.1.12]{Gros} we have an isomorphism in $D(W_n\Omega_Y-\text{dga})$ (with $W_n\Omega_Y-\text{dga}$ the category of differential graded $W_n\Omega_Y$-modules)
\eq{2.1.10.1}{Rf_*W_n\Omega_X \simeq \bigoplus_{i=0}^r W_n\Omega_Y(-i)[-i].}
This gives the first statement of (1) and in particular an isomorphism in $D(W_n\Omega_Y-\text{dga})$
\eq{2.1.10.2}{Rf_*W_n\Omega_X \simeq \bigoplus_{i=0}^r R^if_*W_n\Omega_X[-i]}
and the morphism $e$ from (1) is the projection to the $r$-th summand. On the other hand $(f_{Y/Y*})_n$ is induced by a derived category morphism \eqref{derived-pf} 
 \eq{2.1.10.3}{Rf_*W_n\Omega_X\to W_n\Omega_Y(-r)[-r].}
 Since the composition
\[R^if_*W_n\Omega_X[-i]\xr{\text{via } \eqref{2.1.10.2}} Rf_*W_n\Omega_X\xr{\eqref{2.1.10.3}} W_n\Omega_Y(-r)[-r]\]
is an element in $\Ext^{i-r}(R^if_*W_n\Omega_X, W_n\Omega_Y(-r))$, this composition is zero for all $i<r$.
Thus the morphism \eqref{2.1.10.3} factors over $e$. This implies the second statement of (1).

(2) It suffices to check that $f_{Y/Y*}$ induces an isomorphism; $i_*$ will then automatically be the inverse for any section $i$ of $f$.
(Notice that by \eqref{2.1.10.1} we know $R^r f_*W_\bullet \Omega_X\simeq W_\bullet \Omega_Y(-r)$ but we don't know that this isomorphism is induced by $f_{Y/Y*}$.)
To this end, let $R$ be the Cartier-Dieudonn\'e-Raynaud ring of $k$, see \cite[I]{IR}. 
($R$ is a graded (non-commutative) $W$-algebra generated by symbols $F$ and $V$ in degree 0 and $d$ in degree 1, satisfying the obvious relations). 
Set $R_n:=R/V^nR+dV^nR$. We denote by $D^b(Y,R)$ the bounded derived category of $R$-modules. 
Any de Rham-Witt module is in particular an $R$-module. Thus we may view  $R^i\hat{f}_*W_\bullet\Omega_X=R^if_*W\Omega_X$, $i\ge 1$, as an object in $D^b(Y,R)$.
In $D^b(Y, R)$ we have the following isomorphisms 
\[R_n\otimes^L_R R^if_*W\Omega_X\simeq R_n\otimes_R^L W\Omega_Y(-i)\simeq W_n\Omega_Y(-i)\simeq R^if_*W_n\Omega_X, \]
where the first and the last isomorphism follow from \cite[I, Thm. 4.1.11, Cor. 4.1.12]{Gros} and the second isomorphism is \cite[II, Thm (1.2)]{IR}.
Since $f_{Y/Y*}$ is clearly a morphism of $R$-modules it follows that the projection of $f_{Y/Y*}$ to the $n$-th level, identifies with
\[R_n\otimes^L_R f_{Y/Y*}: R_n\otimes^L_R R^rf_*W\Omega_X\to R_n\otimes^L_R W\Omega_Y(-r).\]
We thus have to show that this is an isomorphism for all $n$. Now Ekedahl's version of Nakayama's Lemma  \cite[I, Prop. 1.1, Cor. 1.1.3]{EII} (also \cite[Prop. 2.3.7]{IlE}) implies that it is in fact enough
to prove that the projection of $f_{Y/Y*}$ to level 1 - in the following simply denoted by $f_*$ - is an isomorphism, i.e. we have to show (for all $q\ge 0$)
\[f_*: R^r f_*\Omega_X^q\xr{\simeq} \Omega_Y^{q-r}.\]
In fact, $f_*$ is given by $H^r$ of the following morphism in $D^b_c(\sO_Y)$
\mlnl{Rf_*\Omega_X^q\simeq Rf_*R\sHom(\Omega_X^{N_X-q}, \omega_X)\simeq R\sHom(Rf_*\Omega_X^{N_X-q}, \omega_Y)[-r]\\ \xr{D(f^*)[-r]} R\sHom(\Omega^{N_X-q}_Y, \omega_Y)[-r]\simeq \Omega^{q-r}_Y[-r], }
with $N_X=\dim X$.
Thus it is enough to show that $H^r(D(f^*)[-r])=H^0(D(f^*))$ is an isomorphism. This follows from the well-known formula 
\[Rf_*\Omega_X^{N_X-q}\simeq \bigoplus_{i=0}^r \Omega^{N_X-q-i}_Y[-i],\]
which might be proved using the K\"unneth decomposition $\Omega^q_X=\oplus_{i+j=q}f^*\Omega_Y^i\otimes_{\sO_X}\Omega^j_{X/Y}$ and \cite[Exp. XI, Thm 1.1, (ii)]{SGA7II}.
\end{proof}

\begin{proposition}\label{1.7.12}
 Let 
\[\xymatrix{(X',\Phi')\ar[r]^{f'}\ar[d]_{g_X} & (Y',\Psi')\ar[d]^{g_Y}\\
            \XP\ar[r]^f & \YP}\]
be a cartesian diagram, with $f,f'$ morphisms in $(Sm_*/S)$ and $g_X,g_Y$ morphisms in $(Sm^*/S)$. Assume that {\em one} of the following conditions is satisfied
\begin{enumerate}
 \item $g_Y$ is flat, or
 \item $g_Y$ is a closed immersion and $f$ is transversal to $g_Y$ (i.e. ${f'}^*\sN_{Y'/Y}=\sN_{X'/X}$).
\end{enumerate}
Then the following diagram commutes 
\eq{1.7.12.1}{\xymatrix{ \sH((X',\Phi')/S)\ar[r]^{f'_*} & \sH((Y',\Psi')/S)\\
                  \sH(\XP/S)\ar[u]^{g_X^*}\ar[r]^-{f_*} & \sH(\YP/S)\ar[u]_{g_Y^*} } }
and also with $\sH$ replaced by $\hat{\sH}$.
\end{proposition}
In \cite[II, Prop. 2.3.2 and p.49 ]{Gros} a version of this Proposition (in the derived category) is proved in the case without supports with $f$ a closed immersion and $g_Y$ transversal to $f$.
(This case is also covered here by factoring $g_Y$ as a closed immersion followed by a smooth projection.)

\begin{proof} We distinguish two cases.

{\em 1. case: $f$ is a closed immersion and $g_Y$ is either flat or transversal to $f$.}  Since $\Phi$ and $\Phi'$ are also families of supports on $Y$ and $Y'$  the push-forward $f_*$ factors over $\sH((Y,\Phi)/S)$ and $f'_*$ factors over
$\sH((Y',\Phi')/S)$. Hence we may assume $\Psi=\Phi$ and $\Psi'=\Phi'$. Furthermore by Definition \ref{2.1.2}, the pull-back $g_X^*$ factors over $\sH((X', g_X^{-1}(\Phi))/S)$ and $g_Y^*$ factors over 
$\sH((Y',g_Y^{-1}(\Phi))/S)$.
Hence we may assume that in the diagram \eqref{1.7.12.1} we have $\Psi=\Phi$, $\Phi'=g_X^{-1}(\Phi)$, $\Psi'=g_Y^{-1}(\Phi)$. 
We may further assume, that $X,X',Y,Y'$ are equidimensional and set $c:=\dim Y-\dim X=\dim Y'-\dim X'$ and  $h:=f\circ g_X=g_Y\circ f'$.

We consider the following diagram in $D^+(\drw_Y)$
\eq{1.7.12.1.5}{\xymatrix{ Rh_*W_\bullet\Omega_{X'}\ar[r]^-{Rg_{Y*}(f'_*)}  & Rg_{Y*}R\ul{\Gamma}_{X'}W_\bullet\Omega_{Y'}(c)[c]\\
                           Rf_*W_\bullet\Omega_X\ar[u]^{Rf_*(g_X^*)}\ar[r]^-{f_*} & R\ul{\Gamma}_X W_\bullet\Omega_Y(c)[c]\ar[u]_{R\ul{\Gamma}_X(g_Y^*)},          }}
where the horizontal (resp. vertical) morphisms are induced by \eqref{derived-pf} (resp. \eqref{2.1.2.1}). By our assumption on the supports we obtain diagram
\eqref{1.7.12.1} if we apply $\oplus H^i(Ra_\Phi(-))$ (resp. $\oplus H^i(R\varprojlim Ra_\Phi(-))$) to diagram \eqref{1.7.12.1.5}, where $a: Y\to S$ is the structure morphism. 
Thus it suffices to show that \eqref{1.7.12.1.5} commutes.
By Proposition \ref{1.7.9} the diagram \eqref{1.7.12.1.5} is isomorphic to the outer square of the following diagram
\eq{1.7.12.2}{\xymatrix{  Rh_*W_\bullet\Omega_{X'}\ar[r]^-{Rg_{Y*}(f'_*)} & Rg_{Y*}\sH^c_{X'}(W_\bullet\Omega_{Y'})(c)\\
                          h_*W_\bullet\Omega_{X'}\ar[r]^-{g_{Y*}(f'_*)}\ar[u]^{\text{nat.}} & g_{Y*}\sH^c_{X'}(W_\bullet\Omega_{Y'})(c)\ar[u]_{\text{nat.}}\\
                  f_* W_\bullet\Omega_X\ar[u]^{g_X^*}\ar[r]^-{f_*} & \sH_X^c( W_\bullet\Omega_Y)(c).\ar[u]_{g_Y^*}}}
The upper square obviously commutes. Thus it suffices to check that the lower square of sheaves of de Rham-Witt systems commutes.
This may be checked locally and we may hence assume that the ideal sheaf of $X$ in $\sO_Y$ is generated by a regular sequence
$t= t_1, \ldots, t_c$. By our assumptions on  $g_Y$, the sequence $g_Y^*t= g_Y^*t_1, \ldots, g_Y^*t_c$ is again a regular sequence generating the ideal sheaf of $X'$ 
in $\sO_{Y'}$. Now the statement follows from the explicit formula \eqref{1.7.9.1}.

{\em 2. case: $X=\P^r_Y$ and $f:X\to Y$ is the projection map.} 
 In this situation we also have $X'=\P^r_{Y'}$ and $f':X'\to Y'$ is the projection map. The morphisms $f$ and $f'$ factor over $(X,f^{-1}(\Psi))\to\YP$ and  
$(X', {f'}^{-1}(\Psi'))\to (Y',\Psi')$, which are morphisms in  $(Sm_*/S)$. 
  Similar to the first case we conclude that it is enough to consider the case $\Phi=f^{-1}(\Psi)$, $\Psi'=g_Y^{-1}(\Psi)$ and $\Phi'= {f'}^{-1}(\Psi')$ . 
Now Lemma \ref{2.1.10}, (1) reduces us to show the commutativity of the following diagram (with $h=f\circ g_X=g_Y\circ f'$) 
\eq{1.7.12.5}{\xymatrix{ R^rh_*W_\bullet\Omega_{X'}\ar[rr]^-{H^0(Rg_{Y*}(f'_*))}  & & g_{Y*}W_\bullet\Omega_{Y'}(-r)\\
                         R^rf_*W_\bullet\Omega_X\ar[u]^{g_X^*}\ar[rr]^-{H^0(f_*)} & & W_\bullet\Omega_Y(-r).\ar[u]_{g_Y^*}         }   }
But now let $i: Y\to X$ be a section and $i': Y'\to X'$ its pull-back, viewed as morphisms over $Y$. We obtain push-forwards  
\[i_*:  W_\bullet\Omega_Y(-r)\to R^rf_*W_\bullet\Omega_X,\quad i'_*: g_{Y*}W_\bullet\Omega_{Y'}(-r)\to R^rh_*W_\bullet\Omega_{X'}.\]
And since $g_X$ is transversal to $i$ we may apply case 1. to obtain
\[(g_Y^*f_*-f'_*g_X^*)\circ i_*= g_Y^*f_*i_*- f'_*i'_*g_Y^*=0.\]
Thus the commutativity of \eqref{1.7.12.5} follows from Lemma \ref{2.1.10}, (2), which tells us that $i_*$ is an isomorphism (with inverse $f_*$).

{\em Proof in the general case.} Since $f: \XP\to \YP$ is quasi-projective we may factor it as follows
\[\XP\xr{i} (U,\Phi)\xr{j} (P,\Phi)\xr{\pi} \YP,\]
where $P=\P^r_Y$,  $i:\XP\to (U, \Phi)$ is a closed immersion, $j:(U,\Phi)\to (P,\Phi)$ is an open immersion and $\pi: (P,\Phi)\to \YP$ is the projection.
Notice that we use the properness of $f_{|\Phi}$ to conclude that $\Phi$ is also a family of supports on $U$ and on $P$ and that all the maps above
are in $(Sm_*/S)$. By base change we obtain a diagram
\[\xymatrix{(X',\Phi')\ar@/^2pc/[rrr]^{f'}\ar[r]^{i'}\ar[d]_{g_X} & (U_{Y'},\Phi')\ar[d]_{g_U}\ar[r]^{j'} & (P_{Y'},\Phi')\ar[r]^{\pi'}\ar[d]_{g_P} & (Y',\Psi')\ar[d]_{g_Y}\\  
             \XP\ar@/_2pc/[rrr]_f\ar[r]^{i}                     & (U_Y,\Phi)\ar[r]^{j}                 & (P,\Phi)\ar[r]^{\pi}                    & \YP,
           }\]
with the horizontal maps in $(Sm_*/S)$ and the vertical maps in $(Sm^*/S)$. Notice that if $g_Y$ satisfies condition (1) or (2) from the statement of the Proposition, then so does $g_U$.
By the functoriality of the push-forward the commutativity of the square \eqref{1.7.12.1} follows from the first two cases and the fact
that $j_*$ and $j'_*$ are just the excision isomorphisms by Proposition \ref{2.1.9}, (2). This proves the proposition.

\end{proof}

\section{Correspondences and Hodge-Witt cohomology}

\noindent In this section all schemes are assumed to be quasi-projective over $k$. We fix a $k$-scheme $S$.

\subsection{Exterior product}

\subsubsection{}\label{2.5.1} Let $X$ and $Y$ be two $k$-schemes and $M$ and $N$ two complexes of de Rham-Witt systems on $X$ and $Y$ respectively. Then for all $n\ge 1$ we set
\[M_n\boxtimes N_n:= p_1^{-1}M_n\otimes_\Z p_2^{-1}N_n\quad \text{as a complex of abelian sheaves on } X\times Y,\]
where $p_1: X\times Y\to X$ and $p_2: X\times Y \to Y$ are the two projection morphisms. It is obvious, that the maps
$\pi_M\otimes \pi_N$ make the family $\{M_n\boxtimes N_n\}_{n\ge 1}$ into a pro-complex of abelian sheaves on $X\times Y$, which we denote by
\eq{2.5.1.1}{M\boxtimes N.}

\subsubsection{Godement resolution}\label{2.5.2}
Let $X$ be a $k$-scheme. For a sheaf of abelian groups $\sA$ on $X$ we 
denote by $G(\sA)$ its Godement resolution. Then there is a natural way to equip the family $\{G(W_n\Omega_X)\}_{n\ge 1}$ with the structure of a complex of de Rham-Witt systems on $X$,
which we denote by $G(W_\bullet\Omega_X)$. It follows from the exactness of $G(-)$ and the surjectivity of the transition maps $W_{n+1}\Omega_X\to W_n\Omega_X$, that
each term $G^q(W_\bullet\Omega_X)$ is a flasque de Rham-Witt system (in the sense of \ref{flasque-dRW}). Thus
the natural augmentation map $W_\bullet\Omega_X\to G(W_\bullet\Omega_X)$, makes $G(W_\bullet\Omega_X)$ a flasque resolution of de Rham-Witt systems of $X$.

\subsubsection{Exterior product for Godement resolutions}\label{2.5.3}
Let $X$ and $Y$ be two $k$-schemes with families of supports.
 It follows from the general construction in \cite[II, \S 6. 1]{God}, that there is a morphism of {\em pro-complexes of abelian sheaves on $X\times Y$}
\[G(W_\bullet\Omega_X)\boxtimes G(W_\bullet\Omega_Y)\lra G(W_\bullet\Omega_X\boxtimes W_\bullet\Omega_Y).\]
Therefore, multiplication induces a morphism of complexes
\eq{2.5.3.1}{G(W_\bullet\Omega_X)\boxtimes G(W_\bullet\Omega_Y)\lra G(W_\bullet\Omega_{X\times Y}).}

\begin{definition}\label{2.5.4}
 Let $(X,\Phi)$ and $(Y,\Psi)$ be in $(Sm_*/S)$ and denote the structure maps by $a: X\to S$ and $b: Y\to S$. Then we define the exterior products
\[\times : \sH^i(\XP/S)\boxtimes \sH^j(\YP/S)\to \sH^{i+j}((X\times Y, \Phi\times \Psi)/S\times S),\]
\[\times : \hat{\sH}^i(\XP/S)\boxtimes \hat{\sH}^j(\YP/S)\to \hat{\sH}^{i+j}((X\times Y, \Phi\times \Psi)/S\times S)\]
as the composition of $H^{i+j}((a\times b)_{\Phi\times\Psi}\eqref{2.5.3.1})$ (resp. $H^{i+j}(\widehat{(a\times b)}_{\Phi\times\Psi} \eqref{2.5.3.1})$ )
with the natural map 
\mlnl{\sH^i(\XP/S)\boxtimes \sH^j(\YP/S) \to H^{i+j}((a\times b)_{\Phi\times\Psi}( G(W_\bullet\Omega_X)\boxtimes G(W_\bullet\Omega_Y) ) )}
(resp. with $(\hat{\phantom{-}})$ ), where we use $\sH^i(\XP/S)= H^i(a_\Phi G(W_\bullet\Omega_X))$, etc.
\end{definition}

\begin{lemma}\label{2.5.5}
 Let $X=\Spec A$ and $Y=\Spec B$ be two smooth affine $k$-schemes and $Z_X\subset X$, $Z_Y\subset Y$ closed subsets.
Assume there are sections $s_1,\ldots,s_i\in A$ and $t_1,\ldots, t_j\in B$ whose vanishing set equals $Z_X$ and $Z_Y$ respectively.
Then for all $n\ge 1$ and $q, r\ge 0$ the morphism
\[\times: H^i_{Z_X}(X, W_n\Omega^q_X)\otimes_\Z H^j_{Z_Y}(Y, W_n\Omega^r_Y)\to H^{i+j}_{Z_X\times Z_Y}(X\times Y, W_n\Omega^{q+r}_{X\times Y}) \]
induced by Definition \ref{2.5.4} is given by
\[\genfrac{[}{]}{0pt}{}{\alpha}{[s_1],\ldots ,[s_i]}\times \genfrac{[}{]}{0pt}{}{\beta}{[t_1],\ldots, [t_j]}= \genfrac{[}{]}{0pt}{}{p_1^*\alpha\cdot p_2^*\beta}{p_1^*[s_1],\ldots, p_1^*[s_i], p_2^*[t_1],\ldots, p_2^*[t_j]}, \]
where we use the notation of \ref{1.7.6}, $\alpha\in W_n\Omega_A^q$, $\beta\in W_n\Omega^r_B$ and $p_1: X\times Y\to X$ and $p_2: X\times Y\to Y$ are the two projection maps. 
\end{lemma}

\begin{proof}
 We denote by $C$ the tensor product $A\otimes_k B$ and by $K_A$, $K_B$ and $K_C$ the respective Koszul complexes (see \ref{1.7.6}) $K^\bullet([t], W_n\Omega_A^q)$, $K^\bullet([s], W_n\Omega_B^r)$
and $K^\bullet(p_1^*[t], p_2^*[s], W_n\Omega^{q+r}_C)$. 
For any ordered tuple $L=(l_1<\ldots < l_\xi)$ of natural numbers $l_\nu\in [1, i+j]$ we denote  $\sigma(L):=\max\{\nu\in [0,\xi]\,|\, l_\nu\le i\}$, where we set $l_0:=0$. 
Notice that $\sigma(L)=0$ iff $l_1>i$ or $L=\emptyset$.
We define
\[L^{\le i}=\begin{cases}(l_1<\ldots < l_{\sigma(L)}), & \text{ if } \sigma(L)\in [1, \xi], \\ \emptyset, &\text{ if } \sigma(L)=0, \end{cases}\]
\[L^{>i}=\begin{cases} ((l_{\sigma(L)+1}-i)< \ldots < (l_\xi-i)), &\text{ if } \sigma(L)< \xi,\\
                                \emptyset, & \text{ if } \sigma(L)=\xi\text{ or } L=\emptyset.
               \end{cases}     \]
Thus the natural numbers appearing in the tuple $L^{\le i}$ are contained in $[1,i]$, whereas those appearing in $L^{\ge i}$ are contained in $[1,j]$.

Let $\{e_1,\ldots e_i\}$, $\{f_1,\ldots, f_j\}$ and $\{h_1,\ldots, h_{i+j}\}$ be the standard bases of $W_n(A)^i$, $W_n(B)^j$ and $W_n(C)^{i+j}$, respectively.
If $L=(l_1<\ldots <l_\xi)$ is an ordered tuple with $l_\nu\in [1, i]$, then we write $e_L= e_{l_1}\wedge \ldots \wedge e_{l_\xi}\in \bigwedge^\xi W_n(A)^i$.
In case $L$ is the emptyset we define $e_L$ to be $1\in W_n(A)$. Similarly we define the notations $f_L$ and $h_L$. 

With this notation we can define the morphism
\mlnl{\epsilon:\bigoplus_{\mu+\nu=\xi}\left(\Hom_{W_n(A)}(\bigwedge^\mu W_n(A)^i, W_n\Omega_A^q)\times \Hom_{W_n(B)}(\bigwedge^\nu W_n(B)^j, W_n\Omega_B^r)\right) \\ 
                \lra \Hom_{W_n(C)}(\bigwedge^{\xi} W_n(C)^{i+j}, W_n\Omega_C^{q+r}),}
by sending a tuple of pairs of morphisms $\oplus_{\mu+\nu=\xi}(\varphi^\mu, \psi^\nu)$ to the morphism, which is uniquely determined by
\[h_L\mapsto p_1^*\varphi^{\sigma(L)}(e_{L^{\le i}})\cdot p_2^*\psi^{\xi-\sigma(L)}(f_{L^{>i}}).\] 
It is straightforward to check, that $\epsilon$ induces a morphism of complexes $K_A\otimes_\Z K_B\to K_C$, which in degree zero equals
\[\epsilon^0: W_n\Omega^q_A\otimes_\Z W_n\Omega^r_B\to W_n\Omega^{q+r}_C,\quad \alpha\otimes \beta\mapsto p_1^*\alpha\cdot p_2^*\beta\]
and in degree $i+j$ is given by
\[\epsilon^{i+j}: K^i_A\otimes_\Z K^j_B\to K^{i+j}_C, \quad \epsilon^{i+j}(\varphi\otimes \psi)(h_{[1, i+j]})= p_1^*\varphi(e_{[1,i]})\cdot p_2^*\psi(f_{[1,j]}).\] 
By the very definition of the symbols $\genfrac{[}{]}{0pt}{}{\alpha}{[s_1],\ldots, [s_i]}$, etc. in \ref{1.7.6} it remains to show, that the following diagram commutes:
\eq{2.5.5.1}{\xymatrix{H^i(K_A)\otimes H^j(K_B)\ar[r]^{\text{nat.}}\ar[d]^{\eqref{1.7.6.1}} & H^{i+j}(K_A\otimes K_B)\ar[r]^{\epsilon} & H^{i+j}(K_C)\ar[d]^{\eqref{1.7.6.1}}\\
                       H^i_{Z_X}\otimes H^j_{Z_Y}\ar[rr]^\times & & H^{i+j}_{Z_X\times Z_Y},    }}
where we abbreviate $H^i_{Z_X}=H^i_{Z_X}(X, W_n\Omega^q_X)$, etc. This is a  
straightforward calculation.
\end{proof}

\subsubsection{Exterior product for Cousin complexes}\label{2.5.6}
Let $X$ and $Y$ be smooth $k$-schemes and denote by  $Z^\bullet_X$ and $Z^\bullet_Y$ their respective filtrations by codimension. Then we obtain from \ref{2.5.2} a natural morphism
\[\frac{\ul{\Gamma}_{Z_X^a}G(W_\bullet\Omega_X)\boxtimes \ul{\Gamma}_{Z_Y^b}G(W_\bullet\Omega_Y) }{\ul{\Gamma}_{Z_X^{a+1}}\boxtimes\ul{\Gamma}_{Z_Y^b}+ \ul{\Gamma}_{Z_X^a}\boxtimes\ul{\Gamma}_{Z_Y^{b+1}}}
  \to \frac{\ul{\Gamma}_{Z_{X\times Y}^{a+b}}G(W_\bullet\Omega_{X\times Y})}{\ul{\Gamma}_{Z_{X\times Y}^{a+b+1}}G(W_\bullet\Omega_{X\times Y})}, \]
where we abbreviate the denominator on the left in the obvious way. With the notation from \ref{CousinWitt}, there is an obvious morphism from 
$$\sH^i_{Z^a_X/Z^{a+1}_X}(W_\bullet\Omega_X)\boxtimes \sH^j_{Z^b_X/Z^{b+1}_Y}(W_\bullet\Omega_Y)$$
to the $(i+j)$-th cohomology of the left hand side.
In composition we obtain a morphism
\[\sH^i_{Z^a_X/Z^{a+1}_X}(W_\bullet\Omega_X)\boxtimes \sH^j_{Z^b_Y/Z^{b+1}_Y}(W_\bullet\Omega_Y)\to \sH^{i+j}_{Z^{a+b}_{X\times Y}/Z^{a+b+1}_{X\times Y}}(W_\bullet\Omega_{X\times Y}),\]
which is compatible with Frobenius. It is straightforward to check that this pairing induces a morphism of {\em pro-complexes of sheaves of $W$-modules on $X\times Y$}
\eq{2.5.6.1}{E(W_\bullet\Omega_X)\boxtimes E(W_\bullet\Omega_Y)\to E(W_\bullet\Omega_{X\times Y}),}
where $E(-)$ denotes the Cousin complex with respect to the codimension filtration (see \ref{CousinWitt}).

\begin{proposition}\label{2.5.7}
The exterior products defined above satisfy the following properties:
\begin{enumerate}
 \item The exterior products (for $\sH$ and $\hat{\sH}$) can also be calculated by using the morphism \eqref{2.5.6.1} in Definition \ref{2.5.4} instead of the morphism \eqref{2.5.3.1}.
 \item The exterior products are associative.
 \item Let $f:\XP\to \YP$ and $f': (X', \Phi')\to (Y',\Psi')$ be two morphisms in $(Sm^*/S)$ and $\alpha\in \sH(\YP/S)$, $\alpha'\in\sH((Y', \Psi')/S)$, then
          \[(f^*\alpha)\times ({f'}^*\alpha')= (f\times f')^*(\alpha\times \alpha')\quad \text{in } \sH((X\times X', \Phi\times\Phi')/S\times S).\]
       Similar with $\hat{\sH}$.
 \item Let $f:\XP\to \YP$ be a morphism in $(Sm_*/S)$ and $\alpha\in \sH(\XP/S)$, $\alpha'\in\sH((X', \Phi')/S)$, then
        \[(f_*\alpha)\times \alpha'= (f\times \id_{X'})_*(\alpha\times \alpha')\quad \text{in } \sH((Y\times X', \Psi\times\Phi')/S\times S).\]
Similar with $\hat{\sH}$. 
 \item Let $\XP$ and $\YP$ be smooth $k$-schemes over $S$ with families of supports and $\alpha\in \sH^{i,q}(\XP/S)$, $\beta\in \sH^{j,r}(\YP/S)$. 
       Then (similar properties hold for $\hat{\sH}$):
       \begin{enumerate}
          \item The switching isomorphism 
                     \[\sH((X\times Y, \Phi\times \Psi)/S\times S)\cong \sH((Y\times X, \Psi\times \Phi)/S\times S)\]
                sends $\alpha\times \beta$ to $(-1)^{qr+ij}\beta\times\alpha$.
          \item \[F(\alpha)\times F(\beta)=F(\alpha\times\beta),\quad \pi(\alpha)\times \pi(\beta)=\pi(\alpha\times \beta),\]
                \[d(\alpha\times\beta)= (d \alpha)\times \beta+ (-1)^q \alpha\times d(\beta),\]
                for $\alpha,\beta\in \sH_n$.
          \item \[V(\alpha)\times\beta=V(\alpha\times F(\beta)),\quad \ul{p}(\alpha)\times \beta= \ul{p}(\alpha\times \pi(\beta))\]
                for $\alpha\in \sH_n$, $\beta\in\sH_{n+1}$.
          \end{enumerate}
\end{enumerate}
\end{proposition}

\begin{proof}
 (1) follows from \cite[II, Thm. 6.2.1]{God}. (2), (3) and (5) easily follow from the properties \cite[II, 6.5, (b), (d), (e)]{God} and the corresponding properties
of the multiplication map $W_\bullet\Omega_X\boxtimes W_\bullet\Omega_Y\to W_\bullet\Omega_{X\times Y}$. It remains to check (4).
Since $f$ is quasi-projective we may factor it as the composition of a regular closed immersion followed by an open
embedding into a projective space over $Y$ followed by the projection to $Y$. Using similar arguments as in the proof of Proposition \ref{1.7.12} (together with Lemma \ref{2.1.10}),
we can reduce to the case that $\Phi=\Psi$ and that $f: (X, \Phi)\inj (Y,\Phi)$ is a regular closed immersion of pure codimension $c$. By (1) it thus suffices to show, that the following diagram commutes
\[\xymatrix{f_*(E(W_\bullet\Omega_X))\boxtimes E(W_\bullet\Omega_{X'})\ar[d]_{f_*\boxtimes \id}\ar[r]^-{\eqref{2.5.6.1}} &  E(W_\bullet\Omega_{X\times X'})\ar[d]^{(f\times\id_{X'})_*}\\
            (\ul{\Gamma}_{X}(E(W_\bullet\Omega_Y))(c)[c])\boxtimes E(W_\bullet\Omega_{X'})\ar[r]^-{\eqref{2.5.6.1}} &  \ul{\Gamma}_{X\times X'}E(W_\bullet\Omega_{Y\times X'})(c)[c],    }\]
where we denote by $f_*$ and $(f\times\id_{X'})_*$ the morphism defined in \eqref{2.1.6.1}. The question is local and we may therefore assume $X=\Spec A$, $X'=\Spec A'$ and $Y=\Spec B$ are affine
and that the ideal of $Y$ in $X$ is generated by a regular sequence $\tau_1,\ldots, \tau_c$. Furthermore we may check the commutativity in each degree separately  and thus by
the description of the terms of the Cousin complex in \ref{CousinWitt}, (2), we are reduced to show the commutativity of the following diagram:
\[\xymatrix{H^i_{Z_X}(X, W_n\Omega^q_X)\otimes H^i_{Z_{X'}}(X', W_n\Omega^q_{X'})\ar[d]_{f_*\otimes \id}\ar[r]^-\times & H^{i+j}_{Z_X\times Z_{X'}}(X\times X', W_n\Omega_{X\times X'}^{q+r})\ar[d]^{(f\times \id)_*}\\
            H^{i+c}_{Z_X}(Y, W_n\Omega^{q+c}_{Y\times X'})\otimes H^i_{Z_{X'}}(X', W_n\Omega^q_{X'})\ar[r]^-\times & 
                                                                  H^{i+j+c}_{Z_X\times Z_{X'}}(Y\times X', W_n\Omega_{Y\times X'}^{q+r+c}),  }\]
where $Z_X\subset X$ and $Z_{X'}\subset X'$ are integral closed subschemes of codimension $i$ and $j$ respectively. But this follows from the explicit formulas in 
Corollary \ref{2.1.9.5} and Lemma \ref{2.5.5}.
\end{proof}

\begin{remark}
 Notice that in the situation of (4) above the two elements  $\alpha'\times (f_*\alpha)$ and $(\id_{X'}\times f)_*(\alpha'\times \alpha)$ are in general only equal up to a sign.
This is the reason for the sign in \eqref{formula-KT}. 
\end{remark}

\subsection{The cycle class of Gros}\label{Gros-results}
We review some results of \cite{Gros}. 
\subsubsection{Witt log forms} Let $X$ be a smooth $k$ scheme. We denote by 
\[W_n\Omega^q_{X, \text{log}}\]
the abelian sheaf on the small \'etale  site $X_{\text{\'et}}$ of $X$ defined in \cite[I, 5.7]{IlDRW}; it is the abelian subsheaf of $W_n\Omega_{X_{\text{\'et}}}^q$ which is locally generated by
sections of the form 
\[\frac{d[x_1]}{[x_1]}\cdots \frac{d[x_q]}{[x_q]},\quad\text{with } x_1,\ldots, x_d\in \sO^\times_{X_{\text{\'et}}}.\]
By definition $W_n\Omega^0_{X,\text{log}}=\Z/p^n\Z$.

\subsubsection{The cycle class}\label{2.6.2}
Let $X$ be a smooth $k$-scheme and $Z\subset X$ a closed integral subscheme of codimension $c$. Then
by \cite[II, (4.1.6)]{Gros} the restriction map 
\eq{2.6.2.1}{H^c_Z(X_{\text{\'et}}, W_n\Omega^c_{X,\text{log}})\xr{\simeq}H^c_{Z\setminus Z_{\text{sing}}}((X\setminus Z_{\text{sing}})_{\text{\'et}}, W_n\Omega^c_{X,\text{log}})}
 is an isomorphism (where $Z_\text{sing}$ denotes the singular locus of $Z$). Therefore  following Gros we can define the {\em log-cycle class of $Z$ of level $n$} 
\[cl_{\text{log},n}(Z)\in H^c_Z(X_{\text{\'et}}, W_n\Omega_{X,\text{log}}^c)\]
as the image of $1$ under the Gysin isomorphism (\cite[II, Thm 3.5.8 and (3.5.19)]{Gros})
\eq{2.6.2.2}{H^0((Z\setminus Z_{\text{sing}})_{\text{\'et}}, \Z/p^n\Z)\xr{\simeq}H^c_{Z\setminus Z_{\text{sing}}}((X\setminus Z_{\text{sing}})_{\text{\'et}}, W_n\Omega^c_{X,\text{log}}) }
composed with the inverse of \eqref{2.6.2.1} (see \cite[II, Def. 4.1.7.]{Gros}). We define the {\em cycle class of $Z$ of level $n$}  
\[cl_n(Z)\in H^c_Z(X, W_n\Omega_X)\]
to be the image of $cl_{\text{log},n}(Z)$ under the natural morphism $H^c_Z(X_{\text{\'et}}, W_n\Omega_{X,\text{log}}^c)\to H^c_Z(X, W_n\Omega_X^c)$.
We have 
\[\pi(cl_{\text{log}, n}(Z))= cl_{\text{log}, n-1}(Z),\quad \pi(cl_ n(Z))= cl_{n-1}(Z). \]
Therefore $(cl_{\text{log},n}(Z))_n$ (resp. $(cl_n(Z))_n$) give rise to an element in the projective system $H^c_Z(X_{\text{\'et}}, W_\bullet\Omega^c_{X,\text{log}})$ (resp. $H^c_Z(X, W_\bullet\Omega^c_X)$),
which we simply denote by $cl_{\text{log}}(Z)$ (resp. $cl(Z)$).

Notice that we have an isomorphism
\eq{2.6.2.3}{H^c_Z(X, W\Omega^c_X)\cong \varprojlim H^c_Z(X, W_n\Omega^c_X).}
(Indeed by Proposition \ref{derived-functors-exist}, (5) we have a spectral sequence 
\[R^i\varprojlim H^j_Z(X, W_n\Omega^c_X)\Longrightarrow H^{i+j}_Z(X, W\Omega^c_X),\]
 which by Lemma \ref{vanishing-of-Rlim}, (1)
gives a short exact sequence
\[0\to R^1\varprojlim H^{c-1}_Z(X, W_n\Omega_X^c)\to H^c_Z(X, W\Omega_X^c)\to \varprojlim H^c_Z(X, W_n\Omega^c_X)\to 0.\]
But by Lemma \ref{drw-is-CM} $H^{c-1}_Z(X, W_n\Omega_X^c)$ equals $0$ for all $n$. Hence \eqref{2.6.2.3}.)
We define
\[\hat{c}l(Z)\in H^c_Z(X, W\Omega^c_X)\]
as the image of $\varprojlim_n cl_n(Z)$ via \eqref{2.6.2.3}.

\subsubsection{Properties of the cycle classes}\label{2.6.3}
The above cycle classes have the following properties (we only list them for $cl$ there are analogous properties for $cl_{\text{log}}$ and $\hat{c}l$):
\begin{enumerate}
 \item Let $Z\subset X$ be a closed integral subscheme of codimension $c$  and $U\subset X$ be open such that $Z\cap U$ is smooth. Denote by $i: Z\cap U\inj (U, Z\cap U)$ the morphism in $(Sm_*/\Spec k)$ which is induced
by the inclusion of $Z$ in $X$ and by $j:(U, Z\cap U)\inj (X, Z)$ in $(Sm^*/\Spec k)$ the morphism induced by the open immersion $U\subset X$. Then
     \[j^*cl(Z)=cl(Z\cap U)= i_*(1) \quad \text{in } H^c_Z(X, W_\bullet\Omega^c_X),\]
    where $1$ is the multiplicative unit in the pro-ring $H^0(Z\cap U, W_\bullet\sO_{Z\cap U})$.
\item Let $Z$ and $Z'$ be two closed integral subschemes of $X$ of codimension $c$ and $c'$ respectively intersecting each other properly (i.e. each irreducible component of $Z\cap Z'$ has codimension $c+c'$ in $X$).
      Then
       \[cl(Z.Z')= \Delta^*(cl(Z)\times cl(Z'))\quad \text{in } H^{c+c'}_{Z\cap Z'}(X, W_\bullet\Omega^{c+c'}). \]
   Here $Z.Z'$ is the intersection cycle $\sum_T n_T [T]$, where the sum is over the irreducible components of $Z\cap Z'$ and $n_T$ are the intersection multiplicities (computed via Serres tor-formula),
     $cl(Z.Z')$ is defined to be $\sum_T n_i cl(T)$ and $\Delta: (X, Z\cap Z')\to (X\times X, Z\times Z')$ in $(Sm^*/\Spec k)$ is induced by the diagonal morphism.
\item For a line bundle $\sL$ on $X$ we denote by $c_1(\sL)_{\text{log}}\in H^1(X_{\text{\'et}}, W_\bullet\Omega^1_{X,\text{log}})$ the sequence of elements $(dlog_n([\sL]))_n$, where
      $[\sL]$ denotes the class of $\sL$ in $H^1(X, \sO^\times_X)=H^1(X_{\text{\'et}}, \G_m)$ and $dlog_n$ is induced by taking $H^1$ of the map $\G_m\to W_n\Omega^1_{X,\text{log}}$, $a\to dlog [a]$.
      We denote by $c_1(\sL)$ the image of $c_1(\sL)_{\text{log}}$ in $H^1(X, W_\bullet\Omega^c_X)$. 

      Let $D\subset X$ be an integral subscheme of codimension 1, then
       \[cl(D)= c_1(\sO(D))\quad \text{in } H^1(X, W_\bullet\Omega^1_X).\]
\item Let $Z$ be as in (1), then:
\[F(cl(Z))=\pi(cl(Z)),\quad V(cl(Z))=\ul{p}(cl(Z)), \quad d( cl(Z))=0.\]
\end{enumerate}
(1) follows from the fact, that the Gysin morphism \eqref{2.6.2.2} is induced by the pushforward $i_{0*}$, where $i_0: Z\setminus Z_{\text{sing}}\to (X\setminus Z_{\text{sing}}, Z\setminus Z_{\text{sing}})$
in $(Sm_*/\Spec k)$ is induced by the inclusion. This follows from \cite[II, 3.4.]{Gros} (or alternatively from the description of the Gysin morphism in \cite[II, Prop. 3.5.6]{Gros} together with Proposition \ref{1.7.9}.)
(2) is \cite[II, Prop. 4.2.12.]{Gros} and (3) is \cite[II, Prop. 4.2.1]{Gros}. The first and the last equality in (4) follow from the definition, the second equality is implied by the first and $\pi(cl_n(Z))=cl_{n-1}(Z)$. 

\subsection{Hodge-Witt cohomology as weak cohomology theory with supports}

\subsubsection{Weak cohomology theories with supports}
Weak cohomology theories with supports have been introduced in \cite{CR}. 
For the convenience of the reader we recall the definitions. 

First, recall the definition of the categories $Sm_*=(Sm_*/k)$ and 
$Sm^*=(Sm^*/k)$ in
Definition \ref{definition-Sm_*^*}. For both categories $Sm^*$ and $Sm_*$ finite coproducts exist:
$$
(X,\Phi)\coprod (Y,\Psi)=(X\coprod Y,\Phi \cup \Psi).
$$  
For $(X,\Phi)$ let $X=\coprod_i X_i$ be the decomposition into connected components, 
then 
$$
(X,\Phi)=\coprod_i (X_i,\Phi\cap \Phi_{X_i}).
$$
In general products don't exist, and we define 
$$
(X,\Phi)\otimes (Y,\Psi):= (X\times Y,\Phi\times \Psi)
$$
which together with the unit object $\mathbf{1}={\rm Spec}(k)$ and the 
obvious isomorphism $(X,\Phi)\otimes (Y,\Psi) \xr{} (Y,\Psi)\otimes (X,\Phi)$ makes
$Sm_*$ and $Sm^*$ to a symmetric monoidal category (see \cite[VII.1]{ML}).

\subsubsection{}\label{triples}
A weak cohomology theory with supports consists of the following data $(F_*,F^*,T,e)$: 
\begin{enumerate}
\item Two functors $F_*:Sm_*\xr{}\GrAb$ and $F^*:(Sm^*)^{op} \xr{} \GrAb$ such that 
$F_*(X)=F^*(X)$ as abelian groups for every object $X\in {\rm ob}(Sm_*)={\rm ob}(Sm^*)$. 
We will simply write $F(X):=F_*(X)=F^*(X)$. We use lower indexes 
for the grading on $F_*(X)$, i.e. $F_*(X)=\oplus_i F_i(X)$, and upper indexes for $F^*(X)$.  
\item For every two objects $X,Y\in {\rm ob}(Sm_*)={\rm ob}(Sm^*)$ a morphism of graded abelian 
groups (for both gradings):
$$
T_{X,Y}: F(X)\otimes_{\Z} F(Y) \xr{} F (X\otimes Y).
$$
\item A morphism of abelian groups $e:\Z\xr{} F(\Spec(k))$. For all smooth schemes $\pi:X\xr{} \Spec(k)$ we denote by $1_X$ the image of 
$1\in \Z$ via the map 
$
\Z\xr{e} F^*(\Spec(k))\xr{F^*(\pi)} F^*(X).
$ 
\end{enumerate}

\subsubsection{} \label{triplescond} These data are required to satisfy the following conditions:
\begin{enumerate}
\item \label{triplescondsum} The functor $F_*$ preserves coproducts and $F^*$ maps coproducts to products. 
Moreover, for $(X,\Phi_1),(X,\Phi_2)\in {\rm ob}(Sm_*)$ with $\Phi_1\cap \Phi_2=\{\emptyset\}$
we require that the map  
$$
F^*(\jmath_1)+F^*(\jmath_2):F^*(X,\Phi_1)\oplus F^*(X,\Phi_2) \xr{} F^*(X,\Phi_1\cup \Phi_2),
$$
with $\jmath_1:(X,\Phi_1\cup \Phi_2)\xr{} (X,\Phi_1)$ and 
$\jmath_2:(X,\Phi_1\cup \Phi_2)\xr{} (X,\Phi_2)$ in $Sm^*$, is an isomorphism.  
\item The data $(F_*,T,e)$ and $(F^*,T,e)$ respectively define a (right-lax) symmetric monoidal functor.  
\item (Grading) For $(X,\Phi)$ such that $X$ is connected the equality   
$$F_i(X,\Phi)=F^{2\dim X-i}(X,\Phi)$$ holds for all $i$.
\item \label{MP} For all cartesian diagrams (cf.~\ref{definition-Sm_*^*})
\begin{equation*}
\xymatrix
{
(X',\Phi') \ar[r]^{f'} \ar[d]^{g_X}
&
(Y',\Psi')  \ar[d]^{g_Y}
\\
(X,\Phi) \ar[r]^{f} 
&
(Y,\Psi)  
}
\end{equation*}
with $g_X,g_Y\in Sm^*$ and  $f,f'\in Sm_*$ such that
either $g_Y$ is smooth or $g_Y$ is a closed immersion and $f$ is transversal to $g_Y$ the following
equality holds:
\begin{equation*}
 F^*(g_Y)\circ F_*(f) = F_*(f')\circ F^*(g_X). 
\end{equation*}
\end{enumerate}

Recall that $f$ is transversal to $g_Y$ if $(f')^*N_{Y'/Y}=N_{X'/X}$ where $N$ denotes
the normal bundle. The case $X'=\emptyset$ is also admissible; in this case the equality
\ref{triplescond}(\ref{MP}) reads:
$$
 F^*(g_Y)\circ F_*(f) = 0.
$$
 
In \cite[1.1.7]{CR} it is spelled out what it means for $(F_*,T,e)$ (and $(F^*,T,e)$) to be 
a right-lax symmetric monoidal functor. 

\begin{definition}\label{definition-wct}
If the data $\triples$ as in \ref{triples} satisfy the conditions 
\ref{triplescond} then we call $\triples$ a \emph{weak cohomology theory with supports.}   
\end{definition}

\begin{definition}\label{definitionT}
Let $\triples, (G_*,G^*,U,\epsilon)$ be two weak cohomology theories with support. 
By a morphism 
\begin{equation}\label{morphismsT}
\triples \xr{} (G_*,G^*,U,\epsilon)
\end{equation}
we understand a morphism of graded abelian groups (for both gradings)
$$
\phi: F(X)\xr{} G(X) \quad \text{for every $X\in {\rm ob}(V_*)= {\rm ob}(V^*)$,}
$$
such that $\phi$ induces a natural transformation of (right-lax) symmetric monoidal functors 
$$
\phi:(F_*,T,e) \xr{} (G_*,U,\epsilon)\; \text{and}\; \phi:(F^*,T,e) \xr{} (G^*,U,\epsilon),
$$ 
i.e. $\phi$ induces natural transformations $F_*\xr{} G_*, F^*\xr{} G^*$, and 
\begin{equation}\label{morphismTsecond}
\phi\circ T = U \circ(\phi \otimes \phi), \quad \phi\circ e=\epsilon. 
\end{equation}
\end{definition}

\subsubsection{Chow theory as a weak cohomology theory with supports} \label{subsubsection-chow-theory}
An example of a weak cohomology theory with supports 
are the Chow groups $(\CH_*,\CH^*,\times,e)$ \cite[1.1]{CR}. We will briefly recall the definitions
for the convenience of the reader.

\begin{definition}[Chow groups with support]\label{definition-chowgroups-with-support}
Let $\Phi$ be a family of supports on $X$. We define: 
$$
\CH(X,\Phi)={\varinjlim}_{W\in \Phi}\CH(W).
$$
\end{definition}

The group $\CH(X,\Phi)$ is graded by dimension. 
We set
$$
\CH_{*}(X,\Phi)=\bigoplus_{d\ge 0} \CH_{d}(X,\Phi)[2d],
$$
where the bracket $[2d]$ means that $\CH_{d}(X,\Phi)$ is considered to be in degree $2d$. 

There is also a grading by codimension. Let $X=\coprod_i X_i$ be the decomposition into 
connected components then 
$
\CH^*(X,\Phi)=\bigoplus_i \CH^*(X_i,\Phi\cap \Phi_{X_i})
$ 
and 
$$
\CH^*(X_i,\Phi\cap \Phi_{X_i}) = \bigoplus_{d\geq 0} \CH^{d}(X_i,\Phi\cap \Phi_{X_i})[2d]
$$
where $\CH^{d}(X_i,\Phi\cap \Phi_{X_i})$ is generated by cycles $[Z]$ with $Z\in \Phi\cap \Phi_{X_i}$, $Z$ irreducible, and $\codim_{X_i}(Z)=d$.

Let $f:(X,\Phi)\xr{} (Y,\Psi)$ be a morphism in $Sm_*$.
There is a push-forward
\begin{equation}
\CH_*(f):\CH(X,\Phi)\xr{} \CH(Y,\Psi),
\end{equation}
defined by the usual push-forward of cycles 
$(f_{\mid W})_*:\CH(W)\xr{} \CH(f(W))\xr{} \CH(Y,\Psi)$ 
(for this we need that $f_{\mid W}$ is proper and $f(\Phi)\subset \Psi$).
We obtain a functor 
\begin{equation}
\CH_*: Sm_* \xr{} \GrAb, \quad \CH_*(X,\Phi):=\CH(X,\Phi),\quad f\mapsto \CH_*(f).
\end{equation}

In order to define a functor 
$$
\CH^*:(Sm^*)^{op} \xr{} \GrAb
$$
we use Fulton's work on refined Gysin morphisms \cite[\textsection6.6]{F}.

Let $f:X\xr{} Y$ be a morphism between smooth schemes and let $V\subset Y$ be a closed
subscheme, thus $f:(X,f^{-1}(V))\xr{} (Y,V)$ is a morphism in $Sm^*$. We define 
$$
\CH^*(f):=f^!:\CH(Y,V)=\CH(V)\xr{} \CH(f^{-1}(V))= \CH(X,f^{-1}(V)). 
$$ 
For the general case let $f:(X,\Phi)\xr{} (Y,\Psi)$ be any morphism in $Sm^*$.
For every $V\in \Psi$ the map $f$ induces $(X,f^{-1}(V))\xr{} (Y,V)$ in $Sm^*$.
We may define  
$$
\CH^*(f): \CH(Y,\Psi)={\varinjlim}_{V\in \Psi}\CH(Y,V) \xr{} {\varinjlim}_{W\in \Phi}\CH(X,W)=\CH(X,\Phi).
$$

The assignment 
\begin{equation*}
\CH^*:(Sm^*)^{op} \xr{} \GrAb, \quad \CH^*(X,\Phi)=\CH(X,\Phi), \quad f\mapsto \CH^*(f)
\end{equation*}
defines a functor. Together with the exterior product $\times$ (see \cite[\textsection1.10]{F}) and the obvious unit $1:\Z\xr{} \CH(\Spec(k))$, we obtain a weak cohomology theory
with supports $(\CH_*,\CH^*,\times,1)$.

\subsubsection{}\label{section-forhat}
From Proposition \ref{2.1.9} and Proposition \ref{2.1.3} we obtain two functors
\begin{align*}
\hat{H}_*&:Sm_*\xr{} \GrAb,\\
\hat{H}^*&:(Sm^*)^{op}\xr{} \GrAb, 
\end{align*} 
where for an object $(X,\Phi)\in {\rm ob}(Sm_*)={\rm ob}(Sm^*)$ with structure morphism $a:X\xr{} \Spec(k)$ we have
$$
\hat{H}(X,\Phi):=\hat{H}_*(X,\Phi):=\hat{H}^*(X,\Phi):=\hat{\sH}(\XP/\Spec k)=\bigoplus_{p,q\ge 0} R^q\hat{a}_\Phi W_\bullet\Omega^p_X 
$$
as abelian groups (see \eqref{fPhihat} for the definition of $\hat{a}_\Phi$). In view of Proposition \ref{derived-functors-exist} we get an equality 
$$
R^q\hat{a}_\Phi=H^q\circ R\varprojlim \circ Ra_* \circ R\ul{\Gamma}_\Phi.
$$ 
In particular, if $\Phi=\Phi_X$ is the set of all closed subsets of $X$ then 
$$
\hat{H}(X):=\hat{H}(X,\Phi_X)= \bigoplus_{p,q\ge 0} H^q(R\varprojlim R\Gamma (W_\bullet\Omega^p_X)). 
$$ 
Note that $R\varprojlim \circ R\Gamma = R\Gamma \circ R\varprojlim$ and 
$R\varprojlim(W_\bullet\Omega^p_X) = \varprojlim(W_{\bullet} \Omega^p_X)=W\Omega^p_X$ 
by Lemma \ref{vanishing-of-Rlim}. Thus we obtain 
$$
\hat{H}(X)=\bigoplus_{p,q\ge 0} H^q(X,W\Omega^p_X).
$$
In the following we will write $\hat{H}(X)$ for $\hat{H}(X,\Phi_X)$.

The grading for $\hat{H}^*$ is defined by 
$$
\hat{H}^i(X,\Phi)=\bigoplus_{p+q=i} R^q\hat{a}_\Phi W_\bullet\Omega^p_X.
$$
For $\hat{H}_*(X,\Phi)$ the grading is defined such that  \ref{triplescond}(4) holds.

For two objects $(X,\Phi), (Y,\Psi)$ with structure maps $a:X\xr{} \Spec(k), 
b:Y\xr{} \Spec(k)$, we define 
\begin{equation}\label{KT}
T: \hat{H}(X,\Phi)\otimes_{\Z} \hat{H}(Y,\Psi) \xr{} \hat{H}(X\times Y,\Phi\times \Psi)
\end{equation}
by the formula
\begin{equation}\label{formula-KT}
T(\alpha_{i,p}\otimes \beta_{j,q})=(-1)^{(i+p)\cdot j}(\alpha_{i,p}\times \beta_{j,q}),
\end{equation}
where $\alpha_{i,p}\in R^i\hat{a}_\Phi W_\bullet\Omega^p_X , 
\beta_{j,q}\in R^j\hat{b}_\Psi W_\bullet\Omega^q_Y,$
and $\times$ is the map in \ref{2.5.4}.

We define
$$
e:\Z\xr{} \hat{H}(\Spec(k))=W(k)
$$
by $e(n)=n$ for all $n\in \Z$.

\begin{proposition}\label{proposition-kuenneth}
The triples $(\hat{H}_*, T, e)$ and $(\hat{H}^*,T, e)$  define right-lax symmetric monoidal functors.
\begin{proof}
The morphism $T$ respects the grading $\hat{H}^*$ and the grading $\hat{H}_*$. 
In the following we will work with the upper grading $\hat{H}^*$. 
All arguments
will also work for the lower grading $\hat{H}_*$ because the difference 
between lower and upper grading is an even integer.

The axioms which involve $e$ do obviously hold.
By using the associativity of $\times$ (Proposition \ref{2.5.7}(2)) it  is 
straightforward to prove the associativity of $T$. 
Let us prove the commutativity of $T$, i.e. that the diagram 
\begin{equation}
\label{proof-kuenneth-Tcomm}
\xymatrix
{
H(X,\Phi)\otimes H(Y,\Psi) \ar[r]^{T} \ar[d]
&
H(X\times Y,\Phi\times \Psi) \ar[d]^{\epsilon}
\\
H(Y,\Psi)\otimes H(X,\Phi) \ar[r]^{T}
&
H(Y\times X,\Psi\times \Phi)
}
\end{equation}
is commutative. The left vertical map is defined by $\alpha\otimes \beta\mapsto 
(-1)^{\deg(\alpha)\deg(\beta)} \beta\otimes \alpha$. By using Proposition \ref{2.5.7}(5)(a)
we obtain 
\begin{multline*}
\epsilon(T(\alpha_{i,p}\otimes \beta_{j,q}))=\epsilon((-1)^{(i+p)j}\alpha_{i,p}\times \beta_{j,q})=(-1)^{(i+p)j}(-1)^{pq+ij}\beta_{j,q}\times\alpha_{i,p}\\ =(-1)^{pq+pj}\beta_{j,q}\times\alpha_{i,p}. 
\end{multline*}
On the other hand, 
\begin{multline*}
T((-1)^{(i+p)(j+q)} \beta_{j,q}\otimes \alpha_{i,p})=(-1)^{(i+p)(j+q)}(-1)^{(j+q)i}\beta_{j,q}\times \alpha_{i,p} \\ =(-1)^{pq+pj}\beta_{j,q}\times\alpha_{i,p}.
\end{multline*}
This proves the commutativity of diagram \eqref{proof-kuenneth-Tcomm}. 

The functoriality of $T$ with respect to $\hat{H}^*$ follows immediately 
from Proposition \ref{2.5.7}(3). The 
functoriality of $T$ with respect to $\hat{H}_*$ follows from Proposition
\ref{2.5.7})(4).
\end{proof}
\end{proposition}

\begin{thm}\label{thm-sH-wct}
The datum $(\hat{H}_*,\hat{H}^*,T,e)$ is a weak cohomology theory with supports (cf.~Definition \ref{definition-wct}). 
\begin{proof}
We have to verify the properties in section \ref{triplescond}. Property \ref{triplescond}(1) is obvious. Proposition \ref{proposition-kuenneth} 
implies \ref{triplescond}(2). The compatibility of the gradings (\ref{triplescond}(3)) is satisfied by definition. Proposition \ref{1.7.12} yields \ref{triplescond}(4). 
\end{proof}
\end{thm}

\subsubsection{}
In the same way as in section \ref{section-forhat} we define 
\begin{align*}
P\hat{H}_*&:Sm_*\xr{} \GrAb,\\
P\hat{H}^*&:(Sm^*)^{op}\xr{} \GrAb, 
\end{align*} 
where for an object $(X,\Phi)\in {\rm ob}(Sm_*)={\rm ob}(Sm^*)$ with structure morphism $a:X\xr{} \Spec(k)$ we have
$$
P\hat{H}(X,\Phi)=P\hat{H}_*(X,\Phi)=P\hat{H}^*(X,\Phi)=\bigoplus_{p\ge 0} R^p\hat{a}_\Phi W_\bullet\Omega^p_X 
$$
as abelian groups. The grading is defined in the obvious way: 
$$
P\hat{H}^{2p}(X,\Phi)=R^p\hat{a}_\Phi W_\bullet\Omega^p_X,
$$
and zero for all odd degrees.
Of course, for $P\hat{H}_*(X,\Phi)$ the grading is defined such that  \ref{triplescond}(4) holds.

For two objects $(X,\Phi), (Y,\Psi)$, we get an induced map 
\begin{equation*}
T: P\hat{H}(X,\Phi)\otimes_{\Z} P\hat{H}(Y,\Psi) \xr{} P\hat{H}(X\times Y,\Phi\times \Psi)
\end{equation*}
from \eqref{KT}.

Theorem \ref{thm-sH-wct} implies the following statement.

\begin{corollary}\label{corollary-sH-wct}
The datum $(P\hat{H}_*,P\hat{H}^*,T,e)$ is a weak cohomology theory with supports (cf.~Definition \ref{definition-wct}). Induced by the inclusion 
$
P\hat{H}(X,\Phi)\subset \hat{H}(X,\Phi),
$
we get a morphism of weak cohomology theories with supports (cf.~Definition \ref{definitionT})
$$
(P\hat{H}_*,P\hat{H}^*,T,e) \xr{} (\hat{H}_*,\hat{H}^*,T,e).
$$
\end{corollary} 

\subsection{Cycle classes}
One of the main goals of \cite{CR} was to give a criterion when a weak cohomology 
theory $(F_*,F^*,T,e)$ admits a morphism 
$$
(\CH_*,\CH^*,\times,1) \xr{} (F_*,F^*,T,e).
$$  
In order to apply \cite[Theorem~1.2.3]{CR} in our situation, i.e.~for 
$(P\hat{H}_*,P\hat{H}^*,T,e)$, we first need to prove semi-purity.

Recall that $(P\hat{H}_*,P\hat{H}^*,T,e)$ satisfies the semi-purity condition (cf. \cite[Definition~1.2.1]{CR}) if the following holds:
\begin{itemize}
\item For all smooth schemes $X$ and irreducible closed subsets $W\subset X$ the 
groups 
$
P\hat{H}_i(X,W) 
$
vanish if $i>2\dim W$.
\item For all smooth schemes $X$, closed subsets $W\subset X$, and open
sets $U\subset X$ such that $U$ contains the generic point of every irreducible 
component of $W$, we require the map 
$$
P\hat{H}^*(\jmath):P\hat{H}_{2\dim W}(X,W)\xr{} P\hat{H}_{2\dim W}(U,W\cap U), 
$$ 
induced by $\jmath:(U,W\cap U)\xr{} (X,W)$ in $Sm^*$, to be \emph{injective}.
\end{itemize}

\begin{proposition}\label{proposition-semi-purity}
The weak cohomology theory with supports $(P\hat{H}_*,P\hat{H}^*,T,e)$ satisfies the 
semi-purity condition. 
\begin{proof}
To verify the two conditions we may assume that $X$ is connected of dimension $d$.
Since 
$$
P\hat{H}_{2i}(X,W)=H^{d-i}_{W}(X,W\Omega_X^{d-i}),
$$
we need to show that $H^{d-i}_W(X,W\Omega_X^{d-i})$ vanishes if $2i>2\dim W$ 
(or equivalently, $d-i<\codim_X W$). Indeed, we have 
$$
H^{d-i}_W(X,W\Omega_X^{d-i})=R^{d-i}a_*R\ul{\Gamma}_{W}R\varprojlim (W_n\Omega_X^{d-i})
$$ 
with $a:X\xr{} \Spec(k)$ the structure map. Since $R\ul{\Gamma}_{W} R\varprojlim=R\varprojlim R\ul{\Gamma}_{W}$, we need
to show that 
\eq{semi-purity1}{R^j\ul{\Gamma}_{W}W_n\Omega_X^{d-i}=0,}

for $j<\codim_X W$. This follows from Lemma \ref{drw-is-CM} (which follows from the fact that
the graded pieces of the standard filtration on $W_n\Omega^{d-i}_X$ are  extensions of locally free $\sO_X$-modules  \cite[I, Cor. 3.9]{IlDRW}).

For the second condition, let $W\subset X$ be a closed subset and $U\subset X$ open,
such that $U$ contains the generic point of every irreducible component of $W$.
We need to prove that the restriction map
$$
H^{d-\dim W}_{W}(X,W\Omega_X^{d-\dim W}) \xr{} H^{d-\dim W}_{U\cap W}(U,W\Omega_U^{d-\dim W})
$$ 
is injective. But the kernel is a quotient of $H^{d-\dim W}_{W\backslash U}(X,W\Omega_X^{d-\dim W})$, which 
vanishes because $d-\dim W<\codim_X (W\backslash U)$.
\end{proof}
\end{proposition} 

\subsubsection{}
In view of \cite[Theorem~1.2.3]{CR} there is at most one morphism of
weak cohomology theories with support 
$$
(\CH_*,\CH^*,\times,1) \xr{} (P\hat{H}_*,P\hat{H}^*,T,e).
$$
 
\begin{proposition}\label{proposition-cl}
There is exactly one morphism (in the sense of Definition \ref{definitionT})
$$
(\CH_*,\CH^*,\times,1) \xr{} (P\hat{H}_*,P\hat{H}^*,T,e).
$$
\begin{proof}
We need to verify the criteria given in  \cite[Theorem~1.2.3]{CR}.

The condition \cite[Theorem~1.2.3]{CR}(1) is satisfied by Proposition \ref{2.1.8}.

For the second condition \cite[Theorem~1.2.3]{CR}(2) we need to show that for the $0$-point $\imath_0:\Spec(k)\xr{}\P^1$ and the $\infty$-point 
$\imath_{\infty}:\Spec(k)\xr{}\P^1$ the following equality holds:
$$P\hat{H}_*(\imath_0)\circ e=P\hat{H}_*(\imath_{\infty})\circ e.$$ 
By \ref{2.6.3}(1) the left hand side is $\hat{c}l(\{0\})$ and the right hand 
side equals $\hat{c}l(\{\infty\})$. In view of \ref{2.6.3}(3) we obtain 
$$
\hat{c}l(\{0\}) =c_1(\OO_{\P^1}(1))= \hat{c}l(\{\infty\}).
$$

For $W\subset X$ an irreducible closed subset, we set $cl_{(X,W)}:=\hat{c}l(W)$.
Then \ref{2.6.3}(1) implies that condition \cite[Theorem~1.2.3]{CR}(4) holds.

Finally, we prove \cite[Theorem~1.2.3]{CR}(3). Let $\imath:X\xr{} Y$ be a closed immersion between smooth schemes, and let $D\subset Y$ be an effective smooth divisor such that 
\begin{itemize}
\item $D$ meets $X$ properly, thus $D\cap X:=D\times_{Y} X$ is a divisor on X,
\item $D':=(D\cap X)_{red}$ is smooth and connected, and thus $D\cap X=n\cdot D'$
as divisors (for some $n\in \Z, n\geq 1$).
\end{itemize}
We denote by $\imath_X:X\xr{}(Y,X), \imath_{D'}:D'\xr{} (D,D')$ 
the morphisms in $Sm_*$ induced by $\imath$, and we define $g:(D,D')\xr{} (Y,X)$ in $Sm^*$ by  
the inclusion $D\subset Y$. Then the following equality is required to hold:
\begin{equation}\label{proof-cl-condition3}
P\hat{H}^*(g)(P\hat{H}_*(\imath_X)(1_X))=n\cdot P\hat{H}_*(\imath_{D'})(1_{D'}).
\end{equation}
Obviously, we may assume that $X$ and $Y$ are connected; we set $c:=\codim_YX$.
A priori, we need to prove the equality \eqref{proof-cl-condition3} in $H^c_{D'}(D,W\Omega^c_D)$. Since $\hat{H}^*(g)$, $\hat{H}_*(\imath_X)$ and $\hat{H}_*(\imath_{D'})$
are morphisms in $\cdrw_k$, they commute with Frobenius. Thus both sides of \eqref{proof-cl-condition3} are already contained in the part which is invariant under the Frobenius
$
H^c_{D'}(D,W\Omega^c_D)^F.
$
Denote by  $f:(D,D')\xr{} (Y,D')$ the morphism in $Sm_*$ which is induced by the inclusion $D\subset Y$.

{\em We claim that 
$$P\hat{H}_*(f): H^c_{D'}(D,W\Omega^c_D)\to H^{c+1}_{D'}(Y,W\Omega^{c+1}_Y) $$
 is injective on $H^c_{D'}(D,W\Omega^c_D)^F$.}

Indeed, by \cite[I, Thm. 5.7.2]{IlDRW} there is an exact sequence of pro-sheaves on $D_{\text{\'et}}$
\[0\to W_\bullet\Omega^c_{D,\text{log}}\to W_\bullet\Omega^c_D\xr{1-F} W_\bullet\Omega^c_D\to 0.\]
This yields an exact sequence of pro-groups
\[H^{c-1}_{D'}(D, W_\bullet\Omega^c_D)\to H^c_{D'}(D_{\text{\'et}}, W_\bullet\Omega^c_{D,\text{log}}) \to H^c_{D'}(D,W_\bullet\Omega^c_D)^F\to 0.\]
(Notice that $H^i_{D'}(D, W_\bullet\Omega^q_D)=H^i_{D'}(D_{\text{\'et}}, W_\bullet\Omega^q_D)$, since the $W_n\Omega^q_D$ are quasi-coherent on $W_n D$ and $ D_{\text{\'et}}=(W_n D)_{\text{\'et}}$.)
But $H^{c-1}_{D'}(D, W_\bullet\Omega^c_D)=0$ (see e.g. \eqref{semi-purity1}) and thus
\[\varprojlim_n H^c_{D'}(D_{\text{\'et}}, W_\bullet\Omega^c_{D,\text{log}}) =\varprojlim_n H^c_{D'}(D,W_\bullet\Omega^c_D)^F\stackrel{\eqref{2.6.2.3}}{=}H^c_{D'}(D,W\Omega^c_D)^F.\]
Combining this with Gros' Gysin isomorphism \eqref{2.6.2.2} we obtain an isomorphism
\[\Z_p\xr{\simeq} H^c_{D'}(D,W\Omega^c_D)^F,\quad 1 \mapsto \hat{c}l(D')\stackrel{\mbox{\tiny \ref{2.6.3}, (1)}}{=}P\hat{H}_*(\imath_{D'})(1_{D'}). \]
In the same way we also obtain an isomorphism
\[\Z_p\xr{\simeq} H^{c+1}_{D'}(Y,W\Omega^{c+1}_Y)^F,\quad 1 \mapsto \hat{c}l(D')\stackrel{\mbox{\tiny \ref{2.6.3}, (1)}}{=}P\hat{H}_*(f)\circ P\hat{H}_*(\imath_{D'})(1_{D'}). \]
This yields the claim.

Thus it suffices to prove the following equality
$$
P\hat{H}_*(f)P\hat{H}^*(g)(P\hat{H}_*(\imath_X)(1_X))=n\cdot P\hat{H}_*(f)P\hat{H}_*(\imath_{D'})(1_{D'}).
$$
For this denote by $g':D\xr{} (Y,D)$ in $Sm_*$  and $\Delta_Y:(Y,D')\xr{} (Y\times Y,D\times X)$ in $Sm^*$ the morphisms induced by the inclusion and the diagonal respectively.
Then
\begin{align*}
 P\hat{H}_*(f)P\hat{H}^*(g)(P\hat{H}_*(\imath_X)(1_X)) & =P\hat{H}_*(g')(1_D)\cup P\hat{H}_*(\imath_X)(1_X) \\
                                                &= P\hat{H}^*(\Delta_Y)(\hat{c}l(D)\times \hat{c}l(X))& \ref{2.6.3}, (1)\\
                                                &= \hat{c}l(D.X) & \ref{2.6.3}, (2)\\
                                                &=n \cdot \hat{c}l(D')\\
                                                &=n\cdot P\hat{H}_*(f\circ \imath_{D'})(1_{D'})&  \ref{2.6.3}, (1),\\
                                                &=n\cdot  P\hat{H}_*(f)P\hat{H}_*(\imath_{D'})(1_{D'})
\end{align*}
where $a\cup b:= P\hat{H}^*(\Delta_Y)(T(a\otimes b))$ and hence the first equality holds by the projection formula, see \cite[Proposition~1.1.11]{CR}.
This finishes the proof.
\end{proof}
\end{proposition}

\begin{definition}\label{definition-cl}
We denote by
$$
\hat{c}l:(\CH_*,\CH^*,\times,1) \xr{} (\hat{H}_*,\hat{H}^*,T,e)
$$
the composition of the morphism in Proposition \ref{proposition-cl}
and Corollary \ref{corollary-sH-wct}. Note that there is no conflict
with the notation in section \ref{Gros-results}.
\end{definition}

\subsubsection{}
To a weak cohomology theory $F=(F_*,F^*,T,e)$ we can attach a graded additive 
category $Cor_F$ (\cite[Definition~1.3.5]{CR}). By definition we have
${\rm ob}(Cor_F)={\rm ob}(Sm_*)={\rm ob}(Sm^*)$, and 
$$
\Hom_{Cor_F}((X,\Phi),(Y,\Psi))=F(X\times Y,P(\Phi,\Psi)),
$$
where 
\begin{multline*}\label{definitionP}
P(\Phi,\Psi):=\{Z\subset X\times Y; \text{$Z$ is closed, ${\rm pr}_2 \mid Z$ is proper,} \\ 
\text{$Z\cap {\rm pr}_1^{-1}(W)\in {\rm pr}_2^{-1}(\Psi)$ for every $W\in \Phi$} \}.
\end{multline*}
The composition for correspondences is as usual:
$$
b\circ a= F_*(p_{13})(F^*(p_{12})(a)\cup F^*(p_{23})(b)),
$$ 
with $p_{ij}:X_1\times X_2\times X_3\xr{} X_i\times X_j$ the projection; the product $\cup$ is defined by
$
a\cup b:= F^*(\Delta)(T(a\otimes b))
$
(see \cite[1.3]{CR} for the details).

Moreover, there is a functor 
$$
\rho_F:Cor_F \xr{} ({\rm Ab})
$$
to  abelian groups defined by
\begin{equation*}
\begin{split}
  \quad &\rho_F(X,\Phi)=F(X,\Phi) \\
&\rho_F(\gamma)=(a\mapsto F_*({\rm pr}_2)(F^*({\rm pr}_1)(a)\cup \gamma)), \quad \text{for $\gamma\in F(X\times Y,P(\Phi,\Psi))$,}
\end{split}
\end{equation*}
(see \cite[1.3.9]{CR}).

The morphism 
$$
\hat{c}l: \CH=(\CH_*,\CH^*,\times,1) \xr{} \hat{H}=(\hat{H}_*,\hat{H}^*,T,e)
$$ 
(Definition \ref{definition-cl}) induces  a functor (see \cite[1.3.6]{CR}) 
$$
Cor(\hat{c}l): Cor_{\CH}\xr{} Cor_{\hat{H}}.
$$

For $(X,\Phi)\in Cor_{\hat{H}}$, the group $\rho_{\hat{H}}(X,\Phi)=\hat{H}(X,\Phi)$ is a de Rham-Witt module over $k$ (i.e.~an object in $\widehat{\drw}_k$). 
The next theorem states that the Chow correspondences act as morphism 
of de Rham-Witt modules.

\begin{thm}
The composition 
$$
Cor_{\CH}\xr{} Cor_{\hat{H}} \xr{\rho_{\hat{H}}} (Ab) 
$$
induces a functor 
\begin{equation}\label{cor-functor-k}
Cor_{\hat{c}l}: Cor_{\CH}\xr{} \widehat{\drw}_k.
\end{equation}
\begin{proof}
Let $(X,\Phi), (Y,\Psi)\in Sm$. 
Note that $Cor_{\CH}((X,\Phi),(Y,\Psi))$ is graded and is only non-trivial 
in even degrees. By definition 
$$
Cor^{2i}_{CH}((X,\Phi),(Y,\Psi))=\CH^{\dim X+i}(X\times Y,P(\Phi,\Psi)),
$$
(we may assume that $X$ is equidimensional),
where $\CH^{\dim X+i}(X\times Y,P(\Phi,\Psi))$ is generated by cycles $Z$ with
$\codim_{X\times Y}Z=\dim X+i$, as in \ref{subsubsection-chow-theory}.

We will show that a cycle 
$Z\in Cor^{2i}_{CH}((X,\Phi),(Y,\Psi))$ defines a morphism 
$$
\rho_{\hat{H}}(\hat{c}l(Z)): \hat{H}(X,\Phi)(-i)\xr{} \hat{H}(Y,\Psi)
$$
in $\widehat{\drw}_k$
(the shift by $-i$ concerns only the grading and the differential, which is multiplied with $(-1)^i$). 

Because pullback is a morphism of de Rham-Witt modules (Definition \ref{2.1.2})
and pushforward is a morphism of de Rham-Witt modules up to the
shift with the relative dimension (Definition \ref{2.1.6}), we can 
easily reduce the statement to the following claim. For all $(X,\Phi)$ and 
all irreducible closed subsets $Z\subset X$ of codimension $c$ the 
map 
\begin{align*}
\hat{H}(X,\Phi)&\xr{} \hat{H}(X\times X,\Phi\times \Phi_Z)(c) \\
a &\mapsto T(a\otimes \hat{c}l(Z))
\end{align*}
is a morphism in $\widehat{\drw}_k$. The compatibility with $F$ and $V$
follows immediately from Proposition \ref{2.5.7}(5) and the invariance 
$F(\hat{c}l(Z))=\hat{c}l(Z)$ (results of Gros \ref{2.6.3}(4)). 
The compatibility with $d$ also follows from Proposition \ref{2.5.7}(5)
and $d( \hat{c}l(Z))=0$ (\ref{2.6.3}(4)):
\begin{align*}
T(da\otimes \hat{c}l(Z))&\overset{{\rm dfn}}{=}(-1)^{\deg(da)\cdot c} da\times \hat{c}l(Z) &\text{\eqref{formula-KT},} \\
                        &= (-1)^{(\deg(a)+1)\cdot c} d(a\times  \hat{c}l(Z)) &\text{\ref{2.5.7}(5), \ref{2.6.3}(4),} \\
                        &=(-1)^{c}dT(a\otimes  \hat{c}l(Z))  &\text{\eqref{formula-KT}.} 
\end{align*}
\end{proof}
\end{thm}

\begin{proposition} \label{proposition-comparison-with-pull-push} 
\begin{itemize}
\item[(i)] For every $f:(X,\Phi)\xr{} (Y,\Psi)$ in $Sm^*$ the transpose of the graph 
$\Gamma(f)^t\subset Y\times X$ defines an element in 
$Cor_{\CH}^0((Y,\Psi),(X,\Phi))$ (i.e.~has degree $0$). The morphism 
$$
Cor_{\hat{c}l}(\Gamma(f)^t): \hat{H}(Y,\Psi)\xr{} \hat{H}(X,\Phi)
$$ 
is the same as $f^*$ in Definition \ref{2.1.2}. 
\item[(ii)] For every $f:(X,\Phi)\xr{} (Y,\Psi)$ in $Sm_*$ the graph 
$\Gamma(f)\subset X\times Y$ defines an element in 
$Cor_{\CH}^{-2r}((X,\Phi),(Y,\Psi))$ (i.e.~has degree $-2r=2(\dim(Y)-\dim(X))$). 
The morphism 
$$
Cor_{\hat{c}l}(\Gamma(f)): \hat{H}(X,\Phi)(r) \xr{} \hat{H}(Y,\Psi)
$$ 
is the same as $f_*$ in Definition \ref{2.1.6}.
\end{itemize}
\begin{proof}
For (i). Let $i_2:(X,f^{-1}\Psi) \xr{} (Y\times X,\pr_Y^{-1}\Psi)$ be 
the morphism 
in $Sm^*$ induced by the morphism of schemes $(f,id_X)$. 
Let $i_1:X\xr{} (Y\times X,\Gamma(f)^t)$ and  $i_3:(X,f^{-1}\Psi)\xr{} (Y\times X,\Gamma(f)^t\cap \pr_Y^{-1}\Psi)$ 
 be the morphism in $Sm_*$ induced
by $(f,id_X)$. 
For all $a\in \hat{H}(Y,\Psi)$ we use the projection formula 
\cite[Proposition~1.1.11]{CR} to obtain:
\begin{align*}
Cor_{\hat{c}l}(\Gamma(f)^t)(a)
&=\hat{H}_*(\pr_X)(\hat{H}^*(\pr_Y)(a)\cup \hat{c}l(\Gamma(f)^t)) \quad \text{by definition,}\\
&=\hat{H}_*(\pr_X)(\hat{H}^*(\pr_Y)(a)\cup \hat{H}_*(i_1)(1)) \quad \text{functoriality of $\hat{c}l$ (cf. \ref{definition-cl})}\\
&=\hat{H}_*(\pr_X)(\hat{H}_*(i_3)(\hat{H}^*(i_2)\hat{H}^*(\pr_Y)(a)) \quad \text{projection formula,} \\
&=\hat{H}_*(\pr_X)(\hat{H}_*(i_3)(\hat{H}^*(f)(a)) \\
&=\hat{H}^*(f)(a).
\end{align*}
The proof of (ii) is similar.
\end{proof}
\end{proposition}
 
\subsection{Correspondence action on relative Hodge-Witt cohomology}\label{Corr-on-rel-HW} For a scheme $S$ over $k$ we need a relative version of 
the functor \eqref{cor-functor-k} with values in de Rham-Witt modules $\cdrw_S$ over $S$.   

Let $f:X\xr{} S$ and $g:Y\xr{} S$ be two schemes over $S$. Suppose that $X,Y$ are smooth over $k$. Since $X\times_S Y\subset X\times Y$ defines a 
closed subset we can define the family of supports on $X\times Y$
\eq{relsupport}{P(X\times_S Y):= P(\Phi_X,\Phi_Y)\cap X\times_S Y,}
this is the family of supports consisting of all closed subsets of $X\times_S Y$, which are proper over $Y$.
(Recall that $\Phi_X$ denotes the family of all closed subsets of $X$.)
If  $h: Z\to S$ is another $S$-schemes which is smooth over $k$, then 
\begin{equation*}
\begin{cases} \text{$p_{13} \mid p_{12}^{-1}(P(X\times_S Y))\cap p_{23}^{-1}(P(Y\times_S Z))$ is proper,} 
\\
p_{13}(p_{12}^{-1}(P(X\times_S Y))\cap p_{23}^{-1}(P(Y\times_S Z))\subset P(X\times_S Z). \end{cases}
\end{equation*} 
By using the fact that $\CH$ defines a weak cohomology theory with supports we obtain a composition (see \cite[1.3]{CR} for details)
\begin{equation}\label{C_S-composition}
\CH(X\times Y, P(X\times_S Y)) \times \CH(Y\times Z, P(Y\times_S Z))  \xr{\circ} \CH(X\times Z, P(X\times_S Z)),
\end{equation}
$$
(a,b)\mapsto b\circ a: = \CH_*(p_{13})(\CH^*(p_{12})(a)\cup \CH^*(p_{23})(b)), 
$$
with morphisms $p_{ij}$ induced by the projections:
\begin{align*} 
p_{12}:(X\times Y\times Z, P(X\times_S Y)\times Z) &\xr{} (X\times Y, P(X\times_S Y)) \in Sm^*, \\
p_{23}:(X\times Y\times Z, X\times P(Y\times_S Z)) &\xr{} (Y\times Z, P(Y\times_S Z)) \in Sm^*, \\
p_{13}:(X\times Y\times Z, P(X\times_S Y)\cap P(Y\times_S Z)) &\xr{} (X\times Z, P(X\times_S Z)) \in Sm_*. 
\end{align*}
 
\begin{definition}\label{definition-C_S}
We define $\mc{C}_S$ to be the (graded) additive category whose objects are given by $k$-morphisms 
$
f:X\xr{} S,
$
where $X$ is a smooth scheme over $k$. Sometimes we will by abuse of notation write $X$ instead of $f:X\to S$.
The morphisms are defined by 
$$
\Hom_{\mc{C}_S}(X,Y):=\CH(X\times Y, P(X\times_S Y)),
$$
and the composition is as in \eqref{C_S-composition}.
For $f:X\xr{} S$ the identity element in $\Hom_{\mc{C}_S}(X,X)$ is the
diagonal $\Delta_X$.
\end{definition}

The verification that the composition defines a category is a straightforward
calculation. Notice that if $X\to S$ and $Y\to S$ are proper, then $\Hom_{\mc{C}_S}(X,Y)= \CH(X\times_S Y)$. The full subcategory of $\mc{C}_S$ whose objects are proper $S$-schemes, which are smooth over $k$, 
has been defined and studied in \cite[Definition~2.8]{CH}. 

\subsubsection{} Let $a\in \Hom_{\mc{C}_S}(f:X\xr{} S,g:Y\xr{} S)=\CH(X\times Y, P(X\times_S Y))$, with $X$ and $Y$ integral. Suppose 
that $a$ is of degree $i$, i.e.~is contained in $\CH^{\dim X+i}$. 
For every open subset $U\subset S$
we get by restriction an induced cycle 
$$
a_U\in \CH(f^{-1}(U)\times g^{-1}(U), P(f^{-1}(U)\times_U g^{-1}(U))).
$$
By using the functor $Cor_{\hat{c}l}$ \eqref{cor-functor-k} we obtain a 
morphism in $\cdrw_k$: 
$$
Cor_{\hat{c}l}(a_U):\hat{H}(f^{-1}(U))(-i) \xr{} \hat{H}(g^{-1}(U)). 
$$
Recall that $\hat{H}(f^{-1}(U))=\bigoplus_{p,q\ge 0} H^q(f^{-1}(U),W\Omega^p_{f^{-1}(U)})$;
we denote by 
\begin{equation*}
U\mapsto \sF(U)=\hat{H}(f^{-1}(U))(-i), \quad U\mapsto \sG(U):=\hat{H}(g^{-1}(U)), 
\end{equation*}
the obvious presheaves on $S$. 

\begin{lemma}\label{lemma-correspondences-relative} 
Let $a\in \Hom_{\mc{C}_S}(f:X\xr{} S,g:Y\xr{} S)=\CH(X\times Y, P(X\times_S Y))$.
The collection $(Cor_{\hat{c}l}(a_U))_{U\subset S}$ defines a morphism of presheaves
$$
(Cor_{\hat{c}l}(a_U))_{U\subset S}: \sF \xr{} \sG.
$$
By sheafification we obtain a morphism of sheaves 
\begin{equation}\label{lemma-correspondences-relative-statement}
(Cor_{\hat{c}l}(a_U))_{U\subset S}: \bigoplus_{p,q\ge 0} R^qf_*W\Omega^p_X(-i) \xr{} \bigoplus_{p,q\ge 0} R^qg_*W\Omega^p_Y.
\end{equation}
\begin{proof}
Let $U,V$ be two open subsets of $S$ such that $V\subset U$.
We denote by $\Gamma(f^{-1}(V)\xr{} f^{-1}(U))^t\subset f^{-1}(U)\times f^{-1}(V)$ 
the transpose of the graph of the inclusion $f^{-1}(V)\xr{} f^{-1}(U)$, and 
similarly for $\Gamma(g^{-1}(V)\xr{} g^{-1}(U))^t$. 

Since the restrictions $\sF(U)\xr{} \sF(V)$ and $\sG(U)\xr{} \sG(V)$ are induced by 
the correspondences
$Cor_{\hat{c}l}(\Gamma(f^{-1}(V)\xr{} f^{-1}(U))^t)$ and 
$Cor_{\hat{c}l}(\Gamma(g^{-1}(V)\xr{} g^{-1}(U))^t)$, respectively, we only need to prove that
$$
\Gamma(g^{-1}(V)\xr{} g^{-1}(U))^t \circ a_U= a_V \circ \Gamma(f^{-1}(V)\xr{} f^{-1}(U))^t
$$
as morphisms in $Cor_{CH}$, i.e.~as cycles in 
$$\CH(f^{-1}(U)\times g^{-1}(V), P(f^{-1}(U)\times_S g^{-1}(V))).$$ 
Via the identification $f^{-1}(U)\times_S g^{-1}(V)=f^{-1}(V)\times_S g^{-1}(V)$ both 
sides equal $a_V$. 

Since the sheafification of $\sF$ is $\bigoplus_{p,q\ge 0} R^qf_*W\Omega^p_X(-i)$,
and similarly for $\sG$, we obtain \eqref{lemma-correspondences-relative-statement}. 
\end{proof}
\end{lemma}

\begin{proposition}\label{proposition-C_S-to-dRW_S}
The assignment 
\begin{align*}
\hat{\sH}(?/S)&:\mc{C}_S \xr{} \cdrw_S, \\
\hat{\sH}(X/S)&:= \bigoplus_{p,q\ge 0} R^qf_*W\Omega^p_X, \\
\hat{\sH}(a/S)&:= (Cor_{\hat{c}l}(a_U))_{U\subset S},
\end{align*}
(cf.~Lemma \ref{lemma-correspondences-relative}) defines a functor to de Rham-Witt modules over $S$.
\end{proposition}

For the proof of the Proposition we will need the following lemma.
\begin{lemma} \label{lemma-linear2}
Let $S=\Spec(R)$. Let $X,Y$ be  $S$-schemes which are smooth over
$k$. We denote by $f^*:W(R)\xr{} H^0(X,W\OO_X)$ and 
$g^*:W(R)\xr{} H^0(Y,W\OO_Y)$ the maps induced by $f:X\xr{} S$
and $g:Y\xr{} S$, respectively. We denote by $\pr_1:X\times Y\xr{} X$
and $\pr_2:X\times Y\xr{} Y$ the projections.

Let $Z\subset X\times_S Y$ be an irreducible closed subset which is proper over $Y$; we set 
$c=\codim_{X\times Y}Z$. Then 
$$
 \hat{H}^*(\pr_2)(g^*(r))\cup \hat{c}l([Z])=  \hat{H}^*(\pr_1)(f^*(r))\cup \hat{c}l([Z])
$$
in $H^c_Z(X\times Y, W\Omega^c_{X\times Y})$.
\begin{proof}
Choose an open set $U\subset X\times Y$ such that $Z\cap U$ is nonempty and smooth. 
Since the natural map 
$$
H^c_Z(X\times Y, W\Omega^c_{X\times Y})\to H^c_{Z\cap U}(U, W\Omega^c_{U})
$$ 
is injective (Proposition \ref{proposition-semi-purity}),
it suffices to check the equality on $H^c_{Z\cap U}(U, W\Omega^c_{U})$. 
We write $\imath_1:Z\cap U\xr{} (U,Z\cap U)$ in $Sm_*$ and $\imath_2:Z\cap U \xr{} U$ in $Sm^*$ for the obvious morphisms. 
By using the projection formula and 
$\hat{c}l([Z\cap U])=\hat{H}_*(\imath_1)(1)$ (\ref{2.6.3}(1)) we reduce to the statement 
$$
\hat{H}^*(\imath_2)\hat{H}^*(\pr_2)(g^*(r))=\hat{H}^*(\imath_2)\hat{H}^*(\pr_1)(f^*(r)).
$$
This follows from $g\circ \pr_2 \circ \imath_2 = f\circ \pr_1 \circ \imath_2$. 
\end{proof}
\end{lemma}

\begin{proof}[Proof of Proposition \ref{proposition-C_S-to-dRW_S}] 
For two composable morphisms $a,b$ in $\mc{C}_S$, we clearly have 
$$
(b\circ a)_U=b_U\circ a_U
$$
for every open $U\subset S$. 
This implies $\hat{H}(b\circ a/S)=\hat{H}(b/S)\circ \hat{H}(a/S)$.

Moreover, $Cor_{\hat{c}l}(a_U)$ is a morphism in $\cdrw_k$ for all $U$, 
thus $\hat{\sH}(a/S)$ commutes with $F,V,d$. 
Finally, we need to show that $\hat{\sH}(a/S)$ induces a morphism of $W(\OO_S)$-modules. 
For this, we may assume that $S=\Spec(R)$ is affine, and it suffices to show that 
$$
Cor_{\hat{c}l}(a):\hat{H}(X)(-i) \xr{} \hat{H}(Y) 
$$
is $W(R)$-linear. We proceed as in the proof of \cite[Theorem~3.2.3]{CR}.
 The ring homomorphism  $f^*:W(R)\xr{} H^0(X,W\OO_X)$
and $g^*:W(R)\xr{} H^0(Y,W\OO_Y)$ induce the $W(R)$-module structures on 
$\hat{H}(X)$ and $\hat{H}(Y)$ via the $\cup$-product: 
$$
r\cdot a= f^*(r)\cup a, 
$$
for all $r\in W(R)$ and $a\in \hat{H}(X)$; similarly for $\hat{H}(Y)$.
We have to prove the following equality for all $r\in R, a\in \hat{H}(X)$, and 
irreducible closed subsets $Z\subset X\times_S Y$:
$$g^*(r)\cup  \hat{H}_*(\pr_2)( \hat{H}^*(\pr_1)(a)\cup \hat{c}l([Z]))= \hat{H}_*(\pr_2)( \hat{H}^*(\pr_1)(f^*(r)\cup a)\cup \hat{c}l([Z])).$$
For this, it is enough to show that
$$
 \hat{H}^*(\pr_2)(g^*(r))\cup \hat{c}l([Z])=  \hat{H}^*(\pr_1)(f^*(r))\cup \hat{c}l([Z])
$$
in $H^c_Z(X\times Y, W\Omega^c_{X\times Y})$, with $c$ the codimension of $Z$.
Lemma \ref{lemma-linear2} implies the claim.
\end{proof}

\begin{proposition}\label{proposition-comparison-with-pull-push-over-S}
Let $X$ and $Y$ be two $S$-schemes which are smooth over $k$.
\begin{itemize} 
\item[(i)] If $Z$ is the transpose of the graph of a morphism $h:X\xr{} Y$ over $S$ then $\hat{\sH}([Z]/S)$ is the pullback morphism defined in \ref{2.1.2}.
\item[(ii)] If $Z$ is the graph of a proper morphism $h:X\xr{} Y$ over $S$ then $\hat{\sH}([Z]/S)$ is the pushforward morphism defined in \ref{2.1.6}.
\end{itemize}
\begin{proof}
The statement follows from Proposition \ref{proposition-comparison-with-pull-push}.
\end{proof}
\end{proposition}

\subsubsection{Local cup product}\label{local-cup-product}
Let $X$ be a smooth equidimensional $k$-scheme and $Z\subset X$ an integral closed subscheme of codimension $c$. We have (see e.g. \eqref{semi-purity1})
\[\sH^i_Z(X, W\Omega_X)=0,\quad \text{for all } i<c.\]
Hence there is a natural morphism in the derived category of de Rham-Witt modules on $X$
\eq{supp-sheaf-to-derived-cat}{\sH^c_Z(W\Omega_X)\to R\ul{\Gamma}_Z W\Omega_X[c]}
inducing an isomorphism
\eq{supp-local-to-global}{H_Z^c(X, W\Omega_X)\cong H^0(X, \sH^c_Z(W\Omega_X)).}
We may thus define a local version of the cup product with the cycle class of $Z$, $\hat{c}l(Z)\in H^0(X, \sH^c_Z(W\Omega_X))$
\eq{localcup}{W\Omega_X\to \sH^c_Z(W\Omega_X)(c),\quad \alpha\mapsto \alpha\cup \hat{c}l([Z]):=\Delta^*((-1)^{c\cdot\deg \alpha}(\alpha\times \hat{c}l([Z]))),}
as the composition 
\mlnl{W\Omega_X\xr{\text{nat.}}\pr_{1*}\pr_1^{-1}W\Omega_X\xr{(-1)^{c\cdot\deg \alpha}(-\times \hat{c}l([Z]))} \pr_{1*}\sH^c_{X\times Z}(W\Omega_{X\times X})(c)\\
                 \xr{\Delta^*} \pr_{1*}\Delta_*\sH^c_{Z}(W\Omega_X)(c)=\sH^c_{Z}(W\Omega_X)(c) }
where $\Delta: (X, Z)\to (X\times X, X\times Z)$ in $(Sm^*/X\times X)$ is induced by the diagonal, $\Delta^*$ is the pullback constructed in \ref{2.1.2} and
 $\times$ is the exterior product from \ref{2.5.4}. Notice that by Lemma \ref{2.5.5} $\alpha\cup \hat{c}l(Z)$ equals $(-1)^{c\cdot\deg \alpha}\alpha\cdot \hat{c}l([Z])$.

\begin{lemma}\label{comparison-global-local-cup}
In the above situation the cup product with $\hat{c}l([Z])$,
\[H^i(X, W\Omega_X)\to H^{i+c}_Z(X, W\Omega_X)(c),\quad a\mapsto a\cup \hat{c}l([Z]),\]
factors via the local cup product, i.e. equals the composition
\[H^i(X, W\Omega_X)\xr{\eqref{localcup}} H^i(X, \sH^c_Z(W\Omega_X)(c))\xr{\eqref{supp-sheaf-to-derived-cat}} H^{i+c}_Z(X, W\Omega_X)(c).\] 
\end{lemma}
\begin{proof}
 Recall that for $a\in H^i(X, W\Omega^q_X)$ the cup product $a\cup \hat{c}l([Z])$ equals 
\[H^*(\Delta)((-1)^{(i+q)c}(a\times \hat{c}l([Z]))),\]
where $\Delta: (X, Z)\to (X\times X, X\times Z)$ in $(Sm^*/\Spec\,k)$ is the diagonal. Therefore it suffices to show that the exterior product
\eq{comparison-local-global-cup1}{\times \hat{c}l([Z]): H^i(X, W\Omega_X^q) \to H^{i+c}_{X\times Z}(X\times X, W\Omega_{X\times X}^{q+c})}
factors via $H^i(X,-)$ applied to
\[\pr_1^{-1}W\Omega^q_X\to \sH^c_{X\times Z}(W\Omega_{X\times X}),\quad \alpha\mapsto (-1)^{ic}(\alpha\times \hat{c}l([Z])),\]
composed with the natural map $\sH^c_{X\times Z}(W\Omega_{X\times X})\to R\ul{\Gamma}_{X\times Z}W\Omega_{X\times X}[c]$.
Let $E(W_\bullet\Omega_X)$ be the Cousin complex of $W_\bullet\Omega_X$ (see \ref{CousinWitt}, Lemma \ref{drw-is-CM}), then the complex $\ul{\Gamma}_Z E(W_\bullet\Omega_X)$
equals zero in all degrees $<c$, hence there is a morphism of {\em complexes} 
\eq{comparison-local-global-cup2}{\sH^c_Z(W_\bullet\Omega_X)\to  \ul{\Gamma}_Z E(W_\bullet\Omega^c_X)[c],}
which, after applying $\varprojlim$ to it, represents \eqref{supp-sheaf-to-derived-cat}.
We obtain the following commutative diagram
\[\xymatrix{\pr_1^{-1}E(W_\bullet\Omega^q_X)\ar[r]^-{\boxtimes \,\hat{c}l([Z])} & E(W_\bullet\Omega^q_X)\boxtimes \sH^c_Z(W_\bullet\Omega^c_X)\ar[r] & \ul{\Gamma}_{X\times Z}E(W_\bullet\Omega^{q+c}_{X\times X})[c]\\
           \pr_1^{-1} W_\bullet\Omega^q_X\ar[u]\ar[r]^-{\boxtimes \,\hat{c}l([Z])} & W_\bullet\Omega^q_X\boxtimes \sH^c_Z(W_\bullet\Omega^c_X)\ar[r]^\times\ar[u] & 
                                                                                                                       \sH^c_{X\times Z}(W_\bullet\Omega^{q+c}_{X\times X})\ar[u]_{\eqref{comparison-local-global-cup2}},
            }\]
where the top right arrow is the composition 
\mlnl{E(W_\bullet\Omega_X^q)\boxtimes \sH^c_Z(W_\bullet\Omega^c_X) \xr{\id\boxtimes \eqref{comparison-local-global-cup2}} E(W_\bullet\Omega^q_X)\boxtimes \ul{\Gamma}_Z E(W_\bullet\Omega^c_X)[c]\\ 
       \xr{\simeq} (E(W_\bullet\Omega^q_X)\boxtimes \ul{\Gamma}_Z E(W_\bullet\Omega^c_X))[c]\xr{\eqref{2.5.6.1}} \ul{\Gamma}_{X\times Z}E(W_\bullet\Omega^{q+c}_{X\times X})[c]. }
Notice that the isomorphism 
\eq{shift-sign}{(E\boxtimes \ul{\Gamma}_Z E[c])\cong (E\boxtimes \ul{\Gamma}_Z E)[c]}
is given by multiplication with $(-1)^{ic}$ on $E^i\boxtimes \ul{\Gamma}_Z E^j$, for all $i,j$ (see the sign convention in \cite[(1.3.6)]{Co}).
Now we apply $H^i(X,-)\circ\varprojlim$ to the above diagram, use Proposition \ref{2.5.7}, (1), take care about the sign from \eqref{shift-sign} and use \eqref{2.6.2.3} and   
we see that \eqref{comparison-local-global-cup1} factors as desired.
\end{proof}

\subsubsection{}\label{CStoCTviaforget}
Let $h: S\to T$ be a morphism of $k$-schemes. Then any two objects $X, Y\in \mc{C}_S$ naturally define objects in $\mc{C}_T$ (via $h$) and $X\times_S Y\subset X\times_T Y$ is a closed subscheme.
This gives a natural map $\CH(X\times Y, P(X\times_S Y))\to \CH(X\times Y, P(X\times_T Y))$. In this way $h$ induces a functor 
\[\mc{C}_S\to \mc{C}_T.\]
If $h$ is fixed, we denote the image of $a\in \Hom_{\mc{C}_S}(X,Y)$ in  $\Hom_{\mc{C}_T}(X, Y)$ via this functor again by $a$. But notice that this functor is in general not faithful.

\begin{proposition}\label{correspondences-and-change-of-basis}
Let $h: S\to T$ be a morphism of $k$-schemes. Let $f: X\to S$ and $g: Y\to S$ be two objects in $\mc{C}_S$ and assume that $X$ and $Y$ are integral and  $f$ and $g$ are \emph{affine}.
Let $Z\subset X\times_S Y$ be a closed integral subscheme which is \emph{finite and surjective} over $Y$, therefore giving rise to a morphism in $\cdrw_S$ (by Proposition \ref{proposition-C_S-to-dRW_S})
\[\hat{\sH}([Z]/S): f_*W\Omega_X\to g_* W\Omega_Y.\]
Then we have the following equality of morphisms in $\cdrw_T$
\[\hat{\sH}([Z]/T)=\bigoplus_i  R^ih_*\hat{\sH}([Z]/S): \bigoplus_i R^i(h f)_*W\Omega_X\to \bigoplus_i R^i(h g)_*W\Omega_Y. \]
\end{proposition}

\begin{proof}
We consider the following composition  in the derived category of abelian sheaves on $S$:
\begin{align}\label{local-corr}
 f_* W\Omega_X & \xr{\pr_1^*} R(f\pr_1)_*W\Omega_{X\times Y}\\
               & \xr{\eqref{localcup}} R(f\pr_1)_*\sH^c_Z(W\Omega_{X\times Y})(c)\nonumber\\
               & \xr{\simeq, \, Z\subset X\times_S Y} R(g\pr_2)_*\sH^c_Z(W\Omega_{X\times Y})(c) \nonumber\\
               & \xr{\eqref{supp-sheaf-to-derived-cat}} g_*R{\pr_2}_*R\ul{\Gamma}_Z (W\Omega_{X\times Y})[c](c)\nonumber\\
               & \xr{\eqref{derived-pf}} g_*W\Omega_Y\nonumber.
\end{align}
Notice that the third arrow only exists in the category of abelian sheaves, it is not respecting the $W\sO_S$-module structure.
We claim that the composition \eqref{local-corr} equals $\hat{\sH}([Z]/S)$ and that $\oplus_i R^ih_*\eqref{local-corr}$ equals $\hat{\sH}([Z]/T)$.
This will prove the statement. Clearly it suffices to prove the last claimed equality, the first then follows with $h=\id$. 

To this end, let $U\subset T$ be an open subset. We denote by $X_U$, $Y_U$ and $Z_U$  the pullbacks of $X$, $Y$ and $Z$  over $U$.
Then $Cor_{\hat{c}l}([Z_U]): H^i(X_U, W\Omega_{X_U})\to H^i(Y_U, W\Omega_{Y_U})(c) $ is given by the following composition 
\begin{align*}
H^i(X_U, W\Omega_{X_U}) & \xr{\pr_1^*}  H^i(X_U, R\pr_{1*}W\Omega_{X_U\times Y_U})\\
                        &\xr{\simeq} H^i(Y_U, R\pr_{2*}W\Omega_{X_U\times Y_U})\\
                        & \xr{\cup\, \hat{c}l([Z_U])} H^i(Y_U, R\pr_{2*}R\ul{\Gamma}_{Z_U} W\Omega_{X_U\times Y_U}[c](c))\\
                        & \xr{\eqref{derived-pf}} H^i(Y_U, W\Omega_{Y_U}).
\end{align*}
Since $f$ and $g$ are affine this composition equals by Lemma \ref{comparison-global-local-cup} the composition $H^i(h^{-1}(U), \eqref{local-corr})$.
By definition $\hat{\sH}([Z]/T)$ is the sheafification of $U\mapsto Cor_{\hat{c}l}([Z_U])$ and the sheafification of
$U\mapsto H^i(h^{-1}(U), \eqref{local-corr})$ is $R^ih_*(\eqref{local-corr})$. This proves the claim.

\end{proof}

\subsection{Vanishing results} \label{section-vanishing-results}
Recall from Proposition \ref{proposition-C_S-to-dRW_S} that we have 
a functor 
\begin{align*}
\hat{\sH}(?/S)&:\mc{C}_S \xr{} \cdrw_S, \\
\hat{\sH}(X/S)&:= \bigoplus_{p,q\ge 0} R^qf_*W\Omega^p_X 
\end{align*}
(cf. Definition \ref{definition-C_S} for $\mc{C}_S$). Let 
$f:X\xr{} S,g:Y\xr{} S$
be objects in $\mc{C}_S$, i.e. $S$-schemes  which are smooth over $k$. 
For simplicity let us assume that $X$ and $Y$ are connected.
For an equidimensional closed subset $Z\subset X\times_S Y$ which is proper over $Y$ and has 
codimension $\dim X+i$ in $X\times Y$ we obtain a morphism in $\cdrw_S$:
$$
\hat{\sH}([Z]/S):\left(\bigoplus_{p,q\ge 0} R^qf_*W\Omega^p_X\right)(-i)\xr{} \bigoplus_{p,q\ge 0} R^qg_*W\Omega^p_Y.
$$

\begin{lemma}\label{lemma-van1}
Let $Z\subset X\times_S Y$ be closed subset which is proper over $Y$. Let $r\geq 0$ be an integer. Suppose that for 
every point $z\in Z$  the image $\pr_2(z)$ is a point of codimension $\geq r$
in $Y$.
Then there is a natural number $N\geq 1$ such that 
$$
N\cdot {\rm image}(\hat{\sH}([Z]/S)) \subset \bigoplus_{p\ge r,q\ge r} R^qg_*W\Omega^p_Y.
$$ 
In other words, the projection of ${\rm image}(\hat{\sH}([Z]/S))$ to 
$$
\bigoplus_{p<r\text{ or }q<r} R^qg_*W\Omega^p_Y
$$
is killed by $N$.
\begin{proof}
We may assume that $Z$ is irreducible and $X,Y$ are connected. We set $i=\dim Y-\dim Z=\codim_{X\times Y} Z-\dim X$.

By using an alteration we can find 
a smooth equidimensional scheme $D$ of dimension $\dim Y-r$ together with a 
proper morphism $\pi:D\xr{} Y$ such that $\pi(D)\supset \pr_2(Z)$. 
In particular, $Z$ is contained in the image of the map $id_X\times_S \pi:X\times_S D\xr{} X\times_S Y$. 

Let $Z_D\subset X\times_S D$ be an irreducible closed subset with 
$\dim(Z_D)=\dim(Z)$ and $(id_X\times \pi)(Z_D)=Z$. ($Z_D$ is automatically proper over $D$.) We define $N$ to be 
the degree of the field extension $k(Z)\subset k(Z_D)$. 

We obtain three maps:
\begin{align*}
\hH(Z/S)&:\hH(X/S)(-i)\xr{}\hH(Y/S), \\
\hH(Z_D/S)&:\hH(X/S)(-i)\xr{}\hH(D/S)(-r), \\
\hH(\Gamma(\pi)/S)&:\hH(D/S)(-r)\xr{}\hH(Y/S),
\end{align*}
($\Gamma(\pi)$ denotes the graph of $\pi$).
We claim that 
$$
N\cdot \hH(Z/S)= \hH(\Gamma(\pi)/S) \circ\hH(Z_D/S).
$$
Indeed, by functoriality it is sufficient to prove 
$$
N\cdot [Z] =  [\Gamma(\pi)] \circ [Z_D]
$$
where $\circ$ is the composition in $\mc{C}_S$ (see \eqref{C_S-composition}).
This is an easy computation in intersection theory.

Thus it is sufficient to show that 
$$
{\rm image}(\hat{\sH}([\Gamma(\pi)]/S)) \subset \bigoplus_{p\ge r,q\ge r} R^qg_*W\Omega^p_Y.
$$
Proposition \ref{proposition-comparison-with-pull-push-over-S} implies that
$\hat{\sH}([\Gamma(\pi)]/S)$ is the push-forward $\pi_*$ defined in \ref{2.1.6}.
Thus 
$$
\hat{\sH}([\Gamma(\pi)]/S)(R^q(g\circ \pi)_*W\Omega^p_D)\subset R^{q+r}g_*W\Omega^{p+r}_Y, 
$$
for all $(p,q)$, which completes the proof.
\end{proof}
\end{lemma}

\begin{lemma}\label{lemma-van2}
Let $Z\subset X\times_S Y$ be closed subset, which is proper over $Y$. Suppose that $X$ is equidimensional
of dimension $d$. Let $r\geq 0$ be an integer. Suppose that for 
every point $z\in Z$  the image $\pr_1(z)$ is a point of codimension $\geq r$
in $X$.
Then there is a natural number $N\geq 1$ such that 
$$
N\cdot \left(\bigoplus_{p>d-r\text{ or }q>d-r} R^qf_*W\Omega^p_X\right) \subset  {\rm ker}(\hat{\sH}([Z]/S)).
$$
\begin{proof}
The proof is analogous to the proof of Lemma \ref{lemma-van1}.

We may assume that $Z$ is irreducible and $Y$ connected. 
We set $i=\dim Y-\dim Z=\codim_{X\times Y} Z-\dim X$.
  
By using an alteration we can find 
a smooth equidimensional scheme $D$ of dimension $d-r$ together with a 
proper morphism $\pi:D\xr{} X$ such that $\pi(D)\supset \pr_1(Z)$. 
In particular, $Z$ is contained in the image of the map $\pi\times_S id_Y:D\times_S Y\xr{} X\times_S Y$. 

Let $Z_D\subset D\times_S Y$ be an irreducible closed subset with 
$\dim(Z_D)=\dim(Z)$ and $(\pi\times id_Y)(Z_D)=Z$. ($Z_D$ is automatically proper over $Y$.) We define $N$ to be 
the degree of the field extension $k(Z)\subset k(Z_D)$. 

We obtain three maps:
\begin{align*}
\hH(Z/S)&:\hH(X/S)(-i)\xr{}\hH(Y/S), \\
\hH(Z_D/S)&:\hH(D/S)(-i)\xr{}\hH(Y/S), \\
\hH(\Gamma(\pi)^t/S)&:\hH(X/S)\xr{}\hH(D/S),
\end{align*}
($\Gamma(\pi)^t$ denotes the transpose of the graph of $\pi$).
We claim that 
$$
N\cdot \hH(Z/S)=\hH(Z_D/S)\circ \hH(\Gamma(\pi)^t/S).
$$
Indeed, by functoriality it is sufficient to prove 
$$
N\cdot [Z] = [Z_D] \circ [\Gamma(\pi)^t]
$$
where $\circ$ is the composition in $\mc{C}_S$ (see \eqref{C_S-composition}).
This is a straightforward calculation in intersection theory.

Because $\dim D=d-r$, the map $\hH(\Gamma(\pi)^t/S)$ vanishes on 
$$
\bigoplus_{p>d-r\text{ or }q>d-r} R^qf_*W\Omega^p_X,
$$ which proves the statement.
\end{proof}
\end{lemma}

\begin{remark}
Resolution of singularities implies that  
Lemma \ref{lemma-van1} and \ref{lemma-van2} hold for $N=1$.
\end{remark}

\subsection{De Rham-Witt systems and modules modulo torsion}\label{section-dRW-mod-torsion}

\subsubsection{} Let $\mc{A}$ be an abelian category (only $\mc{A}=\drw_X$ and $\mc{A}=\widehat{\drw}_X$ 
will be important for us). 
There are 
two natural Serre subcategories attached to torsion objects of $\mc{A}$.
We define 
$$
\mc{A}_{b-tor}:=\{X\in {\rm ob}(\mc{A})\mid \exists n\in \Z\backslash\{0\}: n\cdot id_X=0\}
$$
as full subcategory of $\mc{A}$. We define 
\begin{multline*}
\mc{A}_{tor}:=\{X\in {\rm ob}(\mc{A})\mid \exists (X_i)_{i\in I}\in \mc{A}_{b-tor}^I \exists \phi:\bigoplus_{i\in I} X_i\xr{} X  
\text{ epimorphism} \}
\end{multline*}
as full subcategory of $\mc{A}$.
Note that the index set $I$ is  not finite in general. Obviously,
$$
\mc{A}_{b-tor}\subset \mc{A}_{tor}.
$$ 
\subsubsection{}
If $\mc{A}$ is well-powered, e.g.~$\mc{A}=\widehat{\drw}_X$, then the 
quotient categories $\mc{A}/\mc{A}_{b-tor}$ and $\mc{A}/\mc{A}_{tor}$ exist and
are abelian categories. We refer to \cite[Chapitre~III]{G} for quotient categories.
The functors 
$$
q:\mc{A}\xr{} \mc{A}/\mc{A}_{b-tor}, \quad q':\mc{A}\xr{} \mc{A}/\mc{A}_{tor},
$$
are exact. Moreover, $q(X)=0$ if and only if $X\in \mc{A}_{b-tor}$; the same
statement holds for $q'$ and $\mc{A}_{tor}$. There is an obvious 
factorization 
$$
\mc{A}\xr{} \mc{A}/\mc{A}_{b-tor}\xr{} \mc{A}/\mc{A}_{tor}.
$$
If $\mc{A}$ is the category of (left) modules over a ring $R$ then 
$$
\mc{A}/\mc{A}_{tor}\cong (\text{$R\otimes_{\Z}\Q$-modules}).
$$
We define 
$$
\mc{A}_{\Q}:=\mc{A}/\mc{A}_{b-tor}.
$$
For future reference, we record the following special case.

\begin{definition}\label{definition-drw-Q}
Let $X$ be a scheme over $k$. We define
$$
\widehat{\drw}_{X,\Q}:=\widehat{\drw}_X/\widehat{\drw}_{X,b-tor}
$$
as quotient category. We denote by $q$ the projection functor 
$$
q:\widehat{\drw}_{X}\xr{} \widehat{\drw}_{X,\Q}.
$$
We use the same definitions for $\drw_X$.
\end{definition} 

The main reason for working with the quotient  $\mc{A}/\mc{A}_{b-tor}$ instead 
of $\mc{A}/\mc{A}_{tor}$ is that the Homs are well-behaved.

\begin{proposition}\label{proposition-Hom-A-Q}
Let $\mc{A}$ be a well-powered abelian category. Let $X,Y\in {\rm ob}(\mc{A})$.
Then 
$$
\Hom_{\mc{A}_{\Q}}(q(X),q(Y)) 
$$
is naturally a $\Q$-module and the map 
$$
\Hom_{\mc{A}}(X,Y)\otimes_{\Z}\Q \xr{} \Hom_{\mc{A}_{\Q}}(q(X),q(Y)) 
$$
is an isomorphism.
\begin{proof}
For any $n\in \Z\backslash \{0\}$ the morphism 
$q(Y)\xr{\cdot n} q(Y)$ in $\mc{A}_{\Q}$ is invertible; therefore
$$
\Q\subset \Hom_{\mc{A}_{\Q}}(q(Y),q(Y)).
$$
Via the composition
$$
\Hom(q(X),q(Y))\times \Hom(q(Y),q(Y))\xr{} \Hom(q(X),q(Y))
$$
we see that $\Hom(q(X),q(Y))$ is a $\Q$-module. The $\Q$-module 
structure induced by $\Q\subset \Hom(q(X),q(X))$ and 
$$
 \Hom(q(X),q(X)) \times \Hom(q(X),q(Y))\xr{} \Hom(q(X),q(Y))
$$
is the same, because 
$$
f\circ (q(X)\xr{\cdot n} q(X))= (q(Y)\xr{\cdot n} q(Y))\circ f=nf.
$$
We need to show that the canonical map 
$$
\Hom_{\mc{A}}(X,Y)\otimes_{\Z}\Q \xr{} \Hom_{\mc{A}_{\Q}}(q(X),q(Y)) 
$$
is an isomorphism.

For the injectivity it is sufficient to prove that for all 
$f\in \Hom_{\mc{A}}(X,Y)$ with $q(f)=0$ it follows that $n\cdot f=0$ 
for some $n\in \Z\backslash \{0\}$. Indeed, $q(f)=0$ implies ${\rm image}(f)\in \mc{A}_{b-tor}$; 
thus there is an integer $n\neq 0$ with $n\cdot id_{{\rm image}(f)}=0$. It follows 
that $nf=0$.

For the surjectivity, let $f:X\xr{} Y$ be a morphism in $\mc{A}$ such that 
$\ker(f),{\rm coker}(f)\in \mc{A}_{b-tor}$; equivalently $q(f)$ is an isomorphism.
We need to show that the inverse map $q(f)^{-1}$ is contained in the image of 
$\Hom_{\mc{A}}(Y,X)\otimes_{\Z}\Q$. Choose an integer $n_1\neq 0$ such that 
$n_1\cdot id_{\ker(f)}=0$. Then there exists $g_1:\im(f)\xr{} X$ such that 
$g_1\circ f=n_1\cdot id_{X}$. Let $n_2\neq 0$ be an integer such that 
$n_2\cdot id_{{\rm coker}(f)}=0$. Then $Y\xr{\cdot n_2} Y$ factors through 
$\im(f)$ and we obtain $g_2:Y\xr{\cdot n_2} \im(f)\xr{g_1} X$. 
Thus the image of $g_2\otimes (n_1\cdot n_2)^{-1}$ is the inverse
of $q(f)$. 
\end{proof}
\end{proposition}

\begin{corollary}\label{corollary-F-Q}
Let $F:\mc{A}\xr{} \mc{B}$ be a functor between well-powered abelian categories.
There is a natural functor $F_{\Q}:\mc{A}_{\Q}\xr{} \mc{B}_{\Q}$ defined by 
$$
F_{\Q}(q(X))=F(X)
$$
for every object $X$ in $\mc{A}$, and 
\begin{align*}
F_{\Q}&:\Hom_{\mc{A}_{\Q}}(q(X),q(Y))\xr{} \Hom_{\mc{B}_{\Q}}(F_{\Q}q(X),F_{\Q}q(Y))\\
F_{\Q}&:=F\otimes_{\Z} id_{\Q}
\end{align*}
via the isomorphism of Proposition \ref{proposition-Hom-A-Q}. 
\end{corollary}

\begin{remark}
The statement of Corollary \ref{corollary-F-Q} also follows from
$$
F(\mc{A}_{b-tor})\subset \mc{B}_{b-tor}
$$ 
by using the universal property of the quotient category.
\end{remark}

\begin{proposition}
Let $F:\mc{A}\xr{} \mc{B}$ be a left-exact functor between well-powered 
abelian categories. Suppose that the left derived functor 
$$
RF:D^+(\mc{A})\xr{} D^+(\mc{B})
$$ 
exists. Suppose that there exist sufficiently many $F$-acyclic objects in $\mc{A}$.
Then 
$$
F_{\Q}:\mc{A}_{\Q}\xr{} \mc{B}_{\Q}
$$
is left-exact and the left derived functor 
$$
RF_{\Q}:D^+(\mc{A}_{\Q})\xr{} D^+(\mc{B}_{\Q})
$$
exists. Moreover, the diagram 
$$
\xymatrix
{
D^+(\mc{A})\ar[r]^{RF}\ar[d]_{D^+(q)}
& 
D^+(\mc{B})\ar[d]^{D^+(q)}
\\
D^+(\mc{A}_{\Q})\ar[r]^{RF_{\Q}}
& 
D^+(\mc{B}_{\Q})
}
$$
is commutative.
\begin{proof}
Let 
\begin{equation}\label{equation-X'Y'Z'-exact-sequence}
0\xr{} X' \xr{f'} Y' \xr{g'} Z' \xr{} 0 
\end{equation}
be an exact sequence in $\mc{A}_{\Q}$. By 
\cite[Chapitre~III,\textsection1, Corollaire~1]{G} we can find 
an exact sequence 
\begin{equation}\label{equation-XYZ-exact-sequence}
0\xr{} X \xr{f} Y \xr{g} Z \xr{} 0 
\end{equation}
in $\mc{A}$ and isomorphisms $X'\xr{\cong} q(X),Y'\xr{\cong} q(Y),Z'\xr{\cong}q(Z),$
such that the diagram 
$$
\xymatrix
{
0\ar[r] 
&
X' \ar[r]^{f'} \ar[d]^{\cong}
&
Y' \ar[r]^{g'} \ar[d]^{\cong}
&
Z' \ar[r] \ar[d]^{\cong}
& 
0 
\\
0\ar[r] 
&
q(X) \ar[r]^{f} 
&
q(Y) \ar[r]^{g} 
&
q(Z) \ar[r] 
& 
0 
}
$$
is commutative. Since $q$ is exact and $qF=F_{\Q}q$, it follows that $F_{\Q}$
is left exact.

We define 
$$
P_{\Q}:=\{X'\in {\rm ob}(\mc{A}_{\Q})\mid \exists X\in {\rm ob}(\mc{A}): q(X)\cong X',\; R^iF(X)\in \mc{B}_{b-tor} \; \text{for all $i>0$}\}
$$
as a full subcategory of $\mc{A}_{\Q}$. If $X\in {\rm ob}(\mc{A})$ is an $F$-acyclic
object then $q(X)\in P_{\Q}$. Therefore every object $Y\in {\rm ob}(\mc{A}_{\Q})$
admits a monomorphism $\imath:Y\xr{} X$ with $X\in {\rm ob}(P_{\Q})$.

Suppose that in the exact sequence \eqref{equation-X'Y'Z'-exact-sequence} we know
that $X'\in {\rm ob}(P_{\Q})$. We claim that the following holds
$$
Y'\in {\rm ob}(P_{\Q}) \Leftrightarrow Z'\in {\rm ob}(P_{\Q}).
$$
As above, we may prove the claim for \eqref{equation-XYZ-exact-sequence} instead
of \eqref{equation-X'Y'Z'-exact-sequence}.
It is easy to see that  $R^iF(X)\in \mc{B}_{b-tor}$ for all $i>0$, and 
we conclude that 
$$
qR^iF(Y)\xr{\cong} qR^iF(Z) \quad \text{for all $i>0$.}
$$
This implies the claim.

In the same way we can see that if \eqref{equation-X'Y'Z'-exact-sequence} is 
an exact sequence with objects in $P_{\Q}$ then 
$$
0\xr{} F_{\Q}(X')\xr{} F_{\Q}(Y')\xr{} F_{\Q}(Z')\xr{} 0
$$
is exact. From \cite[I, Cor. 5.3, $\beta$]{Ha} it follows that the left derived
functor 
$$
RF_{\Q}:D^{+}(\mc{A}_{\Q})\xr{} D^{+}(\mc{B}_{\Q})
$$ 
exists. Moreover, if $X$ is $F$-acyclic then $q(X)$ is $F_{\Q}$ acyclic. 
Therefore the diagram 
$$
\xymatrix
{
D^+(\mc{A})\ar[r]^{RF}\ar[d]_{D^+(q)}
& 
D^+(\mc{B})\ar[d]^{D^+(q)}
\\
D^+(\mc{A}_{\Q})\ar[r]^{RF_{\Q}}
& 
D^+(\mc{B}_{\Q})
}
$$
is commutative.

\end{proof}
\end{proposition}

\begin{remark}
Let $f:X\to Y$ be a morphism between $k$-schemes and $\Phi$ a family of supports on $X$.
The assumptions of the Proposition are satisfied for the functors 
\begin{align*}
\ul{\Gamma}_\Phi&: \drw_X\to \drw_X,\\
f_*&: \drw_X\to \drw_Y,\\
\varprojlim &: \drw_X\to \cdrw_X,\\
f_\Phi&: \drw_X \to \drw_Y,\\
\hat{f}_\Phi&: \drw_X\to \cdrw_Y,
\end{align*}
of Proposition \ref{derived-functors-exist}.
\end{remark}

\begin{notation}\label{Q-notation}
In general, we will denote by a subscript $\Q$ the image of an object of $\widehat{\drw}_X$ (resp.~of $D^+(\widehat{\drw}_X)$) under the localization functor
 $q$ (resp.~$D^+(q)$). If $F:\widehat{\drw}_X\to \widehat{\drw}_Y$ is a functor we will 
by abuse of notation denote $F_{\Q}$ again by $F$ and $RF_{\Q}$ again by $RF$. 
 Thus $q(Rf_*W\Omega_X)$ will be denoted by $Rf_* W\Omega_{X,\Q}$, etc. (Warning: $W\Omega_{X,\Q}$ is not the same as  $W\Omega_X\otimes_\Z \Q$.)
\end{notation}

\section{Witt-rational singularities}
All schemes in this section are quasi-projective over $k$.
\subsection{The Witt canonical system}
\subsubsection{}\label{recall-Witt-residual}
Recall from Definition \ref{definition-Wittresidual} the notion of a Witt residual complex. Let $X$ be a $k$-scheme with structure map $\rho_X: X\to \Spec k$, then there is a canonical Witt residual complex 
$K_X=\rho_X^\Delta W_\bullet \omega$ (see Notation \ref{KX}). This complex has the following properties:
\begin{enumerate}
 \item If $X$ is smooth of pure dimension $d$, then there is a quasi-isomorphism of graded complexes $\tau_X: W_\bullet\omega_X\to K_X(-d)[-d]$, which is compatible with localization.
        (Ekedahl, see Theorem \ref{1.6.12}.) 
 \item If  $j: U\inj X$ is an open subscheme, then $j^*K_X\cong K_U$ (see Proposition \ref{1.6.11}).
 \item If $f:X\to Y$ is a morphism of $k$-schemes, then there is a canonical isomorphism $f^\Delta K_Y\cong K_X$ induced by $c_{f, \rho_Y}$, where $\rho_Y$ is the structure map of $Y$.
       This isomorphism is compatible with composition and localization and in case $f$ is an open embedding also with the isomorphism in (2) (via $f^*\cong f^\Delta$).
        (See Proposition \ref{1.6.11}.)
 \item For a proper $k$-morphism $f:X \to Y$, there is a trace map $\Tr_f: f_*K_X\to K_Y$, which is a morphism of complexes of Witt quasi-dualizing systems (see Definition \ref{1.4.1}); it is compatible
       with composition and localization (see Lemma \ref{1.6.13}).
\item We have a functor
   \[D_X:=\sHom(-,K_X): C(\drw_{X,{\rm qc}})^o \to  C(\drw_X).\]
      It preserves quasi-isomorphisms and hence also induces a functor from $D(\drw_{X, {\rm qc}})^o$ to $D(\drw_X)$ (see \ref{1.6.14}).
\item Let $f: X\to Y$ be a finite morphism. We denote by
         \[\vartheta_f: f_*D_X(-)\to D_Y(f_*(-))\]
      the composition 
\[f_*\sHom(-, K_X)\xr{\rm nat.} \sHom(f_*(-), f_*K_X)\xr{\Tr_f} \sHom(-, K_Y).\] 
      Then $\vartheta_f$ is an isomorphism of functors on $C^-(\drw_{X, {\rm qc}})$; it is compatible with composition and localization.
      (It is an isomorphism on each level by duality theory, see \ref{1.6.3}, (7); for the other assertions see Proposition \ref{1.6.14.4.5}.)
\end{enumerate}

\begin{definition}\label{definition-Witt-canonical}
 Let $X$ be a $k$-scheme of pure dimension $d$. The {\em Witt canonical system on $X$} is defined to be the $(-d)$-th cohomology of $K_X$ sitting in degree $d$ and is denoted by $W_\bullet\omega_X$, i.e.
    \[W_\bullet\omega_X:=H^{-d}(K_X)(-d).\]
 Since $W_\bullet \sO_X$ is a Witt system (i.e. a de Rham-Witt system with zero differential), 
       $W_\bullet\omega_X$  inherits the structure of a Witt system from the canonical isomorphism $K_X\cong \sHom(W_\bullet\sO_X, K_X)$ and \ref{recall-Witt-residual}, (5).
       We denote the limit with respect to $\pi$ by
\[W\omega_X:=\varprojlim_\pi W_\bullet\omega_X.\]
\end{definition}

\begin{remark}\label{Witt-canonical-old-new}
 Notice that in case $X$ is smooth the above definition of $W_\bullet\omega_X$ coincides (up to canonical isomorphism) with our earlier
 definition of $W_\bullet\omega_X=W_\bullet\Omega_X^d$, by \ref{recall-Witt-residual}, (1). But observe, in Example \ref{1.4.4}, (2) we viewed it as a Witt dualizing system,
whereas now as a Witt system. This will not cause any confusion.
\end{remark}

\begin{proposition}\label{properties-Witt-canonical}
Let $X$ be a $k$-scheme of pure dimension $d$. Then $W_\bullet\omega_X$ has the following properties:
\begin{enumerate}
\item The sheaf $W_1\omega_X$ equals the usual canonical sheaf on $X$ if $X$ is normal.
 \item The complex $K_X$ is concentrated in degree $[-d, 0]$, hence there is a natural morphism of complexes
        \[W_\bullet\omega_X\to K_X(-d)[-d].\]
        This morphism is a quasi-isomorphism if $X$ is CM.
\item For each $n$ the sheaf $W_n\omega_X$ is a coherent sheaf on $W_n X$, which is  $S_2$.
\item Let $j: U\inj X$ be an open subscheme which contains all 1-codimensional points of $X$. Then
      we have an isomorphism of Witt systems
       \[W_\bullet\omega_X\xr{\simeq,\, \text{nat.}} j_*j^*W_\bullet\omega_X= j_*W_\bullet\omega_U.\]
        If $X$ is normal, $U$ can be chosen to be smooth, in which case we have an isomorphism 
               \[j_*W_\bullet\omega_U\simeq j_*W_\bullet\Omega^d_U,\]
        which is induced by \ref{recall-Witt-residual}, (1) and (2). In particular the transition maps $W_n\omega_X\to W_{n-1}\omega_X$ are surjective if $X$ is normal.
\item  Assume $X$ is normal. Then there is a natural isomorphism for all $n$
        \[\sHom(W_n\omega_X, W_n\omega_X)\cong H^{-d}(D_{X,n}(W_n \omega_X)(-d)),\]
        where $D_{X,n}(-)=\sHom(-, K_{X,n})$.
      Therefore $\sHom(W_\bullet\omega_X, W_\bullet\omega_X)$ is naturally equipped with the structure of a Witt system and
      multiplication induces an isomorphism of Witt systems
       \[W_\bullet\sO_X\xr{\simeq} \sHom(W_\bullet\omega_X, W_\bullet\omega_X).\]
\item Let $f:X\to Y$ be a proper morphism between $k$-schemes, which are both of pure dimension $d$. 
      Then we  define the pushforward $f_*: f_*W_\bullet\omega_X\to W_\bullet\omega_Y$ as the composition in $\drw_Y$
         \[f_*: f_*W_\bullet\omega_X=f_*H^{-d}(K_X)(-d)=H^{-d}(f_* K_X)(-d)\xr{\Tr_f} H^{-d}(K_Y)(-d)= W_\bullet\omega_Y\]
         This morphism is compatible with composition and localization and in case $X$ and $Y$ are smooth coincides with the pushforward defined in Definition \ref{2.1.6}
          (for $S=Y$.)
\item   The sequence of $W_n\sO_X$-modules
       \[0\to i_{n*}W_{n-1}\omega_X\xr{\ul{p}} W_n\omega_X\xr{F^{n-1}} i_* F^{n-1}_{X*}\omega_X\]
       is exact for any $n\ge 1$.  Furthermore, if $X$ is CM, then the map on the right is also surjective. (Here we write $\omega_X:=W_1\omega_X$.)
\end{enumerate}
\end{proposition}

\begin{proof}
(1) is clear by construction. (In the derived category of coherent sheaves $K_{X,1}$ is isomorphic to $\rho_X^! k$, where $\rho_X: X\to \Spec \,k$ is the structure map.)
It suffices to prove the other statements on a fixed finite level $n$.

(2). The codimension function associated to $K_{X,n}$ (see \ref{1.6.1}) is given by 
\[d_{K_{X,n}}(x)=-{\rm trdeg} (k(x)/k)=\dim \sO_{X,x}-d\]
 (see \cite[(3.2.4)]{Co}).
  This already gives the first statement of (2). For the second statement the same argument as in \cite[I, 2.]{EI} works (there for smooth schemes). For the convenience
 of the reader we recall the argument. Let $X$ be CM. We have to show that $H^i(K_{X,n})=0$ for all $i\neq -d$. For $n=1$ this follows from (1). We have an exact sequence of coherent $W_n\sO_X$-modules
\[0\to i_*F_{X*}^{n-1}\sO_X\xr{V^{n-1}} W_n\sO_X\xr{\pi} i_{n*} W_{n-1}\sO_X,\]
where $F_X$ denotes the absolute Frobenius on $X$ and $i$, $i_n$ denote the closed embeddings $X\inj W_n X$, $W_{n-1} X\inj W_n X$. Applying $R\sHom_{W_n\sO_X}(-, K_{X,n})$ to it and using
duality for the finite morphisms $iF_X^{n-1}$ and $i_n$ (cf. \ref{recall-Witt-residual}, (6)) we obtain a triangle in $D(W_n\sO_X)$
\eq{Witt-residual-triangle}{i_{n*}K_{X,n-1}\to K_{X,n}\to (iF^{n-1}_X)_*K_{X,1}\to i_{n*}K_{X,n-1}[1].}
Therefore the statement follows by induction.

(3). Since $K_{X,n}$ is a residual complex, it has coherent cohomology by definition. To prove the $S_2$ property of $W_n\omega_X$ it suffices to show
\eq{S2-property}{\begin{cases}
   \Ext^0_{W_n\sO_{X,x}}(k(x), W_n\omega_{X,x})=0, & \text{for all } x\in X^{(1)}\\
   \Ext^i_{W_n\sO_{X,x}}(k(x), W_n\omega_{X,x})=0, & \text{for } i=0,1 \text{ and all } x\in X^{(2)},
  \end{cases}
} 
where $X^{(c)}$ denotes the points of codimension $c$. By the formula for the codimension function associated to $K_{X,n}$ above and  \cite[V, \S 7]{Ha}, we have 
\eq{ext-for-residual}{\Ext^i_{W_n\sO_{X,x}}(k(x), K_{X,n})=\begin{cases}
                                           0, & \text{if } x\not\in X^{(i+d)},\\
                                           k(x),& \text{if } x\in X^{(i+d)}. 
                                       \end{cases}
}
Thus the vanishing \eqref{S2-property} can easily be deduced from the spectral sequence 
\[E_2^{i,j}=\Ext^i(k(x), H^j(K_{X,n}))\Rightarrow \Ext^{i+j}(k(x), K_{X,n})\]
and the vanishing $E^{i,j}_2=0$ if $i<0$ or $j\not\in [-d,0]$ (by (2)).

(4). The first morphism is bijective by (3) and \cite[Exp. III, Cor. 3.5]{SGA2}.

(5). The first isomorphism is obtained by considering the spectral sequence 
 \[\Ext^i(W_n\omega_X, H^j(K_{X,n}))\Rightarrow \Ext^{i+j}(W_n\omega_X, K_{X,n})= H^{i+j}(D_{X,n}(W_n\omega_X))\]
   and using that $H^j(K_{X,n})\neq 0$ only for $j\in [-d,0]$, by (2).
 The second isomorphism can easily be deduced from (4), the isomorphism $W_n\sO_X\xr{\simeq}j_*W_n\sO_U$ ($W_n\sO_X$ is $S_2$) and the corresponding statement for smooth schemes, see \ref{1.4.8}.

(6) The equality $f_*H^{-d}(K_X)=H^{-d}(f_* K_X)$ follows from the spectral sequence $R^if_*H^j(K_X)\Rightarrow R^{i+j}f_*K_X$, (2) above and from $Rf_*K_X= f_*K_X$.
    The other statements follow from \ref{recall-Witt-residual}, (4).

(7) Recall that $F$ and $\ul{p}$ on $W_\bullet\omega_X$ are defined by the isomorphism $K_X\cong \sHom(W_\bullet\sO_X, K_X)$ and the formulas in \ref{dualizing-functor}.
     Thus the exact sequence in (7) is the result of applying $H^{-d}$ to the triangle \eqref{Witt-residual-triangle} above and using the isomorphisms $\ul{p}\circ \epsilon_i$
     and $V\circ \epsilon_{\sigma i}$ from \ref{1.4.3}.
\end{proof}

\begin{definition}\label{pb-pf-for-Witt-canonical}
Let $f:X\to Y$ be a finite and surjective $k$-morphism between integral normal schemes both of dimension $d$. 
\begin{enumerate}
 \item We define a pullback morphism in $\drw_Y$ 
      \[f^*: W_\bullet\omega_Y\to f_*W_\bullet\omega_X\]
       as follows: Choose open and smooth subschemes $j_X: U\inj X$ and $j_Y: V\inj Y$, which contain all $1$-codimensional points of $X$ and $Y$ respectively and such that $f$ restricts to a morphism
        $f': U\to V$. We define $f^*$ as the composition
       \mlnl{W_\bullet\omega_Y\xr{\simeq,\, \ref{properties-Witt-canonical}, (5)} j_{Y*}W_\bullet\Omega^d_V \xr{{f'}^*} j_{Y*}f'_*W_\bullet\Omega^d_U\xr{\simeq,\, \ref{Witt-canonical-old-new}} f_*j_{X*}W_\bullet\omega_U\\
                                                         \xr{\simeq,\, \ref{properties-Witt-canonical}, (5)} f_*W_\bullet\omega_X.}
       It is straightforward to check, that this morphism is independent of the choice of $U$. We also write $f^*$ for the sum of the natural pullback on $W_\bullet\sO$ with the just defined pullback, i.e.
        \[f^*: W_\bullet\sO_Y\oplus W_\bullet\omega_Y\to f_*(W_\bullet\sO_X\oplus W_\bullet\omega_X).\]
       Taking the limit we obtain a pullback in $\widehat{\drw}_Y$
       \[f^*: W\sO_Y\oplus W\omega_Y\to f_*(W\sO_X\oplus W\omega_X).\]
\item We define a pushforward in $\drw_Y$
        \[f_*: f_*W_\bullet\sO_X\to W_\bullet\sO_Y\]
        as the composition
\begin{align*}
f_*W_\bullet\sO_X & \xr{\simeq,\, \ref{properties-Witt-canonical}, (5)} f_*\sHom(W_\bullet\omega_X, W_\bullet\omega_X)\\
                  & \xr{\simeq,\, \ref{properties-Witt-canonical}, (5)} H^{-d}(f_*D_X(W_\bullet\omega_X)(-d))\\
                  & \xr{\simeq,\,\vartheta_f} H^{-d}(D_Y(f_*W_\bullet\omega_X)(-d))\\
                  & \xr{D_Y(f^*), \, (1)} H^{-d}(D_Y(W_\bullet\omega_Y)(-d))\\
                  & \xr{\simeq, \ref{properties-Witt-canonical}, (5)} \sHom(W_\bullet\omega_Y, W_\bullet\omega_Y)\\
                  & \xr{\simeq,\, \ref{properties-Witt-canonical}, (5)} W_\bullet\sO_Y.
\end{align*}
We also write $f_*$ for the sum of the pushforward on $f_*W_\bullet\omega_X$ defined in Proposition \ref{properties-Witt-canonical}, (6) with the just defined pushforward.
Taking the limit we obtain a pushforward in $\widehat{\drw_Y}$
\[f_*: f_*(W\sO_X\oplus W\omega_X)\to W\sO_Y\oplus W\omega_Y.\]
\end{enumerate}
\end{definition}

\begin{lemma}\label{properties-can-pb-pf}
 Let $f:X\to Y$ be a finite and surjective $k$-morphism between normal integral schemes of dimension $d$. 
\begin{enumerate}
 \item Let $j_Y: V\inj Y$ be an open smooth subscheme, which contains all 1-codimensional points of $Y$ and such that $U:=f^{-1}(V)\subset X$ is smooth and contains all 1-codimensional points of $X$
        (e.g. $V=Y\setminus f(X_{\text{sing}})$). Denote by $j_X: U\inj X$ and $f': U=f^{-1}(V)\to V$ the pullbacks.
    Then
     \[j_X^*\circ f^*= {f'}^*\circ j_Y^*,\quad j_Y^*\circ f_*=  f'_*\circ j_X^*,\]
      where ${f'}^*$ is the pullback defined in Definition \ref{2.1.2} and $f'_*$ is the pushforward defined in Definition \ref{2.1.6} (with $S=V$).
\item The composition $f_*\circ f^*$ equals the multiplication with the degree of $f$.
\end{enumerate}
\end{lemma}

\begin{proof}
The first part of (1) follows from the fact that all the maps in the definitions of $f_*$ and $f^*$ are compatible with localization; the second part follows immediately from the definitions.
For (2) we have to check that $f_*\circ f^*-\deg f=0 $ in $\sHom(W_\bullet\sO_Y, W_\bullet\sO_Y)=W_\bullet\sO_Y$ and in $\sHom(W_\bullet\omega_Y, W_\bullet\omega_Y)= W_\bullet\sO_Y$.
It suffices to check this on some open $V\subset Y$ and hence the assertion follows from (1) and Proposition \ref{2.1.8} (Gros). 
\end{proof}

\subsection{Topological finite quotients}\label{section-top-finite-quot}

\subsubsection{Universal homeomorphisms}\label{univ-homeo} Recall that a morphism of $k$-schemes $u:X\to Y$ is a {\em universal homeomorphism} if for any $Y'\to Y$ the base change morphism
$u': X\times_Y Y'\to Y'$ is a homeomorphism. By \cite[Cor. (18.12.11)]{EGAIV4} this is equivalent to say that $u$ is finite, surjective and radical. 
In case $X$ and $Y$ are integral and $Y$ is normal it follows from \cite[Prop. (3.5.8)]{EGAI}, that $u$ is a universal homeomorphism if and only if 
$u$ is finite, surjective and purely inseparable (i.e. $k(Y)\subset k(X)$ is purely inseparable).

\subsubsection{Relative Frobenius}\label{rel-Frob}
We denote by $\sigma: \Spec k \to \Spec k$ the Frobenius (we will not use the notation from section \ref{drws} any longer); for a $k$-scheme $X$
we denote by $X^{(n)}$ the pullback of $X$ along $\sigma^n: \Spec k \to \Spec k$ and by $\sigma^n_X : X^{(n)}\to X$ the projection.
Notice that since $k$ is perfect, $\sigma_X^n$ is an isomorphism of $\mathbb{F}_p$-schemes. 
The $n$-th relative Frobenius of $X$ over $k$ is by definition the unique $k$-morphism $F_{X/k}^n: X\to X^{(n)}$ which satisfies 
$\sigma_X^n\circ F^n_{X/k}= F^n_X$, where $F_X: X\to X$ is the absolute Frobenius morphism of $X$. Clearly $F^n_{X/k}$ is a universal homeomorphism.

\begin{lemma}\label{univ-homeo-and-Frobenius}
Let $u: X\to Y$ be a morphism between integral $k$-schemes, which is a universal homeomorphism and assume that $Y$ is normal.
Then $\deg u= p^n$ for some natural number $n$ and there exits a universal homeomorphism $v: Y\to X^{(n)}$ such that the following diagram commutes
\[\xymatrix{X\ar[r]^-{F^n_{X/k}}\ar[d]_u &    X^{(n)}\ar[d]^{u\times\id_{\Spec k}},\\
            Y\ar[ur]^v\ar[r]_-{F^n_{Y/k}} &   Y^{(n)}.
            }\]

\end{lemma}

\begin{proof}
 We may assume $X=\Spec B$ and $Y=\Spec A$. Then by \ref{univ-homeo} $u^*: A\inj B$ is an inclusion of $k$-algebras which makes $k(A)\subset k(B)$ a purely inseparable field extension of degree $p^n$.
$F^{n*}_{X/k}$ is given by $B\otimes_{k,\sigma^n} k\to B$, $b\otimes\lambda\mapsto b^{p^n}\lambda$. But $b^{p^n}\in k(A)\cap B=A$ by the normality of $Y$. Therefore
$F^{n*}_{X/k}$ factors via $u^*: A\inj B$. We obtain a homomorphism of $k$-algebras $B\otimes_{k,\sigma^n} k\to A$, which gives rise to a $k$-morphism $v: Y\to X^{(n)}$. It follows that $v$ is a universal homeomorphism,
which makes the diagram in the statement commutative.
\end{proof}

\begin{lemma}\label{pb-pf-univ-homeo}
Let $u: X\to Y$ be a universal homeomorphism between two integral and normal schemes. 
Let $u_*$ and $u^*$ be the pushforward and the pullback from Definition \ref{pb-pf-for-Witt-canonical}.
Then
\[u_*u^*=\deg u\cdot \id_{(W\sO_Y\oplus W\omega_Y)},\quad u^*u_*=\deg u\cdot\id_{u_*(W\sO_X\oplus W\omega_X)}.\] 
\end{lemma}
\begin{proof}
The equality on the left is a particular case of Lemma \ref{properties-can-pb-pf}, (2). To prove the equality on the right we may assume that $X$ and $Y$ are smooth (by Proposition \ref{properties-Witt-canonical} and the
corresponding statement for $W\sO$). Then by Lemma \ref{properties-can-pb-pf}, (1) and Proposition \ref{proposition-comparison-with-pull-push-over-S}, we have
 $u_*=\hat{\sH}([\Gamma]/Y)$ and $u^*=\hat{\sH}([\Gamma^t]/Y)$, where $\Gamma\subset X\times Y$ is the graph of $u$ and $\Gamma^t$ its transpose.
Therefore we are reduced to show 
\eq{pb-pf-univ-homeo1}{[\Gamma^t]\circ[\Gamma]=\deg u\cdot [\Delta_X] \quad \text{in } \CH(X\times X, P(X\times_Y X)),}
where $\Delta_X\subset X\times X$ is the diagonal. But since $u$ is flat (being a finite and surjective morphism between integral and smooth schemes) and a homeomorphism, we have
\[[\Gamma^t]\circ[\Gamma]=[X\times_Y X]=\dim_{k(X)}(k(X)\otimes_{k(Y)} k(X))\cdot [\Delta_X]=\deg u \cdot [\Delta_X].\]
\end{proof}

\begin{definition}\label{finite-quotients}
Let $X$ be a normal and equidimensional scheme. We say that $X$ is a {\em finite quotient} if there exists a finite and surjective morphism from a smooth scheme $Y\to X$.
We say that $X$ is a {\em tame finite quotient} if this morphism can be chosen to have its degree (as a locally constant function on $Y$) not divisible by $p$.
We say that $X$ is a {\em topological finite quotient} if there exists a universal homeomorphism $u: X\to X'$ to a finite quotient $X'$.
\end{definition}

\begin{remark}
If $X'$ is a finite quotient, so is ${X'}^{(n)}$. It follows from Lemma \ref{univ-homeo-and-Frobenius} that if $X$ is normal and equidimensional
       and if there exists a universal homeomorphism $u: X'\to X$ with a source
a finite quotient $X'$, then $X$ is a topological finite quotient.
\end{remark}

Topological finite quotients are good to handle because of the following Proposition.

\begin{proposition}\label{top-fin-quot-vs-smooth}
Consider morphisms
\[\xymatrix{           & Z\ar[d]^u \\ 
            X\ar[r]^f  & Y,}\]
where $X$ is smooth and integral, $Y$, $Z$ are normal and integral, $f$ is finite and surjective and $u$ is a universal homeomorphism.
Let $a: Y\to S$ be a morphism to some $k$-scheme $S$.
Set 
\[\beta:=\pr_{1,3*}[X\times_Y Z\times _Y X] \text{ in } \Hom_{\mc{C}_S}^0(X,X)= \CH^{\dim X}(X\times X, P(X\times_S X)), \]
where $\pr_{1,3}: X\times Z\times X \to X\times X$ is the projection (see \ref{Corr-on-rel-HW} for the notation).
Then for all i the composition
\[f^*\circ u_*: R^i (au)_*(W\sO_Z\oplus W\omega_Z)_{\Q}\to R^i(af)_*(W\sO_X\oplus W\omega_X)_{\Q} \]
induces an isomorphism in $\widehat{\drw}_S$ 
\[R^i (au)_*(W\sO_Z\oplus W\omega_Z)_{\Q}\cong \hat{\sH}(\beta/S)(R^i(af)_*(W\sO_X\oplus W\omega_X)_{\Q}),\]
where $\hat{\sH}(\beta/S)$ is the morphism from Proposition \ref{proposition-C_S-to-dRW_S} (see Notation \ref{Q-notation} for the meaning of the subscript $\Q$).
\end{proposition}

\begin{proof}
First of all notice that both $f$ and $u$ are finite and universally equidimensional. It follows that $\beta$ defines an element in $\CH^{\dim X}(X\times X, P(X\times_Y X))$,
a fortiori in $\CH^{\dim X}(X\times X, P(X\times_S X))$. In particular, Proposition \ref{proposition-C_S-to-dRW_S} yields a morphism
\[\hat{\sH}(\beta/Y): f_*(W\sO_X\oplus W\omega_X)\to f_*(W\sO_X\oplus W\omega_X).\]
We claim 
\eq{top-fin-quot-vs-smooth1}{\hat{\sH}(\beta/Y)= f^* u_*u^*f_*,}
with $f_*, u_*,u^*, f^*$ as in Definition \ref{pb-pf-for-Witt-canonical}.
By Proposition \ref{properties-Witt-canonical}, (4) and Lemma \ref{properties-can-pb-pf}, (1) and since $\hat{\sH}(\beta/Y)$ is compatible with localization in $Y$ (just by construction),
we may assume, that $X$, $Y$ and $Z$ are smooth. Then 
\[\beta= [\Gamma_f^t]\circ[\Gamma_u]\circ[\Gamma_u^t]\circ[\Gamma_f],\]
where we denote by $\Gamma_f^t\subset Y\times X$ the transpose of the graph of $f$, etc. Therefore claim \eqref{top-fin-quot-vs-smooth1} follows from Proposition \ref{proposition-comparison-with-pull-push-over-S}.
Thus Proposition \ref{correspondences-and-change-of-basis} implies that
\mlnl{\hat{\sH}(\beta/S)(R^i(af)_*(W\sO_X\oplus W\omega_X)\\
         =\text{Image}(f^* u_*u^*f_*: R^i(af)_*(W\sO_X\oplus W\omega_X)\to R^i(af)_*(W\sO_X\oplus W\omega_X)).}
Thus the assertion follows from $(u^*f_*)(f^*u_*)=\deg f\deg u\cdot\id$ (by Lemma \ref{properties-can-pb-pf}, (a) and Lemma \ref{pb-pf-univ-homeo}). 
\end{proof}

\subsection{Quasi-resolutions and relative Hodge-Witt cohomology}

\begin{definition}\label{quasi-res}
 We say that a morphism between two integral $k$-schemes $f:X \to Y$ is a {\em quasi-resolution} if the following conditions are satisfied: 
\begin{enumerate}
\item $X$ is a topological finite quotient, 
\item $f$ is projective, surjective, and 
generically finite, 
\item the extension of the function fields $k(Y)\subset k(X)$ is purely inseparable.
\end{enumerate}
\end{definition}

Condition (2) and (3) are for example satisfied if $f$ is projective and birational.

Let $X,Y$ be integral and normal. 
In general, every projective, surjective and generically finite morphism
$f:X\xr{} Y$ can be factored through the normalization $Y'$ of $Y$ in the 
function field $k(X)$ of $X$:
$$
f:X\xr{f'} Y' \xr{u} Y. 
$$
The morphism $f'$ is birational and $u$ is finite.
If $k(Y)\subset k(X)$ is purely inseparable then $u$ is a  universal homeomorphism.

\begin{remark}\label{de-Jong-did-it}
Let $X$ be an integral $k$-scheme. It follows from the proof of \cite[Cor. 5.15]{deJong2} (cf. also \cite[Cor. 7.4]{deJong1}), that 
there exists a finite and surjective morphism from a normal integral scheme $u:X'\to X$, such that $k(X)\subset k(X')$ is purely inseparable and a smooth, integral and quasi-projective scheme $X''$, 
with a finite group $G$ acting on it  such that there is a projective and birational morphism $f: X''/G\to X'$. In particular $X$ has a quasi-resolution
\[X''/G\xr{f} X'\xr{u} X.\]
\end{remark}

\begin{thm}\label{vanishing-HDI}
Let $Y$ be a topological finite quotient and $f: X\to Y$ a quasi-resolution.
Then the pullback $f^*$ and the pushforward (see Proposition \ref{properties-Witt-canonical}, (6)) induce isomorphisms in $D^b(\widehat{\drw}_{Y,\Q})$
\[f^*: W\sO_{Y,\Q}\xr{\simeq} Rf_*W\sO_{X,\Q},\quad  Rf_*W\omega_{X,\Q}\cong f_*W\omega_{X,\Q}[0]\xr{f_*,\, \simeq} W\omega_{Y,\Q}. \]
\end{thm}
\begin{proof}
We can assume that $X$ and $Y$ are integral schemes of dimension $d$ and (by Lemma \ref{pb-pf-univ-homeo}) also that they are finite quotients. Thus there exist smooth integral schemes $X'$ and $Y'$ together
with finite and surjective morphisms $a: X'\to X$ and $b: Y'\to Y$. Let
   \[X\xr{\pi}X_1\xr{u} Y,\]
 be a factorization of $f$, with $\pi$ projective and birational, $X_1$ normal and $u$ a universal homeomorphism. We can find a non-empty smooth open subscheme $U_0\subset Y$, such that 
  $U_1:=u^{-1}(U_0)$, $U_2:= \pi^{-1}(U_1)$, $U_0':=b^{-1}(U_0)$ and $U_2':= a^{-1}(U_2)$ are smooth and $\pi|U_2$ is an isomorphism. 
Notice that $a|U_2'$, $b|U_0'$ and $u|U_1$ are then automatically flat.
Set $Z_0':= Y'\setminus U_0'$ and $Z_2':=X'\setminus U_2'$. 
We define
\[\alpha:=[X'\times_X X']\in \CH^d(X'\times X', P(X'\times_X X')),\]
\[\beta:= [Y'\times_Y Y']\in \CH^d(Y'\times Y', P(Y'\times_Y Y')),\]
\[\gamma:= [X'\times_Y Y']\in \CH^d(X'\times Y', P(X'\times_Y Y')).\]
(See \ref{Corr-on-rel-HW} for the notation.) These cycles are well-defined since $a: X'\to X$ and $b: Y'\to Y$ are universally equidimensional.
We claim
\eq{vanishing-HDI1}{\deg b\cdot(\gamma\circ \alpha)-\deg a\cdot(\beta\circ \gamma)\in \text{image}(\CH(Z'_2\times_Y Z'_0)),}
\eq{vanishing-HDI2}{\deg a\cdot(\gamma^t\circ \beta)-\deg b\cdot(\alpha\circ \gamma^t)\in \text{image}(\CH(Z'_0\times_Y Z'_2)),}
\eq{vanishing-HDI3}{(\gamma^t\circ \gamma\circ \alpha)-(\deg a\deg b\deg u)\cdot \alpha\in \text{image}(\CH(Z'_2\times_Y Z'_2)),}
\eq{vanishing-HDI4}{(\gamma\circ \gamma^t\circ \beta)-(\deg a\deg b\deg u)\cdot \beta\in \text{image}(\CH(Z'_0\times_Y Z'_0)).}
Observe that $(U'_2\times_Y Y')\cup (X'\times_Y U'_0)= U'_2\times_{U_0} U'_0$ etc. Thus using the localization sequence we see that it suffices to prove
\begin{align*}
\deg b\cdot(\gamma_{|U_2'\times U_0'}\circ \alpha_{|U_2'\times U_2'})& =\deg a\cdot(\beta_{|U_0'\times U_0'}\circ \gamma_{|U_2'\times U'_0})  \in \CH(U'_2\times_{U_0} U'_0),\\
\deg a\cdot(\gamma^t_{|U_0'\times U_2'}\circ \beta_{|U_0'\times U_0'})& =\deg b\cdot(\alpha_{|U_2'\times U_2'}\circ \gamma^t_{|U_0'\times U_2'})  \in \CH(U_0'\times_{U_0} U_2'),\\
(\gamma^t_{|U_0'\times U_2'}\circ \gamma_{|U_2'\times U_0'}\circ \alpha_{|U_2'\times U_2'})& =(\deg a\deg b\deg u)\cdot \alpha_{|U_2'\times U_2'}  \in \CH(U_2'\times_{U_0} U_2'),\\
(\gamma_{|U_2'\times U_0'}\circ \gamma^t_{|U_0'\times U_2'}\circ \beta_{|U_0'\times U_0'})& =(\deg a\deg b\deg u)\cdot \beta_{|U_0'\times U_0'} \in \CH(U_0'\times_{U_0} U_0').
\end{align*}
Obviously
\begin{align*}
\alpha_{|U_2'\times U_2'} & = [\Gamma_{a|U_2'}^t]\circ [\Gamma_{a|U_2'}],\\
\beta_{|U_0'\times U_0'} & = [\Gamma_{b|U_0'}^t]\circ [\Gamma_{b|U_0'}],\\
\gamma_{|U_2'\times U_0'}& = [\Gamma_{b|U_0'}^t]\circ [\Gamma_{u|U_1}]\circ [\Gamma_{\pi|{U_2}}]\circ [\Gamma_{a|U_2'}].
\end{align*}
Thus the claim follows from
\[[\Gamma_{u|U_1}]\circ [\Gamma_{u|U_1}^t]=\deg u\cdot[\Delta_{U_0}],\quad [\Gamma_{u|U_1}^t]\circ [\Gamma_{u|U_1}]=\deg u\cdot[\Delta_{U_1}];\]
see \eqref{pb-pf-univ-homeo1} for the equality on the right.

In view of the vanishing Lemmas \ref{lemma-van1} and \ref{lemma-van2}, we see that  \eqref{vanishing-HDI1} implies, that 
$\hat{\sH}(\gamma/Y)$ induces a morphism
\[\hat{\sH}(\alpha/Y)\left(\bigoplus_i R^i(fa)_*(W\sO_{X'}\oplus W\omega_{X'})_\Q\right)\to \hat{\sH}(\beta/Y)\left( b_*(W\sO_{Y'}\oplus W\omega_{Y'})_\Q\right) \]
and  \eqref{vanishing-HDI2} implies that $\hat{\sH}(\gamma^t/Y)$ induces a morphism
\[\hat{\sH}(\beta/Y)\left(b_*(W\sO_{Y'}\oplus W\omega_{Y'})_\Q\right)\to \hat{\sH}(\alpha/Y)\left(\bigoplus_i R^i(fa)_*(W\sO_{X'}\oplus W\omega_{X'})_\Q\right).\]
By \eqref{vanishing-HDI3} and \eqref{vanishing-HDI4} these morphisms are inverse to each other, up to multiplication with $(\deg a\deg b\deg u)$.
By Proposition \ref{top-fin-quot-vs-smooth} $a^*$ induces for all $i$ an isomorphism
$R^if_* (W\sO_X\oplus W\omega_X)_\Q\to \hat{\sH}(\alpha/Y)(R^i(fa)_*(W\sO_{X'}\oplus W\omega_{X'})_\Q)$ and
$b^*$ induces an isomorphism $(W\sO_Y\oplus W\omega_Y)_\Q\to \hat{\sH}(\beta/Y)( b_*(W\sO_{Y'}\oplus W\omega_{Y'})_\Q)$. 
This gives $R^if_*W\sO_{X,\Q}=0=R^if_*W\omega_{X,\Q}$, for all $i\ge 1$. It also gives isomorphisms in cohomological degree 0, but it is not immediately clear that these coincide with 
pullback and pushforward. But since $X_1$ is normal and $\pi: X\to X_1$ is birational, the pullback $\pi^*: W\sO_{X_1}\to \pi_*W\sO_{X}$ clearly is an isomorphism and
hence so is $f^*: W\sO_{Y,\Q}\to f_*W\sO_{X,\Q}$, by Lemma \ref{pb-pf-univ-homeo}. Now the statement follows from the next lemma.
\end{proof}

\begin{lemma}\label{pf-of-Womega}
Let $f:X\to Y$ be a proper morphism between $k$-schemes of the same pure dimension $d$ and assume
that the pullback morphism $f^*: W\sO_{Y,\Q}\xr{\simeq} Rf_* W\sO_{X,\Q}$ is an isomorphism in $D(\widehat{\drw}_{Y,\Q})$. Then
the pushforward (see Proposition \ref{properties-Witt-canonical}, (6))
\[f_*: f_*W\omega_{X,\Q}\xr{\simeq} W\omega_{Y,\Q}\]
is an isomorphism. More precisely, there exists a natural number $N\ge 1$ such that kernel and cokernel of $f_*W_n\omega_X\to W_n\omega_Y$ are $N$-torsion for all $n\ge 1$.
\end{lemma}
\begin{proof}
By assumption we find a natural number $N\ge 1$ such that  $R^if_*W\sO_X$,  $i\ge 1$, as well as kernel and cokernel of $f^*: W\sO_Y\to f_*W\sO_X$ are all $N$-torsion.
It follows from the short exact sequence 
\[0\to W\sO\xr{V^n}W\sO\to W_n\sO\to 0\]
that there exists a natural number $M\ge 1$ (e.g. $M=N^2$) such that $R^if_*W_\bullet\sO_X$, $i\ge 1$, as well as kernel and cokernel of $f^*: W_\bullet\sO_Y\to f_*W_\bullet\sO_X$
are $M$-torison. Let $C_n$ be the cone of $f^*: W_n\sO_Y\to Rf_*W_n\sO_X$ in $D^b_{\rm qc}(W_n\sO_Y)$. (Here we write $f_*$ instead of $W_n(f)_*$ etc.)
Then the above can be rephrased by saying, that $H^i(C_n)$ is $M$-torsion for all $i\in \Z$ and all $n\ge 1$.
Now applying the dualizing functor $D_{W_nY}$ to the triangle in $D^b_{\rm qc}(W_n\sO_Y)$
\[W_n\sO_Y\to Rf_*W_n\sO_X\to C_n\to W_n\sO_Y[1]\]
and using the duality isomorphism $D_{W_nY} Rf_*\cong Rf_*D_{W_nX}$
yields a triangle in $D^b_{\rm qc}(W_n\sO_Y)$
\[D_{W_nY}(C_n)\to Rf_*K_{X,n}\to K_{Y,n}\to D_{W_nY}(C_n[-1]). \]
Taking $H^{-d}$, we obtain an exact sequence
\[\sExt^{-d}(C_n, K_{Y,n})\to R^{-d}f_*K_{X,n}\to W_n\omega_Y\to \sExt^{-d}(C_n[-1], K_{Y,n}).\]
As in Proposition \ref{properties-Witt-canonical}, (6) the morphism in the middle is just the pushforward $f_{*}: f_*W_n\omega_X\to W_n\omega_Y$.
Now the filtration of the spectral sequence 
\[E_2^{i,j}=\sExt^i(H^{-j}(C_n), K_{Y,n})\Rightarrow \sExt^{i+j}(C_n, K_{Y,n})\]
is finite and the $E_2$-terms are $M$-torsion. Thus the groups $\sExt^{-d}(C_n[-1], K_{Y,n})$ and $\sExt^{-d}(C_n, K_{Y,n})$ are $M^r$-torsion, for some $r>>0$.
In fact $r$ only depends on the length of the filtration of the above spectral sequence and since this length is bounded for all $n\ge 1$, we may choose $r$ to work for all $n$.
It follows that kernel and cokernel of $f_*: f_*W_n\omega_X\to W_n\omega_Y$ are $M^r$-torsion; hence kernel and cokernel of the limit
$f_*: f_*W\omega_X\to W\omega_Y$ as well.  This yields the statement.
\end{proof}

\subsection{Rational and Witt-rational singularities}

A special class of singularities which appear naturally in higher dimensional geometry
are the \emph{rational} singularities. Essentially, rational singularities do not affect
the cohomological properties of the structure sheaf.

\begin{definition}[{\cite[p. 50]{kempf}}]\label{definition-rational-singularities}
Let $S$ be a normal variety 
and $f:X\xr{} S$ a resolution of singularities (i.e. $f$ is projective and birational and $X$ is smooth). We say that 
$f$ is a \emph{rational resolution} if 
\begin{enumerate}
\item $R^if_*\OO_X=0$, for $i>0$,
\item $R^if_*\omega_X=0$ for $i>0$ (this always holds by Grauert-Riemenschneider  if the characteristic of the ground field is zero, but is needed in positive characteristic).
\end{enumerate}

We say that $Y$ has \emph{rational singularities} if a rational resolution exists. 
\end{definition}   

An immediate problem with the definition of rational singularities in  
positive characteristic is that the existence of a resolution of singularities
is assumed. For example, tame quotient singularities are rational singularities
provided that a resolution of singularities exists \cite[Theorem~2]{CR}.

If an integral normal scheme $Y$ over a field has one rational resolution, then all resolutions are rational,
i.e.~rational singularities are an intrinsic property of $Y$.
(In characteristic zero, this was proved by Hironaka, see \cite[Theorem 1]{CR} for the characteristic $p$ case.)

In characteristic zero, Kov\'acs \cite{K} observed that one can replace condition (1)
in Definition \ref{definition-rational-singularities} by the following condition:
there is an alteration $f:X\xr{} S$ such that the natural morphism 
$$
\OO_S\xr{} Rf_*\OO_X
$$ 
admits a splitting in the derived category of coherent sheaf on $S$. The main
tool in the proof is Grauert-Riemenschneider vanishing, and this characterization
does not hold in positive characteristic.

In order to study congruence formulas for the number of points of a 
variety over a finite field, Blickle and Esnault \cite{BE} introduced the notion
of Witt-rational singularities.

\begin{definition}[{\cite[Def. 2.3]{BE}}]\label{definition-Witt-rational-BE}
Let $S$ be an integral $k$-scheme and $f:X\xr{} S$ a generically \'etale alteration with $X$ a smooth $k$-scheme.
We say that $S$ has \emph{BE-Witt-rational singularities} if the 
natural morphism 
$$
W\OO_S\otimes_\Z \Q \xr{} Rf_*W\OO_X\otimes_\Z\Q
$$
admits a splitting in the derived category of sheaves of abelian groups on $X$. 
\end{definition}

We call the singularities defined in \cite{BE} BE-Witt-rational singularities,
rather than Witt-rational singularities, because we will redefine 
Witt-rational singularities in \ref{definition-Witt-rational}.
We remark:

\begin{proposition}\label{one-then-all}
 The notion of  BE-Witt-rational singularities is independent of the chosen generically \'etale alteration.
More precisely, if an integral scheme $S$ has BE-Witt-rational singularities, then for any alteration
(not necessarily generically \'etale) $g: Y\to S$, with $Y$ smooth, the pullback $g^*: W\sO_S\otimes_\Z \Q\to Rg_*W\sO_Y\otimes_\Z \Q$
admits a splitting in the derived category of sheaves of abelian groups on $X$. 
\end{proposition}

\begin{proof}
Obviously it suffices to prove that if $f: X\to Y$ is an alteration between smooth schemes, then the composition 
\[ W\sO_Y\xr{f^*} Rf_* W\sO_X\xr{f_*} W\sO_Y\]
is multiplication with $[k(X): k(Y)]$, where $f_*$ is the pushforward from Definition \ref{2.1.6}.
It suffices to check this on some non-empty open subscheme $U$ of $Y$ such that $f|f^{-1}(U)$ is finite. 
Thus the statement follows from Proposition \ref{2.1.8} (Gros). 
\end{proof}

\begin{definition}\label{definition-Witt-rational}
We say that an integral $k$-scheme $S$ has {\em Witt-rational singularities} if for any quasi-resolution (see Definition \ref{quasi-res}) $f: X\to S$ the following conditions are satisfied:
\begin{enumerate}
 \item  $f^*: W\sO_{S,\Q}\xr{\simeq} f_*W\sO_{X,\Q}$ is an isomorphism.
 \item $R^if_* W\sO_{X,\Q}=0$, for all $i\ge 1$.
 \item $R^if_* W\omega_{X,\Q}=0$, for all $i\ge 1$.
\end{enumerate}
In case only the first two properties are satisfied we say $S$ has {\em $W\sO$-rational singularities}. (See Notation \ref{Q-notation} for the meaning of the subscript $\Q$.)
\end{definition}

\begin{remark}\label{Witt-rational-and-normal}
Notice that if $S$ is normal, then condition (1) above is automatically satisfied.
Indeed, each quasi-resolution $f: X\to S$ can be factored as $X\xr{\pi} X_1\xr{u} S$ with $X_1$ normal, $\pi$ projective and birational and $u$ an universal homeomorphism; thus
condition (1) is satisfied by Lemma \ref{pb-pf-univ-homeo}.
\end{remark}

\begin{proposition}\label{properties-Witt-rational-sing}
 Let $S$ be an integral $k$-scheme. Then the following statements are equivalent:
        \begin{enumerate}
          \item $S$ has Witt-rational singularities.
          \item  There exists a quasi-resolution $f:X\to S$ satisfying (1), (2), (3) of Definition \ref{definition-Witt-rational}.
          \item There exists a quasi-resolution $f:X\to S$, such that there are isomorphisms in $D^b(\widehat{\drw}_{S,\Q})$
                   \eq{pWrs1}{f^*: W\sO_{S,\Q}\xr{\simeq} Rf_*W\sO_{X,\Q}, \quad Rf_*W\omega_{X,\Q}\cong f_*W\omega_{X,\Q}[0]\xr{f_*,\,\simeq} W\omega_{S,\Q}.}
          \item For all quasi-resolutions $f:X\to S$ the morphisms \eqref{pWrs1} are isomorphisms.
         \end{enumerate}
\end{proposition}

\begin{proof}
Clearly (1) $\Rightarrow$ (2) and (4) $\Rightarrow$ (1). (2) $\Rightarrow$ (3) follows from Lemma \ref{pf-of-Womega}. For (3) $\Rightarrow$ (4) notice 
that by de Jong (see Remark \ref{de-Jong-did-it}) any two quasi-resolutions of $S$ can be dominated by a third one. Thus the statement follows from Theorem \ref{vanishing-HDI}.
\end{proof}
There is an obvious analog of this proposition for $W\sO$-rational singularities.

\begin{corollary}\label{finite-quotients-are Witt-rational}
Topological finite quotients over $k$ have Witt-rational singularities.
\end{corollary}

\begin{remark}\label{finite-quot-not-CM}
In characteristic $0$ finite quotients always have rational singularities (see e.g. \cite[Prop. 5.13]{KM}). In characteristic $p>0$ this is not the case.
Indeed let $G=\Z/p^n\Z$ act linearly on a finite dimensional $k$-vector space, where $k$ is assumed to be algebraically closed.
Then it is shown in \cite{Ellingsrud}, that  $\A(V)/G$ is not CM, provided that $\dim_k V> \dim_k V^G +2$. In particular $\A(V)/G$ cannot have
rational singularities in the sense of Definition \ref{definition-rational-singularities}, which are always CM.
This also shows that Witt-rational singularities do no need to be CM.
\end{remark}

\begin{proposition}\label{proposition-Witt-rational-universal-homeomorphism}
Let $u:Y\xr{} Y'$ be a universal homeomorphism between normal schemes.
Then $Y$ has Witt-rational singularities if and only if $Y'$ has Witt-rational 
singularities.
\begin{proof}
If $f:X\xr{} Y$ is a quasi-resolution then $u\circ f$ is a quasi-resolution. 
For all $i>0$ we get 
$$
R^i(u\circ f)_*W\OO_{X,\Q}=u_*R^if_*W\OO_{X,\Q}, \quad R^i(u\circ f)_*W\omega_{X,\Q}=u_*R^if_*W\omega_{X,\Q},
$$
and thus
\begin{align*}
R^i(u\circ f)_*W\OO_{X,\Q}=0 &\Leftrightarrow R^if_*W\OO_{X,\Q}=0,\\
R^i(u\circ f)_*W\omega_{X,\Q}=0 &\Leftrightarrow R^if_*W\omega_{X,\Q}=0.
\end{align*}
\end{proof}
\end{proposition}

\begin{definition}\label{quasi-bir}
 Let $S$ be a $k$-scheme and $X$ and $Y$ two integral $S$-schemes. We say that $X$ and $Y$ are {\em quasi-birational over $S$} if there exists a commutative diagram
\[\xymatrix@-1pc{            &  Z\ar[dl]_{\pi_X}\ar[dr]^{\pi_Y}  &  \\
               X\ar[dr] &                                   &  Y\ar[dl]  \\
                        &   S,                               &  }\]
with  $\pi_X$ and $\pi_Y$ quasi-resolutions (see Definition \ref{quasi-res}). We say that the triple $(Z, \pi_X,\pi_Y)$  (or just $Z$ if we do not need to specify $\pi_X$ and $\pi_Y$)
is a {\em quasi-birational correspondence between $X$ and $Y$}. 
\end{definition}
Since quasi-resolutions always exist (see Remark \ref{de-Jong-did-it}), two integral {\em projective} $S$-schemes $X$ and $Y$ are quasi-birational over $S$ if and only if
the generic points of $X$ and $Y$ map to the same point $\eta$ in $S$ and there exists a field  $L$ with a homomorphism $\sO_{S,\eta}\to L$ and inclusions
of $\sO_{S,\eta}$-algebras $k(X)\inj L$, $k(Y)\inj L$, which make $L$ a finite and purely inseparable field extension of $k(X)$ and $k(Y)$.
In particular this is the case if $k(X)$ and $k(Y)$ are isomorphic as $\sO_{S,\eta}$-algebras.

\begin{corollary}\label{independence}
Let $S$ be a $k$-scheme and $f: X\to S$ and $g: Y\to S$ two $S$-schemes, which are integral, have Witt-rational singularities and are quasi-birational over $S$ to each other.
Then the choice of a quasi-birational correspondence between $X$ and $Y$ induces isomorphisms in $D^b(\widehat{\drw}_S)$
\eq{independence1}{ Rf_*W\sO_{X,\Q}\cong Rg_*W\sO_{Y,\Q}, \quad Rf_*W\omega_{X,\Q}\cong Rg_*W\omega_{Y,\Q}.}
Moreover, two quasi-birational correspondences $Z$ and $Z'$ induce the same isomorphisms if there exists a field $L$ with a homomorphism $\sO_{S,\eta}\to L$ ($\eta\in S$ being the common image of the generic points
of $X$ and $Y$) and inclusions
of $\sO_{S,\eta}$-algebras $k(Z)\inj L$, $k(Z')\inj L$, which make $L$ a finite and purely inseparable field extension of $k(Z)$ and $k(Z')$, such that
the composite inclusions $k(X), k(Y)\inj k(Z)\inj L$ and $k(X), k(Y)\inj k(Z')\inj L$ are equal.
\end{corollary}
 
\begin{proof}
A quasi-birational correspondence $(Z,\pi_X,\pi_Y)$ between $X$ and $Y$ induces an isomorphism
\[Rf_*W\sO_{X,\Q}\xr{\pi_X^*, \,\simeq} Rf_* R\pi_{X*}W\sO_{Z,\Q}\cong Rg_*R\pi_{Y*}W\sO_{Z,\Q}\xl{\simeq,\,\pi_Y^*} Rg_*W\sO_{Y,\Q},\]
and similar for $W\omega$. For the second statement first notice that if $(Z,\pi_X, \pi_Y)$ and $(Z',\pi'_X,\pi'_Y)$ are two
quasi-birational correspondences between $X$ and $Y$ and if there is a quasi-resolution  $a: Z'\to Z $ such that $\pi'_X=\pi_X\circ a$ and $\pi'_Y=\pi_Y\circ a$, then
they induce the same isomorphisms \eqref{independence1}. If we are given two arbitrary quasi-birational correspondences $Z$ and $Z'$ between $X$ and $Y$ and a field $L$
as is the statement of the corollary, then we can take a quasi-resolution of the closure of the image of $\Spec L\to Z\times_S Z'$ to obtain a quasi-birational correspondence $Z''$ between
$X$ and $Y$ mapping via a quasi-resolution to $Z$ and $Z'$ and is compatible with $\pi_X$, $\pi'_X$, $\pi_Y$, $\pi'_Y$ in the obvious sense. 
This proves the statement. 
\end{proof}

\begin{corollary}\label{independence2}
 In the situation of Corollary \ref{independence} assume that $S$ is integral and $f$ and $g$ are generically finite and purely inseparable.
Then any quasi-birational correspondence between $X$ and $Y$ induces the {\em same} isomorphism in $D^b(\widehat{\drw}_S)$
\[ Rf_*W\sO_{X,\Q}\cong Rg_*W\sO_{Y,\Q}, \quad Rf_*W\omega_{X,\Q}\cong Rg_*W\omega_{Y,\Q}.\]
\end{corollary}

\begin{corollary}
Let $S$ be a $k$-scheme and $f:X\to S$ be an integral and projective $S$-scheme, which has Witt-rational singularities. Let $k(X)^{\rm perf}$ be the perfect closure of $k(X)$ and 
$\eta\in S$ the image of the generic point of $X$. Then ${\rm End}_{\sO_{S,\eta}-{\rm alg}}(k(X)^{\rm perf})$ is acting on  $Rf_*W\sO_{X,\Q}$ and $Rf_*W\omega_{X,\Q}$ as objects in
$D^b(\widehat{\drw}_S)$. 
\end{corollary}

\begin{proof}
 An element $\sigma$ in ${\rm End}_{\sO_{S,\eta}-{\rm alg}}(k(X)^{\rm perf})$ will when composed with $k(X)\inj k(X)^{\rm perf}$ factor over a finite and purely inseparable extension $L$
of $k(X)$. By the remark after Definition \ref{quasi-bir} it hence gives rise to a quasi-birational correspondence of $X$ with itself and thus yields the promised
well-defined action by Corollary \ref{independence}.
\end{proof}

\begin{remark}
 The above corollaries have obvious analogs for $W\sO$-rational singularities.
\end{remark}

\begin{corollary}\label{rig-slope1}
 Let $X$ and $Y$ be two integral $k$-schemes, which have $W\sO$-rational singularities and are  quasi-birational over $k$.
Then in $\widehat{\drw}_{k,\Q}$
\[H^i(X,W\sO_{X,\Q})\cong H^i(Y, W\sO_{Y,\Q}), \quad \text{for all }i\ge 0.\]
In particular, if $X$ and $Y$ are projective then \cite[Thm 1.1]{BBE} yields a Frobenius equivariant isomorphism
\[H^i_{\rm rig}(X/K)^{<1}\cong H^i_{\rm rig}(Y/K)^{<1}\quad \text{for all } i\ge 0,\]
where $K={\rm Frac}(W(k))$ and $H^i_{\rm rig}(X/K)^{<1}$ denotes the part of rigid cohomology on which the Frobenius acts with slope $<1$. 
\end{corollary}

\begin{proof}
 Apply Corollary \ref{independence} in the case $S=k$.
\end{proof}

We will also give some results on the torsion, see Theorem \ref{torison1} and Theorem \ref{torsion2}.

\begin{corollary}\label{points}
 Let $k$ be a finite field. Let $X$ and $Y$ be two quasi-birational integral and projective $k$-schemes, which have $W\sO$-rational singularities.
Then for any finite field extension $k'$ of $k$ we have
\[|X(k')|\equiv |Y(k')|\quad \text{mod } |k'|.\]
\end{corollary}
\begin{proof}
This follows from Corollary \ref{rig-slope1} and \cite[Cor.~1.3]{BBE}.
\end{proof}

In the case where $X$ and $Y$ are smooth, integral and proper, the above corollary was proved in \cite[Cor. 3, (i)]{E83}.
In case there is a morphism $f: X\to Y$, which is birational and $X$ is smooth and projective and $Y=Z/G$ is the quotient under a finite group $G$ of a smooth projective scheme $Z$,
this was proved in \cite[Thm 4.5.]{C}.

We investigate the properties of Witt-rational singularities a little bit further.

\begin{proposition}\label{comparison}
 Consider the following properties on an integral $k$-scheme $S$:
\begin{enumerate}
 \item $S$ has rational singularities.
 \item $S$ has Witt-rational singularities.
\item $S$ has $W\sO$-rational singularities.
\item $S$ has BE-Witt-rational singularities.
\end{enumerate}
Then 
\[(1)\Rightarrow (2)\Rightarrow (3)\Rightarrow (4).\]
Furthermore the first implication is strict, by Remark \ref{finite-quot-not-CM} above.
\end{proposition}

\begin{proof}
 (1) $\Rightarrow$ (2): By assumption there exists a resolution $f:X\to S$. The exact sequences  
\[0\to W_{n-1}\sO_X\xr{V}W_n\sO_X\to \sO_X\to 0\]
 and (see Proposition \ref{properties-Witt-canonical}, (7)) 
 \[0\to W_{n-1}\omega_X\xr{\ul{p}} W_n\omega_X\xr{F^{n-1}} \omega_X\to 0\]
give us $R^if_*W_n\sO_X=0=R^if_*W_n\omega_X$ for all $n,i\ge 1$ and also $W_n\sO_S\cong f_*W_n\sO_X$. Since $S$ is CM, the isomorphism $W_n\sO_S\cong Rf_*W_n\sO_X$ also gives the isomorphism 
$f_*W_n\omega_X\cong W_n\omega_S$ via duality (see Proposition \ref{properties-Witt-canonical}, (2)). Now the statement follows from the exact sequence (see Lemma \ref{vanishing-of-Rlim})
\[0\to R^1\varprojlim_n R^{i-1}f_*E_n\to R^if_*(\varprojlim_n E_n)\to \varprojlim_n R^if_*E_n\to 0,\]
where $E_n\in \{W_n\sO_X, W_n\omega_X\}$. (For the vanishing of $R^1f_*W\omega_X$ use that $f_*W_n\omega_X\cong W_n\omega_S$ and hence the projection maps $f_*W_n\omega_X\to f_*W_{n-1}\omega_X$ are surjective,
by Proposition \ref{properties-Witt-canonical}, (4).)

(2) $\Rightarrow$ (3): trivial.

(3) $\Rightarrow$ (4): By Remark \ref{de-Jong-did-it} we find a quasi-resolution of $S$ of the form
                          $f:X\to S$ with  $X$ a finite quotient.
                       We thus find a smooth scheme $X'$ with a finite and surjective morphism $g: X'\to X$. Then $h:=f g: X'\to S$ is an alteration and
                       $h^*: W\sO_{S,\Q}\to Rg_*W\sO_{X',\Q}$ splits by Proposition \ref{properties-Witt-rational-sing}, (1) $\Rightarrow$ (2), and Lemma \ref{properties-can-pb-pf}, (2);
                      a fortiori $h^*: W\sO_S\otimes\Q\to Rg_*W\sO_{X'}\otimes\Q$ splits.
 \end{proof}

\subsection{Complexes and sheaves attached to singularities of schemes}

\begin{corollary}\label{independence3}
 Let $S$ be an integral scheme and $f: X\to S$ and $g: Y\to S$ be quasi-resolutions.
Then $X$ and $Y$ are quasi-birational over $X$ and any quasi-birational correspondence between $X$ and $Y$ induces the {\em same} isomorphism in  the  derived category $D^b(\widehat{\drw}_S)$
\[ Rf_*W\sO_{X,\Q}\cong Rg_*W\sO_{Y,\Q}, \quad Rf_*W\omega_{X,\Q}\cong Rg_*W\omega_{Y,\Q}.\]
Furthermore, if $Z$ is another integral scheme and $h:Z\to S$ a quasi-resolution, then the isomorphisms
\[ Rf_*W\sO_{X,\Q}\cong Rh_*W\sO_{Z,\Q}, \quad Rf_*W\omega_{X,\Q}\cong Rh_*W\omega_{Z,\Q}\]
induced by any quasi-birational correspondence between $X$ and $Z$ equals the isomorphism obtained by composing the isomorphisms
induced by quasi-birational correspondences between  $X$ and $Y$ and between  $Y$ and $Z$. 
\end{corollary}
\begin{proof}
 First we show that $X$ and $Y$ are quasi-birational over $S$. For this notice that $k(X)\otimes_{k(S)} k(Y)$ is a local Artin algebra. Denote by $L$ its residue field.
Then we can take a quasi-resolution of the closure of the image of $\Spec L$ in $X\times_S Y$ to obtain a quasi-birational correspondence between $X$ and $Y$. 
Let $V$ and $V'$ be two quasi-birational correspondences between $X$ and $Y$. Denote by $L''$ the residue field of the local Artin algebra $k(V)\otimes_{k(S)} k(V')$.
This is a purely inseparable field extension of $k(S)$. Hence there is only one embedding over $k(S)$ of $k(X)$ and $k(Y)$ into $L''$.
Thus by Corollary \ref{independence} the two isomorphisms induced by $V$ and $V'$ are equal. Finally if we have the three quasi-resolutions $f,g,h$, we find a diagram
of $S$-morphisms
\[\xymatrix{ V''\ar[r]\ar[d]& V'\ar[d]\ar[r] & Z\\
             V\ar[d]\ar[r]  & Y\        &  \\
             X, }\]
in which $V$ is a quasi-birational correspondence between $X$ and $Y$, $V'$ is a quasi-birational correspondence between $Y$ and $Z$ and $V''$ is a quasi-birational correspondence between $V$ and $V'$
(and also between $X$ and $Z$). The last statement of the Corollary follows.
\end{proof}

\begin{definition}\label{canonical-complex-for-singularities}
 Let $S$ be an integral $k$-scheme of dimension $d$. We define in $D^b(\cdrw_{S,\Q})$
\[\mc{WS}_{0,S}:=Rf_*W\sO_{X,\Q}, \quad \mc{WS}_{d, S}:= Rf_*W\omega_{X,\Q},\]
where $f:X\to S$ is any quasi-resolution. This definition is independent (up to a canonical isomorphism) of the choice of the quasi-resolution $f$  by Corollary \ref{independence3}.
\end{definition}

It follows, that $S$ has $W\sO$-rational singularities if and only if $\mc{WS}_{0,S}\cong W\sO_S$ and it has Witt-rational singularities if and only if in addition
we have $H^i(\mc{WS}_{d, S})=0$ for all $i\ge 1$ (which is equivalent to $\mc{WS}_{d, S}\cong W\omega_{S,\Q}$).

Next we want to give a characterization of Witt-rational singularities using alterations.

\begin{proposition}\label{alterations-and-projectors}
Let $S$ be an integral $k$-scheme of dimension $d$ and $f: X\to S$ an alteration with $X$ smooth.  
Set 
\[\epsilon_f:=\frac{1}{\deg f} [X\times_S X]^0\in \CH^d(X\times X, P(X\times_S X))_\Q,\]
where $[X\times_S X]^0$ is the cycle associated to the closure of $X_\eta\times_\eta X_\eta$ in $X\times_S X$  with $\eta$ the generic point of $S$.
Further set
\[\hat{\sH}^{*,0}(X/S):=\bigoplus_i R^if_*W\sO_X,\quad \hat{\sH}^{*, d}(X/S):=\bigoplus_i R^if_*W\Omega^d_X,\]
\[\hat{\sH}^{*,(0,d)}(X/S):= \hat{\sH}^{*,0}(X/S)\oplus \hat{\sH}^{*,d}(X/S).\]
Then:
\begin{enumerate}
 \item The restriction of $\hat{\sH}(\epsilon_f/S)$ to $\hat{\sH}^{*,(0,d)}(X/S)_\Q$ is a projector, which we denote by $e_f$. (See Proposition \ref{proposition-C_S-to-dRW_S} for the notation.)
 \item The pullback $f^*$ induces a natural morphism of Witt modules over $S$
      \[f^*: W\sO_{S,\Q} \to e_f \hat{\sH}^{*,(0,d)}(X/S)_\Q.\]
 \item If $g: Y\to S$ is another alteration with $Y$ smooth. Then set 
       \[\gamma_{f,g}:= \frac{1}{\deg f}[X\times_S Y]^0 \in \CH^d(X\times Y, P(X\times_S Y)),\]
        where $[X\times_S Y]^0$ is the cycle associated to the closure of $X_\eta\times_\eta Y_\eta$ in $X\times_S Y$.
        Then $\hat{\sH}(\gamma_{f,g}/S): \hat{\sH}(X/S)\to \hat{\sH}(Y/S)$  induces an isomorphism 
         \eq{alterations-and-projectors1}{\hat{\sH}(\gamma_{f,g}/S): e_f(\hat{\sH}^{*,(0,d)}(X/S))_\Q\xr{\simeq} e_g(\hat{\sH}^{*,(0,d)}(Y/S))_\Q,}
          which is compatible with the pullback morphism from (2).
\end{enumerate}
\end{proposition}

\begin{proof}
Take a non-empty open and smooth subset $U$ of $S$ such that the pullback of $f$ - and in case (3) also of $g$ - over $U$ is finite and surjective and hence also flat; 
the pullbacks are denoted by $f_U: X_U\to U$ and $g_U: X_U\to U$.
Notice that the restriction of $[X\times_S X]^0$ and $[X\times_S Y]^0$ over $U$ equal $[X_U\times_U X_U]$ and $[X_U\times_U Y_U]$.

 (1) We have to show  $e_f\circ e_f- e_f=0$ on $\hat{\sH}^{*,(0,d)}(X/S)_\Q$. By the same argument as in the proof of Theorem \ref{vanishing-HDI} 
(using the vanishing Lemmas \ref{lemma-van1} and \ref{lemma-van2}) it suffices to prove
\[[X_U\times_U X_U]\circ [X_U\times_U X_U]-\deg f \cdot [X_U\times_U X_U]=0. \]
This follows immediately from $[X_U\times_U X_U]=[\Gamma_{f_U}^t]\circ [\Gamma_{f_U}]$.

(2) The morphism $e_f$, in particular gives a projector $e_f: f_*W\sO_{X,\Q}\to f_*W\sO_{X,\Q}$. Thus we need to show, that $e_f\circ f^*=f^*$ on $W\sO_{S,\Q}$.
    It suffices to prove this over $U$; thus the statement follows from $\deg f\cdot e_{f_U}=f_{U}^*\circ f_{U*}$.

(3) To prove that $\gamma_{f,g}$ is an isomorphism on $e_f(\hat{\sH}^{*, (0,d)}(X/S)_\Q)$ with inverse $\frac{1}{\deg g}\hat{\sH}([Y\times_S X]^0/S)$ it suffices
  to show that the following cycles in $\CH^d(X\times Y, P(X\times_S Y))$, etc.,
\begin{align*}
 \deg g\cdot[X\times_S Y]^0\circ [X\times_S X]^0- \deg f\cdot [Y\times_S Y]^0\circ [X\times_S Y]^0\\
 \deg f\cdot [Y\times_S X]^0\circ [Y\times_S Y]^0- \deg g \cdot[X\times_S X]^0\circ [Y\times_S X]^0\\
 [Y\times_S X]^0\circ [X\times_S Y]^0 \circ [X\times_S X]^0-(\deg f\deg g)\cdot [X\times_S X]^0\\
[X\times_S Y]^0\circ [Y\times_S X]^0\circ [Y\times_S Y]^0- (\deg f\deg g)\cdot [Y\times_S Y]^0
\end{align*}
act as zero on $\hat{\sH}^{*,(0,d)}(X/S)_\Q$ and $\hat{\sH}^{*,(0,d)}(Y/S)_\Q$ respectively.
By the same argument as in the proof of Theorem \ref{vanishing-HDI} (using the vanishing Lemmas \ref{lemma-van1} and \ref{lemma-van2}) it suffices to prove that
the pullback along $U$ of the above cycles vanish. This follows easily from $[X_U\times_U X_U]=[\Gamma_{f_U}^t]\circ [\Gamma_{f_U}]$, 
$[X_U\times_U Y_U]= [\Gamma_{g_U}^t]\circ [\Gamma_{f_U}]$ and $[\Gamma_{f_U}]\circ [\Gamma_{f_U}^t]= \deg f\cdot [\Delta_U]$, etc.
This yields the isomorphism \eqref{alterations-and-projectors1}. It is compatible with the pullback from (2), since on $W\sO_{S,\Q}$ we have
$\hat{\sH}(\gamma_{f,g}/S)f^*= g^*$. Indeed over $U$ we have $\hat{\sH}(\gamma_{f_U,g_U}/U)=\frac{1}{\deg f}\cdot g_U^*\circ f_{U*}$; thus it holds on $U$ and hence on all of $S$.
\end{proof}

\begin{definition}\label{RS}
Let $S$ be an integral $k$-scheme of dimension $d$ and $f:X\to S$ an alteration with $X$ smooth. 
Then using the notations from Proposition \ref{alterations-and-projectors} we define 
\[\RS^0(X/S):= {\rm coker} (W\sO_{S,\Q}\xr{f^*} e_f\hat{\sH}^{*, 0}(X/S)_\Q)= \frac{f_*W\sO_{X,\Q} }{W\sO_{S,\Q}} \bigoplus_{i\ge 1} e_f R^if_*W\sO_{X,\Q} \]
and 
\[\RS^d(X/S):= e_f\hat{\sH}^{>0,d}(X/S)=\bigoplus_{i\ge 1} e_f R^if_*W\Omega^d_{X,\Q}.\]
Then by Proposition \ref{alterations-and-projectors} $\RS^0(X/S)$ and $\RS^d(X/S)$ are independent of $f:X\to S$ up to a canonical isomorphism and are therefore
denoted 
\[\RS^0(S) := \RS^0(X/S), \quad \RS^d(S) := \RS^d(X/S), \quad \RS(S):=\RS^0(S)\oplus \RS^d(S). \] 
\end{definition}

\begin{remark}\label{direct-summand}
If $S$ is normal, then $\RS(S)$ is a direct summand of $\sH^{*,(0,d)}(X/S)_\Q$ for any alteration $f: X\to S$ with $X$ smooth.
\end{remark}

\begin{thm}\label{Witt-rational-via-alterations}
Let $S$ be an integral $k$-scheme. Then there are isomorphisms in $\widehat{\drw}_{S,\Q}$
\[\RS^0(S)\cong \frac{H^0(\mc{WS}_{0,S})}{W\sO_{S,\Q}}\bigoplus_{i\ge 1} H^i(\mc{WS}_{0,S}), \quad \RS^d(S)\cong \bigoplus_{i\ge 1} H^i(\mc{WS}_{d,S}).\] 
In particular:
\begin{enumerate}
 \item $S$ has $W\sO$-rational singularities $\Longleftrightarrow$ $\RS^0(S)=0$.
 \item $S$ has Witt-rational singularities $\Longleftrightarrow$ $\RS(S)=0$.
\end{enumerate}
\end{thm}

\begin{proof}
By de Jong (see Remark \ref{de-Jong-did-it}) there exists an alteration $f:X\to S$ which factors as 
\[f: X\xr{h} Y\xr{g} S,\]
where $X$ is smooth, $Y=X/G$, for $G$ a finite group, $h$ is the quotient map and $g$ is a quasi-resolution. 
Thus the following equalities hold by definition for all $i$
\[H^i(\mc{WS}_{0,S})= R^ig_*W\sO_{Y,\Q},\quad H^i(\mc{WS}_{d,S})= R^ig_* W\omega_{Y,\Q}.\]
Since $Y$ is normal and $h$ is finite and surjective it is  also universally equidimensional.
Thus $[X\times_Y X]\in \CH^d(X\times X, P(X\times_Y X))$, where $d=\dim S$. We denote by $[X\times_S X]^0$ the cycle associated to the closure of $X_\eta\times_\eta X_\eta$ in $X\times_S X$, where $\eta$ is the generic point of $S$. We claim 
 \eq{Witt-rational-via-alterations1}{[X\times_S X]^0=\deg g\cdot [X\times_Y X]\quad \text{in } \CH^d(X\times X, P(X\times_S X)).}
Indeed it suffices to check this over a smooth and dense open subscheme $U$ of $S$, over which $g$ is a universal homeomorphism and $Y$ is smooth. 
But then 
\begin{align*}
[X_U\times_U X_U]^0 & = [X_U\times_U X_U]\\
                    & = [\Gamma_{f_U}^t]\circ[\Gamma_{f_U}]\\
                    & =[\Gamma_{h_U}^t]\circ [\Gamma_{g_U}^t]\circ [\Gamma_{g_U}]\circ [\Gamma_{h_U}]\\
                    & = \deg g\cdot [\Gamma_{h_U}^t]\circ [\Gamma_{h_U}] & \text{by \eqref{pb-pf-univ-homeo1}}\\
                    & =\deg g\cdot[X_U\times_{Y_U} X_U].\\
\end{align*}
Hence the claim.
Now $\hat{\sH}([X\times_Y X]/Y)$ induces an endomorphism of $h_*(W\sO_X\oplus W\Omega^d_X)$ which actually equals $ h^*\circ h_*$, for $h_*$ and $h^*$ as in Definition \ref{pb-pf-for-Witt-canonical}
(see \eqref{top-fin-quot-vs-smooth1}.) It thus follows from Proposition \ref{correspondences-and-change-of-basis} and \eqref{Witt-rational-via-alterations1} that $e_f=\frac{1}{\deg f}\hat{\sH}([X\times_S X]^0/S)$ factors 
for each $i\ge 0$ as
\[\xymatrix{ R^if_*(W\sO_X\oplus W\Omega^d_{X})_\Q\ar[rr]^{e_f}\ar[dr]_{\frac{1}{\deg h}\cdot h_*}    &     &  R^if_*(W\sO_X\oplus W\Omega^d_{X})_\Q\\
                                            &             R^ig_*(W\sO_Y\oplus W\omega_Y)_\Q\ar[ur]_{h^*}.   }\]
Taking into account that $h_*\circ h^*$ is multiplication with the degree of $h$, this yields the statement of the theorem.
\end{proof}

\subsection{Isolated singularities}

In this section we will relate the sheaf $\RS^0(S)$ (Definition \ref{RS}) for a normal variety $S$ with an 
isolated singularity  to the Witt vector cohomology of the exceptional set
in a suitably good resolution of singularities $\tilde{S}$. For this, we
need to compute the higher direct images of $W(\OO_{\tilde{S}})_{\Q}$.
  
\begin{proposition}\label{proposition-vanishing-WI}
Let $f:\tilde{X}\xr{} \tilde{Y}$ be a proper morphism of schemes. Let 
$Y_0\subset \tilde{Y}$ be a closed subset, we denote by $Y=\tilde{Y}\backslash Y_0$
the complement. We consider the cartesian diagrams 
$$
\xymatrix
{
X \ar[r]\ar[d] \ar@{}[dr]|\square
&
\tilde{X}\ar[d]
&
X_0 \ar[l]\ar[d] \ar@{}[dl]|\square
\\
Y \ar[r]
&
\tilde{Y}
&
Y_0. \ar[l]
}
$$ 
Let $\mc{I}\subset \OO_{\tilde{X}}$ be a sheaf of ideals for $X_0$. Suppose that
$
R^if_*\OO_X=0 
$ 
for all $i>0$.
Then 
$$
R^if_*W(\mc{I})_{\Q}=0 \quad \text{for all $i>0$.} 
$$
\end{proposition}

In order to prove Proposition \ref{proposition-vanishing-WI} we need several
Lemmas. 

\begin{lemma}\label{lemma-I-to-Ia}
Let $X$ be a scheme and $\mc{I}\subset \OO_X$ a sheaf of ideals. 
For all integers $a\geq 1$ the natural map 
$$
W(\mc{I}^a)_{\Q}\xr{} W(\mc{I})_{\Q}
$$
is an isomorphism.
\begin{proof}
The proof is the same as in \cite[Proposition~2.1(ii)]{BBE}.
\end{proof}
\end{lemma} 

\begin{lemma}\label{lemma-I-vanishing}
With the assumptions of Proposition \ref{proposition-vanishing-WI}.
There are $N,a\geq 0$ such that for all $r\geq N$
and all $n\geq 1$ the morphism 
$$
R^if_*W_n(\mc{I}^{r+a})\xr{}R^if_*W_n(\mc{I}^{r}), 
$$
induced by $\mc{I}^{r+a}\subset \mc{I}^{r}$,
vanishes for all integers $i>0$.
\begin{proof}
\cite[Lemma~2.7(i)]{BBE}.
\end{proof}
\end{lemma}

\begin{proof}[Proof of Proposition \ref{proposition-vanishing-WI}]
Choose $N$ as in Lemma \ref{lemma-I-vanishing}. By using Lemma \ref{lemma-I-to-Ia}
we may replace $\mc{I}$ by $\mc{I}^N$. Thus we may assume that $N=1$.
The image of the Frobenius acting on $W_n(\mc{I})$ is contained in 
$W_n(\mc{I}^{p})$:
$$
\xymatrix
{
W_n(\mc{I}) \ar[rr]^{F}\ar[dr]^{\exists !}
&
&
W_n(\mc{I})
\\
&
W_n(\mc{I}^{p})\arir[ur]
&
}
$$
and therefore $F^a$ (with $a$ as in Lemma \ref{lemma-I-vanishing}) is zero on 
$R^if_*W_n(\mc{I})$ for all $i>0$.

We continue as in the proof of \cite[Theorem~2.4]{BBE}. Since $F^a$ acts 
as zero on $\varprojlim_n R^if_*W_n(\mc{I})$ and 
$R^1\varprojlim_n R^if_*W_n(\mc{I})$ for all $i\geq 1$, we obtain via the 
exact sequence 
\[0\to R^1\varprojlim_n R^{i-1}f_*W_n(\mc{I})\to R^if_*(W(\mc{I}))\to \varprojlim_n R^if_*W_n(\mc{I})\to 0\]
that $F^{2a}$ acts as zero on $R^{i}f_*W(\mc{I})$ for all $i>0$ (we use that 
$R^1\varprojlim_n f_*W_n(\mc{I})=0$). Thus the 
relation $FV=p$ implies that $p^{2a}$ kills $R^{i}f_*W(\mc{I})$ for all $i>0$.
\end{proof}

\begin{corollary}\label{corollary-direct-image-isolated-singularity}
Let $S$ be an integral  
scheme with an isolated singularity at the closed point $s\in S$.
Let $f:X\xr{} S$ be projective and birational. Suppose that $f$
is an isomorphism over $S\backslash \{s\}$; we denote by $E$ the closed 
set $f^{-1}(s)$ equipped with some scheme structure.
Then 
$$
R^if_*W\OO_{X,\Q}=H^i(E,W\OO_{E,\Q}) \quad \text{for all $i>0$,}
$$  
where $H^i(E,W\OO_{E,\Q})$ is considered as skyscraper sheaf supported in $s$.
Moreover, the morphism is compatible with the Frobenius and Verschiebung action.
\begin{proof}
Let $\mc{I}$ be the sheaf of ideals for $E$. We obtain an exact sequence
$$
0\xr{} W\mc{I}_{\Q} \xr{} W\OO_{X,\Q} \xr{} W\OO_{E,\Q}\xr{} 0.
$$
Proposition \ref{proposition-vanishing-WI} implies that the higher direct 
images of $W\mc{I}_{\Q}$ vanish, which proves the statement.
\end{proof}
\end{corollary}

\begin{lemma}\label{lemma-sectral-sequence-E_i}
Let $E$ be a scheme. Write $E=\cup_{i=1}^r E_i$ with $E_i$ irreducible for all $i$.
There is a spectral sequence 
$$
E_1^{s,t}=\bigoplus_{1\leq \imath_0 <\imath_1< \dots <\imath_{s}\leq r} H^t(\cap_{j=0}^sE_{\imath_{j}},W\OO_{\cap_{j=0}^sE_{\imath_{j}},\Q})\Rightarrow H^{s+t}(E,W\OO_{E,\Q}).
$$
The spectral sequence is compatible with the Frobenius and Verschiebung operation.
\begin{proof}
By \cite[Proposition~2.1(i)]{BBE} the sheaf $W\OO_{\cap_{j=0}^sE_{\imath_{j}},\Q}$ doesn't 
depend on the choice of the scheme structure for $\cap_{j=0}^sE_{\imath_{j}}$. 

From \cite[Corollary~2.3]{BBE} we get an exact sequence 
\begin{equation*}
W\OO_{E,\Q}\xr{} \bigoplus_{\imath_0}W\OO_{E_{\imath_0},\Q}\xr{} \bigoplus_{\imath_0<\imath_1}W\OO_{E_{\imath_0}\cap E_{\imath_1},\Q}\xr{} \dots,
\end{equation*}
where the maps are sums of restriction maps and thus compatible with $F$ and $V$.


For every affine open set $U\subset E$ we have 
$$
H^{t}(U\cap_{j=0}^sE_{\imath_{j}},W\OO)=0, \quad \text{for all $t>0$,}
$$ 
because $U\cap_{j=0}^sE_{\imath_{j}}$ is affine. Therefore we obtain the 
spectral sequence in the statement. 
\end{proof}
\end{lemma}

\begin{proposition}\label{proposition-spectral-sequence-degenerates}
Let $E$ be a projective scheme over a finite field $k=\mathbb{F}_{p^a}$. Write $E=\cup_{i=1}^r E_i$ with $E_i$ irreducible for all $i$. Suppose that for all 
$s\geq 1$ and all $1\leq \imath_0< \dots <\imath_{s}\leq r$ the set theoretic
intersection $\cap_{j=0}^sE_{\imath_{j}}$, equipped with the reduced scheme 
structure, is smooth. 

For all $n\geq 2$ the differential of $E_n^{s,t}$, induced by the 
spectral sequence of Lemma \ref{lemma-sectral-sequence-E_i}, vanishes.
In other words, $E_2^{s,t}\Rightarrow H^{s+t}(E,W\OO_{E,\Q})$ is degenerated and 
$$
E_{\infty}^{s,t}=E_2^{s,t} \quad \text{for all $s,t$.}
$$
\begin{proof}
By definition $E_{n}^{s,t}$ is a subquotient of $E_{1}^{s,t}$. By assumption,
$\cap_{j=0}^sE_{\imath_{j}}$ is smooth and projective. We know that 
$$
H^t(\cap_{j=0}^sE_{\imath_{j}},W\OO_{\cap_{j=0}^sE_{\imath_{j}},\Q})\cong H^t_{{\rm crys}}(\cap_{j=0}^sE_{\imath_{j}}/K)_{[0,1[},
$$
where the right hand side is the slope $<1$-part of crystalline cohomology ($K=W(k)$). This isomorphism is compatible with the $F$-operation.
It follows from \cite{KaMe}
that the characteristic polynomial of $F^a$ acting on $H^t_{{\rm crys}}(\cap_{j=0}^sE_{\imath_{j}}/K)$ is the characteristic polynomial of the geometric
Frobenius acting on $H^t_{\text{\'et}}((\cap_{j=0}^sE_{\imath_{j}})\times_{k}\bar{k},\Q_{\ell})$ for any 
prime $\ell\neq p$.  Thus $F^a$ acting on 
$H^t_{{\rm crys}}(\cap_{j=0}^sE_{\imath_{j}}/K)_{[0,1[}$ 
has algebraic eigenvalues with absolute value $p^{\frac{ta}{2}}$ with respect to 
any embedding into $\C$. In particular, this holds for the 
eigenvalues of $E_{n}^{s,t}$ for $n\geq 1$ and all $s$. Thus the 
differential 
$$
E_{n}^{s-n,t+n-1}\xr{} E_{n}^{s,t}\xr{} E_{n}^{s+n,t-n+1}
$$
vanishes if $n\geq 2$. 
\end{proof}
\end{proposition}

\begin{thm}\label{thm-criterion-WO-rational}
Let $k$ be a finite field. 
Let $S$ be a normal integral scheme with an isolated singularity at the closed point $s\in S$.
Let $f:X\xr{} S$ be projective and birational. Suppose that $f$
is an isomorphism over $S\backslash \{s\}$; we denote by $E$ the closed 
set $f^{-1}(s)$.  Write $E=\cup_{i=1}^r E_i$ with $E_i$ irreducible for all $i$. Suppose that for all 
$s\geq 1$ and all $1\leq \imath_0< \dots <\imath_{s}\leq r$ the set theoretic
intersection $\cap_{j=0}^sE_{\imath_{j}}$, equipped with the reduced scheme 
structure, is smooth.

Then $S$ has $W\OO$-rational singularities (Definition \ref{definition-Witt-rational}) if and only if
the spectral sequence of Lemma \ref{lemma-sectral-sequence-E_i}
satisfies
$$
E_2^{s,t}=0 \quad \text{for  all $(s,t)\neq (0,0)$.}
$$ 
\begin{proof}
It follows from Corollary \ref{corollary-direct-image-isolated-singularity}
that $S$ has $W\OO$-rational singularities if and only if
$$
H^i(E,W\OO_{E,\Q})=0 \quad \text{for all $i>0$.}
$$
The assertion follows from Proposition \ref{proposition-spectral-sequence-degenerates}.
\end{proof}
\end{thm}

The case $t=0$ is a combinatorial condition on the exceptional set $E$. 
For any smooth and proper scheme $X$ over $k$ we have
$$
H^0(X,W\OO_X)\otimes_{W(k)}W(\bar{k}) \cong H^0_{\text{\'et}}(X\times_k \bar{k},\Q_p)\otimes_{\Q_p} W(\bar{k}).
$$ 
Therefore the condition $E^{s,0}_2=0$, for all $s\geq 1$, is equivalent to 
the vanishing of the cohomology of the complex
\begin{multline*}
\bigoplus_{\imath_0}H^0(E_{\imath_0}\times_{k}\bar{k},\Q_p)\xr{} 
\bigoplus_{\imath_0<\imath_1}H^0((E_{\imath_0}\cap E_{\imath_1})\times_{k}\bar{k},\Q_p)\xr{}\\
\bigoplus_{\imath_0<\imath_1<\imath_2}H^0((E_{\imath_0}\cap E_{\imath_1}\cap E_{\imath_2})\times_{k}\bar{k},\Q_p) \xr{} \dots
\end{multline*}
in degree $\geq 1$. Of course, the cohomology in degree $=0$ equals 
$H^0(E\times_{k}\bar{k},\Q_p)=\Q_p$, because $S$ is normal. 

For a surface $S$ the conditions for $t\geq 1$ are equivalent to 
$E_{i,{\rm red}}\times_k \bar{k}\cong \coprod_{j} \P^1$
for all $i$. Indeed, for a smooth curve $E_i$ we have 
$$
\dim H^1(E_i,W\OO_{E_i,\Q})\geq \dim H^1(E_i,\OO_{E_i}),
$$
and $H^1(E_i,\OO_{E_i})=0$ if and only if $E_i\times_k \bar{k}$ is a disjoint 
union of $\P^1$s.

Therefore $S$ has $W\OO$-rational singularities if and 
only if the exceptional divisor of a minimal resolution over $\bar{k}$ is a tree of $\P^1$s.

\subsection{Cones and Witt-rational singularities}\label{section-on-cones}
In \cite[Ex. 2.3]{BE} Blickle and Esnault give an example of an $\mathbb{F}_p$-scheme which has BE-rational singularities but not rational singularities.
Their proof in fact shows, that their example also has $W\sO$-rational singularities. In this section we slightly generalize their example and show that it
also has Witt-rational singularities.

\subsubsection{}\label{Grauert-criterion}
Let $X$ be a proper $k$-scheme, $\sL$ an invertible sheaf on $X$ and $V(\sL)=\Spec (\oplus_{n\ge 0} \sL^{\otimes n})$.
We denote by $s_0: X\inj V(\sL)$ the zero-section.
By Grauert's criterion (see \cite[Cor. (8.9.2)]{EGAII}) $\sL$ is ample on $X$ if and only if there exists a $k$-scheme $C$ together
with a $k$-rational point $v\in C$ and a proper morphism $q: V(\sL)\to C$, such that $q$ induces an isomorphism $V(\sL)\setminus s_0(X)\xr{\simeq} C\setminus\{v\}$ and
$q^{-1}(v)_{\rm red}=s_0(X)_{\rm red}$. If $\sL$ is ample, we call any triple $(C,q,v)$ as above a {\em contractor of the zero-section of $V(\sL)$}.
One can choose $C$ for example to be the cone $\Spec S$, where $S$ is the graded ring $k\oplus_{n\ge 1} H^0(X, \sL^{\otimes n})$ and $v$ is the vertex,
i.e. the point corresponding to the ideal $S_+$.

The following proposition is well-known; we prove it for lack of reference.

\begin{proposition}\label{GR-Kodaira}
Let $X$ be a proper and smooth $k$-scheme and $\sL$ an ample sheaf on $X$. Then the following statements are equivalent:
\begin{enumerate}
 \item For any $n\ge 1$ and all contractors $(C,q, v)$ of the zero-section of $V(\sL^{\otimes n})$ we have $R^iq_*\omega_{V(\sL^{\otimes n})}=0$ for all $i\ge 1$.
\item $H^i(X, \omega_X\otimes\sL^{\otimes n})=0$ for all $i,n\ge 1$.
\end{enumerate}
\end{proposition}
\begin{proof}
Let $(C,q,v)$ be a contractor of the zero-section of $V(\sL)$. Denote by $\sI$ the ideal sheaf of the zero section of $X$ in $V(\sL)$.
We have 
\[s_0^*\sI=\sL,\quad s_{0*}\sL^{\otimes n}= \sI^n/\sI^{n+1}.\]
We set $Y_n:= \Spec (\sO_{V(\sL)}/\sI^n)$; in particular $Y_1=s_0(X)$.
The sheaves $R^iq_*\omega_{V(\sL)}$ have support in $\{v\}$ and since $\sqrt{\fm_v\sO_{V(\sL)}}= \sI$, where $\fm_v\subset \sO_{X,v}$ is the ideal of $v$, the theorem on formal functions yields
\[R^iq_*\omega_{V(\sL)}=0 \Longleftrightarrow \varprojlim_n H^i(Y_n,\omega_{V(\sL)|Y_n})=0. \]
Further $\omega_{V(\sL)|Y_1}\cong \omega_X\otimes \sL$, hence tensoring the exact sequence $0\to \sI^n/\sI^{n+1}\to \sO_{Y_{n+1}}\to \sO_{Y_n}\to 0$
with $\omega_{V(\sL)|Y_{n+1}}$ yields the exact sequence
\eq{GR-Kodaira1}{0\to \omega_X\otimes \sL^{\otimes {n+1}}\to \omega_{V(\sL)|Y_{n+1}}\to \omega_{V(\sL)|Y_n}\to 0.}

(2)$\Rightarrow$ (1): We have $H^i(Y_1, \omega_{V(\sL)|Y_1})=H^i(X, \omega_{X}\otimes \sL)=0$ for all $i\ge 1$, by assumption.
                      Now the statement for $(C,q,v)$ follows from \eqref{GR-Kodaira1} and induction. Since $\sL$ is any ample sheaf, 
                       we can replace $\sL$ by $\sL^{\otimes n}$ in the above argument and obtain the statement also for 
                      a contractor of the zero-section of $V(\sL^{\otimes n})$.

(1)$\Rightarrow$ (2): Let $d$ be the dimension of $X$. Then the maps $H^d(Y_{n+1},\omega_{V(\sL)|Y_{n+1}})\to H^d(Y_n,\omega_{V(\sL)|Y_n})$ are surjective for all $n\ge 1$.
                      Since $\varprojlim_n H^d(Y_n,\omega_{V(\sL)|Y_n})=0$ by assumption, we get in particular
                        \[H^d(Y_1,\omega_{V(\sL)|Y_1})= H^d(X, \omega_X\otimes \sL)=0.\]
                       Replacing $\sL$ by $\sL^{\otimes n}$ in the above argument thus gives us
                        \[H^d(X, \omega\otimes \sL^{\otimes n})=0, \quad\text{for all } n\ge 1.\]
                       Assume we proved $H^{i+1}(X, \omega\otimes \sL^{\otimes n})=0$, for all $n\ge 1$. Then 
                        \eqref{GR-Kodaira1} yields that $H^i(Y_{n+1},\omega_{V(\sL)|Y_{n+1}})\to H^i(Y_n,\omega_{V(\sL)|Y_n})$
                        is surjective and as above we conclude that in particular 
                      $H^i(Y_1,\omega_{V(\sL)|Y_1})=H^i(X, \omega_X\otimes\sL)$ equals zero (if $i\ge 1$). Again we can do the argument with $\sL$
                      replaced by $\sL^{\otimes n}$ and obtain  $H^i(X, \omega_X\otimes\sL^{\otimes n})=0$ for all $n\ge 1$.
                      This finishes the proof.
\end{proof}

\begin{remark}
 Notice that in characteristic zero condition (1) is always satisfied because of the Grauert-Riemenschneider vanishing theorem; condition (2) is always satisfied because of
 the Kodaira vanishing theorem.
\end{remark}

\begin{thm}\label{example-Witt-rational}
 Let $X_0$ be a smooth, projective and geometrically connected $k$-scheme and $\sL$ an ample sheaf on $X_0$.
Assume $(X_0, \sL)$ satisfies the following condition:
\eq{Kodaira-van}{ H^i(X_0, \omega_{X_0}\otimes \sL^{\otimes n})=0,\quad \text{for all } n, i\ge 1.}
Let $X$ be the projective cone of $(X_0,\sL)$, i.e. $X={\rm Proj}(S[z])$, where $S$ denotes the graded $k$-algebra $k\oplus_{n\ge 1} H^0(X_0,\sL^{\otimes n})$.

Then $X$ is integral and  normal;  it has  Witt-rational singularities if and only if 
\eq{WO-van}{H^i(X_0, W\sO_{X_0})_\Q=0,\quad \text{for all } i\ge 1.}
Furthermore if $\sL$ is very ample, then $X$ is $CM$ if and only if the following condition is satisfied:
 \eq{CM-condition}{ H^i(X_0,\sL^{\otimes n})=0,\quad  \text{for all } n\in \Z \text{ and all } 1\le i\le \dim X_0-1.}
\end{thm}

\begin{proof}
Let $v$ be the $k$-rational point of $X$, corresponding to the homogenous ideal in $S[z]$ generated by $S_+$.
Then the pointed projective cone $X\setminus\{v\}$ is an $\A^1$-bundle over $X_0={\rm Proj\,} S$. Thus $X$ has an isolated singularity at $v$ and 
it suffices to consider the affine cone $C=\Spec S$ of $(X_0,\sL)$.  Since $X_0$ is integral, projective and geometrically connected and $k$ is perfect we have $H^0(X_0,\sO_{X_0})=k$, by Zariski's connectedness theorem.
Thus $C$ is integral and normal by \cite[Prop. (8.8.6), (ii)]{EGAII}. Set $V:=V(\sL)=\Spec(\oplus_{n\ge 0}\sL^{\otimes n})$; it is an $\A^1$-bundle over $X_0$, hence is smooth.
By \cite[Cor. (8.8.4)]{EGAII} there exists a projective morphism $q: V\to C$, such that the triple $(C, q, v)$ becomes a contractor of the zero-section of $V$.
In particular, $q: V\to C$ is a projective resolution of singularities of $C$. (In case $\sL$ is very ample, $q:V\to C$ is the blow-up of $C$ in the closed point $v$, see \cite[Rem. (8.8.3)]{EGAII}.)
By Proposition \ref{properties-Witt-rational-sing} and Remark \ref{Witt-rational-and-normal} it suffices to prove
\[R^iq_*W\sO_{V,\Q}=0,\quad R^iq_*W\omega_{V,\Q}=0,\quad\text{for all } i\ge 1.\]
By Corollary \ref{corollary-direct-image-isolated-singularity} the vanishing $R^iq_*W\sO_{V,\Q}=0$ is equivalent to the vanishing $H^i(X_0, W\sO_{X_0})_{\Q}=0$.
Thus it suffices to prove the vanishing $R^iq_*W\omega_{V,\Q}=0$ under the assumptions \eqref{Kodaira-van} and $Rq_*W\sO_{V,\Q}\cong W\sO_{C,\Q}$.
Proposition \ref{GR-Kodaira} yields $R^iq_*\omega_V=0$ for all $i\ge 1$. Hence also $R^iq_*W_n\omega_V=0$ for all $i,n\ge 1$, by the exact sequence in Proposition \ref{properties-Witt-canonical}, (7)
and induction. Therefore the exact sequence
\[0\to R^1\varprojlim_n R^{i-1}q_* W_n\omega_V\to R^iq_*W\omega_V\to \varprojlim_n R^iq_*W_n\omega_V\to 0\]
 immediately yields 
\eq{example-Witt-rational1}{R^iq_*W\omega_{V,\Q}=0,\quad \text{for all } i\ge 2.}
To conclude also the vanishing of $R^1q_*W\omega_{V,\Q}$ we have to prove the vanishing $(R^1\varprojlim_n q_*W_n\omega_V)_{\Q}=0$.
To this end denote by $I_\bullet$, $K_\bullet$ and $C_\bullet$ the image, the kernel and the cokernel of $q_*W_\bullet\omega_V\to W_\bullet\omega_C$, respectively.
Notice that $K_n$ and $C_n$ are coherent $W_n\sO_C$-modules, whose support is concentrated in the closed point $v$; hence these modules have finite length. Therefore
$K_\bullet$ and $C_\bullet$ satisfy the Mittag-Leffler condition, in particular $R^1\varprojlim C_n=0=R^1\varprojlim K_n$.
Furthermore, $R^2\varprojlim K_n=0$ by Lemma \ref{vanishing-of-Rlim}, (1). Thus the exact sequence $0\to K_\bullet\to q_*W_\bullet\omega_V\to I_\bullet\to 0$ gives
\[R^1\varprojlim q_*W_n\omega_V\cong R^1\varprojlim I_n.\]
We also have the vanishing $R^1\varprojlim W_n\omega_C=0$, since the transition maps are surjective (by Proposition \ref{properties-Witt-canonical}, (4)) and
therefore the exact sequence $0\to I_\bullet\to W_\bullet\omega_C\to C_\bullet\to 0$ gives a surjection 
\[\varprojlim C_n\to R^1\varprojlim I_n\to 0.\]
By Lemma \ref{pf-of-Womega}  $(\varprojlim C_n)_\Q=0$, hence also $(R^1\varprojlim q_*W_n\omega_V)_\Q=0$.

For the last statement notice that $X$ is CM if and only if $H^i_v(C,\sO_C)=0$ for all $i\le d-1=\dim X_0$. Since $C$ is affine the long exact localization sequence gives us
an exact sequence
\[0\to H^0_v(C,\sO_C)\to H^0(C,\sO_C)\to H^0(C\setminus\{v,\},\sO_C)\to H^1_v(C,\sO_C)\to 0\]
and isomorphisms
\[H^{i-1}(C\setminus\{v\},\sO_C)\cong H^i_v(C,\sO_C),\quad \text{for all }i\ge 2.\]
Let $\pi: C\setminus \{c\}\to X_0$ be the pointed affine cone over $X_0={\rm Proj}\, S$.
Since $\sL$ is very ample the graded ring $S$ is generated by $S_1$  and we obtain 
\[\pi_*\sO_{C\setminus\{v\}}=\bigoplus_{n\in\Z} \sL^{\otimes n}.\]
Therefore $H^0(C\setminus\{v\}, \sO_C)=\bigoplus_{n\ge 0} H^0(X_0,\sL^{\otimes n})=S=H^0(C,\sO_C)$.
Thus $H^0_v(C,\sO_C)$ and $H^1_v(C,\sO_C)$ always vanish and $H^i_v(C, \sO_C)$ vanishes for $2\le i\le \dim X_0 $ if and only if 
$\bigoplus_{n\in \Z} H^{i-1}(X_0,\sL^{\otimes n})$ vanishes, which is exactly condition \eqref{CM-condition}. This proves the Theorem.
\end{proof}

\begin{remark}\label{remark-cone}
\begin{enumerate}
 \item Condition \eqref{Kodaira-van} in the above theorem holds in characteristic zero by Kodaira vanishing, which in general is wrong in positive characteristic; for a counter example see \cite{Raynaud}.
       By \cite[Cor. 2.8.]{Deligne-Illusie} this condition holds if $X_0$ lifts to a smooth $W_2(k)$-scheme and has dimension $\le p$.
        Notice also that condition \eqref{Kodaira-van} and \eqref{CM-condition} are always satisfied  if $X_0$ is a smooth hypersurface in some $\P^n_k$ and $\sL=\sO_{\P^n_k}(1)_{|X_0}$.
\item The vanishing \eqref{WO-van} is satisfied if the degree map induces an isomorphism $\CH_0(X_0\times_k \overline{k(X_0)})_\Q\cong \Q$, which is for example the case if $X_{0,\bar{k}}$ is rationally chain connected, 
      where $\bar{k}$ is an algebraic closure of $k$.  See \cite[Ex. 2.3]{BE}, alternatively one can also use Bloch's decomposition of the diagonal and the vanishing Lemmas \ref{lemma-van1} and \ref{lemma-van2}.
       Another case where \eqref{WO-van} holds is when $X_0$ has a smooth projective model $\mc{X}_0$ over a complete discrete valuation ring $R$  of mixed characteristic and with residue field $k$, such that
        the generic fiber $\mc{X}_{0,\eta}$ satisfies $H^i(\mc{X}_{0,\eta}, \sO_{X_{0,\eta}})=0$ for all $i\ge 1$. This follows from $p$-adic Hodge theory
        (`` the Newton polygon of $H^i_{\rm crys}(X_0/W)\otimes \text{Frac}(R)\cong H^i_{dR}(\mc{X}_{0,\eta})$ lies above the Hodge polygon''). 
\item The example in \cite{BE} alluded to at the beginning of this section is the following: By \cite[Prop. 3]{Shioda} the Fermat hypersurfaces $X_0\subset \P^{2r+1}_k$ given by
       $x_0^n+\ldots +x_{2r+1}^n$, where $n$ is such that $p^\nu\equiv -1$ mod $n$, for some $\nu\ge 1$,  are unirational over $\bar{k}$. Hence $X_0$ satisfies the conditions 
       \eqref{Kodaira-van}, \eqref{WO-van} and \eqref{CM-condition} of the theorem and we obtain that the projective cone $X$ of $X_0$ is normal, CM and has Witt-rational singularities. 
       (In \cite[Ex. 2.3.]{BE} it was shown that $X$ is $W\sO$-rational.)
        Choosing $X_0$ of degree larger than $2r+2$ we see, that $H^{\dim X_0}(X_0,\sO_{X_0})\neq 0$ and it follows that $X$ cannot have rational singularities.
\item Let $X$ be as in the theorem and let $\pi: Y\to X$ be any resolution of singularities.  Then $R^i\pi_*W\omega_{Y,\Q}=0$ for all $i\ge 1$.
      As it follows from the proof and \cite[Thm. 1]{CR}, we even have the stronger vanishing $R^i\pi_*\omega_Y=0$ for all $i\ge 1$, which is implied by condition \eqref{Kodaira-van} solely.
      It is an obvious question whether the vanishing $R^i\pi_*W\omega_{Y,\Q}=0$ also holds without condition \eqref{Kodaira-van}, i.e. does some version of Grauert-Riemenschneider vanishing hold.
      In view of Proposition \ref{GR-Kodaira}, this question should be linked to some kind of Kodaira vanishing, of which at the moment even the formulation is not clear.
      (But see \cite[Cor. 1.2]{BBE} for a first result in this direction.)
\end{enumerate}
\end{remark}

\subsection{Morphisms with rationally connected generic fibre}

The goal of this section is to prove the following theorem.

\begin{thm} \label{thm-rational-connected-fibres}
Let $X$ be an integral scheme with Witt-rational singularities.
Let $f:X\xr{} Y$ be a projective morphism to an integral and normal scheme $Y$.  
We denote by $\eta$ the generic point of $Y$, and $X_{\eta}$ denotes the generic
fibre of $f$. Suppose that $X_{\eta}$ is smooth and for every field extension 
$L\supset k(\eta)$ the degree map 
$$
\CH_0(X_{\eta}\times_{k(\eta)}L)\otimes_{\Z} \Q\xr{} \Q
$$
is an isomorphism.
Then 
$$
\bigoplus_{i>0} R^if_*W\OO_{X,\Q}\cong \bigoplus_{i>0}H^i(\mc{WS}_{0,Y}), \quad \bigoplus_{i>0} R^if_*W\omega_{X,\Q}\cong  \bigoplus_{i>0}H^i(\mc{WS}_{\dim Y,Y}), 
$$
where $\mc{WS}$ is defined in Definition \ref{canonical-complex-for-singularities}.
\end{thm}

\subsubsection{} \label{notation-rational-connected-fibres}
Let $f:X\xr{} Y$ be as in the assumptions of 
Theorem \ref{thm-rational-connected-fibres}. We can choose a factorization
(cf.~Remark \ref{de-Jong-did-it})
$$
X'\xr{\pi} X'/G \xr{h} X''\xr{u} X, 
$$
such that 
\begin{itemize}
\item $X'$ is smooth, integral and quasi-projective, $G$ is a finite group acting on $X'$,
\item $h$ is birational and projective,
\item $X''$ is normal and $u$ is a universal homeomorphism. 
\end{itemize}

We obtain a commutative diagram 
\begin{equation} \label{diagram-quasi-resolution}
\xymatrix
{
X' \ar[d]^{\pi}
&
\\
X'/G \ar[d]^{h}
&
\\
X'' \ar[d]^{u} \ar[r]
& 
Y'' \ar[d]^{u'}
\\
X\ar[r]^{f}
&
Y
}
\end{equation}
where $Y''$ is the Stein factorization of $X''\xr{} Y$. Therefore $u'$
is a universal homeomorphism. Note that $Y''$ is also the Stein factorization
of $X'/G\xr{} Y$.

\begin{lemma}\label{lemma-first-step-rational-connected-fibres}
With the assumptions of Theorem \ref{thm-rational-connected-fibres}.
Set $f'=f\circ u \circ h\circ \pi$ with the notation 
as in diagram \eqref{diagram-quasi-resolution}. We consider $X'$ via $f'$ 
as a scheme over $Y$.
Let $P=\sum_{g\in G}[\Gamma(g)] \in \CH(X'\times_{Y}X')$, with $\Gamma(g)$
the graph of $g$. 
Then there is a natural isomorphism 
$$
\hH(P/Y)\left(\bigoplus_{i>0} R^if'_*(W\OO_{X',\Q}\oplus W\omega_{X',\Q})\right)
\xr{} 
\bigoplus_{i>0} R^if_*(W\OO_{X,\Q} \oplus W\omega_{X,\Q})
$$
\begin{proof}
The cycle $P=\sum_{g\in G}[\Gamma(g)]$ is already defined in 
$\CH(X'\times_{X'/G}X')$. Theorem \ref{Witt-rational-via-alterations} 
and Corollary \ref{finite-quotients-are Witt-rational} imply 
\begin{equation}\label{multline-P-image}
\hH(P/(X'/G))\left(\pi_*W\OO_{X',\Q} \oplus \pi_*W\omega_{X',\Q}\right)
\cong W\OO_{X'/G,\Q} \oplus W\omega_{X'/G,\Q}.
\end{equation}
It follows   from Proposition \ref{correspondences-and-change-of-basis} that 
$$
\hH(P/Y)=\bigoplus_i R^i(f\circ u\circ h)_*\hH(P/(X'/G)).
$$
Since $P^2=\#G\cdot P$, we obtain from \eqref{multline-P-image} that  
\begin{multline*}
\hH(P/Y)\left(\bigoplus_{i>0} R^if'_*(W\OO_{X',\Q} \oplus W\omega_{X',\Q})\right)=\\
\bigoplus_{i>0} R^i(f\circ u\circ h)_*(W\OO_{X'/G,\Q} \oplus W\omega_{X'/G,\Q}).
\end{multline*}
Because $X$ has Witt-rational singularities and $X'/G\xr{u\circ h} X$ is a
quasi-resolution we get
$$
R^j(u\circ h)_*(W\OO_{X'/G,\Q} \oplus W\omega_{X'/G,\Q})=\begin{cases} 0 & \text{for all $j>0$} \\ W\OO_{X,\Q} \oplus W\omega_{X,\Q} & \text{for $j=0$,} 
\end{cases}
$$
which implies the statement.
\end{proof}
\end{lemma}

\begin{proposition}\label{proposition-characterization-RS}
Let $f':X'\xr{} Y$ be a projective and surjective morphism. Suppose that
$Y$ is normal, and $X'$ is smooth and connected. We denote by $\eta$ the 
generic point of $Y$. 
Let $Y'\subset X'$ be a closed irreducible subset such that $Y'\xr{} Y$
is generically finite. We denote by $A_1$ and $A_2$ the  closure of $Y'_{\eta}\times_{\eta}X'_{\eta}$ and $X'_{\eta}\times_{\eta} Y'_{\eta}$ in $X'\times_Y X'$, 
respectively. Then there are natural isomorphisms 
\begin{align}
\label{align-characterization-RS-1} 
\hH([A_1]/Y)\left(\bigoplus_{i>0} R^if'_*W\OO_{X',\Q}\right)&\cong \RS^0(Y)\\
\label{align-characterization-RS-2} 
\hH([A_2]/Y)\left(\bigoplus_{i>0} R^if'_*W\omega_{X',\Q}\right)&\cong \RS^{\dim Y}(Y).
\end{align}
(See Definition \ref{RS} for $\RS$.)
\begin{proof}
We will prove \eqref{align-characterization-RS-2}, the identity \eqref{align-characterization-RS-1} can be proved in the same way.

Let $\tilde{Y}\xr{} Y'$ be an alteration of $Y'$, such that $\tilde{Y}$ is smooth. 
We denote by $\imath:\tilde{Y}\xr{} X'$ the composition. The graph of $\imath$
defines a morphism 
$$
\hH(\Gamma(\imath)/Y): \hH(\tilde{Y}/Y)(-r)\xr{} \hH(X'/Y),
$$
with $r:=\dim X'-\dim Y$. The closure $Z$ of $X'_{\eta}\times_{\eta} \tilde{Y}_{\eta}$
in $X'\times_Y \tilde{Y}$ defines a morphism 
$$
\hH(Z/Y):\hH(X'/Y)\xr{} \hH(\tilde{Y}/Y)(-r).
$$
We have $({\rm id}_{X'}\times \imath)_*([Z])=d\cdot [A_2]$ for some $d\neq 0$. 
Finally, we define $Q$ to be the closure of $\tilde{Y}_{\eta}\times_{\eta}\tilde{Y}_{\eta}$ in 
$\tilde{Y}\times_Y \tilde{Y}$. We have the following relations 
\begin{align*}
[\Gamma(\imath)]\circ [Z]-d[A_2]&\in \ker(\CH(X'\times_Y X')\xr{} \CH(X'_{\eta}\times_\eta X'_{\eta})),\\
[Z]\circ [\Gamma(\imath)]-[Q]&\in \ker(\CH(\tilde{Y}\times_Y \tilde{Y})\xr{} \CH(\tilde{Y}_{\eta}\times_\eta \tilde{Y}_{\eta})),\\
[A_2]\circ [A_2]- e[A_2] &\in \ker(\CH(X'\times_Y X')\xr{} \CH(X'_{\eta}\times_\eta X'_{\eta})),\\
[Q]\circ [Q] -g[Q]&\in \ker(\CH(\tilde{Y}\times_Y \tilde{Y})\xr{} \CH(\tilde{Y}_{\eta}\times_\eta \tilde{Y}_{\eta})),
\end{align*} 
for some $e,g\neq 0$. Because of Lemma \ref{lemma-van2} we conclude that 
$$
\hH([A_2]/Y)\left(\bigoplus_{i>0} R^if'_*W\omega_{X',\Q}\right)\cong 
\hH([Q]/Y)\left(\bigoplus_{i>0} R^i(f'\circ \imath)_*W\omega_{\tilde{Y},\Q}\right).
$$ 
By Definition \ref{RS} the right hand side is $\RS^{\dim Y}(Y)$. 
\end{proof}
\end{proposition}
 
\begin{proof}[Proof of Theorem~\ref{thm-rational-connected-fibres}]
We use diagram \ref{diagram-quasi-resolution} and the constructions of 
\ref{notation-rational-connected-fibres}.
The first step is to prove 
\begin{equation}\label{equation-1-thm-rational-connected-fibres}
\CH_0(X'_{\eta}\times_{k(\eta)}L)^G\otimes \Q\cong \Q
\end{equation}
for every field extension $L\supset k(\eta)$. 
We have a push-forward map 
$$
\alpha:\CH_0(X'_{\eta}\times_{k(\eta)}L)\otimes \Q \xr{} \CH_0(X_{\eta}\times_{k(\eta)}L)\otimes \Q,
$$
and because $X_{\eta}$ is smooth over $k(\eta)$ we have a pull-back map 
$$
\beta:\CH_0(X_{\eta}\times_{k(\eta)}L)\otimes \Q \xr{} \CH_0(X'_{\eta}\times_{k(\eta)}L)\otimes \Q.
$$
We have the following formula for the composition 
$$
\beta\circ \alpha=\deg(u)\sum_{g\in G} g_*,
$$
provided that $G$ acts faithfully on $X'$ (we may assume this).
Thus \eqref{equation-1-thm-rational-connected-fibres} follows from 
the assumptions. 

Let 
$$P=\sum_{g\in G}[\Gamma(g)] \in \CH(X'\times_{Y}X'),$$ 
with $\Gamma(g)$
the graph of $g$. We denote by $P_{\eta}$ the image of $P$ in 
$\CH(X'_{\eta}\times_\eta X'_{\eta})$. The cycle $P_{\eta}$ is invariant 
under the action of $G$ on the left and the right factor of $X'_{\eta}\times_\eta X'_{\eta}$.

By the same arguments as in \cite[Proposition~1]{BS} we see that there
are non-zero integers $N_1,N_2,M_1,M_2$, effective divisors $D_1,D_2$ on
$X'_{\eta}$, and a closed point $\alpha$ of $X'_{\eta}$, such that 
\begin{align}\label{align-1-thm-rational-connected-fibres}
N_1P_{\eta}+N_2 [\alpha\times_{\eta} X'_{\eta}] &\in {\rm image}(\CH(X'_{\eta}\times_{\eta} D_2)),\\
 M_1P_{\eta}+M_2 [X'_{\eta} \times_{\eta} \alpha ] &\in {\rm image}(\CH(D_1\times_{\eta} X'_{\eta})).\label{align-2-thm-rational-connected-fibres}
\end{align}

Let $\bar{D}_1, \bar{D}_2,$ and $Y_{\alpha}$ be the closures of $D_1,D_2,$ and $\alpha$ in $X'$.  
The map $Y_{\alpha}\xr{} Y$ is generically finite and 
$Y_{\alpha,\eta}=\alpha$. Let $A_1$ and $A_2$ be the closures of $\alpha\times_{\eta} X'_{\eta}$ and $X'_{\eta}\times_{\eta} \alpha$ in $X'\times_{Y} X'$. 
Because of the formulas \eqref{align-1-thm-rational-connected-fibres}, \eqref{align-2-thm-rational-connected-fibres} over the generic point $\eta$, there is a Weil divisor $S\subset Y$
such that 
\begin{align*}
N_1P+N_2 [A_1] &\in {\rm image}(\CH(X'\times_Y \bar{D}_2))+{\rm image}(\CH(X'_S\times_S X'_S)),\\
 M_1P+M_2 [A_2] &\in {\rm image}(\CH(\bar{D}_1\times_Y X'))+{\rm image}(\CH(X'_S\times_S X'_S)).
\end{align*}
Lemma \ref{lemma-van1} implies that $\hH(P/Y)$ acts as $\hH([A_1]/Y)$ on 
$$
\bigoplus_{i>0} R^if'_*W\OO_{X',\Q}.
$$
From Lemma \ref{lemma-van2} we conclude that $\hH(P/Y)$ acts as $\hH([A_2]/Y)$ on 
$$
\bigoplus_{i>0} R^if'_*W\omega_{X',\Q}.
$$
In view of Proposition \ref{proposition-characterization-RS}, Theorem \ref{Witt-rational-via-alterations}, and   
Lemma \ref{lemma-first-step-rational-connected-fibres} this implies the statement
of the theorem.
\end{proof}

\section{Further applications}
\subsection{Results on torsion}

\subsubsection{The Cartier-Dieudonn\'e-Raynaud algebra }\label{Recall-CDR}
Recall from \cite[I, (1.1)]{IR} that the Cartier-Dieudonn\'e-Raynaud algebra is the graded (non-commutative) $W$-algebra, generated by formal symbols $F$ and $V$ in degree $0$ and by $d$ in degree $1$ which satisfy the following relations
\[F\cdot a= F(a)\cdot F,\quad a\cdot V= V\cdot F(a), \quad (a\in W), \quad F\cdot V=p= V\cdot F\]
\[a\cdot d=d\cdot a \quad(a\in W), F\cdot d \cdot V= d,\quad d\cdot d=0.\]
(Here $F$ is a formal symbol, whereas $F(a)$ denotes the Frobenius on the Witt vectors of $k$ applied to the Witt vector $a$.)
Thus $R$ is concentrated in degree 0 and 1. Notice that any de-Rham-Witt module and any Witt-module on a scheme $X$ naturally becomes an $R$-module (the latter with $d$ acting as 0).
Therefore we have an exact functor 
\[\widehat{\drw}_X\to {\rm Sh}(X,R):=(\text{sheaves-of $R$-modules on $X$})\]
which trivially derives to a functor
\eq{drw-to-R}{\phi: D^b(\widehat{\drw}_X)\to D^b({\rm Sh}(X, R)):=D^b(X,R).}
  We set
\[R_n:= R/(V^n \cdot R+d\cdot V^n\cdot R).\]
Notice that this is a left $R$-module. We obtain a functor 
\[R_n\otimes_R: {\rm Sh}(X,R)\to {\rm Sh}(X, W_n[d]):=(\text{sheaves of $W_n[d]$-modules on $X$}),\]
where $W_n[d]$ is the graded $W_n$-algebra $W_n\oplus W_n\cdot d$, with $d^2=0$. By \cite[Prop. (3.2)]{IR} the following sequence is  an exact sequence of right $R$-modules
\eq{R-exact-seq}{0\to R(-1)\xr{(F^n,-F^nd)} R(-1)\oplus R\xr{dV^n+V^n} R\to R_n\to 0.} 
This allows us to calculate the left-derived functor
\[R_n\otimes^L_R -: D^-({\rm Sh}(X,R))\to D^-({\rm Sh}(X,W_n[d])):=D^-(X, W_n[d]).\]
One obtains the following.
\begin{enumerate}
 \item Assume $X$ is smooth. Then $R_n\otimes^L_R W\Omega_X\cong W_n\Omega_X$.
 \item Let $f:X\to Y$ be a morphism of $k$-schemes, then
       \[R_n\otimes^L_R Rf_*(-)\cong Rf_*(R_n\otimes^L_R(-)): D^b(X,R)\to D^b(Y,W_n[d]).\]
 \item Let $X$ be a scheme and $Z\subset X$ be a closed subscheme, then 
       \[R_n\otimes^L_R R\ul{\Gamma}_Z(-)\cong R\ul{\Gamma}_Z(R_n\otimes^L_R(-)): D^b(X,R)\to D^b(X, W_n[d]).\]
\end{enumerate}
The first statement is \cite[II, Thm (1.2)]{IR}, the last two statements follow directly from \eqref{R-exact-seq}.

\begin{proposition}[{\cite[I, Prop. 1.1.]{EII}}]\label{Ekedahl-Nakayama}
Let $S$ be a $k$-scheme and $M\in D^b(S,R)$ a complex of $R$-modules, which is bounded in both directions, i.e. there exists a natural number $m$, 
such that $H^i(M)^j$ is non-zero only for $(i,j)\in [-m,m]\times [-m,m]$.
\begin{enumerate}
 \item Assume there exist integers $r,s\in \Z$, such that $H^i(R_1\otimes^L_R M)^j=0$ for all pairs $(i,j)$ satisfying one of the following conditions
        \eq{Ekedahl-Nakayama1}{(i+j=r,\, i\ge s)  \text{ or } (i+j=r-1,\, i\ge s+1) \text{ or } (i+j=r+1, \, i\ge s+1).}
       Then 
            \[H^i(R_n\otimes^L_R M)^j=0,\quad \text{for all }n \text{ and all }(i,j) \text{ with } i+j=r, \, i\ge s.\]\
\item Assume there exist integers $r,s\in \Z$, such that $H^i(R_1\otimes^L_R M)^j=0$ for all pairs $(i,j)$ satisfying one of the following conditions
       \eq{Ekedahl-Nakayama2}{(i+j=r,\, i\le s) \text{ or } (i+j=r+1,\, i\le s) \text{ or } (i+j=r-1,\, i\le s-1).}
        Then 
        \[H^i(R_n\otimes_R^L M)^j=0\quad \text{for all }n \text{ and all } (i,j) \text{ with } i+j=r,\, i\le s.\]
\end{enumerate}
\end{proposition}
This proposition is called Ekedahl's Nakayama Lemma because applying it to the cone of a morphism $f: M\to N$ in $D^b(X, R)$ between two complexes, which are bounded in both directions,
we obtain that $R_n\otimes_R f$ is an isomorphism for all $n$, if $R_1\otimes_R^L f$ is. 

\begin{corollary}\label{cor-Ekedahl-Nakayama}
Let $S$ be a $k$-scheme and $f: X\to S$, $g: Y\to S$ two $S$-schemes which are smooth over $k$.
Let $\varphi: Rf_*W\Omega_X\to Rg_*W\Omega_Y$ be a morphism in $D^b(S,R)$.
Then $R_1\otimes_R \varphi$ is a morphism 
    \[\varphi_1:=R_1\otimes_R^L \varphi: Rf_*\Omega_X\to Rg_*\Omega_Y\quad \text{ in } D^b(S, k[d])\]
and furthermore:
\begin{enumerate}
 \item If there exists  an natural number $a\ge 0$ such that 
       $\varphi_1: \oplus_{j\le a} Rf_*\Omega_X^j\xr{\simeq} \oplus_{j\le a}Rg_*\Omega_Y^j$ is an isomorphism, then 
        \[\hat{\varphi}:=R\varprojlim_n(R_n\otimes_R \varphi): \bigoplus_{j\le a} Rf_* W\Omega^j_X\xr{\simeq} \bigoplus_{j\le a} Rg_*W\Omega^j_Y\]
        is an isomorphism.
\item If there exists a natural number $a\ge 0$ such that 
       $\varphi_1: \oplus_{j\ge a} Rf_*\Omega_X^j\xr{\simeq} \oplus_{j\ge a}Rg_*\Omega_Y^j$ is an isomorphism, then 
        \[\hat{\varphi}=R\varprojlim_n(R_n\otimes_R \varphi): \bigoplus_{j\ge a} Rf_* W\Omega^j_X\xr{\simeq} \bigoplus_{j\ge a} Rg_*W\Omega^j_Y\]
        is an isomorphism.
\end{enumerate}
\end{corollary}

\begin{proof}
 First of all notice, that by \ref{Recall-CDR}, (1) and (2) $R_n\otimes^L_R\varphi$ indeed is a morphism $Rf_*W_n\Omega_X\to Rg_*W_n\Omega_Y$
and that $R\varprojlim(R_n\otimes^L_R\varphi)$ is a morphism $Rf_*W\Omega_X\to Rg_*W\Omega_Y$, 
since $R\varprojlim Rf_*W_n\Omega_X= Rf_*R\varprojlim_n W_n\Omega_X=Rf_*W\Omega_X$.

Denote by $C$ the cone of $\varphi$ in $D^b(S, R)$. It is clearly bounded in both directions (in the sense of Proposition \ref{Ekedahl-Nakayama}).
Assume we are in the situation of (1). Then $H^i(R_1\otimes_R^L C)^j= 0$ for all $i\in \Z$ and $j\le a$.
Choose $i_0\in \Z$ and $j_0\le a$ and set $s:=i_0$ and $r:=j_0+i_0$. Then $H^i(R_1\otimes_R^L C)^j= 0$ for all $(i,j)$ as in \eqref{Ekedahl-Nakayama1}.
Thus $H^{i_0}(R_n\otimes_R^L C)^{j_0}=0$ for all $n$.
Therefore $R_n\otimes^L_R\varphi$ is an isomorphism for all $n$ in degree $\le a$, which gives (1).
Now assume we are in the situation of (2). Then $H^i(R_1\otimes_R^L C)^j=0$ for all $i\in \Z$ and $j\ge a$. Choose $i_0\in \Z$ and $j_0\ge a$
and set $s:=i_0$ and $r:= i_0+j_0$. Then $H^i(R_1\otimes_R^L C)^j=0$ for $(i,j)$ as in \eqref{Ekedahl-Nakayama2}.
Thus $H^{i_0}(R_n\otimes^L_R C)^{j_0}=0$, which implies (2) as above.
\end{proof}

\begin{lemma}\label{Cousin-and-Rn}
 Let $X$ be a smooth scheme and $E(W_\bullet\Omega_X)$ the Cousin complex of $W_\bullet\Omega$ with respect to the codimension filtration on $X$ (see \ref{CousinWitt}).
Then there is a natural commutative diagram of isomorphisms in $D^b(X, W_n[d])$
\[\xymatrix{ R_n\otimes^L_R W\Omega_X \ar[r]^-\simeq\ar[d]_\simeq    &   R_n\otimes^L_R \varprojlim E(W_\bullet\Omega_X)\ar[d]^\simeq\\
               W_n\Omega_X\ar[r]^\simeq & E(W_n\Omega_X).  }\]
\end{lemma}

\begin{proof}
 There is an obvious map 
\[R_n\otimes_R \varprojlim E(W_\bullet\Omega_X)= \varprojlim E(W_\bullet\Omega_X)/(V^n+dV(\varprojlim E(W_\bullet\Omega_X)))\to E(W_n\Omega),\]
yielding a morphism $R_n\otimes^L_R \varprojlim E(W_\bullet\Omega_X)\to E(W_n\Omega)$, which clearly fits into a diagram as in the statement.
By Lemma \ref{drw-is-CM}, $E(W_\bullet\Omega_X)$ is a flasque resolution of $W_\bullet\Omega_X$ and thus, by Lemma \ref{properties-of-flasque}, (4),
$\varprojlim E(W_\bullet\Omega_X)$ is a resolution of $W\Omega_X$. Hence the two horizontal arrows in the diagram are isomorphisms, and by \ref{Recall-CDR}(1) the vertical arrow on the left is an isomorphism. 
This proves the claim. 
\end{proof}

\subsubsection{}\label{pb-Rn} Let $f: X\to Y$ be a morphism between $k$-schemes. Then we have the pullback in $D^b(\drw_{Y})$ (see Definition \ref{2.1.2})
                 \[f^*: W_\bullet\Omega_Y\to Rf_* W_\bullet\Omega_X.\]
                 Using \eqref{R-exact-seq} and \ref{Recall-CDR}, (1), (2) one checks 
                      \[f^*_n= R_n\otimes_R^L(\phi(R\varprojlim f^*)) \quad \text{in } D^b(Y,W_n[d]),\]
                   where $\phi$ is the forgetful functor from \eqref{drw-to-R}, $f^*_n$ denotes the image of $f^*$ under the projection from $D^b(\drw_Y)$ to $ D^b(Y, W_n[d])$ 
and 
                     $R\varprojlim: D^b(\drw_Y)\to D^b(\cdrw_Y)$ (see Proposition \ref{derived-functors-exist}.)
\subsubsection{}\label{pf-Rn} Let $f: X\to Y$ be a morphism of pure relative dimension $r$ between smooth schemes and let $Z\subset X$ be a closed subscheme, such that $f_{|Z}$ is proper.
                 Then we have the pushforward in $D^b(\drw_Y)$ (see \eqref{derived-pf})
                  \[f_*: Rf_*R\ul{\Gamma}_Z W_\bullet\Omega_X\to W_\bullet\Omega_Y(-r)[-r].\]
                  Using \eqref{2.1.6.1} and Lemma \ref{Cousin-and-Rn} one checks
                      \[f_{*,n}=R_n\otimes_R^L(\phi(R\varprojlim f_*)) \quad \text{in } D^b(Y,W_n[d]).\]

\subsubsection{}\label{cup-Rn} Let $X$ be smooth equidimensional scheme and $Z\subset X$ an integral closed subscheme of codimension $c$. Then we have the cup-product with $cl([Z])$ (see \ref{2.6.2}) in $D^b(\drw_X)$ 
                 \[(-)\cup cl([Z]): W_\bullet\Omega_X\xr{\eqref{localcup}} \sH^c_{Z}(W_\bullet\Omega)(c)\xr{\eqref{supp-sheaf-to-derived-cat}} R\ul{\Gamma}_Z(W_\bullet\Omega)(c)[c].\]
                  Notice that in \ref{local-cup-product} we defined this maps only in the limit, but they can clearly be constructed on each level and the limit of the above gives \ref{local-cup-product}.
                  Also notice that the first map is given by $\alpha\mapsto (-1)^{c\cdot\deg \alpha} \alpha\cdot cl([Z])$, which is  a morphism of de Rham-Witt systems 
                  (since $F( cl([Z]))=\pi(cl([Z]))$ and $d(cl([Z]))=0$.)
                  Using \eqref{R-exact-seq} and \ref{Recall-CDR}, (1), (3) one checks
                   \[ (-\cup cl_n([Z]))= R_n\otimes^L_R \phi(R\varprojlim( (-)\cup cl([Z]))) \quad \text{in } D^b(X, W_n[d]).\]

\subsubsection{}\label{Corr-R} Let $S$ be a $k$-scheme, $f:X\to S$ and $g: Y\to S$ be two $S$-schemes which are smooth over $k$ and of the same pure dimension $N$.
                 Let $Z\subset X\times_S Y$ be a closed integral subscheme of dimension $N$, which is proper over $Y$. Then 
                   we define the morphism
                     \[\sR([Z]/S): Rf_*W\Omega_X\to Rg_*W\Omega_Y\quad\text{in } D^b(S,R)\]
                   as the composition 
                   \begin{align*}
                     Rf_*W\Omega_X &\xr{\pr_1^*}  Rf_*R\pr_{1*}W\Omega_{X\times Y}\\
                                   &\xr{\cup cl([Z])}  Rf_*R\pr_{1*}R\ul{\Gamma}_Z W\Omega_{X\times Y}(N)[N]\\
                                   &\xr{\simeq, \, Z\subset X\times_S Y}   Rg_*R\pr_{2*}R\ul{\Gamma}_Z W\Omega_{X\times Y}(N)[N]\\
                                   &\xr{pr_{2*}}  Rg_*W\Omega_{Y}.
                   \end{align*}
                   Here we simply write $\pr_{1}^*$ instead of $\phi(R\varprojlim(\pr_1^*))$ etc. Notice that the third arrow does not exist in $D^b(\cdrw_{S})$, since it is not compatible
                    with the $W\sO_S$-action. By \ref{Recall-CDR} and \ref{pb-Rn}, \ref{pf-Rn}, \ref{cup-Rn}, $R_n\otimes^L_R \sR([Z]/S)$ equals the following composition in $D^b(S, W_n[d])$, for all $n\ge 1$
                    \begin{align}
                     Rf_*W_n\Omega_X &\xr{\pr_{1,n}^*}  Rf_*R\pr_{1*}W_n\Omega_{X\times Y}\\
                                   &\xr{\cup cl_n([Z])}  Rf_*R\pr_{1*}R\ul{\Gamma}_Z W_n\Omega_{X\times Y}(N)[N]\nonumber\\
                                   &\xr{\simeq, \, Z\subset X\times_S Y}   Rg_*R\pr_{2*}R\ul{\Gamma}_Z W_n\Omega_{X\times Y}(N)[N]\nonumber\\
                                   &\xr{pr_{2,n*}}  Rg_*W_n\Omega_{Y}\nonumber.
                    \end{align}
                  This also shows that
          \eq{Corr-R1}{\sR([Z]/S)= R\varprojlim(R_n\otimes^L_R \sR([Z]/S)).}
\begin{lemma}\label{Corr-R-gives-Corr-H}
In the situation of \ref{Corr-R} we have
\[\oplus_i H^i(\sR([Z]/S))= \hat{\sH}([Z]/S): \hat{\sH}(X/S)\to \hat{\sH}(Y/S),\]
in particular $H^i(\sR([Z]/S))$ is $W\sO_S$-linear. (See Proposition \ref{proposition-C_S-to-dRW_S} for the notation.)
Further,
\[\oplus_i H^i(R_1\otimes^L_R \sR([Z]/S)): \oplus_i R^if_*\Omega_X\to \oplus_i R^i g_*\Omega_Y\]
equals the the morphism $\rho_H(Z/S)$ from \cite[Prop. 3.2.4]{CR}; in particular it is $\sO_S$-linear.
\end{lemma}

\begin{proof}
 By Lemma \ref{comparison-global-local-cup} (and its proof for the level $n=1$ case) the global cup product
$(-)\cup \hat{c}l([Z]): H^i(X\times Y, W\Omega_{X\times Y})\to H^{i+N}_Z(X\times Y, W\Omega_{X\times Y}(N))$ 
 (resp. $(-)\cup cl_1([Z]): H^i(X\times Y, \Omega_{X\times Y})\to H^{i+N}_Z(X\times Y, \Omega_{X\times Y}(N))$)
is given by applying $H^i(X\times Y,R\varprojlim (-))$ to the  cup product of \ref{cup-Rn} (resp. 
applying $H^i(X\times Y, R_1\otimes^L_R (-))$ to \ref{cup-Rn} ).
Now the lemma follows from going through the definitions (cf. also the proof of Proposition \ref{correspondences-and-change-of-basis} and of \cite[Lem. 4.1.3]{CR}.) 
\end{proof}

\begin{thm}\label{torison1}
 Let $S$ be a $k$-scheme and $f: X\to S$ and $g: Y\to S$ be two $S$-schemes which are integral and smooth over $k$ and have dimension $N$.
Assume $X$ and $Y$ are properly birational over $S$, i.e. there exists a closed integral subscheme $Z\subset X\times_S Y$, such that the projections
  $Z\to X$ and $Z\to Y$ are proper and birational.
Then $R_n\otimes^L_R \sR([Z]/S)$ (see \ref{Corr-R}) induces isomorphisms in $D^b(S, W_n)$, for all $n\ge 1$
\[Rf_*W_n\OO_X\cong Rg_*W_n\OO_Y,\quad Rf_*W_n\Omega_X^N\cong Rg_*W_n\Omega^N_Y.\]
Therefore $\sR([Z]/S)=R\varprojlim(R_n\otimes^L_R \sR([Z]/S))$ induces isomorphisms in $D^b(S,R)$
\[Rf_*W\OO_X\cong Rg_*W\OO_Y,\quad Rf_*W\Omega_X^N\cong Rg_*W\Omega^N_Y.\]
Taking cohomology we obtain isomorphisms of $W\sO_S$-modules which are compatible with Frobenius and Verschiebung 
\[R^if_*W\sO_X\cong R^ig_*W\sO_Y, \quad R^i f_*W\Omega^N_X\cong R^i g_* W\Omega^N_Y,\quad \text{for all }i\ge 0.\] 
\end{thm}
\begin{proof}
 By Ekedahl's Nakayama Lemma (Corollary \ref{cor-Ekedahl-Nakayama}) it suffices, to show that 
\[H^i(R_1\otimes_R^L \sR([Z]/S)): R^i f_*\Omega_X\to R^i g_*\Omega_Y\]
is an isomorphism for all $i\ge 0$. This follows from Lemma \ref{Corr-R-gives-Corr-H} and \cite[Thm. 3.2.6]{CR}. The second statement follows again from 
Lemma \ref{Corr-R-gives-Corr-H}.
\end{proof}

\begin{remark}
Notice that we do not need to assume here that any of the schemes $X$, $Y$ or $S$ are quasi-projective. We had to assume it in section 2 and 3
since, we can only prove the compatibility of $\sH(-/S)$ with composition of correspondences, in the quasi-projective case. But in the argument
above we only need that the maps exist in $D^b(S, R)$ and can then reduce to the result of \cite{CR}, where no quasi-projectiveness assumption is needed.
\end{remark}

Modulo torsion the $W\sO$-part of the following corollary was proved by Ekedahl in \cite{E83}.

\begin{corollary}\label{cor-torsion1}
Let $X$ and $Y$ be two smooth and proper $k$-schemes, which are birational and of pure dimension $N$. Then there are isomorphisms of $W[F, V]$-modules
\[H^i(X, W\sO_X)\cong H^i(Y, W\sO_Y), \quad H^i(X, W\Omega^N_X)\cong H^i(Y, W\Omega^N_Y), \quad \text{for all } i\ge 0\]
and also for all $n$, isomorphisms of $W_n$-modules
\[H^i(X, W_n\sO_X)\cong H^i(Y, W_n\sO_Y), \quad H^i(X, W_n\Omega^N_X)\cong H^i(Y, W_n\Omega^N_Y),\quad \text{for all } i\ge 0.\]
\end{corollary}

In the case where $X$ and $Y$ are tame finite quotients (see Definition \ref{finite-quotients}) we have no map like $\sR([Z]/S)$. This is why we have to assume
that there exists a morphism in this case:
 
\begin{thm}\label{torsion2}
 Let $f: X\to Y$ be a proper and birational $k$-morphism between two tame finite quotients. Then we have isomorphisms
\[f^*: W\sO_Y\xr{\simeq} Rf_*W\sO_X, \quad Rf_*W\omega_X\cong f_*W\omega_X[0]\xr{\simeq,\, f_*} W\omega_Y,\]
where $W\omega$ is defined in Definition \ref{definition-Witt-canonical} and $f_*$ is the pushforward from Proposition \ref{properties-Witt-canonical}, (6).
There are also corresponding isomorphisms on each finite level.
\end{thm}

\begin{proof}
By \cite[Cor. 4.3.3]{CR} we have isomorphisms
\[f^*: \sO_Y\xr{\simeq} Rf_*\sO_X, \quad Rf_*\omega_X\cong f_*\omega_X[0]\xr{\simeq,\, f_*} \omega_Y.\]
By \cite[Prop. 5.7]{KM} tame finite quotients are CM. 
Now the statement follows by induction from the two exact sequences (where $X$ is any pure dimensional CM scheme)
\[0\to W_{n-1}\sO_X\xr{V} W_n\sO_X\to \sO_X\to 0\]
and (see Proposition \ref{properties-Witt-canonical}, (7))
\[0\to W_{n-1}\omega_X\xr{\ul{p}} W_n\omega_X\xr{F^{n-1}} \omega_X\to 0.\]
\end{proof}

Notice that the $W\sO$-part of the theorem is a direct consequence of \cite[Cor. 4.3.3]{CR} and does not need any of the techniques developed in this paper.

\begin{corollary}\label{cor-torsion2}
 In the situation of Theorem \ref{torsion2} we have isomorphisms of $W[V,F]$-modules
\[H^i(X, W\sO_X)\cong H^i(Y, W\sO_Y),\quad H^i(X, W\omega_X)\cong H^i(Y, W\omega_Y), \quad \text{for all }i \ge 0.\]
There are also corresponding isomorphisms on each finite level.
\end{corollary}


\end{document}